\documentclass[12pt,PhD,twoside]{muthesis}
\usepackage{amsmath,amssymb, mathrsfs}
\usepackage{graphicx,psfrag}
\usepackage{amsthm, tikz}
\usepackage{extpfeil}

\newtheorem*{lem}{Lemma}
\newtheorem*{lemA}{Lemma A}
\newtheorem*{lemB}{Lemma B}

\newtheorem*{thm}{Theorem}
\newtheorem*{prop}{Proposition}
\newtheorem*{rem}{Remark}
\newtheorem*{cor}{Corollary}

\newcommand{\C}{\mathbb{C}}
\newcommand{\Z}{\mathbb{Z}}
\newcommand{\N}{\mathbb{N}}
\newcommand{\Q}{\mathbb{Q}}
\newcommand{\K}{\mathbb{K}}

\newcommand{\tG}{Q}

\renewcommand{\H}{H}
\newcommand{\tR}{R}
\newcommand{\tK}{\widetilde{K}}

\newcommand{\g}{\mathfrak{g}}
\newcommand{\hh}{\mathfrak{h}}
\newcommand{\kk}{\mathfrak{k}}
\newcommand{\n}{\mathfrak{n}}

\newcommand{\gl}{\mathfrak{gl}}
\newcommand{\h}{\mathfrak{k}}
\newcommand{\li}{\mathfrak{l}}
\newcommand{\pp}{\mathfrak{p}}
\newcommand{\tg}{\mathfrak{q}}
\renewcommand{\sl}{\mathfrak{sl}}
\newcommand{\mm}{\mathfrak{m}}
\newcommand{\tr}{\mathfrak{r}}
\newcommand{\z}{\mathfrak{z}}

\newcommand{\rank}{\textnormal{rank}}
\newcommand{\Specm}{\textnormal{Specm}}
\newcommand{\ZZ}{\mathcal{Z}}
\newcommand{\ind}{\textnormal{ind}}
\newcommand{\End}{\text{End}}

\newcommand{\ad}{\text{ad}}

\newcommand{\Hom}{\text{Hom}}

\newcommand{\Ad}{\textnormal{Ad}}
\newcommand{\gr}{\textnormal{gr}}
\newcommand{\grk}{\textnormal{gr}_{\sf K}}
\newcommand{\grko}{\textnormal{gr}^0_{\sf K}\,}
\newcommand{\spn}{\textnormal{span}}
\newcommand{\Lie}{\textnormal{Lie}}

\newcommand{\codim}{\textnormal{codim}}

\newcommand{\sgn}{\textnormal{sgn}}
\newcommand{\chr}{\textnormal{char}}
\newcommand{\Ind}{\textnormal{Ind}}

\newcommand{\cm}{\textnormal{-mod}}
\newcommand{\ab}{\textnormal{ab}}
\newcommand{\reg}{\textnormal{reg}}

\renewcommand{\th}{\textnormal{th}}
\newcommand{\Prim}{\textnormal{Prim\,}}
\newcommand{\Primo}{\textnormal{Prim}}
\newcommand{\mult}{\textnormal{mult}}
\newcommand{\Ann}{\textnormal{Ann}}
\newcommand{\Mat}{\textnormal{Mat}}
\newcommand{\Th}{\Theta}
\newcommand{\Comp}{\textnormal{Comp}}

\renewcommand{\dim}{\textnormal{dim}}

\newcommand{\Zas}{\mathcal{Z}}
\newcommand{\X}{\mathcal{X}}
\newcommand{\CC}{\mathcal{C}}
\newcommand{\m}{\mathfrak{m}}
\newcommand{\HH}{\mathfrak{H}}
\newcommand{\NN}{\mathfrak{N}}
\newcommand{\PP}{\mathcal{P}_\epsilon}
\newcommand{\Pp}{\mathcal{P}}
\newcommand{\ii}{\textnormal{\textbf{i}}}
\newcommand{\jj}{\textnormal{\textbf{j}}}
\newcommand{\Oo}{\mathcal{O}}
\renewcommand{\SS}{\mathcal{S}}
\newcommand{\Sf}{\mathfrak{S}}
\newcommand{\FF}{\mathcal{F}}
\newcommand{\EE}{\mathcal{E}}

\newcommand{\V}{\mathcal{W}}
\newcommand{\cG}{\mathcal{G}}
\newcommand{\Ni}{\mathcal{N}}
\newcommand{\BB}{\mathcal{B}}
\newcommand{\II}{\mathcal{D}}
\newcommand{\LL}{\mathcal{L}}
\newcommand{\AC}{\mathcal{AC}}
\newcommand{\VA}{\mathcal{VA}}
\newcommand{\MF}{\mathcal{MF}}

\newcommand{\ra}{\rightarrow}

\begin{document}
\title{Centralisers in Classical Lie Algebras}
\author{Lewis William Topley}
\school{Mathematics}
\faculty{Engineering and Physical Sciences}
\def\wordcount{xxxxx}

\tablespagefalse

\figurespagefalse


\beforeabstract

In this thesis we shall discuss some properties of centralisers in classical Lie algebas and related structures. Let $\K$ be an
algebraically closed field of characteristic $p \geq 0$. Let $G$ be a simple algebraic group over $\K$. We shall
denote by $\g = \Lie(G)$ the Lie algebra of $G$, and for $x \in \g$ denote by $\g_x$ the centraliser. Our results
follow three distinct but related themes: the modular representation theory of centralisers, the sheets of simple Lie
algebras and the representation theory of finite $W$-algebras and enveloping algebras.

When $G$ is of type $\sf A$ or $\sf C$ and $p > 0$ is a good prime for $G$, we show that the invariant algebras $S(\g_x)^{G_x}$
and $U(\g_x)^{G_x}$ are polynomial algebras on $\rank \, \g$ generators, that the algebra $S(\g_x)^{\g_x}$ is generated
by $S(\g_x)^p$ and $S(\g_x)^{G_x}$, whilst $U(\g_x)^{\g_x}$ is generated by $U(\g_x)^{G_x}$ and the $p$-centre, generalising a classical theorem of Veldkamp. We apply the latter result
to confirm the first Kac-Weisfeiler conjecture for $\g_x$, giving a precise upper bound for the dimensions of simple $U(\g_x)$-modules.
This allows us to characterise the smooth locus of the Zassenhaus variety in algebraic terms. These results correspond to
an article \cite{Top}, soon to appear in the Journal of Algebra.

The results of the next chapter are particular to the case $x$ nilpotent with $G$ connected of type $\sf{B}, \sf{C}$ or $\sf D$ in any characteristic good for $G$.
Our discussion is motivated by the theory of finite $W$-algebras which shall occupy our discussion in the final chapter, although we make several deductions of
independent interest.
We begin by describing a vector space decomposition for $[\g_x \g_x]$ which in turn allows us to give a formula for $\dim\, \g_x^\ab$
where $\g_x^\ab := \g_x / [\g_x \g_x]$. We then concoct a combinatorial parameterisation of the sheets of $\g$ containing $x$ and use it to classify the
nilpotent orbits lying in a unique sheet. We call these orbits \emph{non-singular}. Subsequently we give a formula for the maximal rank of sheets containing
$x$ and show that it coincides with $\dim\, \g_x^\ab$ if and only if $x$ is non-singular. The latter result is applied to show for any (not necessarily nilpotent) $x\in \g$ lying in a unique sheet, that the orthogonal
complement to $[\g_x \g_x]$ is the tangent space to the sheet, confirming a recent conjecture.

In the final chapter we set $p = 0$ and consider the finite $W$-algebra $U(\g,x)$, again with $G$ of type $\sf B, \sf{C}$ or $\sf D$. The one dimensional representations are parameterised by the
maximal spectrum of the maximal abelian quotient $\EE = \Specm \, U(\g, x)^\ab$ and we classify the nilpotent elements in classical types for which $\EE$ is isomorphic to an affine
space $\mathbb{A}^d_\K$: they are precisely the non-singular elements of the previous chapter. The component group $\Gamma$ acts naturally on $\EE$ and the fixed point space lies in bijective correspondence
with the set of primitive ideals of $U(\g)$ for which the multiplicity of the correspoding primitive quotient is one. We call them \emph{multiplicity free}. We show that this fixed point space is always
an affine space, and calculate its dimension. Finally we exploit Skryabin's equivalence to study parabolic induction of multiplicity free ideals.
In particular we show that every multiplicity free ideals whose associated variety is the closure of an induced orbit is itself induced from a completely
prime primitive ideals with nice properties, generalising a theorem of M{\oe}glin. The results of the final two chapters make up a part of a joint work with
Alexander Premet \cite{PTop}.

\afterabstract

\prefacesection{Acknowledgements}
I would like to offer special thanks to Professor Alexander Premet for his excellent guidance throughout this research. His mastery of the subject has been inspirational, and his instruction has always been incisive. I am grateful to Anne Moreau, Simon Goodwin and Oksana Yakimova for having taken an interest in my work, for reading early drafts and making helpful suggestions. I extend further thanks to Simon and Anne for inviting me to visit their universities, and for giving me the opportunity to present my work in seminars. I am also grateful to the organisers of all of the other conferences and meeting which I have attended, especially those that I have been fortunate to be invited to speak at. Finally I would like to thank Iain Gordon who first inspired my interest in algebra when I was an undergraduate.

I am especially grateful to my family who have been greatly supportive over the past few years.

\newpage
\vspace*{\fill}
\begingroup
\emph{\centering{This thesis is dedicated to the memory of\\}} 
\emph{\centering{Dominick Charles Hamilton\\}}
\emph{\centering{1954 - 2013\\}}
\endgroup
\vspace*{\fill}

\afterpreface

\chapter{Introduction}

\newcounter{parno}
\renewcommand{\theparno}{\thesection.\arabic{parno}}
\newcommand{\p}{\refstepcounter{parno}\noindent\textbf{\theparno .} \space} 

This introduction shall serve to present the notation and the mathematical objects which we shall discus in this thesis. It is by no means a comprehensive summary of the theory needed to understand the work. Less common notions will be covered in more detail. We assume that the reader has an understanding of the theories of algebraic groups, Lie algebras and nilpotent orbits, as exposited in \cite{Bor}, \cite{TY}, \cite{Dix}, \cite{Jac}, \cite{CM}, \cite{Jan}.

\section{Notations}\label{notations}
\setcounter{parno}{0}

\p Before we begin we shall introduce some standard notations which shall be used without exception. Natural numbers are denoted $\N$, we adopt the convention $0 \notin \N$ and $\N_0 = \N\cup \{0\}$. The integers, rationals and complex numbers are denoted by the usual ``blackboard'' letters $\Z, \Q, \C$. For $n \in \N$ the integers modulo $n$ shall be denoted $\Z_n$. Furthermore $\K$ shall always denote an algebraically closed field of characteristic $p\geq 0$.

\p Once a choice of field is made explicit, all algebraic varieties and vector spaces shall be constructed over that field until further notice. Unless otherwise stated the word group shall mean algebraic group and these shall be denoted by the upper case latin font $G, K, \tG$, and their respective Lie algebras by the corresponding lower case gothic letters $\g, \kk, \tg$. Toplogical notions, like open and continuous, shall always pertain to the Zariski toplogy.

\p Our point of departure will be a reductive group $\tG$ or a stabiliser $\tG_x$ in a reductive group, where $x \in \tg$. We assume throughout that $\chr(\K)$ is good for the underlying group. We remind the reader that:
\begin{itemize}
\item{$p = 2$ is bad for $\tG$ when $\tG$ has a factor not of type $\sf A$;}
\item{$p = 3$ is bad for $\tG$ when $\tG$ has a factor of exceptional type;}
\item{$p = 5$ is bad for $\tG$ when $\tG$ has a factor of type $\sf E_8$.}
\end{itemize}
Then $\chr(\K)$ is said to be good for $\tG$ if it is not a bad prime.

\p If our choice of group $\tG$ is clear then we shall denote the adjoint representations of $\tG$ and $\tg$ in $\tg$ by the standard letters $\Ad$ and $\ad$, and the coadjoint representations by $\Ad^\ast$ and $\ad^\ast$. Furthermore, for $x \in \tg$ we may write $\Oo_x = \Ad(\tG) x$ for the adjoint orbit of $x$. Once the notion of the partition of an orbit in a classical algebra has been introduced in Section~\ref{nilpotentorbits} we might also use the notation $\Oo_\lambda$ to denote an orbit with partition $\lambda$.

\p \label{genlinnots} As mentioned above, almost all of our work shall pertain to classical groups and the centralisers therein. As such, it will be convenient for us to adhere to specific notations for each of the classical groups: for the remainder of this thesis we let $N \in \N$, let $V$ be an $N$-dimensional vector space over $\K$, and let $G$ denote the general linear group $GL(V)$. When we refer to a group of type $\sf A$ we shall actually mean $G$ and the reader will notice that all of our results regarding $GL(V)$ may easily be translated over to $SL(V)$.

\p \label{classnots} Whenever we discuss groups of type $\sf B$, $\sf C$ or $\sf D$ we make the assumption that $\chr(\K) \neq 2$. Introduce a symmetric or anti-symmetric non-degenerate bilinear form $(\cdot, \cdot) : V \times V \ra \K$. Let $K$ denote the connected group of matrices preserving that form: $g \in K$ if and only if $(gu, gv) = (u,v)$ for all $u,v \in V$ and $\det(g) = 1$. We shall write $(u,v) = \epsilon (v,u)$ where $\epsilon = \pm1$. If $(\cdot, \cdot)$ is symmetric $K$ is a special orthogonal group, whilst if $(\cdot, \cdot)$ is anti-symmetric then $K$ is a symplectic group and $N = \dim\, V$ is even. The associated Lie algebra is $\kk$, the set of skew self-adjoint matrices with respect to $(\cdot, \cdot)$. In symbols an element $x\in \g$ lies in $\kk$ if and only if $(xu, v) = - (u, xv)$ for all $u, v \in V$.

\p \label{classnots2} For $x \in \h$ we let $x^\ast$ denote the adjoint of $x$ with respect to $(\cdot,\cdot)$. The map $\sigma : x \mapsto - x^\ast$ is an involution on $\g$ and the fixed point set is $\h$. The $-1$ eigenspace is a $K$-module which we denote by $\pp$. If we choose a basis for $V$ then the bilinear form is represented by a matrix $(u, v) = u^\top X v$. In this case $\sigma$ may be written explicitly as $\sigma(x) = -X^{-1} x^\top X$.

\p \label{rootnotation} Occasionally we shall need to call upon the classical theory of root systems. If $\tG$ is a reductive algebraic group over $\K$ with maximal torus $H$ then the associated root system shall be denoted $\Phi$ with choice of simple roots $\Pi$, and $\Phi^+$ ($\Phi^-$) the set of positive (negative) roots determined by $\Pi$. The ``usual" ordering of the roots shall refer to the labelling convention adopted by Bourbaki and may be seen in \cite[Chapter~6]{Bour}. The root space associated to $\alpha \in \Phi$ shall be denoted $\tg^\alpha$. The Killing form shall always be denoted $\kappa$.

\section{Jordan decomposition and Levi subgroups}\label{jordandecomposition}
\setcounter{parno}{0}
\p \label{jordan1}For each endomorphism $x$ of $V$ the Jordan-Chevalley decomposition theorem states that there exists a unique pair of commuting endomorphisms $x_n$ and $x_s$ such that $x_n$ is nilpotent, $x_s$ is semi-simple and $x = x_n + x_s$. We call $x_n$ ($x_s$) the nilpotent (semi-simple) part of $x$. Now if $\tg \subseteq \g$ is a linear Lie algebra we may apply this decomposition to obtain elements $x_n, x_s \in \g$. In general the nilpotent and semi-simple parts do not lie in $\tg$, however when $\tg$ is simple they do. All of our Lie algebras shall be linear and we shall call an element nilpotent if it is so as an endomorphism of $V$.

\p\label{jordanreduction} As the title of this thesis suggests we shall be discussing centralisers $\tg_x$ where $x \in \tg$ and $\tg$ is classical. We may simplify the discussion by reducing to the nilpotent case as follows. The proof of the Jordan-Chevalley decomposition theorem shows that the endomorphisms $x_n$ and $x_s$ are polynomials in $x$. Therefore an element $y \in \tg$ commutes with $x$ if and only if it commutes with both $x_n$ and $x_s$. We infer that $\tg_x = (\tg_{x_s})_{x_n}$. It is well-known that when $\tg$ is classical the centraliser $\tg_{x_s}$ is a direct sum of classical Lie algebras. We shall make this remark precise in \ref{levistruct}. For now it suffices to observe that many questions about the centraliser $\tg_x$ may be reduced to questions about the centraliser $\tg_{x_n}$.

\p \label{leviint} Before discussing nilpotent orbits we should take a moment to consider semisimple elements and their centralisers. We shall first approach these algebras from a more general perspective. Let $\tG$ be an algebraic group with unipotent radical $U$. There is an exact sequence $$1\ra U \ra \tG \ra \tG/U \ra 1$$ and it is pertinent to ask when this sequence splits, ie. whether there exists a reductive group $L \subseteq \tG$ isomorphic to $\tG/U$ such that $\tG = L \ltimes U$. Over fields of characteristic zero such a Levi factor always exists, however over fields of positive characteristic it is a very subtle problem. There is one situation in which we can give an affirmative answer in all cases. Suppose that $\tG$ is a reductive group over a field of good characteristic and that $P$ is a parabolic subgroup with unipotent radical $U$. Then there exists a Levi decomposition $P = L \ltimes U$. We call $L$ a \emph{Levi factor of $P$}, and the groups obtained in this way are called the \emph{Levi 
subgroups of $\tG$}. The Levi subgroups of the simple groups are well understood. Let $\tG$, $\H$, $\Phi$, $\Phi^+$ and $\Pi$ be as per \ref{rootnotation}, with characteristic good for $\tG$. Every parabolic is $\tG$-conjugate to a standard parabolic and every standard parabolic is obtained in the following way. For $\widehat{\Pi} \subseteq \Pi$ the standard parabolic $P({\widehat{\Pi}})$ of type $\widehat{\Pi}$ is the closed subgroup of $\tG$ with Lie algebra generated by the subspaces $\hh$ and $\tg^\alpha$ with $\alpha \in \Pi \cup (-\widehat{\Pi})$. The Levi factor associated to $P({\widehat{\Pi}})$ is the closed subgroup $L({\widehat{\Pi}}) \subseteq \tG$ with Lie algebra generated by $\hh$ and $\tg^{\alpha}$ with $\alpha \in \pm \widehat{\Pi}$. This Lie algebra shall be denoted $\li({\widehat{\Pi}})$ and is called the \emph{standard Levi subalgebra of $\tg$ of type $\widehat{\Pi}$}.

\p \label{levicent} As mentioned in \ref{leviint}, the Levi subgroups bear a relation to centralisers of semi-simple elements. That relation is now easy to explain. Suppose $h \in \tg$ is semi-simple. Then we may assume that $h$ lies in our chosen maximal torus $\hh$. After conjugating by some element of the Weyl group of $(\tg, \hh)$ we may even assume that the image of $h$ under the Killing isomorphism lies in the closure of the fundamental Weyl chamber. Now if we define $$\widehat{\Pi} = \{ \alpha \in \Pi : \alpha(h) = 0\}$$ then by the previous paragraph we may construct the standard Levi subalgebra of type $\widehat{\Pi}$ and it is easy to see that every element of $\li({\widehat{\Pi}})$ centralises $h$. If fact the converse is true and we have $$\li({\widehat{\Pi}}) = \tg_h.$$

\p \label{levistruct} The single theme which unifies these two perspectives on Levi subalgebras is that we may use inductive principles to reduce either to the case of nilpotent centralisers or to smaller classical subalgebras. In any case, we shall need to understand the structure of the standard Levi subalgebra $\li = \li(\widehat{\Pi})$ when $\h$ is classical (notations of \ref{classnots}). The algebra $\li$ is a direct sum $\mathfrak{c} \oplus [\li \li]$ where the centre $\mathfrak{c}$ is spanned by vectors in $\hh$ dual to the elements of $\Pi \backslash \widehat{\Pi}$. The Dynkin diagram of the semi-simple algebra $[\li \li]$ includes into the Dynkin digram of $\h$ just as $\widehat{\Pi}$ includes into $\Pi$. The dimension of $\mathfrak{c}$ is greater than or equal to the number of connected components of type $\sf A$ in the diagram of $\li$ and so we may supplement each special linear factor of $[\li \li]$ with a one dimensional centre to make general linear factors. The remaining centre may be 
expressed as a product of $\gl_1$'s.  This suggests a nice standard form for the isomorphism types of the Levi subalgebras of $\h$.
\begin{lem}
Every Levi subalgebra of $\h$ is isomorphic to something of the form
$$\gl_\ii \times \mm := \gl_{i_1} \times \cdots \times \gl_{i_l} \times \mm$$
where $\ii = (i_1,...,i_l)$ is a sequence of integers fulfilling $\sum_j i_j \leq \rank\, \h$ and $\mm$ is a simple algebra of the
same type as $\h$ with natural representation of dimension $R_\ii := N - \sum_j 2i_j$ (under the additional restriction that
$R_\ii \neq 2$ when $\h$ has type $\sf D$).
\end{lem}
\noindent In Section~\ref{conjlev} we shall refine this lemma to offer a classification of $\tG$-conjugacy classes of Levi subalgebras in classical types in terms of certain sequences of integers.

\section{Nilpotent orbits}\label{nilpotentorbits}
\setcounter{parno}{0}
\p Recall that $V = \K^N$. Let us consider a linear algebraic group $\tG \subseteq G = GL(V)$ acting on its Lie algebra $\tg \subseteq \g$ by conjugation. A \emph{nilpotent orbit} is an abbreviation for a $\tG$-orbit consisting of nilpotent elements of $\tg$. If there is more than one group under discussion we might also call these \emph{nilpotent $\tG$-orbits}. When $\tG$ is semisimple, the set of nilpotent elements is called the \emph{nilpotent cone of $\tg$} and is denoted $\Ni(\tg)$. It is well known to coincide with the zero locus of the ideal of $\K[\tg]^\tG_+\unlhd \K[\tg]$ generated by non-constant $\tG$-invariant functions; see \cite[\S 7]{Jan} for example. First we review the classification of nilpotent orbits for classical groups.

\p An ordered partition of $N$ is a finite non-increasing sequence summing to $N$; in symbols, $\lambda = (\lambda_1,...,\lambda_n)$ with $\lambda_1 \geq \cdots \geq \lambda_n$ and $\sum \lambda_i = N$. We denote the set of partitions of $N$ by $\Pp(N)$. Recall that partitions can be visually realised as Young tableau. A nilpotent $N\times N$ matrix is called regular if it is $G$-conjugate to one of the form 
$$J_N = \left( \begin{array}{cccccc}
0 & 1 & 0 &\cdots & 0 & 0 \\
0 & 0 & 1 &\cdots & 0 & 0 \\
0 & 0 & 0 & \cdots & 0 & 0 \\
\vdots & \vdots & \vdots & \ddots & \vdots & \vdots \\
0 & 0 & 0 & \cdots & 0 & 1 \\
0 & 0 & 0 & \cdots & 0 & 0
\end{array} \right) .$$

\bigskip

We say that a nilpotent element of $\g$ is in Jordan normal form when it is in block diagonal form diag$(J_{\lambda_1},..., J_{\lambda_n})$ for some $n$ with $\lambda_1 \geq \cdots \geq \lambda_n$. It is clear that any such matrix is nilpotent and gives rise to an ordered partition of $N$ via $\lambda := (\lambda_1,...,\lambda_n)$. Furthermore, it is a theorem that any nilpotent $N \times N$ matrix is $G$-conjugate to a matrix in Jordan normal form. In this way we obtain a bijection 
\begin {eqnarray}\label{GLnorbits}
\Ni(\g)/G \longleftrightarrow \Pp(N).
\end{eqnarray}
If $\Oo_e$ has partition $\lambda$ then $V$ may be decomposed into minimal $e$-stable subspaces $V = \oplus_{i=1}^n V[i]$, and shall call these $V[i]$ the \emph{Jordan block spaces of $e$}. The decomposition is non-unique in general, however when we talk about $V[i]$ it will be implicit that a choice of decomposition has been made. By definition we have $\dim \, V[i] = \lambda_i$.

\p Now let us turn our attention to other classical cases. It is not hard to see that two endomorphisms of $V$ are $SL(V)$-conjugate if and only if they are $GL(V)$-conjugate (here we use the fact that $\K$ is algebraically closed). It follows immediately that the classification of nilpotent $GL(V)$-orbits in (\ref{GLnorbits}) extends to $SL(V)$.

\p \label{classsurj}It remains to classify the nilpotent orbits for types $\sf B$, $\sf C$ and $\sf D$. We assume that $\chr(\K) \neq 2$. The classification of nilpotent $K$-orbits is again given in terms of partitions.
\begin{thm}\label{nilorbits} Let $K$ be a classical group not of type $\sf A$.
\begin{enumerate} 
\item{If $K$ is symplectic then a partition $\lambda$ corresponds to a nilpotent $K$-orbit if and only if every odd part occurs with even multiplicity.}
\item{If $K$ is orthogonal then a partition $\lambda$ corresponds to a nilpotent $K$-orbit if and only if every even part occurs with even multiplicity.}
\end{enumerate}
\end{thm}
\noindent We denote the set of partitions associated to nilpotent $K$-orbits by $\PP(N)$. The above theorem tells us that there is a surjection
\begin{eqnarray}\label{thesurjection}
\Ni(\h)/K \twoheadrightarrow \PP(N).
\end{eqnarray}
\noindent The fibres of this surjection are singletons unless $K$ is of type $\sf D$ and $\lambda$ is \emph{very even}, in which case the fibre contains two orbits which are permuted by the unique outer automorphism of $\h$ coming from the graph automorphism of the associated Dynkin diagram. The very even partitions are those for which all parts are even. They shall play a role in our results on the sheets of Lie algebra. For a survey on the classification of nilpotent $K$-orbits when $\chr(\K) = 2$ see \cite[5.11]{Ca}.

\p The above theorem is easy to read but it does not quite give the whole story. The proof of the ``if'' part takes $e \in \g$ satisfying the parity conditions on its parts and shows that there exists a symmetric or skew bilinear form with respect to which $e$ is skew-self adjoint; see \cite[1.11]{Jan} for example. Such a choice of form necessarily makes an explicit pairing between certain Jordan block spaces. Since the bilinear form determines the algebra $\h$, this pairing is far from arbitrary and plays a role in describing a basis for $\h_e$. The following may be seen as a more technical version of the previous theorem. Recall that $V$ decomposes as $\oplus_{i=1}^n V[i]$ into Jordan block spaces for $e$.
\begin{lem}\label{nilpotents}
Suppose $e \in \Ni(\g)$ has partition $\lambda = (\lambda_1,...,\lambda_n)$. Then $\Ad(G)e$ intersects $\Ni(\h)$ if and only if there exists an involution $i \mapsto i'$ on the set $\{1,...,n\}$ such that:
\begin{enumerate}
\item{$\lambda_i = \lambda_{i'}$ for all $i = 1,...,n$;}

\smallskip

\item{$(V[i], V[j]) = 0$ if $i \neq j'$;}

\smallskip

\item{$i=i'$ if and only if $\epsilon(-1)^{\lambda_i} = -1$.}
\end{enumerate}
\end{lem}
\noindent
Without disturbing the ordering $\lambda_1 \geq \cdots \geq \lambda_n$ we may renumber the Jordan block spaces in such a way that $$i'\in\{i-1,i,i+1\}\ \ \mbox{for all }\ 1\le i\le n.$$ As an immediate consequence of this convention we have that $j'>i'$ whenever $1\le i<j\le n$ and $j\ne i'$. The involution is uniquely determined by $(\cdot, \cdot)$ and $e \in \h$.

\p Before we continue we shall describe another approach to the classification of nilpotent orbits. The maps given in (\ref{GLnorbits}) and (\ref{thesurjection}) are very explicit and may be described in an elementary way using linear algebra. Those methods do not extend well to the same problem for exceptional Lie algebras, where something more sophisticated is required. Bala-Carter theory provides an almost characteristic free parameterisation of nilpotent $\tG$-orbits when $\tG$ is an arbitrary reductive group. Let $\tG$ be such a group. A distinguished nilpotent element $x \in \tg$ is one such that every torus in $\tG_x$ is contained in the centre of $\tG$. This property is obviously preserved under conjugation, and so we may define a distinguished nilpotent orbit to be a nilpotent $\tG$-orbit consisting of distinguished elements. A weak form of the Bala-Carter theorem states that the nilpotent orbits are in bijective correspondence with $\tG$-orbits of pairs $(\li, \Oo)$ where $\li$ is a Levi subgroup 
and $\Oo$ is a distinguished nilpotent orbit \cite[Proposition~4.7]{Jan}. They go on to classify the distinguished nilpotent orbits in terms of so called distinguished parabolic subgroups. Their original method works for $p = 0$ and $p \gg 0$, whilst the argument was extended to good characteristic by Pommering \cite{Pom1, Pom2}. The latter approach still relied on some case by case deductions, which were finally removed by the use of the Kempf-Rosseau theory of optimal tori in \cite{Pre1}.

\section{Induced nilpotent orbits}\label{inducedorbits}
\setcounter{parno}{0}
\p \label{inductionint} There is a powerful construction which allows us to obtain nilpotent $\tG$-orbits from nilpotent $L$-orbits, where $\tG$ is reductive and $L\subseteq \tG$ is a Levi subgroup. It is called induction of nilpotent orbits and was introduced by Lusztig and Spaltenstein in \cite{LS}. It will be vital to the study of the sheets of a Lie algebra in Chapter~\ref{derivedchapter} and finite $W$-algebras in Chapter~\ref{abelianquot}. Let $\tG$ be arbitrary reductive group and let $\pp \subseteq \tg$ be a parabolic subalgebra of $\tg$ with nilpotent radical $\n$ and Levi factor $\li$. Also let $\Oo \subseteq \li$ be a nilpotent $L$-orbit. The definition of an induced orbit is contained in the following.
\begin{thm}
There is a unique nilpotent $tG$-orbit, independent of our choice of parabolic $\mathfrak{p}$ containing $\li$ as a Levi factor, meeting $\Oo + \mathfrak{n}$ in a dense open subset. We denote this orbit by \emph{Ind}$_\li^\tg(\Oo)$. It has dimension $\dim\, \Oo + 2\dim\, \mathfrak{n}$.
\end{thm} 
\noindent The theorem and its proof may be seen as a generalisation of the construction of the regular nilpotent orbit or, more generally, of Richardson orbits (those induced from the zero orbit in some Levi subalgebra) given in \cite{Ric}.

\p Induction of nilpotent orbits enjoys many distinguished properties, the most basic of which can be encapsulated concisely in the following way: let $\tilde{p}(\tG)$ denote the set of all pairs $(\li, \Oo)$ where $\li\subseteq \tg$ is a Levi subalgebra and $\Oo\subseteq \li$ a nilpotent orbit. We let $\leq_\Ind$ be defined by $(\li_1, \Oo_1) \leq_\Ind (\li_2,\Oo_2)$ if $\li_1 \subseteq \li_2$ and $\Oo_2 = \Ind^{\li_2}_{\li_1}(\Oo_1)$. Then $(\tilde{p}(\tG), \leq_\Ind)$ is a poset. When we discuss induction in detail in Chapter~\ref{derivedchapter} it will be useful for the reader to consider the associated Hasse diagram. Other useful properties shall be introduced as they are required.

\p \label{rigidinst} In the third and fourth chapters we shall use induction in order to reduce certain statements about nilpotent orbits to the case where an orbit $\Oo$ cannot be induced. Such an orbit is called \emph{rigid}. The rigid orbits are classified and in classical types they have a nice combinatorial description due to Kempken and Spaltenstein, which we shall recall in Theorem~\ref{rigids}.

\section{The centraliser of a nilpotent element}\label{basisforthecent}
\renewcommand{\p}{\refstepcounter{parno}\noindent\textbf{\theparno .} \space} 
\setcounter{parno}{0}

\p Assume the notations of \ref{genlinnots} and \ref{classnots}. We wish to discuss the centralisers in the classical Lie algebra $\g$ or $\h$. As was explained in Section~\ref{jordanreduction}, we reduce our discussion to the nilpotent case using Levi subalgebras. Let $e \in \Ni(\h)$. In order to carry out explicit computations in $\h_e$ we shall introduce a basis involving parameters which depend on the partition associated to $e$. The purpose of this section is to introduce that basis and recall some basic facts about the stabiliser groups $G_e$ and $K_e$. Thanks to \cite[Theorems~2.5 \& 2.6]{Jan} we may identify $\g_e$ with $\Lie(G_e)$ and $\h_e$ with $\Lie(K_e)$.

\p First of all let us take a nilpotent $G$-orbit $\Oo$ with partition $\lambda = (\lambda_1,...,\lambda_n)$. Pick an element $e \in \Oo$.  There exist vectors $\{w_i : i = 1,...,n\}$ such that $\{e^s w_i : 1 \leq i \leq n, 0 \leq s < \lambda_i\}$ forms a basis for $V$. Let $\xi \in \mathfrak{g}_e$. Then $\xi (e^s w_i) = e^s (\xi w_i)$ showing that $\xi$ is determined by its action on the $w_i$. Define
\begin{eqnarray*}
\xi_i^{j,s} w_k = \left\{ \begin{array}{ll}
         e^sw_j & \mbox{ if $i=k$}\\
        0 & \mbox{ otherwise}\end{array} \right.
\end{eqnarray*}
Provided $\lambda_j > s \geq \lambda_j - \min(\lambda_i,\lambda_j)$ we may extend the action to $V = \spn\{e^s w_i\}$ by the requirement that $\xi_i^{j,s}$ is linear and centralises $e$. The following observation was made in \cite{Yak1}, for example.
\begin{lem} \label{gebasis}
The maps $\xi_i^{j,\lambda_j - 1 - s}$ with $1 \leq i,j \leq n$ and $0 \leq s < \min(\lambda_i, \lambda_j)$ form a basis for $\mathfrak{g}_e$.
\end{lem}
\noindent We shall adopt the convention that when the indexes $i,j$ and $s$ lie outside the prescribed ranges, the map $\xi_i^{j,\lambda_j -1 - s}$ is zero.

\p We define a basis for the dual space $\g_e^\ast$ in the usual manner. Let
\begin{eqnarray*}
(\xi_i^{j,s})^\ast(\xi_k^{l,r}) = \left\{ \begin{array}{ll}
         1 & \mbox{ if $i=k$, $j=l$, $s=r$}\\
        0 & \mbox{ otherwise}\end{array} \right.
\end{eqnarray*}

\p Our next aim is to describe a basis for $\h_e$. The following approach is implicit in \cite{Yak2}, however we shall recover the details for the reader's convenience. Before we write down a spanning set for $\h_e$ we must normalise the basis for $V$. Let $\{w_i\}$ be chosen so that $\{e^s w_i : 1 \leq i \leq n, 0 \leq s < \lambda_i\}$ is a basis for $V$ and fix $1 \leq i \leq n$, $0 < s$. Recall the involution $i \mapsto i'$ defined on $\{1,...,n\}$ in Lemma~\ref{nilpotents}. We have $(e^{\lambda_i-1} w_i, e^sw_{i'}) = (-1)^s (e^{\lambda_i - 1 + s}w_i, w_{i'})$ and $e^{\lambda_i - 1 + s}w_i = 0$ so $e^{\lambda_i-1}w_i$ is orthogonal to all $e^sw_{i'}$ with $s>0$. There is a (unique upto scalar) vector $v \in V[i]$ which is orthogonal to all $e^s w_{i'}$ for $s < \lambda_i - 1$. This $v$ does not lie in the image of $e$ for otherwise it would be othogonal to all of $V[i] + V[i']$. This is not possible since the restriction of $(\cdot,\cdot)$ to $V[i] + V[i']$ is non-degenerate thanks to part 2 of
Lemma~\ref{nilpotents}. It does no harm to replace $w_i$ by $v$ and normalise according to the rule
\begin{eqnarray*}
 (w_i, e^{\lambda_i - 1} w_{i'}) = 1 & \text{ whenever } i \leq i'
 \end{eqnarray*}
 With respect to this basis the matrix of the restriction of $(\cdot,\cdot)$ to $V[i] + V[i']$ is antidiagonal with entries $\pm 1$. Although this normalisation is equivalent to choosing a new bilinear form on $V$ it is easy to see that the involution $i \mapsto i'$ on $\{1,...,n\}$ remains unchanged. 

\p \label{sigmaint} Since $\sigma : \g_e \ra \g_e$ is an involution the maps $\xi + \sigma(\xi)$, with $\xi \in \g_e$, span $\h_e$. Thanks to \ref{gebasis} we may conclude that $$\{\xi_i^{j,\lambda_j - 1 - s} + \sigma(\xi_i^{j,\lambda_j - 1-s}) : 1 \leq i,j \leq n , 0 \leq s < \min(\lambda_i, \lambda_j)\}$$ is a spanning set for $\h_e$. Similar reasoning shows that a spanning set for $\pp_e$ may be found amongst the $\xi_i^{j,\lambda_j - 1 - s} - \sigma(\xi_i^{j,\lambda_j - 1-s}) $. This leaves us with two immediate tasks: evaluate $\sigma(\xi_i^{j,\lambda_j-1-s})$ and determine the linear relations between the spanning vectors. Using the fact that $\xi_i^{j,\lambda_j - 1 - s} + \sigma(\xi_i^{j,\lambda_j - 1-s})$ is skew self-adjoint with respect to $(\cdot,\cdot)$ we deduce that 
\begin{eqnarray}\label{sigmaaction}
\sigma(\xi_i^{j,\lambda_j - 1 - s}) = \varepsilon_{i,j,s}\xi_{j'}^{i', \lambda_i - 1 - s}
\end{eqnarray}
where $\varepsilon_{i,j,s}$ is defined by the relationship $(e^{\lambda_j-1-s}w_j,e^sw_{j'}) = - \varepsilon_{i,j,s}(w_i, e^{\lambda_i-1} w_{i'})$. This requires a little calculation.

\p We now make the notation
\begin{eqnarray*}
\zeta_i^{j,s} = \xi_i^{j,\lambda_j-1-s} +  \varepsilon_{i,j,s}\xi_{j'}^{i', \lambda_i - 1 - s}\\
\eta_i^{j,s} = \xi_i^{j,\lambda_j-1-s} -  \varepsilon_{i,j,s}\xi_{j'}^{i', \lambda_i - 1 - s}
\end{eqnarray*}
The maps $\zeta_i^{j,s}$ span $\h_e$ and the maps $\eta_i^{j,s}$ span $\pp_e$. We define a dual spanning set
\begin{eqnarray*}
(\zeta_i^{j,s})^\ast := (\xi_i^{j,\lambda_j-1-s})^\ast +  \varepsilon_{i,j,s}(\xi_{j'}^{i', \lambda_i - 1 - s})^\ast\\
(\eta_i^{j,s})^\ast := (\xi_i^{j,\lambda_j-1-s})^\ast -  \varepsilon_{i,j,s}(\xi_{j'}^{i', \lambda_i - 1 - s})^\ast.
\end{eqnarray*}
Note that we do \emph{not} have $(\zeta_i^{j,s})^\ast(\zeta_k^{l,r}) \in \{0,1\}$, contrary to the common convention for dual basis vectors. For instance, $\zeta_i^{i,\lambda_i-2} = 2\xi_i^{i,1}$ when $i = i'$ and $\lambda_i \geq 2$, and so $(\zeta_i^{i,\lambda_i-2})^\ast(\zeta_i^{i,\lambda_i-2}) = 4$. Our discussion shall be more concise due to this choice of definition (see Remark~\ref{duals} below).
We make further notation
$$\varpi_{i \leq j} = \left\{ \begin{array}{ll}
         1 & \mbox{if $i \leq j$}\\
        -1 & \mbox{if $i > j$}\end{array} \right. $$
and comparing with Lemma~\ref{nilpotents} we see that $\varpi_{i\leq i'} \varpi_{i'\leq i} = \epsilon(-1)^{\lambda_i-1}$, which shall prove useful in some later calculations.

\p The next lemma was first written in \cite{Top} and settles the question of which linear relations exist between the maps $\zeta_i^{j,s}$.
\begin{lem}\label{spanningdetails}
The following are true:
\begin{enumerate} 
\item{$\varepsilon_{i,j,s} = (-1)^{\lambda_j - s} \varpi_{i\leq i'} \varpi_{j \leq j'}$;}
\smallskip
\item{$\varepsilon_{i,j,s} = \varepsilon_{j',i',s}$;}
\smallskip
\item{The relations amongst the $\zeta_i^{j,s}$ are those of the form $\zeta_i^{j,s} = \varepsilon_{i,j,s} \zeta_{j'}^{i',s}$;}
\smallskip
\item{The relations amongst the $\eta_i^{j,s}$ are those of the form $\eta_i^{j,s} = - \varepsilon_{i,j,s} \eta_{j'}^{i',s}$;}
\smallskip
\item{Analogous relations hold amongst the $(\zeta_i^{j,s})^\ast$ and the $(\eta_i^{j,s})^\ast$. These are the only such relations;}
\smallskip
\item{The $(\zeta_i^{j,s})^\ast$ span $\h_e^\ast$ and the $(\eta_i^{j,s})^\ast$ span $\pp_e^\ast$.}
\end{enumerate}
\end{lem}
\begin{proof}
We have $$\varepsilon_{i,j,s} = \frac{-(e^{\lambda_j-1-s}w_j,e^sw_{j'})}{(w_i, e^{\lambda_i-1} w_{i'})} = \frac{(-1)^{\lambda_j-s}(w_j,e^{\lambda_j-1}w_{j'})}{(w_i, e^{\lambda_i-1} w_{i'})}.$$ We claim that $(w_i, e^{\lambda_i - 1} w_{i'}) = \varpi_{i \leq i'}$. The bilinear form $(\cdot,\cdot)$ is normalised so that $(w_i, e^{\lambda_i - 1} w_{i'}) = 1$ whenever $i \leq i'$, therefore we need only show that $(w_i, e^{\lambda_i - 1} w_{i'}) = -1$ whenever $i > i'$. If $i > i'$ then $$(w_i, e^{\lambda_i - 1} w_{i'}) = (-1)^{\lambda_i - 1} (e^{\lambda_i-1} w_i, w_{i'}) = \epsilon(-1)^{\lambda_i-1} (w_{i'}, e^{\lambda_i-1}w_i) =  \epsilon(-1)^{\lambda_i-1}.$$ However $ \epsilon(-1)^{\lambda_i-1} = -1$ whenever $i \neq i'$ by lemma \ref{nilpotents}, concluding part 1. Next observe that $\varpi_{i\leq i'} \varpi_{i' \leq i} = 1$ if and only if $i = i'$ hence $\varpi_{i\leq i'} \varpi_{i'\leq i} = \epsilon(-1)^{\lambda_i-1}$. Part 2 now follows from part 1.

The equality $\zeta_i^{j,s} = \varepsilon_{i,j,s} \zeta_{j'}^{i',s}$ follows from part 2. To see that these are the only relations we note that $\pp_e$ is spanned by vectors $\xi_i^{j,\lambda_j-1-s} - \varepsilon_{i,j,s} \xi_{j'}^{i',\lambda_i - 1-s}$ and that $\g_e / \pp_e \cong \h_e$. Then the map $ \xi_i^{j,\lambda_j-1-s} + \pp_e\ra \zeta_i^{j,s}$ is well defined and extends to a linear map $\g_e / \pp_e \ra \h_e$. It is surjective and so by dimension considerations  it is a vector space isomorphism. The only linear relations amongst the vectors $\xi_i^{j,\lambda_j-1-s} + \pp_e$ in $\g_e / \pp_e$ are those of the form $\xi_i^{j,\lambda_j-1-s} - \varepsilon_{i,j,s} \xi_{j'}^{i',\lambda_i - 1-s} + \pp_e= 0$. Part 3 follows. An identical argument works for part 4.

Part 5 is a straightforward consequence of duality, whilst for 6 we observe that $(\zeta_i^{j,s})^\ast$ vanishes on each $\eta_k^{l,r}$ and $(\eta_i^{j,s})^\ast$ vanishes on each $\zeta_k^{l,r}$. Since the set of all $(\zeta_i^{j,s})^\ast$ and all $(\zeta_i^{j,s})^\ast$ together span $\g_e^\ast$ it is quite evident that the zetas must span $\h_e^\ast$ and the etas must span $\pp_e^\ast$.
\end{proof}
\begin{rem}\label{duals}
\rm{Using part 3 of the above lemma we shall describe a basis for $\h_e$ by refining the spanning set of elements $\zeta_i^{j,s}$. From this basis it is possible to define a dual basis in the usual manner, and it is easy to prove that this coincides upto scalars with the set of functionals $(\zeta_i^{j,s})^\ast$ which is obtained by carrying out the analogous refinement. The point of this remark is that our definition of $(\zeta_i^{j,s})^\ast$ coincides upto scalars with the more canonical one, although ours is easier to work with. Analogous statements hold for $\pp_e$.}
\end{rem}

\p Let us proceed to refine a basis for $\h_e$ from $\{\zeta_i^{j,s}\}$. Define
\begin{eqnarray*}
H&:=& \{ \zeta_i^{i,s} : i< i' , 0 \leq s < \lambda_i \} \cup \{ \zeta_i^{i,s} : i = i', 0 \leq s < \lambda_i, \lambda_i - s \text{ even}\};\\
N_0 &:=& \{ \zeta_i^{i',s} : i \neq i', 0 \leq s < \lambda_i,  \lambda_i - s \text{ odd} \};\\
N_1 &:=& \{\zeta_i^{j,s} : i < j \neq i', 0 \leq s < \lambda_j\},
\end{eqnarray*}
and
\begin{eqnarray*}
\HH &:=& \spn (H);\\
\NN_0 &:=& \spn (N_0);\\
\NN_1 &:=& \spn (N_1).
\end{eqnarray*}
If $U_0,U_1 \subseteq V$ are vector subspaces then $\End(U_0, U_1)$ shall denote the space of $\K$-linear maps $U_0 \ra U_1$. We consider $\End(U_0, U_1)$ to be a subspace of $\End(V)$ under the natural embedding induced by the inclusions of $U_0$ and $U_1$ into $V$. An analogue of the following lemma holds for $\pp_e$ but we shall not need it.
\begin{lem}\label{subbasis}
The set $H\sqcup N_0 \sqcup N_1$ forms a basis for $\h_e$. Furthermore we have the following characterisation of the first two spaces:
\begin{enumerate}
\item{$\HH$ is precisely the subspace of $\h_e$ which preserves each Jordan block space $V[i]$: $$\HH = \h_e \cap(\bigoplus_{i} \emph{\End}(V[i]));$$}

\item{$\NN_0$ is precisely the subspace of $\h_e$ which interchanges $V[i]$ and $V[i']$ for $i \neq i'$ and kills $V[i]$ for $i=i'$: $$\NN_0 = \h_e \cap(\bigoplus_{i\neq i'} \emph{\End}(V[i], V[i']));$$}

\end{enumerate}
\end{lem}
\begin{proof}
First we show that all elements of $H\sqcup N_0 \sqcup N_1$ are non-zero. Clearly $\zeta_i^{j,s} = 0$ if and only if $\xi_i^{j,\lambda_j - 1-s} = -\varepsilon_{i,j,s} \xi_{j'}^{i',\lambda_i-1-s}$. For this we require that $i = j'$ and $\varepsilon_{i,j,s} = -1$. For $i=j'$ we must have $i = i' = j$ or $i \neq i' = j$. In the first case, $\varepsilon_{i,j,s} = (-1)^{\lambda_j-s}$ which equals $-1$ only if $\lambda_i-s$ is odd. But the maps $\zeta_i^{i,s}$ do not occur in $H$ when $i=i'$ and $\lambda_i-s$ is odd. In the second case $\varepsilon_{i,j,s} = (-1)^{\lambda_i-1-s}$ which equals $-1$ only if $\lambda_i-s$ is even. However, the maps $\zeta_i^{i',s}$ do not occur in $N_0$ when $i\neq i'$ and $\lambda_i-s$ is even.

Next observe that when $\zeta_i^{j,s}\neq 0$ exactly one of the two maps $\zeta_i^{j,s}$ and $\zeta_{j'}^{i',s}$ occur in $H\sqcup N_0\sqcup N_1$, thus showing this set to be a basis by part 3 of Lemma~\ref{spanningdetails}. The three characterisations are clear upon inspection of the definitions of the sets $H, N_0$ and $N_1$.
\end{proof}

\p \label{componentgroupintro} Finally we shall need to describe some basic facts about the centraliser group. Our remarks here will require $\chr(\K)$ to be 0 (the conclusions in positive characteristic are very slightly weaker). We continue to fix a nilpotent element $e$ in the Lie algebra $\h$ of a classical group $K$. By the Jacobson-Morozov theorem we may include $e$ into an $\sl_2$-triple $\phi = (e,h,f)$. By $\sl_2$ theory $\ad(h)$ induces a $\Z$-grading $\h = \oplus_{i \in \Z} \h(i)$ and it is well known that $\h_e \subseteq \oplus_{i \geq 0} \h(i)$. Set $\h_e(i) = \h(i) \cap \h_e$ and $\tr = \oplus_{i > 0} \h_e(i)$. Then $\h_e(0)$ is a Levi factor of $\h_e$ and $\tr$ the nilradical. 

If we let $R$ be the closed, connected subgroup of $K$ with $\tr = \Lie(R)$ then we have $K_e = K_{e, h} \ltimes R$ (see \cite[3.12]{Jan}. Since $R$ is connected it follows that $K_e^\circ = K_{e,h}^\circ \ltimes R$ and $$\Gamma(e) := K_e / K_e^\circ \cong K_{e, h} / K_{e,h}^\circ.$$ This group is finite and is called \emph{the component group of $e$}. By the theory of Dynkin and Kostant the centraliser of the pair $(e,h)$ is actually the centraliser of the $\sl_2$-triple $\phi$ (see \cite[Lemma~3.4.4]{CM} for instance). In this case we get $\Gamma(e) \cong K_{\phi}/K_{\phi}^\circ$. Very similar conclusions hold for an arbitrary reductive group.

\section{Symmetrisation in positive characteristic}\label{Symmetrisation}
\setcounter{parno}{0}

\p Throughout this section we take $\K$ to be an algebraically closed field of characteristic $p > 0$, with $\tG$ a group over $\K$. Since $\tg$ is a Lie algebra of an algebraic group it is restricted, with $p$-operation $x \mapsto x^{[p]}$ corresponding to $p^\text{th}$ powers of derivations of the ring of regular functions $\K[\tG]$.

\p The symmetric and enveloping algebras on $\tg$ shall here and always be denoted $S(\tg)$ and $U(\tg)$. The adjoint representations of $\tG$ and $\tg$ extend to representations in both $S(\tg)$ and $U(\tg)$, and we shall denote these extensions by the same labels. The representation $\Ad$ is extended by $\K$-algebra automorphisms, whilst $\ad$ is extended by $\K$-algebra derivations. As such, the invariant subspaces $U(\tg)^{\tG} \subseteq U(\tg)^\tg$, $S(\tg)^\tG \subseteq S(\tg)^\tg$ are all algebras. Since $\ad(x) u = xu - ux$ for $x \in \tg$ and $u\in U(\tg)$ it follows that $U(\tg)^\tg$ is equal to the centre of $U(\tg)$ which we denote by $Z(\tg)$.

\p Since $\tg$ shall be finite dimensional in all of our cases of interest $S(\tg)$ may be identified with the algebra of polynomial functions on the dual space $\tg^\ast$, which we denote $\K[\tg^\ast]$. There is a grading $S(\tg) = \bigoplus_{k\geq 0} S^k(\tg)$ where $S^k(\tg)$ consists of all homogeneous polynomials of degree $k$ and $S^0(\tg) = \K$. Since $\tg$ and $\tG$ preserve this grading the invariant subalgebras are homogeneously generated and contain the homogeneous parts of their elements. Also note that $S(\tg)$ inherits a filtered structure from its grading. The Poincar\'{e}-Birkhoff-Witt (PBW) theorem  implies that there exists a canonical filtration $\{U_k(\tg)\}_{k\geq 0}$ where $U_0(\tg) =\K$ and $U_k(\tg)$ is generated by expressions $x_{i_1} x_{i_2}\cdots x_{i_k}$ with each $x_{i_j}$ in $\tg$ when $k > 0$. The graded algebra associated to $U(\tg)$ shall be denoted $\gr U(\tg)$ and is isomorphic to $S(\tg)$ (we shall identify them). In general, if $A$ and $B$ are filtered $\tG$-modules and 
$\beta : A \ra B$ is a filtered $\tG$-map then there is an associated graded $\tG$-map $\gr A \ra \gr B$ which we denote by $\gr \beta$.

\p We define the $p^\text{th}$-power subalgebra $S(\tg)^p = \{f^p : f \in S(\tg)\}$. This algebra is $\tg$-invariant, and $S(\tg)$ is a free module of rank $p^{\dim\, \tg}$ over $S(\tg)^p$. The analogue in the enveloping algebra is the $p$-centre $Z_p(\tg)$. This is the central subalgebra of $U(\tg)$ generated by all terms $x^p - x^{[p]}$ with $x \in \tg$. If we endow $Z_p(\tg)$ with the filtration induced from that on $U(\tg)$ then the associated graded algebra is $S(\tg)^p$.

\p If $E$ is a simple $\tg$-module then Quillen's lemma \cite{Qui} asserts that $Z(\tg)$ acts by scalars on $E$. In particular $Z_p(\tg)$ acts by scalars. In fact something stronger is true: there is an associated functional $\chi \in \tg^\ast$ such that each $x^p - x^{[p]}$ acts by $\chi(x)^p$ on $E$ (this was first observed by Kac and Weisfeiller \cite{KW}) . We call $\chi$ the \emph{$p$-character of $E$}. Let $I_\chi$ be the (left and right) ideal of $U(\tg)$ generated by expressions $x^p - x^{[p]} - \chi(x)^p$ and define the reduced enveloping algebra $U_\chi(\tg) := U(\tg)/I_\chi$. Each simple $U(\tg)$-module is a $U_\chi(\tg)$-module for some $\chi$, and conversely each $U_\chi(\tg)$-module is canonically a $U(\tg)$-module. This stratifies the set of simple objects of $U(\tg)$-mod into subsets parameterised by $\tg^\ast$.

\p One of our first results regards the structure of the invariant subalgebras when $\tg$ is a centraliser in a Lie algebra of type $\sf A$ or $\sf C$. The properties of $U(\tg)^\tG$ and $U(\tg)^\tg$ are deduced from those of the corresponding symmetric invariant algebras by the use of a filtration preserving isomorphism of $\tG$-modules. Such a map is known to exist over any field of characteristic zero, and is named the symmetrisation map \cite[2.4.6]{Dix}. The construction there follows more or less from the PBW theorem. Over fields of positive characteristic the construction fails and we use an approach involving representation theory. The map is defined as a composition $$U(\tg) \ra S(U(\tg)) \ra S(\tg).$$

\p We introduce an auxiliary group. Recall that $V$ is a finite dimensional $\K$-vector space and suppose $\tG \subseteq \tR \subseteq G = GL(V)$. The inclusion $\tg \subseteq \tr$ induces inclusions $S(\tg) \subseteq S(\tr)$ and $U(\tg) \subseteq U(\tr)$. Let $x_1,...,x_n$ be an ordered basis for $\tr$ and $x_{j_1}x_{j_2} \cdots x_{j_m}$ ($j_1\leq j_2\leq \cdots \leq j_m$) a basis element for $U(\tr)$. If $I \subseteq \{1,...,m\}$ then $x_I := \prod_{k \in I} x_{j_k}$ is the monomial in $U(\tr)$, with terms ordered as above. Our map $\mu: U(\tr) \ra S(U(\tr))$ acts on this basis element by the rule
\begin{eqnarray*}
\mu(x_{j_1}x_{j_2} \cdots x_{j_m}) = \sum_{I_1\sqcup \cdots\sqcup I_k = \{1,...,m\}} x_{I_1}\circ x_{I_2}\circ \cdots \circ x_{I_k}
\end{eqnarray*}
Here $\circ$ denotes symmetric multiplication in $S(U(\tr))$, and we do not distinguish between two partitions $I_1\sqcup \cdots\sqcup I_k$ and $I'_1\sqcup \cdots\sqcup I'_k$ of $\{1,...,m\}$ if there exists $\tau$ in $\mathfrak{S}_k$, the symmetric group on $k$ letters, such that $I_j = I'_{\sigma(j)}$ for $1\leq j\leq k$. It is already clear that $\mu$ depends upon our choice of basis and our choice of ordering. The map is extended by linearity to all of $U(\tr)$ and enjoys several properties, most notably: $\mu$ is an injective map of $\tR$-modules fulfilling
\begin{eqnarray*}
\mu(x_{j_1}x_{j_2} \cdots x_{j_m}) = \mu(x_{j_1})\mu(x_{j_m})...\mu(x_{j_m})\mod S(U(\tr))_{(m-1)}.
\end{eqnarray*}
This map was constructed by Mil'ner in \cite{Mil}. In that paper he attempted to settle the first Kac-Weisfeiler conjecture in the affirmative, although the argument was eventually found to contain a gap which could not be closed. Nonetheless, much good theory has come out of that attempt and the results of \cite{FP} were accrued in the process of trying to make sense of that article. The main point of reference for our purposes is \cite[$\S 3$]{Pre} where the map is considered especially in the case of centralisers.

\p The next ingredient is a map $S(U(\tr)) \ra S(\tr)$. This shall be induced from a map $U(\tr) \ra \tr$ as follows. We say that $\tr$ \emph{possesses Richardson's property} if there exists an $\Ad(\tR)$-invariant decomposition $\g = \tr \oplus \mathfrak{c}.$ The inclusion $\tr \subseteq \g = \gl(V)$ gives an embedding $U(\tr) \hookrightarrow U(\g)$. Since $V$ is a $U(\g)$-module there is a surjection $U(\g)\ra \g = \End(V)$. The composition $U(\tr) \ra U(\g) \ra \g$ acts as the identity on $\tr \subseteq U(\tr)$, and so the image of this map contains $\tr$. Composing with the projection onto the first factor in the decomposition $\g = \tr \oplus \mathfrak{c}$ we obtain an $\Ad(\tR)$-equivariant map $$\pi : U(\tr) \ra \tr.$$ The symmetric algebra construction is a covariant functor from $\K$-vector spaces to commutative $\K$-algebras, and so we obtain a map $$S(\pi) : S(U(\tr)) \ra S(\tr).$$

\p Continuing, we assume that $\tr$ possesses Richardson's property and define $$\beta = S(\pi) \circ \mu : U(\tr) \ra S(\tr).$$ We say that the subalgebra $\tg$ is \emph{saturated} provided the image of $\pi|_{U_+(\tg)}$ is contained in $\tg$, where $U_+(\tg)$ denotes the augmentation ideal (note that the reverse inclusion holds automatically). Now the following theorem is contained in Property (B1), Proposition~3.4 and Lemma~3.5 of \cite{Pre}.
\begin{thm}\label{milmap}
If $\tG \subseteq \tR \subseteq GL(W)$, $\tr$ possesses Richardson's property and $\tg$ is a saturated subalgebra, then $\beta : U(\tg) \ra S(\tg)$ is an isomorphism of $\tG$-modules such that the associated graded map $\gr(\beta) :  \emph\gr U(\tg) \cong S(\tg) \ra \emph\gr S(\tg) = S(\tg)$ is the identity. Furthermore, if $\tg$ is a centraliser in $\tr$ then it is a saturated subalgebra. 
\end{thm}
\begin{rem}
\rm{By insisting in Richardson's property that the decomposition is $\Ad(\tR)$-invariant we obtain a $\tG$-module isomorphism in the above theorem. In Premet's proof the decomposition is only assumed to be $\ad(\tr)$-invariant, implying that $\beta$ is a map of $\tg$-modules. Inspecting the details of \cite[$\S 3$]{Pre} we can see that this slight alteration will place no extra burden upon the proof.}
\end{rem}

\section{The index of a Lie algebra}\label{Index}
\setcounter{parno}{0}

\p \label{indexintro} Now let $\chr(\K)$ be arbitrary. The notion of index of a Lie algebra was first defined by Dixmier \cite[1.11.6]{Dix} and plays a central role in the representation and invariant theory of Lie algebras in all characteristics. It is defined
\begin{eqnarray}
\ind\, \tg = \min_{\alpha \in \tg^\ast} \dim\, \tg_\alpha
\end{eqnarray}
This definition can be extended to the context of an arbitrary $\tg$-module $W$ by writing $\ind(\tg, W) = \min_{\alpha \in W^\ast} \dim \tg_\alpha$ so that $\ind\, \tg = \ind(\tg, \tg)$. There is an open dense subset of points in $W^\ast$ having $\tg$-stabiliser of dimension $\ind(\tg, W)$ (a straighforward assertion generalising \cite[1.11.5]{Dix}). That open set is denoted $W_\reg$ and the elements are called \emph{regular}. More generally, if $U \subseteq W$ then $U_\reg\subseteq U$ denotes the open subset of elements having a stabiliser of minimal dimension.

\p One example of the importance of the index is given by Rosenlicht's theorem which states that for a linear algebraic group $\tG$ acting on an irreducible variety $X$, the generic orbits are separated by rational invariants; in particular the algebra $\K(\tg^\ast)^\tG$ has transcendence degree $\ind\, \tg$ over the base field \cite{Ros1}.

\p The well known Elashvili conjecture on index, now a theorem, is closely related to our first results and so we recall its statement and history. The conjecture says that for every reductive Lie algebra $\tg$ and each $\alpha \in \tg^\ast$ we have equality $\ind \, \tg_\alpha = \ind\, \tg$. This attracted the attention of several Lie theorists, and the first major positive result was a proof in the classical cases by Yakimova \cite{Yak1}. We shall make use of her methods later. Willem de Graaf confirmed the conjecture in exceptional types with the aid of a computer. Since there are so many nilpotent orbits in Lie algebras of exceptional type, the computations of the latter paper are impossible to present in a concise way. Fortunately Charbonnel and Moreau have presented an almost case free argument in characteristic zero \cite{CMM}. We shall record this theorem now for later use.
\begin{thm}\label{elashvili}
If $\tG$ is a reductive group and $\alpha \in \tg^\ast$ then $\ind\, \tg_\alpha = \ind\, \tg$.
\end{thm}

\p Closely related to the notion of index of $\tg$ is the notion of a generic stabiliser. If $W$ is a $\tG$-module then it is also a $\tg$-module via the differential of the group representation. We say that \emph{there exists a generic stabiliser for $\tg$ in $W$} provided  there exists an open subset $\Pp$ of $W$ such that for all $v,w \in \Pp$ the stabilisers $\tg_v$ and $\tg_w$ are $\Ad(\tG)$-conjugate. In this case, the stabilisers of points in $\Pp$ are called \emph{generic stabilisers}. Some of our early results involve proving the existence of a generic stabiliser in certain modules, and our purpose in doing so is to calculate an index $\ind(\tg, W)$. This is made possible by the following lemma. We cannot find this result stated as such in the literature however the proof is implicit in \cite[Lemme~1.4]{TY1}.
\begin{lem}\label{genreg}
Fix $\gamma \in \tg$ and let $\V_\gamma = \{\delta \in W : \tg_\gamma \subseteq \tg_{\delta}\}$. If \begin{eqnarray*}\varphi: \tG \times \V_\gamma \ra W;\\ \varphi(g,w) = g\cdot w\end{eqnarray*} is a dominant morphism then $\gamma$ is a regular point of $W$, and $\tg_\gamma$ is a generic stabiliser.
\end{lem}
\begin{proof}
Suppose $\mathcal{P}$ is an open subset of $W$ contained in $\varphi(\tG \times \V_\gamma)$. It is well known that there exists an open subset $\mathcal{O}$ of $W$ such that the stabilisers of all points in $\mathcal{O}$ have dimension $\ind(\tg, W)$ (argue in the style of \cite[1.11.5]{Dix}). Then there exists $\delta \in \mathcal{O} \cap \mathcal{P}$, ie. there is $(g, \delta') \in \tG \times \V_\gamma$ such that $\varphi(g,\delta')= \delta$ and $\dim\, \tg_\delta = \ind(\tg, W)$. We conclude that $$\dim\, \tg_\gamma \leq \dim\, \tg_{\delta'} = \dim\, \tg_\delta = \ind(\tg, W)$$ so that $\gamma$ is regular. For all $\gamma' \in \mathcal{O} \cap \mathcal{P}$ we know that $\tg_{\gamma'}$ contains some $\tG$-conjugate of $\tg_\gamma$, say $\Ad(g) \tg_\gamma \subseteq \tg_{\gamma'}$. However by dimension considerations we see that $\tg_{\gamma'} = \Ad(g) \tg_\gamma$. Therefore $\mathcal{O} \cap \mathcal{P}$ is an open set within which the $\tg$-stabilisers of all points are $\tG$-conjugate.
\end{proof}

\section{The sheets of a Lie algebra}\label{sheets}
\setcounter{parno}{0}

\p In this section we shall introduce the geometric objects which shall occupy our interest in the middle of this thesis. For the moment we let $\tG$ be simple and connected. Define the variety
\begin{eqnarray*}
\tg^{(k)} = \{ x \in \tg : \dim\, \tg_x = k\}
\end{eqnarray*}
The irreducible components of these varieties are referred to as the \emph{sheets of} $\tg$. Since $\tG$ is connected, the sets $\tg^{(k)}$ and the sheets are $\tG$-stable. Therefore they are unions of $\tG$-orbits, hence locally closed. In order to get a better idea of how $\tg$ decomposes into sets $\tg^{(k)}$ it is instructive to look at the extremal cases. We have $\tg^{(\dim\, \tg)} = \mathfrak{z}(\tg) = 0$, whilst $\tg^{(\rank\, \tg)}$ is the open, irreducible subset $\tg_\reg \subseteq \tg$ which was discussed in \ref{indexintro}, also known as \emph{the regular sheet}. The regular sheet was studied extensively before sheets were investigated in full generality. It is well known that the adjoint and coadjoint orbits are even dimensional, hence $\tg^{(k)} = \emptyset$ whenever $\dim\, \tg - k$ is odd.

\p\label{jordanclassif} The sheets were classified by Borho \cite{Borho} using the notion of \emph{$\tG$-Jordan classes} (also known as \emph{decomposition classes}), which we now recall. Using the Jordan decomposition of \ref{jordan1} we may define an equivalence relation on $\tg$ by declaring that $x \sim y$ whenever there exists $g \in \tG$ such that $\Ad(g) \tg_{x_s} = \tg_{y_s}$ and $\Ad(g) x_n = y_n$. Here $x = x_s + x_n$ is the Jordan decomposition of $x$. The equivalence classes under this relation are called the $\tG$-Jordan classes, and the set of $\tG$-Jordan classes will be denoted $\mathcal{J}(\tG)$. According to \cite[39.1.5]{TY} the class of $x \in \tg$ coincides with
\begin{eqnarray}\label{jordanclassstruct}
\Ad(\tG)(x_n + \mathfrak{z}(\tg_{x_s})_\reg)
\end{eqnarray}
from whence it follows that the Jordan classes are irreducible, $\tG$-stable and locally closed.

\p Let $P(\tG)$ denote the $\tG$-conjugacy classes of pairs $(\li, \Oo)$ where $\li$ is a Levi subalgebra of $\tg$ and $\Oo \subseteq \Ni(\li)$ is a nilpotent orbit. Here $\tG$ acts diagonally on the two factors. The elements of this set shall be called \emph{$\tG$-pairs} for short and the conjugacy class of $(\li, \Oo)$ will be denoted $(\li, \Oo)/\tG$. To each $x \in \tg$ we may associate a $\tG$-pair as follows. The centraliser $\li = \tg_{x_s}$ is a Levi subalgebra of $\tg$ thanks to \ref{levicent}. We let $L$ denote the Levi subgroup of $\tG$ with Lie algebra $\li$. We then set $\Oo = \Ad(L)x_n$ and send $x$ to the $\tG$-pair $(\li, \Oo)/\tG$. The map $x \mapsto (\li, \Oo)/\tG$ descends to a well-defined bijection $$\mathcal{J}(\tG) \longleftrightarrow P(\tG).$$ From the finitude of nilpotent orbits and conjugacy classes of Levi subalgebras we deduce that $\mathcal{J}(\tG)$ is finite. 

\p Clearly $\dim\, \tg_x$ does not change as $x$ varies over $J \in \mathcal{J}(\tG)$. It follows 
that each sheet is a union of $\tG$-Jordan classes. Since $\mathcal{J}(\tg)$ is finite and each class is irreducible, each sheet
contains a dense open class. Hence the sheets of $\tg$ are classified by identifying those
$\tG$-Jordan classes which are dense in a sheet. It turns out that those classes are precisely the
ones corresponding to $\tG$-pairs $(\li, \Oo)/\tG$ where $\li$ is a Levi and $\Oo \subseteq \li$
is a rigid nilpotent orbit. We remind the reader that rigid orbits are those which cannot be 
obtained via Lusztig-Spaltenstein induction. We record this classification here for later use.
\begin{thm}\label{classifsheets}
The sheets of $\tg$ are in one-to-one correspondence with the $\tG$-pairs $(\li, \Oo)/\tG$ where $\Oo \subseteq \li$
 is a rigid nilpotent orbit in $\li$.
\end{thm}
\noindent If $\mathcal{S}$ corresponds to $(\li, \Oo)/\tG$ under the above bijection we say that 
$\mathcal{S}$ \emph{has data} $(\li, \Oo)/\tG$. It is a theorem that every sheet contains a unique nilpotent orbit, and if $\SS$ has data $(\li, \Oo)/\tG$
then that orbit is $\Ind_\li^\tg(\Oo)$. Later we shall use this to classify the sheets containing a given orbit.




\p\label{dimformula} If $\SS$ has data $(\li, \Oo)/\tG$ then we define \emph{the rank of $\SS$} by $$\rank\, \SS := \dim\, \mathfrak{z}(\li).$$ 
Suppose that $\SS \subseteq \tg^{(k)}$ and let $e_0 \in \Oo$. Then by (\ref{jordanclassstruct}) the map 
$\tG \times (e_0 + \mathfrak{z}(\li)_\reg) \ra \SS$ is dominant. Thus by the theorem on 
dimensions of the fibres of a morphism we have
\begin{lem}
$\dim\, \SS = \dim\, \tg - k + \rank\, \SS$.
\end{lem}

\p The sheets of the classical Lie algebras are smooth thanks to \cite{IH}. By contrast there are sheets in exceptional algebras which are not smooth; an example is given in type ${\sf G}_2$ in \cite{Slo}.

\section{Ideals of the enveloping algebra}
\setcounter{parno}{0}

\p For this section let $\chr(\K) = 0$ and $\tg$ be a finite dimensional simple Lie algebra. The finite dimensional simple modules for $\tg$ may all be obtained as the simple heads of modules induced from a character on a torus, the so called Verma modules. Thus the finite dimensional simple modules are fairly well understood, however the entire class of simple modules is much harder to discuss in any detail. A more tractable problem is to consider the annihilators in the enveloping algebra of simple modules. These are the \emph{primitive ideals of $U(\tg)$}. We denote by $\Prim \, U(\tg)$ the set of all primitive ideals of $U(\tg)$. We shall now describe two important invariants attached to a primitive ideal by the theory of commutative algebra.

\p \label{associatedvariety} Let $I \in \Prim U(\tg)$. The filtration on $U(\tg)$ induces one on $I = \cup_{k\geq 0} I_{(k)}$.
The associated graded $\gr I$ is an ideal in $\gr U(\tg) = S(\tg) = \K[\tg^\ast]$. We may consider the set of common zeroes 
of $\gr I$ in $\tg^\ast$. Identifying $\tg^\ast$ with $\tg$ via the Killing isomorphism we obtain a closed subvariety 
$\mathcal{VA}(I) \subseteq \tg$, which we call the \emph{associated variety of $I$}. If $I \in \Prim \, U(\tg)$ is the annihilator
of $W$ then $Z(\tg)$ acts by a character on $W$, whose kernel is a maximal ideal $\mm \unlhd Z(\tg)$. But then
$\gr I$ contains the ideal of $S(\tg)$ generated by $\gr \mm$, which equals $S(\tg)^\tg_+\unlhd S(\tg)$,
the ideal generated by non-constant invariant polynomials (see \cite[9.2]{Jan} for more details). But since $V(\K[\tg]^\tg_+) = \Ni(\tg)$ we see that
$\VA(I) \subseteq \Ni(\tg)$. The Irreducibility Theorem states that $\mathcal{VA}(I)$ is irreducible,
hence coincides with the closure of a nilpotent orbit. This was first proven in full generality by Joseph \cite{Jos}. More recently Ginzburg has proven
a deep result generalising this theorem: he states that if $A$ is a positively filtered associative algebra such that $\gr A$ 
is a commutative Poisson algebra with finitely many symplectic leaves then the associated variety of any primitive ideal 
of $A$ is the closure of a connected symplectic leaf \cite{Ginz}.

\p \label{multiplicity} We now define the associated cycle $\AC(I)$ of a primitive ideal in $U(\tg)$. Let $M$ be a finitely generated $U(\tg)$-module with $I = \Ann_{U(\tg)}(M)$. Choose generators $m_1,...,m_l$ for $M$  and define a filtration on $M$ by letting $M_{(0)} = \spn_{\K}\{m_1,...,m_l\}$ and $M_{(k)} = U(\tg)_{(k)}M_0$. Then $\gr(M)$ is a finitely generated $S(\tg)$-module. It follows from Chapter IV, $\S 1$, Theorem 1 of \cite{Bou1} that there exist prime ideals $\mathfrak{q}_1, ..., \mathfrak{q}_m$ in $S(\tg)$ and a chain of module inclusions $$0 = M_0 \subseteq M_1 \subseteq M_2 \subseteq \cdots \subseteq M_k = \gr(M)$$ such that $M_i/M_{i-1} \cong S(\tg)/\mathfrak{q}_i$. After reordering the indexes, there exists $k' \leq k$ such that $\mathfrak{q}_1,...,\mathfrak{q}_{k'}$ is the list of primes \emph{without repetition}. For $i = 1,...,k'$ let $n_i$ denote the number of $\mathfrak{q}_j$ ($1\leq j\leq k$) with $\mathfrak{q}_j = \mathfrak{q}_i$. The associated cycle of $M$ is defined to be the 
formal linear combination
$\sum_{i=1}^{k'} n_i[\mathfrak{p}_i]$. Thanks to \cite[Lemma~9.5]{Jan} this definition depends neither upon the filtration of $M$ nor upon our choice of generators for $M$. It can be shown that this definition only depends upon the ideal $I$ and so the notation $\AC(I)$ is justified.

\p Since the variety $\VA(I)$ is irreducible by Joseph's theorem and hence coincides with the Zariski closure of
a nilpotent orbit $\Oo \subset \tg$, we have that $\AC(I)=m_I[J]$ where $m_I\in\N$ and $J=\sqrt{\gr(I)}$, a
prime ideal of $S(\tg^\ast)$. The positive integer $m_I$ will be referred to as the {\it multiplicity} of $\Oo$ in the 
primitive quotient $U(\tg)/I$ and abbreviated as $\mult_{\Oo}(U(\tg)/I)$. It is well known that if $\Oo=\{0\}$ then 
$I$ coincides with the annihilator in $U(\tg)$ of a finite dimensional irreducible $\tg$-module $V$, the radical 
$J=\sqrt{\gr(I)}$ identifies with the ideal $\bigoplus_{i>0} S^i(\tg)$  and $m_I=(\dim \, V)^2$. Much later in the
text we shall, in some sense, classify the \emph{multiplicity free primitive ideals} in enveloping algebras of classical Lie algebras: those with $\mathcal{VA}(I) = \overline{\Oo}$ and
$\mult_\Oo(U(\tg)/I) = 1$. We denote the set of multiplicity free primitive ideals with associated variety $\overline{\Oo}$
 by $\MF_{\Oo}$.

\p \label{goldierank} The following paragraph makes use of Goldie's theory of semisimple rings of fractions, which is surveyed in \cite[Chapter~5]{GW}. For any $I\in\Prim U(\tg)$ the prime Noetherian ring $U(\g)/I$ embeds into a full ring of fractions.
The latter ring is prime Artinian and hence isomorphic to the matrix algebra $\Mat_n(\II_I)$ over a skew-field $\II_I$
called the {\it Goldie field} of $U(\g)/I$. The positive integer is called the Goldie rank of $U(\g)/I$ which
is often abbreviated as $\rank \, U(\g)/I$. Recall that a primitive ideal $I$ is called {\it completely prime} if $U(\g)/I$ is a domain.
It is well known that this happens if and only if $\rank \, U(\g)/I = 1$. Classifying the completely prime primitive ideals of $U(\g)$ is an
long-standing classical problem of Lie Theory. In general, it remains open outside type $\sf A$ although many important
partial results have been obtained. See \cite{McG}, \cite{Moe}, \cite{Moe1} and references therein.

\p Finally let us recall the theory of induced modules and ideals for $U(\tg)$. Let $\mathfrak{p}$ be a parabolic subalgebra
of $\tg$ with Levi decomposition $\li \ltimes \n$. Any simple left $\li$-module $E$ is canonically a $\mathfrak{p}$-module, with $\n$ acting by
zero. Then $\Ind_\pp^{\tg}(E)$ is the parabolically induced left $\tg$-module $U(\tg) \otimes_{U(\mathfrak{p})} E$. These induced modules
are sometimes referred to as generalised Verma modules. If $E$ is a highest weight module then so too is $\Ind_\pp^{\tg}(E)$, with the weight being inherited from $E$.

\p\label{inducedideals} If $I_0$ is a primitive ideal of $U(\li)$ then it is the annihilator of some simple module $E$. If we choose a parabolic subalgebra containing
$\li$ as a Levi factor then we can define the induced ideal $\Ind_\li^\tg(I_0)$ to be the annihilator in $U(\tg)$ of $\Ind_\pp^\tg(E)$.
It is well known that $\Ind^\tg_\li(E)$ coincides with the largest two-sided ideal of $U(\g)$ contained in
the left ideal $U(\g)(I_0 + \mathfrak{n})$ and hence depends only on $\pp$ and $I_0$; see \cite[10.4]{BGR}.
Even more is true: if $\pp_1$ and $\pp_2$ are two parabolics possessing $\li$ as a Levi factor then they are
conjugate by an element $g \in \tG$ stabilising $\li$, say $\pp_2 = \Ad(g)\pp_1$. But then the induced module with twisted action, ${}^g\Ind^\tg_{\pp_1}(E)$, is
itself induced from $\pp_2$, and since the annihilator of the induced module is a two-sided ideal it is fixed by $\Ad(g)$.
This shows that the induced ideal does not depend upon our choice of parabolic containing $\li$ as a Levi factor,
and our notation is justified.
Even though the induced module is usually not simple, the following theorem tells us that ideals induced from completely
prime primitive ideals are very well behaved.
\begin{thm}
If $I_0 \unlhd U(\li)$ is a completely prime primitive ideal then $\Ind_\li^\tg(I_0)$ is also completely prime and primitive.
\end{thm}
\noindent The statement for completely prime ideals is known as Conze's theorem \cite{Con}. The fact that it is primitive follows from
\cite[8.5.7]{Dix}.

\section{Finite $W$-algebras}\label{finitewalgebras}
\setcounter{parno}{0}

\p\label{fwaconstruction} In this section we have $\chr(\K) = 0$. Let $\tG$ be a connected reductive group over $\K$ and let $e \in \tg$ be nilpotent. A finite $W$-algebra is an associative, filtered algebra constructed from the data $(\tg, e)$. We now recall the construction. First include $e \hookrightarrow (e,h,f) = \phi$ into an $\sl_2$-triple in $\tg$. By $\sl_2$ theory there is $\Z$-grading $\tg = \bigoplus_{i \in \Z} \tg(i)$ induced by $\ad(h)$. Let $\chi$ be the element of $\tg^\ast$ associated to $e$ under the Killing isomorphism. Then there is a non-degenerate symplectic form on $\tg(-1)$ given by $\langle x, y\rangle = \chi[x y]$. Choose a Lagrangian (a maximal totally isotropic space with respect to $\langle\cdot,\cdot\rangle$)  $\ell \subseteq \tg(-1)$ and define a nilpotent subalgebra $\mm = \ell \oplus \sum_{i < -1} \tg(i)$. It is easy to check that $\chi$ is a character for $\mm$. We let $U(\tg)\mm_\chi$ be the left ideal of $U(\tg)$ generated by terms $x - \chi(x)$ with $x \in \mm$. The 
generalised Gelfand-Graev module associated to $\chi$ is $\cG_\chi = U(\g)/U(\tg) \mm_\chi$. The finite $W$-algebra $U(\tg, e)$ is defined to be the quantum Hamiltonian reduction $\cG_\chi^{\mm}$. It is clear that the isomorphism type of $U(\tg,e)$ does not change as $e$ varies in a $\tG$-orbit.

\p The $W$-algebra $U(\tg,e)$ was studied by Kostant in the case $e$ is regular, although in a rather different language. He observed that $U(\tg,e) \cong Z(\tg)$ in that situation \cite{Kos}. Later, Kostant's approach was expanded upon in the thesis of his student, Lynch \cite{Lyn}. The definition did not appear in its current generality until Premet brought them into mathematical literature in \cite{Pre2}. By this time, these objects were being studied by physicists using a construction via BRST cohomology. These definitions were shown to coincide in \cite{DC}. Gan and Ginzburg provided a useful variant on Premet's construction \cite{GG} showing that the $W$-algebra does not depend upon the choice of Lagrangian space. Yet another definition was provided by Losev via Fedesov quantization, and all of these definitions have been used to great advantage. The $W$-algebra is closely related to the enveloping algebra $U(\tg)$ (a rather coarse impression of this relationship is imparted by the isomorphism
$U(\tg, 0) \cong U(\tg)$), whilst in \cite{Pre2} it was shown that the algebra can be seen as a filtered deformation of the coordinate algebra $\C[e + \tg_f]$. For these reasons, $U(\tg, e)$ is often referred to as \emph{the enveloping algebra of the Slodowy slice $e + \tg_f$}.

\p \label{ganginz} Before we continue we would like to briefly sketch Gan and Ginzburg's construction. In the notation of \ref{fwaconstruction}, fix \emph{any} isotropic subspace $\overline{\ell} \subseteq \tg(-1)$ with respect to $\langle\cdot,\cdot\rangle$, and let $\overline{\ell}^\perp$ be the annihilator, again with respect to $\langle\cdot,\cdot\rangle$. Define $\overline{\mm} = \overline{\ell} \oplus \sum_{i < -1} \tg(i)$ and $\overline{\n} = \overline{\ell}^\perp \oplus \sum_{i < -1} \tg(i)$. Now $\chi$ is a character on $[\overline{\mm}\, \overline{\mm}]$ and so we may define $\overline{\mm}_\chi$, similar to the above, and set $H_\ell := (U(\tg)/U(\tg)\overline{\mm}_\chi)^{\ad(\overline{\n})}$. The Lagrangian $\ell$ from \ref{fwaconstruction} may be chosen to include $\overline{\ell}$, so that $\overline{\mm} \subseteq \mm$. But then $U(\tg)/U(\tg)\overline{\mm}_\chi$ surjects naturally onto $U(\tg)/U(\tg)\mm_\chi$, and $\ad(\overline{\n})$-invariants are mapped to $\ad(\mm)$-invariants. Then Gan 
and Ginzburg show that the induced map $H_\ell \ra U(\tg,e)$ is an isomorphism. This shows that the construction of $U(\tg,e)$ is independent of our choice of Lagrangian. More importantly, the stabiliser of the $\sl_2$-triple, $\tG_\phi$, preserves $\overline{\mm}$ when $\overline{\ell}$ is taken to be zero. The action on the quotient $U(\tg)/U(\tg)\overline{\mm}_\chi$ also preserves $\ad(\overline{\n})$-invariants and we obtain a $\tG_\phi$-action on $U(\tg,e)$, which will be of great importance to our later work.

\p \label{modularhistory} Let us now give examples of how the representation theory of $W$-algebras interplays with more classical theories. These examples shall also serve as the motivation for our results on $W$-algebras. For $\chi \in \tg^\ast$ we shall denote by $\tg_\chi$ the stabiliser of $\chi$ in $\tg$ under the coadjoint representation. The second Kac-Weisfeiler conjecture states for each simple Lie algebra $\tg$ of a reductive group $\tG$ over a field of characteristic $p > 0$ and for each $\chi \in \tg^\ast$ that $p^{\frac{1}{2}(\dim\, \tg - \dim \, \tg_\chi)}$ divides the dimension of every simple $U_\chi(\tg)$-module $W$. This was settled in the affirmative by Premet in \cite{Pre3} using a modular analogue of the algebra $\mm$ defined above. Humpreys' conjecture on small modular representations states that this lower bound on dimensions is best possible, ie. that every such reduced enveloping algebra actually possesses a module of that dimension. In \cite{Pre4} Premet showed that if each finite 
$W$-algebra has a non-trivial one dimensional representation then Humphrey's conjecture holds under some mild restrictions on $\tG$, provided $p$ is greater than some unknown bound. Examples show that some of the assumptions on $\tG$ are actually necessary. The existence of non-trivial one dimensional representations is now settled (see \ref{1dreps}) in the affirmative and Humphreys' conjecture is now a theorem, given the aforementioned assumptions. One of the main theorems of this thesis describes the geometry of a certain algebraic variety which parameterises the one dimensional $U(\h,e)$-modules (for $\h$ classical) and this will ultimately allow us to learn more about the modular representations in these types.

\p Finite $W$-algebras are best motivated by the powerful results which link their representation theory with that of the enveloping algebra. The most fundamental relationship of this nature is the Skryabin equivalence, proven in the appendix to \cite{Pre2}. We denote by $U(\tg)\cm\chi$ the category of $U(\tg)$-modules upon which $\mm_\chi := \{x - \chi(x) : x\in \mm\}$ acts locally nilpotently. Let $U(\tg,e)\cm$ denote the category of all $U(\tg,e)$-modules. If $W \in U(\tg,e)\cm$ then one may check that $\cG_\chi \otimes_{U(\tg,e)} W \in U(\tg)\cm\chi$. Conversely, if $V \in U(\tg)\cm\chi$ then $V^{\mm_\chi} = \{v\in V : mv = \chi(m)v \text{ for all } m \in \mathfrak{m}\} \in U(\tg,e)\cm$. Define maps
\begin{eqnarray*}
\FF &:& U(\tg,e)\cm \ra U(\tg)\cm\chi \\
& &  V \ra \cG_\chi \otimes_{U(\tg,e)} V \\
& & \\
\widetilde\FF &:& U(\tg)\cm\chi \ra U(\tg,e)\cm\\
& & W \ra W^\mm
\end{eqnarray*}
The Skryabin equivalence may be stated:
\begin{thm}\label{skryabinequivalence}
The maps $\FF$ and $\widetilde{\FF}$ are quasi-inverse equivalences of categories.
\end{thm}

\p \label{quotientsintro} Let $I_c$ be the ideal of $U(\tg,e)$ generated by all commutators $xy-yx$ with $x,y \in U(\tg,e)$. Then the maximal abelian quotient is defined $U(\tg,e)^\text{ab} := U(\tg,e) /I_c$. It is not hard to see that the one dimensional representations of $U(\tg,e)$ are parameterised by the maximal spectrum $\mathcal{E} := \Specm \, U(\tg,e)^\ab$. By our discussion in \ref{ganginz} the group $\tG_\phi$ acts on $U(\tg,e)^\ab$ by algebra automorphisms and so permutes the ideals of $U(\tg,e)^\ab$ of codimension 1. It follows from \cite[Lemma~2.4]{Pre5} that the connected component $\tG_\phi^\circ$ stabilises every such ideal. Recall from \ref{componentgroupintro} that we have a natural isomorphism $$\Gamma:= \Gamma(e) = \tG_e/\tG_e^\circ \cong \tG_\phi/ \tG_\phi^\circ$$ and it follows that $\Gamma$ acts on $\EE$. We shall be interested in the space $\EE^\Gamma$.

\p We shall now explain how this fixed point space is related to multiplicity free ideals of $U(\tg)$. The following theorem is an amalgam of several deep results \cite[Theorem~1.1]{Pre6}, \cite{Lo1}, \cite{Ginz1}, \cite[Theorem~1.2.2]{Lo2}, \cite{Pre7}, \cite[Remark~7.7]{LO}. If $\mathcal{R}$ is a ring and $W $ is a left $\mathcal{R}$-module then set $I_W := \Ann_\mathcal{R} W$.
\begin{thm}\label{bigtheoremF}
Let $e \in \Ni(\tg)$. 
\begin{enumerate}
\item{If $I \in \Prim \, U(\tg)$ then $\VA(I) = \overline{\Oo}_e$ if and only if $I = I_{\FF(W)}$ for some finite dimensional simple $U(\tg,e)$-module $W$;}
\end{enumerate}
Now fix two finite dimensional simple $U(\tg,e)$-modules $W_1$ and $W_2$.
\begin{enumerate}
\item[2.]{If $I_{\FF(W_1)} = I_{\FF(W_2)}$ then $I_{W_1}$ and $I_{W_2}$ are $\Gamma$-conjugate;}
\item[3.]{$\rank \, U(\tg)/ I_{\FF(W_1)}$ divides $\dim(W_1)$;}
\item[4.]{The multiplicity $\mult_{\Oo_e} U(\tg)/I_{\FF(W_1)}$ is equal to $[\Gamma: \Gamma_V] (\dim \, W_1)^2$.}
\end{enumerate}
\end{thm}
\noindent Part 1 tells us, among other things, that every primitive ideal can be obtained from a simple module for some finite $W$-algebra. Part 4 of the theorem
implies that the multiplicity free ideals of $U(\tg)$ with associated variety $\overline{\Oo}_e$ are obtained from elements of $\EE^\Gamma$. Thus we obtain a surjective map
\begin{eqnarray*}
 \EE^\Gamma &\longrightarrow & \mathcal{MF}_{\Oo_e}\\
W &\longmapsto & I_{\FF(W)}
\end{eqnarray*}
Part 2 of the theorem tells us that this map is also injective, and part 3 of the theorem implies that the elements of $\MF_{\Oo_e}$ are completely prime.


\section{Some useful theorems}\label{UsefulTheorems}
\setcounter{parno}{0}

\p We take a moment to list four theorems which will be of central importance to our methods. The first two theorems both apply to the modular representation theory, whilst the second two discuss finite $W$-algebras.

\p \label{kacweisover} We recall all of the notations and conventions of Section~\ref{Symmetrisation}. In his seminal paper \cite{Zas} Zassenhaus observed that there is an upper bound on the dimensions of simple modules for any Lie algebra $\tg$ defined over a field of positive characteristic. We denote that upper bound by $M(\tg)$. A very coarse estimate is given by $M(\tg) \leq p^{\dim\, \tg}$. For $\tg$ reductive it is well known that $M(\tg) = p^{\frac{1}{2}(\dim\,\g - \rank\,\g)}$ (see \cite{PS} for a proof which works in great generality). Let $F(\tg)$ denote the full field of fractions of $Z(\tg)$ and let $F_p(\tg)$ denote that of $Z_p(\tg)$. Clearly $F(\tg)$ is a finite field extension of $F_p(\tg)$ and by the remarks preceding \cite[Lemma~5]{Zas} the dimension of $F(\tg)$ as a vector space over $F_p(\tg)$ is power of $p$. Let us denote the degree by $$[F(\tg) : F_p(\tg)] = p^l.$$ One of the many interesting consequences of Zassenhaus' aforementioned paper is the following, which is explained simply 
in \cite[A.5]{Jan1}.
\begin{thm}
$M(\tg) = p^{\frac{1}{2}(\dim\, \tg - l)}$.
\end{thm}
\noindent We shall discuss the number $M(\tg)$ further in the next section. In particular we shall describe a conjecture which predicts its value.

\p Our very first results describe certain rings of invariants and our method is to satisfy the assumptions of a theorem of Skryabin \cite[Theorem~5.4(i)]{Skr}. The theorem is proved in a more general setting than we require, and we shall state a version sufficient for our purposes. Suppose that $X$ is vector space over $\K$, acted on linearly by $\tG$, with ring of regular functions $\K[X]$ and rational functions $\K(X)$. We denote by $\K[X]^p$ the subalgebra of $p^\text{th}$ powers and $\K(X)^p$ similar. For $f_1,...,f_m \in \K[X]$ and $\alpha \in X$ we have the differential $d_\alpha f_i : T_\alpha X \ra \K$. We denote by $J(f_1,...,f_m)$ the Jacobian locus of $f_1,...,f_m$; that is the closed subset of $X$ such that the differentials $d_\alpha f_i$ are linearly dependent.
\begin{thm}\label{Skryabin}
Suppose that $f_1,...,f_m \in \K[X]^\tg$ where $$m = \dim\, X - \dim \, \tg + \ind(\tg, X).$$ If $\emph{\codim}_X\,  J(f_1,...,f_m) \geq 2$ then
$$
\K[X]^\tg = \K[X]^p[f_1,...,f_m].
$$
\end{thm}
\noindent We shall commonly refer to the bound on the codimension of the Jacobian locus as \emph{the codimension 2 condition}.

\p\label{premettheorem} The remaining two theorems regard $W$-algebras and are specific to fields of zero characteristic. Fix a connected reductive group $\tG$ and a nilpotent element $e \hookrightarrow \phi = (e,h,f)$ included into an $\sl_2$-triple.
Let $\SS_1,...,\SS_l$ be pairwise distinct sheets of $\tg$ containing $e$. This next theorem is an extremely
powerful tool which relates the irreducible components of $\EE(\tg, e)$ to the sheets $\SS_1,...,\SS_l$.
We let $X_i := \SS_i \cap (e + \tg_f)$ and write $\Comp(X_i)$ for the set of irreducible components of $X_i$. The following was proven by Premet in \cite[Theorem~1.2]{Pre4}
\begin{thm}\label{premetsurj}
Suppose that the orbit $\Oo_e$ is induced from a nilpotent orbit in some Levi subalgebra of $\g$. There exists a surjective map $$\sigma : \Comp\, \EE(\tg,e) \twoheadrightarrow \bigsqcup_{i=1}^l \Comp\, X_i.$$
For every irreducible component $\mathcal{Y} \subseteq \EE(\tg,e)$ with $\sigma(\mathcal{Y}) \in \Comp(X_i)$ we have $\dim(\mathcal{Y}) \leq \dim(X_i)$. Furthermore, for each $1 \leq i \leq l$ there exists $\mathcal{Y} \in \sigma^{-1}(\Comp(X_i))$ such that $\dim(\mathcal{Y}) = \dim(X_i)$. 

\end{thm} 
\noindent The proof uses the methods of modular reduction and there is no known construction over fields of characteristic zero.

\p\label{losevtheorem} Notations as per the previous paragraph. The final theorem which we list here is due to Losev and provides a means by which we can compare the varieties $\EE(\tg,e)$ and $\EE(\li, e_0)$
when $\li$ is a Levi subalgebra of $\tg$ and $e_0$ a nilpotent element with $\Oo_e = \Ind_\li^\tg(\Oo_{e_0})$. Losev studies finite $W$-algebras from the perspective of symplectic geometry and
Fedosov quantisation, which we have not touched upon in this thesis as the subject is too broad and too far-removed from our methods. The following theorem is contained in \cite{Lo3}.
\begin{thm}\label{losevembed}
Let $(\li, e_0)$ be defined as above. There exists a completion $U(\li,e_0)'$ of the finite $W$-algebra $U(\li,e_0)$ and an injective algebra
homomorphism $\Xi_\li: U(\tg,e)\rightarrow U(\li,e_0)'$. The abelian quotients $U(\li,e_0)^\ab$ and $U(\li,e_0)'^\ab$ are naturally isomorphic 
giving an injective algebra homomorphism $\xi_\li : U(\tg,e)^\ab \ra U(\li, e_0)^\ab$ such that the resulting morphism of varieties $\xi_\li^\ast : \EE(\li, e_0) \ra \EE(\tg, e)$
is finite. 
\end{thm}

\section{Overview of results}\label{overview}
\newcounter{parna}
\renewcommand{\theparna}{\thesubsection.\arabic{parna}}
\renewcommand{\p}{\refstepcounter{parna}\noindent\textbf{\theparna .} \space} 
\setcounter{parna}{0}

\p We are now ready to describe the results contained herein, and talk a little about their proofs. Our discussion falls roughly into three topics: our primary results are of an invariant theoretic nature, and build on the concepts of \cite{PPY}; in the next chapter we discuss the relationship between the derived subalgebra of the centraliser and sheets in a classical Lie algebra - among other things, we settle a recent conjecture of Izosimov; and finally we apply our deductions from the second section to the theory of finite $W$-algebras in types $\sf B$, $\sf C$ and $\sf D$. The first of these three topics corresponds to a paper written by the author \cite{Top}, soon to appear in the Journal of algebra, whilst the second two correspond to a joint work with Alexander Premet \cite{PTop} which is currently under review.

\subsection{Invariants and representations of centralisers}
\setcounter{parna}{0}
\p Broadly speaking, the purpose of the first section is to bring the invariant theoretic discussion of \cite{PPY}, \cite{BB}, \cite[$\S 8$]{Yak2} into the characteristic $p$ realm, and exploit some combinatorial techniques to study the representation theory of the centralisers $\tg_x$ in type $\sf A$ and $\sf C$.

\p \label{counterkac} When an algebraic group may be reduced modulo $p$, in an appropriate sense, we have groups $\tG$, $\tG_p$ and their respective Lie algebras $\tg$, $\tg_p$. If $\tG$ is reductive and the characteristic of the field is very good for $\tG$ then it is known that $S(\tg_p)^{\tg_p}$ is generated by $S(\tg_p)^p$ and the mod $p$ reduction of $S(\tg)^{\tG}$, and that similar theorem holds for the invariants in the enveloping algebra. This is the most concise description of the algebra of invariants which we could hope for, and Kac asked whether it holds in general \cite{Kac}. Unfortunately a counterexample is given by the three dimensional solvable algebraic Lie algebra spanned by $\{h, a, b\}$ with non-zero brackets $[h,a] = a$ and $[h,b] = b$. In this example $S(\g)^G = \C$ and so after reducing modulo $p$ the element $a^{p-1} b$ is a example of an invariant which is not generated in this way.

\p  Despite failing in general, we shall show that this good behaviour is exhibited by centralisers in type $\sf A$ and $\sf C$, thus giving us a complete description of the symmetric invariant algebras in these cases. In the following we insist $p > 2$ if $\tG$ is of type $\sf C$.
\begin{thm}\label{main1} If $\tG$ is simple of rank $\ell$, of type $\sf A$ or $\sf C$ and $x \in \tg$ then
\begin{enumerate} 
\item{$\K[\tg^\ast_x]^{\tg_x}$ is free of rank $p^{\ell}$ over $\K[\tg^\ast_x]^p$;}
\item{$\K[\tg^\ast_x]^{\tG_x}$ is a polynomial algebra on $\ell$ generators;}
\item{$\K[\tg^\ast_x]^{\tg_x} \cong \K[\tg^\ast_x]^p \otimes_{(\K[\tg^\ast_x]^p)^{\tG_x}} \K[\tg_x^\ast]^{\tG_x}$.}
\end{enumerate}
\end{thm}
\p In \cite{PPY} Premet formulated the following conjecture: if $\tG$ is reductive of rank $\ell$ over an algebraically closed field of characteristic 0 then for $x \in \tg$ the invariant algebra $\K[\tg_x^\ast]^{\tG_x}$ is graded polynomial on $\ell$ generators. The authors went on to settle this conjecture in type $\sf A$. There is a straightforward reduction to the case $\tg$ simple and so part 2 of Theorem~\ref{main1} is a perfect analogue of the above conjecture in characteristic $p > 0$ for type $\sf A$ and $\sf C$. The proof in \cite{PPY} is quite different to ours, although we use the invariants which they constructed as a starting point. The conjecture has since been shown to fail; a long root vector in type $\sf E_8$ provides a counterexample \cite{Yak3}. Despite this, it is hopeful that it will hold when $\tG$ is simple of type $\sf B$ or $\sf D$, but the currently existing methods do not extend well to these cases.

\p Part 1 of Theorem~\ref{main1} has its roots in a theorem of Veldkamp \cite[Theorem~3.1]{Vel}, where a decomposition of the centre of the enveloping algebra was first proven in the reductive case with strong restrictions on $p$ (see also Theorem~\ref{main2} below). In \cite[Theorem~1.6]{Tan} those restrictions are weakened and the case of symmetric invariants is treated. Our proof is similar to that of \cite{Tan}, where the method is to satisfy the assumptions of a theorem of S. Skryabin, described in the previous section (Theorem~\ref{Skryabin}). In order to exhibit such polynomials we use an explicit presentation of type $\sf A$ invariants $x_1,...,x_N \in \K[\g_x^\ast]^{G_x}$, which was first conjectured in characteristic zero in \cite[Conjecture~4.1]{PPY}, and then confirmed in \cite[$\S 6$]{Yak2}. We should mention that the same invariants were derived using very different methods in \cite{BB} and we adopt their notation to reflect the combinatorial nature of our methods. Our first step is to reduce 
these invariants modulo $p$ and show that they are still invariant under the action of the modular analogue of the centraliser (Corollary~\ref{invariants}).

\p Having presented these invariants we go on to show that their Jacobian locus satisfies the required codimension 2 condition in Section~\ref{JacLoc}. To do this we use a line of reasoning not too dissimilar to \cite[$\S 3$]{PPY}. To complete the proof in type $\sf A$ it remains only to observe that these invariants are of the correct number. This follows from Elashvili's conjecture; see Theorem~\ref{elashvili} and $\S \ref{Index}.1$. The proof in type $\sf C$ is similar. In the notation of \ref{classnots} we have $\epsilon = -1$, $K \subseteq G$ and $e\in \h$ a nilpotent element. Using the $K_e$-stable decomposition $\g_e = \h_e \oplus \pp_e$ we may identify $\h_e^\ast$ as a $K_e$-module, with the annihilator in $\g_e^\ast$ of $\pp_e$. Therefore it is possible to consider the restrictions to $\h_e^\ast$ of the type $\sf A$ invariants. Generalising a well-known fact in the case $e = 0$ we show that the invariants of odd subscript restrict to zero (Corollary~\ref{vanishing1}) and accordingly we study the 
invariants $x_{2}|_{\h_e^\ast}, x_{4}|_{\h_e^\ast},...,x_{N}|_{\h_e^\ast}$. In Section~\ref{JacLoc} we show that the codimension 2 property is inherited quite conveniently from the type $\sf A$ case. Once again, that these invariants are of the correct number follows from Elashvili's conjecture.

\p Our next result uses the same methods as the above, however it is not a characteristic $p$ analogue of a known result in characteristic zero. In that respect it seems highly likely that the obvious non-modular versions of this theorem will hold. We mentioned previously that when $\epsilon = 1$ and $e \in \kk$ is nilpotent, the known methods do not lend themselves well to the study of $\K[\h_e^\ast]^{K_e}$. Despite this we may discuss the invariant algebra $\K[\pp_e^\ast]^{K_e}$ in some depth. 
\begin{thm}\label{main3}
Let $\epsilon = 1$ so that $K$ is of type $\sf B$ or $\sf D$. If $e \in \h$ is nilpotent, with associated partition $\lambda$ and $m := (N + |\{ i : \lambda_i \text{ odd}\}|)/2$. Then
\begin{enumerate}
\item{$\K[\pp_e^\ast]^{\h_e}$ is free of rank $p^{m}$ over $\K[\pp_e^\ast]^p$;}
\smallskip
\item{$\K[\pp_e^\ast]^{K_e}$ is a polynomial algebra on $m$ generators;}
\smallskip
\item{$\K[\pp_e^\ast]^{\h_e} \cong \K[\pp_e^\ast]^p \otimes_{(\K[\pp_e^\ast]^p)^{K_e}} \K[\pp_e^\ast]^{K_e}$.}
\end{enumerate}
\end{thm}
\p  We shall briefly discuss the method. Denote by $d_r$ the degree of $x_r$. The degrees $d_r$ have a very combinatorial description given in Section~\ref{elemInv}. The invariants we consider are $\{x_{r}|_{\pp_e^\ast} : r + d_r \text{ even}\}$; the other invariants restrict to zero by Corollary~\ref{vanishing1}. Once again the codimension 2 condition on the Jacobian locus is inherited from the type $\sf A$ case. The argument is very similar to the type $\sf C$ case and our proof shall be quite brief. For Theorem~\ref{main1}, we have the correct number of invariants to apply Skryabin's theorem thanks to Elashvili's conjecture. For Theorem~\ref{main3}, however, the index is not known and we make a detour in Section~\ref{GenEl} to calculate $\ind(\h_e, \pp_e)$ by exhibiting the existence of a generic stabiliser, making use of Lemma~\ref{genreg}. Some other results are obtained regarding indexes and generic stabilisers.

\p The remaining results of the second chapter concern the centres of enveloping algebras. We resume the setting of \ref{main1}. Then $\tg_x$ is a $p$-Lie algebra with $p$-operation $v \mapsto v^{[p]}$ induced by matrix multiplication. Let $U(\tg_x)$ denote the enveloping algebra of $\tg_x$ and $S(\tg_x)$ the symmetric algebra. Since $S(\tg_x)$ and $\K[\tg_x^\ast]$ are canonically isomorphic as $\tg_x$-algebras we may identify their invariant subalgebras.  Our next theorem offers a description of $U(\tg_x)^{\tG_x}$ and $U(\tg_x)^{\tg_x}$ analogous to the one given for symmetric invariant subalgebras. As mentioned earlier this theorem may be seen as an analogue of Veldkamp's theorem \cite[Theorem~3.1]{Vel}. Our method uses the theory of symmetrisation in characteristic $p$ outlined in Section~\ref{Symmetrisation}, and a standard filtration argument.
\begin{thm}\label{main2}  If $\tG$ is simple of rank $\ell$, of type $\sf A$ or $\sf C$ and $x \in \tg$ then
\begin{enumerate}
\item{$Z(\tg_x)$ is free of rank $p^\ell$ over $Z_p(\tg_x)$;}
\item{$U(\tg_x)^{\tG_x}$ is a polynomial algebra on $\ell$ generators;}
\item{$Z(\tg_x) \cong Z_p(\tg_x) \otimes_{Z_p(\tg_x)^{\tG_x}} U(\tg_x)^{\tG_x}$.}
\end{enumerate}
\end{thm}

\p We recall from section \ref{kacweisover} that the dimensions of simple modules of a restricted Lie algebra $\tg$ are bounded, and that the upper bound is finite and is denoted $M(\tg)$. The first Kac-Weisfeiller conjecture predicts that $$M(\tg) = p^{\frac{1}{2}(\dim\, \tg - \ind\, \tg)}.$$ We remind the reader that  in this section $\chr(\K) \neq 2$ whenever $\tG$ is of type $\sf C$. Our final invariant theoretic result is as follows: 
\begin{thm}\label{KW1}
Let $\tG$ be of type $\sf A$ or $\sf C$ and $x \in \tg$. The first Kac-Weisfeiler conjecture holds for $\tg_x$.
\end{thm}
\noindent We supply two proofs for this theorem. The first is quite general and conceptual, whilst the second is geometric and applies only to the case where $x$ is nilpotent and $\tG$ is of type $\sf A$. The second approach does have the advantage of showing that the baby Verma modules for $\g_x$ are generically simple.

\p In finite characteristic $Z(\tg_x)$ is a finitely generated, commutative integral domain. Therefore we may consider the algebraic variety $\Zas(\tg_x) := \Specm\, Z(\tg_x)$, known as the Zassenhaus variety \cite{Zas}, \cite{PT}. The geometry of $\Zas(\tg_x)$ is of some importance as it controls the representation theory of $\tg_x$. Theorem~\ref{KW1} allows us to make some fine deductions about the smooth locus of $\ZZ(\tg_x)$. In particular, we show that the smooth locus of coincides with the set of maximal ideals $\mm$ of $Z(\tg_x)$ such that $U(\tg_x)/U(\tg_x)\mm$ is isomorphic to the ring of $M(\tg_x) \times M(\tg_x)$ matrices over $\K$. This scheme of argument is quite general and was first sketched in \cite[Remark~5.2]{PS}.

\subsection{Sheets and centralisers}\label{shcnoverview}
\setcounter{parna}{0}
\p The next chapter of the thesis discusses some surprising geometric properties of the derived subalgebra of a centraliser and the sheets of a Lie algebra of type $\sf B$, $\sf C$ and $\sf D$.  Our arguments will work equally well in the orthogonal or symplectic case and at first we only need to assume that $\chr(\K) \neq 2$. Throughout we shall work with the notations $(\cdot, \cdot)$, $\epsilon$, $K$, $\sigma$ introduced in \ref{classnots}. We fix $e \in \Ni(\h)$ with partition $\lambda = (\lambda_1,...,\lambda_n) \in \mathcal{P}_\epsilon(N)$. As noted in Section~\ref{basisforthecent}, there is an involution $i \mapsto i'$ on the indexes $\{1,...,n\}$ and we adopt the convention that $\lambda_i = 0$ for $i= 0$ or $i > n$.

\p Our very first result of this chapter consists of a decomposition of the derived subalgebra $[\h_e \h_e]$. We prove this decomposition in a direct manner, by scrutinising the span of products of the basis elements which we introduced in Lemma~\ref{subbasis}. The precise description of the decomposition cannot be understood without a lot of notation which we postpone until Section~\ref{decomposingh_e}. The decomposition theorem is then proven in Section~\ref{decomposingderived}.

\p  Set $\h_e^\ab := \h_e/[\h_e\h_e]$ and write $c(\lambda) := \dim \, \h_e^\ab$. The main motivation for the
aforementioned endeavour is to obtain a combinatorial formula for $c(\lambda)$. The formula is given in terms of indexes $\{1,...,n\}$
fulfilling certain properties. We say that a pair $(i,i+1)$ with $1\leq i < n$ is a \emph{2-step} for $\lambda$ if $i = i'$, $i+1 = (i+1)'$ and
$\lambda_{i-1} \neq \lambda_i \geq \lambda_{i+1} \neq \lambda_{i+2}.$ The set of indexes $i$ such that $(i, i+1)$ is a 2-step for $\lambda$ is denoted
$\Delta(\lambda)$. The following formula for $\dim \, \h_e^\ab$ is proven in Corollary~\ref{expression}.
\begin{lem}
$c(\lambda) = \displaystyle{\sum}_{i>0} \big\lfloor \frac{\lambda_i - \lambda_{i+1}}{2} \big\rfloor + |\Delta(\lambda)|.$
\end{lem}
\smallskip

\p \label{definesingular} Before we move on we shall define a very special class of partitions. Let $\lambda \in \PP(N)$. We say that a 2-step $(i, i+1)$ is \emph{bad} if either of the following two conditions hold:
\begin{itemize}
\item{$\lambda_{i-1} - \lambda_i \in 2\N$;}
\item{$\lambda_{i+1} - \lambda_{i+2} \in 2\N$.}
\end{itemize}
Otherwise $(i, i+1)$ is \emph{good}. If $\lambda$ has a bad 2-step then it is called \emph{singular}, otherwise it is called \emph{non-singular}. Take note that
when $i = 1$ the difference $\lambda_{i-1} - \lambda_i$ is negative and the first condition may be omitted from the definition of a bad 2-step.

\p Our next task is to classify the sheets of $\h$ containing $e$. Recall the classification of sheets described in \ref{classifsheets}.
Our approach is to decide which $K$-pairs $(\li, \Oo)/K$ correspond to sheets containing $e$ under the classification given in \ref{classifsheets}. It is known that a given sheet $\SS$ contains
a unique nilpotent orbit, and that if $\SS$ has data $(\li, \Oo)/K$ then that orbit is $\Ind^\h_\li(\Oo)$. Therefore we must decide which $K$-pairs with a rigid orbit
induce to $\Oo_e$. We identify a subset of the $K$-pairs, including all pairs in which the orbit is rigid, and design an iterative technique for deciding whether or not
the pair induces to $\Oo_e$. We call this process of iteration \emph{the Kempken-Spaltenstein algorithm}. The algorithm itself is introduced in Section~\ref{KSalg}
in purely combinatorial terms.

\p The details of the algorithm are quite technical, however the results which follow are easy enough to present here. The algorithm takes, as its input,
the partition $\lambda$ and a so-called \emph{admissible sequence}. We establish a bijection between the maximal admissible sequences for $\lambda$ (upto rearrangement)
and the sheets of $\h$ containing $e$ (Corollary~\ref{bijection}) and go on to show that there is a unique maximal admissible sequence if and only if $\lambda$ is non-singular (Proposition~\ref{phising}).
We also show that the sequences of length $l$ correspond to sheets of rank $l$, which provides a workable method for calculating the maximal rank of sheets
containing $e$. This maximal rank is denoted $z(\lambda)$ (modulo Corollary~\ref{newz}), and in Section~\ref{KSmaxlength} we calculate a combinatorial formula for
$z(\lambda)$. After comparing $z(\lambda)$ and $c(\lambda)$ it is a simple task to show that $z(\lambda) = c(\lambda)$ if and only if $\lambda$ is non-singular (Corollary~\ref{maxseq}). All of
the results just stated can be summarised in the following:
\begin{thm}\label{firstimpressions}
The following are equivalent:
\begin{enumerate}
\item{the partition $\lambda$ is non-singular;}

\item{$c(\lambda) = z(\lambda)$;}

\item{$e(\lambda)$ lies in a unique sheet;}
\end{enumerate}
\end{thm}
\noindent We remind the reader that this theorem holds in any characteristic $\neq 2$. If the characteristic is zero or $\gg 0$ then these criteria are
equivalent to the assertion that $e$ is a non-singular point on the variety $\h^{(\dim\, \h_e)}$. The theorem and the previous remark shall be proven together in
Corollary~\ref{izosim}. When we refer to non-singular nilpotent orbits, we shall mean those which lie in a unique sheet.

\p\label{izosintroduct} Our motivation for considering the property $c(\lambda) = z(\lambda)$ came directly from the theory of finite $W$-algebras, as we shall soon explain.
There is, however, one very nice consequence of this property which we prove in the final section of Chapter~\ref{derivedchapter}. Using a simple argument, Izosimov proved in \cite{Izos}
that, for any $x \in \h$ lying in a unique sheet $\SS$, the space $[\h_x \h_x]$ is orthogonal to $T_x \SS$ with respect to the Killing form. He went on
to show that in type $\sf A$ these two spaces are mutual orthogonal complements with respect to $\kappa$, and conjectured that the same should hold
in other classical types. Using the previous theorem we are able to confirm his conjecture:
\begin{thm}
For each $x \in \h$ lying in a unique sheet $\SS$ we have $T_x\SS = [\g_x, \g_x]^\perp.$
\end{thm}
\noindent As one would expect, we reduce to the case of a nilpotent element $e$ and conclude that $c(\lambda) = z(\lambda)$. It is then a small matter to show that this last equality
implies that the dimensions of $T_e\SS$ and $[\h_e \h_e]$ sum to $\dim\,  \h$.

\subsection{Commutative quotients of $W$-algebras}
\setcounter{parna}{0}

\p In the final chapter of this thesis we work over a field of zero characteristic. Continue to fix a nilpotent element
$e \in \Ni(\h)$ with partition $\lambda \in \PP(N)$, and choose an $\sl_2$-triple $\phi =(e,h,f)$ containing $e$. Our
aim is to study the one dimensional representations of the finite $W$-algebra $U(\h, e)$. The importance of these representations was
alluded to in \ref{modularhistory} and \ref{quotientsintro}, whilst in \ref{quotientsintro} we observed that they are parameterised by the maximal spectrum of the maximal
abelian quotient $\EE := \EE(\h,e) = \Specm \, U(\h,e)^\ab$.

\p In \cite{Pre4} Premet showed that the algebra $U(\sl_N, e)^\ab$ is a polynomial algebra for any nilpotent element $e$. This implies that the one dimensional representations
are parameterised by an affine space $\mathbb{A}_\C^d$. His approach made use of the explicit presentation for $U(\sl_N, e)$ given by Brundan and Kleschev, which is unavailable
outside type $\sf A$ (the presentation is derived by realising these $W$-algebras as shifted Yangians \cite[Theorem~10.1]{BK}).
Another ingredient to his method was Theorem~\ref{premettheorem} which relates the spectrum $\EE$ to the Katsylo sections of the sheets containing $e$.
It follows immediately from that theorem that if $e$ lies in more than one sheet then $\EE$ is reducible, hence cannot be isomorphic to an affine space. Our main result of the
chapter states that the converse holds for classical Lie algebras:
\begin{thm}
$e$ is non-singular if and only if $\EE$ is isomorphic to an affine space $\mathbb{A}_\C^d$ for some $d$. If this is the case then $d = c(\lambda)$.
\end{thm}

\p The proof of the theorem is completed in Theorem~\ref{class}. Our method is to prove a general criterion to ensure the polynomiality of $U(\h,e)^\ab$.
For the proof of that criterion we do not need to assume that our Lie algebra is classical and so we work with the Lie algebra $\tg$ of an arbitrary reductive group throughout Section~\ref{polynomialityofquot}.
The criterion states that if the orbit of $e$ is not rigid, and the maximal rank of sheets containing $e$ coincides with $c(e) = \codim_{\tg_e}\,  [\tg_e \tg_e]$ then $U(\tg,e)^\ab$ is a polynomial algebra in
$c(e)$ variables. Given Theorem~\ref{firstimpressions} our method of proof for the previous theorem should now be clear.

\p In \ref{quotientsintro} we noted that the stabiliser group $K_\phi$ acts upon $U(\h,e)^\ab$ by algebra automorphisms, and that $K_\phi^\circ$ fixes every two-sided ideal.
It follows that $\Gamma$ acts upon $\EE$, and in \ref{bigtheoremF}
we established a bijection between the multiplicity free primitive ideals of $U(\h)$ and the fixed point space $\EE^\Gamma$. Another motivation for studying this space comes
from \cite{Lo4} where it is explained that quantisations of the $K$-equivariant covers of $\Oo_e$ are in bijective correspondence with $\EE^\Gamma$. Once again we are interested
in the geometry of $\EE^\Gamma$. The answer is surprisingly uniform:
\begin{thm}
The variety $\EE^\Gamma$ is always isomorphic to an affine space $\mathbb{A}_\C^d$.
\end{thm}
\noindent The proof is given in Theorem~\ref{class1} and the exact value of $d$ is given in all cases. We postpone our description of $d$ since it depends upon some
combinatorial properties of $\lambda$ which are not so important right now. The majority of Section~\ref{computingvarieties} is spent preparing to prove the theorem, where
several technical lemmas are required.

\p In the final section we apply our new understanding of $\EE^\Gamma$ to the multiplicity free primitive ideals of $U(\h)$. This work has very classical consequences and our final
result follows in the tradition of the classification of completely prime primitive (CPP) ideals mentioned in \ref{goldierank}. The traditional approach to such a classification comes in two parts:
firstly one would try to show that whenever the associated variety of a CPP ideal is the closure of an induced orbit $\Oo_e$, there is a CPP ideal in the enveloping algebra of a Levi subalgebra
whose associated variety is the closure of a rigid orbit which itself induces to $\Oo_e$ and such that the latter ideal induces to the former. Secondly one would try to describe the CPP ideals whose
associated variety is the closure of a rigid orbit. Together these two steps would give a fairly satisfactory inductive classification of the CPP ideals of $U(\h)$. Unfortunately, outside type $\sf A$ there
are CPP ideals whose associated variety is the closure of an induced orbit, which cannot themselves be induced from any CPP ideal (see \cite{BJ} for more detail).

\p As we noted in \ref{bigtheoremF} the multiplicity free ideals are all CPP. We shall show that for these ideals, the first part of the classification described above works as desired.
\begin{thm} Let $I\in\MF_{\Oo_e}$ be a multiplicity-free primitive ideal such that $\Oo_e$ is not rigid. Then there exists a proper parabolic subalgebra $\mathfrak p$ of $\g$ with a Levi decomposition $\pp = \li \ltimes \n$, a rigid nilpotent orbit $\Oo_{e_0}$ in $\li$ and a completely prime primitive ideal $I_0 \in \Prim_{\Oo_{e_0}}$ such that the following hold:
\begin{enumerate}
\item{$\VA(I_0) = \overline{\Oo}_{e_0}$;}
\item{$\Oo_e = \Ind^\h_\li(\Oo_{e_0})$;}
\item{$I = \Ind^\h_\pp(I_0)$.}
\end{enumerate}
\end{thm}
\noindent The theorem is restated in \ref{E} and proven over the course of Section~\ref{CPPideals}.

\chapter{The Invariant Theory of Centralisers}
\section{The elementary invariants}\label{elemInv}
\renewcommand{\theparno}{\thesection.\arabic{parno}}
\renewcommand{\p}{\refstepcounter{parno}\noindent\textbf{\theparno .} \space} 
\setcounter{parno}{0}

\p Throughout this chapter the characteristic of $\K$ is assumed to be positive unless otherwise stated. We assume the notation of \ref{genlinnots} so that $G = GL(V)$. For the entirety of this section fix $e \in \Ni(\g)$ with associated partition $\lambda$. Our first task is to introduce certain distinguished elements of $\K[\g_e^\ast]^{G_e}$. These invariants are defined combinatorially and their history shall be discussed in Remark~\ref{comparisonrem} below. The first step is to introduce \emph{the sequence of invariant degrees}:
\begin{eqnarray*}
(d_1, ..., d_N) := (\overbrace{1,1,...,1}^{\lambda_1 \text{ 1's}},\overbrace{2,...,2}^{\lambda_2 \text{ 2's}},...,\overbrace{n,...,n}^{\lambda_n \text{ n's}})
\end{eqnarray*}

\p Suppose $\lambda$ has length $n$ and write $\lambda = (\lambda_1,..., \lambda_n)$. A \emph{composition of} $\lambda$ is a finite sequence $\mu = (\mu_1, ..., \mu_n)$ with $0 \leq \mu_k \leq \lambda_k$ for $k = 1,...,n$. We write $|\mu| = \sum \mu_k$ and let $l(\mu)$ denote the number of $k$ for which $\mu_k$ is nonzero. If $1 \leq k \leq l(\mu)$ then $i_k^{(\mu)}$ denotes the $k^{\text{th}}$ index such that $\mu_{i_k}^{(\mu)} \neq 0$ (ordered so that $i_1^{(\mu)} \leq i_2^{(\mu)} \leq \cdots$). If our choice of composition $\mu$ is clear from the context then we write $i_k = i_k^{(\mu)}$.

\p Fix $r \in \{1,...,N\}$ and set $d = d_r$. Denote by $\mathcal{C}_r$ the set of compositions of $\lambda$ fulfilling $|\mu| = r, l(\mu) = d$ and let $\Sf_d$ denote the symmetric group on $d$ objects. If a choice of $w \in \Sf_d$ and a composition $\mu \in \mathcal{C}_r$ is implicit then we write
\begin{eqnarray*}
s_k = \lambda_{i_{wk}^{(\mu)}} - \lambda_{i_k^{(\mu)}} + \mu_{i_k^{(\mu)}} -1
\end{eqnarray*}
Since $\g_e \cong (\g_e^\ast)^\ast$, the ring of regular functions $\K[\g_e^\ast]$ is spanned by monomials in the linear unary forms $\xi_i^{j,s} : \g_e^\ast \mapsto \K$. For $r=1,...,N$ define functions $\Theta_r : \Sf_{d_r} \times \mathcal{C}_r \mapsto \K[\g_e^\ast]$ by the rule
\begin{eqnarray*}
\Theta_r (w, \mu) = \sgn(w)\xi_{i_1}^{i_{w1}, s_1} \cdots \xi_{i_d}^{i_{wd},s_d}.
\end{eqnarray*}

\p Now the elementary invariants may be defined
\begin{eqnarray}\label{xrreform}
x_r = \sum_{(w,\mu) \in S_d\times \mathcal{C}_r} \Theta_r(w,\mu) \in \K[\g_e^\ast]. 
\end{eqnarray}
Each $x_r$ is homogeneous of degree $d_r$.

\begin{rem}\label{comparisonrem}
\rm \begin{enumerate}
\item{Let's review a brief history of these polynomials. The definition of $x_r$ first appeared in \cite[Conjecture~4.1]{PPY} over $\C$, where it was proposed that this would be the explicit formula for the invariants being studied there, which were denoted ${}^eF_1,...,{}^eF_N$. The same invariants were derived using different methods in \cite{BB}, defined over a different basis, but presented in our notation $x_1,...,x_N$. In \cite[$\S 6$]{Yak2} these bases were shown to coincide so \cite[Conjecture~4.1]{PPY} is equivalent to ${}^eF_i = x_i$, and in \cite[Theorem~15]{Yak2} the conjecture was confirmed.}

\item{It is instructive to examine what happens in the extremal examples of nilpotent orbits: when $e = 0$ or when $e$ is regular. Fix $1\leq r \leq N$. In the former case $$\lambda = (\overbrace{1,1,...,1}^{N \text{ 1's}})$$ and $s_k = \mu_{i_k} -1 = 0$ for $1\leq k \leq d_r$. Furthermore $\xi_i^{j,0}$ is just the $i,j$ elementary matrix with a 1 in the $i^\text{th}$ row and $j^\text{th}$ column, and zeros elsewhere. It follows that the polynomial $x_r$ is just the dual to the $r^\text{th}$ standard  matrix invariant obtained as a coefficient in the characteristic equation (see \cite[Theorem~1.4.1]{Smi} for example).

Now consider the case $e$ regular. By a general result of Dixmier we know that $\g_e$ is abelian. More precisely \cite[Theorem~2]{Yak2} tells us that $\g_e$ is spanned by matrix powers of $e$. Then $\mathcal{C}_r$ contains a single element $(r)$, and $d = d_r = 1$ so that $\Sf_d$ is trivial. Hence $x_r = \xi_1^{1, r-1} = e^{r-1}$ and the invariants are just a basis for $\g_e$.}
\end{enumerate}
\end{rem}

\p Note that all groups, spaces and maps discussed in the current section may equally well be defined over $\C$. The following result is contained in \cite[Proposition~0.1]{PPY}. See also the appendix of \cite{PPY} for a proof using elementary methods, due to Charbonnel.
\begin{thm}\label{BrB}
The polynomials $x_1,...,x_N$ are $G_e$-invariant when defined over $\C$.
\end{thm}
We use reduction modulo $p$ to obtain the following.
\begin{cor}\label{invariants}
The polynomials $x_1,...,x_N$ are $G_e$-invariant when defined over $\K$.
\end{cor}

\begin{proof}
We proceed with all maps and spaces defined over $\C$. The matrix $\xi_i^{j,s}$ is nilpotent when $i \neq j$ or $s>0$. The elements $\xi_i^{i,0}$ are semisimple. We shall denote by $M_N(\C)$ the ring of $N\times N$ matrices over $\C$ with the usual associative multiplication and denote by $M_N$ the $\Z$-lattice in $M_N(\C)$ arising from the inclusion $\Z \subseteq \C$. Let $t$ be an indeterminate and denote by $M_N(\C)[t]$ the ring of polynomials in $t$ with coefficients in $M_N(\C)$. Let $i \neq j$, let $\lambda_j - \min(\lambda_i, \lambda_j) \leq s < \lambda_j$ and $0 < r < \lambda_i$. Consider the matrices
\begin{eqnarray*}
g_{t} = 1 + t\xi_i^{j,s} \text{ and } h_{t} = 1 + t\xi_i^{i,r}.
\end{eqnarray*}
They both lie in $M_N(\C)[t]$ and the reader may verify that their inverses are
\begin{eqnarray*}
g_{t}^{-1} = 1 - t\xi_i^{j,s} \text{ and } h_t^{-1} = 1 + \sum_{k=1}^{\infty} (-1)^k t^k \xi_i^{i,kr}.
\end{eqnarray*}
Since all coefficients are integral we may consider the above maps to be elements $M_N[t]$. The elements $x_r$ are also elements of $S(M_N)$, so we get $\Ad(g_t) x_r \in S(M_N[t])$. Hence the expression $\Ad(g_t) x_r$ can be written $A_{0,r} + t A_{1,r} + t^2 A_{2,r} + \cdots$ where $A_{k,r} \in S(M_N)$. By Theorem~\ref{BrB} we see that $x_r$ is $g_t$-stable over $\C$ so $A_{0,r} = x_r$ and $A_{k,r} = 0$ for all $k > 0$. Let $\tilde{A}_{k,r} = A \otimes 1 \in S(\gl_N[t]) \otimes_{\Z} \K$. Then for each $r$ we still have $\tilde{A}_{0,r} = x_r \otimes 1 \in S(M_N)\otimes_{\Z} \K$ and $\tilde{A}_{k,r}= 0 $ for $k > 0$. But $x_r \otimes 1$ is just $x_r$ defined over $\K$. If we now allow $t$ to vary over $\K$ then we conclude that $g_t$ fixes each $x_r$ over $\K$. Similarly each $h_t$ fixes each $x_r$ over $\K$. We proceed over $\K$. Let $U$ denote the subgroup of $G_e$ generated by the affine lines $\{1 + t\xi_i^{j,s}: t \in \K\}$ and $\{1 + t \xi_i^{i,r}: t \in \K\}$ where $i,j,s$ and $r$ vary in the range 
$i \neq j$, $\lambda_j - \min(\lambda_i, \lambda_j) \leq s < \lambda_j$ and $0 < r < \lambda_i$. The above shows that the generators of $U$ fix each $x_r$ over $\K$; so must the closure $\bar{U}$.

Consider the group $T \subset G$ in which the elements act on each $V[i]$ $(i = 1,...,n)$ by an arbitrary nozero scalar. The group is toral, closed, contained in $G_e$ and $\Lie(T) = \spn\{ \xi_i^{i,0} : 1\leq i \leq n\}$. We shall show that $T$ fixes each $x_r$. Recall that $\xi_{i_k}^{i_{wk}, s_k} \in \Hom(V[i_k], V[i_{wk}])$ so if $h \in T$ is defined by $hw_i = t_i w_i$ for $t_i \in \K$ then
\begin{eqnarray*}
\Ad(h)\xi_{i_k}^{i_{wk}, s_k} = (t_{i_{k}}^{-1}t_{i_{wk}}) \xi_{i_k}^{i_{wk}, s_k}
\end{eqnarray*}
and
\begin{eqnarray*}
\Ad(h)x_r = \sum_{(w,\mu) \in \Sf_d \times \mathcal{C}_r}  \sgn(w)\overbrace{(\prod_{k=1}^dt_{i_{k}}^{-1}t_{i_{wk}})}^{=1} \xi_{i_1}^{i_{w1}, s_1} \cdots \xi_{i_d}^{i_{wd}, s_d} = x_r
\end{eqnarray*}
Now the group $\langle \bar{U}, T \rangle$ generated by $\bar{U}$ and $T$ is closed, connected and its Lie algebra contains a basis for $\g_e$. Hence $\langle \bar{U}, T \rangle = G_e$. Both $\bar{U}$ and $T$ fix $x_r$; so must $G_e$.
\end{proof}

\p Since $\K[\g_e^\ast]$ is spanned by monomials in the $\xi_i^{j,s}$, the involution $\sigma : \g_e \mapsto \g_e$ introduced in \ref{classnots2} extends uniquely to a $\K$-algebra automorphism of $\K[\g_e^\ast] \mapsto \K[\g_e^\ast]$ which we shall also denote by $\sigma$. This extension is clearly involutory. In \ref{sigmaint} we noted that
\begin{eqnarray*}
\sigma(\xi_i^{j,\lambda_j - 1 - s}) = \varepsilon_{i,j,s}\xi_{j'}^{i', \lambda_i - 1 - s}
\end{eqnarray*}
\noindent The following proposition shall be pivotal to our understanding of the symmetric invariants for centralisers in classical cases.
\begin{prop}\label{sigmaxr}
$\sigma(x_r) = (-1)^r x_r$
\end{prop}
\begin{proof}
Fix $(w,\mu) \in \Sf_{d} \times \mathcal{C}_r$. We shall show that if $\Theta_r(w,\mu) \neq 0$ then there exists a pair $(w', \mu') \in \Sf_d \times \mathcal{C}_r$ such that $\sigma(\Theta_r(w,\mu)) = (-1)^{r}\Theta_r(w', \mu')$ and $\Theta_r(w',\mu') \neq 0$. In view of the definition of $x_r$ the proposition shall follow. By the above formula we have
\begin{eqnarray*}
\sigma(\Theta_r(w,\mu)) &=& \sgn(w)\sigma(\xi_{i_1}^{i_{w1}, s_1} \cdots \xi_{i_d}^{i_{wd}, s_d})\\ &=& (\prod_{k=1}^d\varepsilon_{i_k, i_{wk}, \lambda_{i_k} - \mu_{i_k}})\sgn(w) \xi_{(i_{w1})'}^{(i_{1})', \mu_{i_1} - 1} \cdots \xi_{(i_{wd})'}^{(i_{d})',\mu_{i_d} - 1}
\end{eqnarray*}
We must examine the coefficient 
\begin{eqnarray*}
\prod_{k=1}^d\varepsilon_{i_k, i_{wk}, \lambda_{i_k} - \mu_{i_k}} &=& \prod_{k=1}^d (-1)^{\lambda_{i_{wk}} - \lambda_{i_k} + \mu_{i_k}} \partial_{i_k \leq i_k'} \partial_{i_{wk} \leq (i_{wk})'}\\ &=& (\prod_{k=1}^d (-1)^{\lambda_{i_{k}}} \partial_{i_{k} \leq i_{k}'})(\prod_{k=1}^d (-1)^{\lambda_{i_{wk}}} \partial_{i_{wk} \leq (i_{wk})'})(\prod_{k=1}^d (-1)^{\mu_{i_k}})
\end{eqnarray*}
By definition $|\mu| = r$ implies $\prod_{k=1}^d (-1)^{\mu_{i_k}} = (-1)^{\sum_{k=1}^d \mu_{i_k}} = (-1)^r$. Since $w$ is a permutation $\prod_{k=1}^d (-1)^{\lambda_{i_{k}}} \partial_{i_{k} \leq i_{k}'} = \prod_{k=1}^d (-1)^{\lambda_{i_{wk}}} \partial_{i_{wk} \leq i_{wk}'}$ and the product $\prod_{k=1}^d\varepsilon_{i_k, i_{wk}, \lambda_{i_k} - \mu_{i_k}}$ reduces to $(-1)^r$.

It remains to prove that $\Theta_r(w', \mu') = \sgn(w) \xi_{(i_{w1})'}^{(i_{1})', s_1+\lambda_{i_{1}} - \lambda_{i_{w1}}} \cdots \xi_{(i_{wd})'}^{(i_{d})',s_d+\lambda_{i_{d}} - \lambda_{i_{wd}}}$ for some choice of $(w', \mu') \in \Sf_d\times \mathcal{C}_r$. For $k = 1,...,d$ let $j_k = (i_{wk})'$. Let $\mu'$ be the composition of $\lambda$ with nonzero entries in positions indexed by the $j_k$ such that $\mu'_{j_k} = \mu_{i_k} + \lambda_{i_{wk}} - \lambda_{i_k}$. We have $l(\mu') = d$ by definition. Furthermore $$|\mu'| = \sum_{i=1}^k \mu'_{j_k} = \sum_{i=1}^k (\mu_{i_k} + \lambda_{i_{wk}} - \lambda_{i_k}) = \sum_{i=1}^k \mu_{i_k} = |\mu| = r.$$ To see that $0 \leq \mu'_k \leq \lambda_k$ we use the fact that $\Theta_r(w,\mu) \neq 0$. We have $0 \leq \lambda_{i_{wk}}-1-s_k \leq \min(\lambda_{i_k}, \lambda_{i_{wk}})$ and by definition of $s_k$ we have $0 \leq \lambda_{i_k} - \mu_{i_k} \leq \lambda_{i_{wk}}$. The left hand inequality gives us $\mu'_{j_k} \leq \lambda_{i_{wk}} = \lambda_{j_k}$, whilst the right 
hand inequality tells us that $0 \leq \lambda_{i_{wk}} - (\lambda_{i_k} - \mu_{i_k}) = \mu'_{j_k}$. It follows that $0 \leq \mu'_k \leq \lambda_k$ for $k=1,...,n$ so that $\mu'$ is a composition of $\lambda$. We have shown that $\mu' \in \mathcal{C}_r$.

We now aim to define $w'\in \Sf_d$. If we define an element $\omega \in \Sf_n$ by $\omega (i_k) = i_{wk}$ for $k = 1,...,d$ and $\omega( i ) = i$ for all $\mu_i = 0$, then $\sgn(\omega) = \sgn(w)$. The element $\nu$ of $S_n$ which sends $(i_{wk})'$ to $(i_k)'$ is $' \circ \omega^{-1} \circ '$, where $\circ$ denotes composition of permutations in $\Sf_n$. But $'$ is an involution so $\nu$ and $\omega^{-1}$ are conjugate in $\Sf_n$ implying $\sgn(\nu) = \sgn(\omega)$. Now $\nu$ preserves the set $\{j_1, ..., j_d\}$ and acts trivially on its complement in $\{1,...,n\}$, therefore we may take $w'$ to be the unique element of $\Sf_d$ such that $\nu(j_k) = j_{w'k}$ for all $k$. This element will have $\sgn(w') = \sgn(\nu) = \sgn(\omega) = \sgn(w)$.

From the above we get 
\begin{eqnarray*}
j_{w' k} = \nu ((i_{wk})') = (i_{k})'\\ \mu'_{j_{k}} = \mu'_{(i_{wk})'} = \mu_{i_k} + \lambda_{i_{wk}} - \lambda_{i_k}\\
\lambda_{j_{w'k}} = \lambda_{i_k}\\ \lambda_{j_k} = \lambda_{i_{wk}}
\end{eqnarray*}
so that $$\xi_{j_k}^{j_{w'k}, \lambda_{j_{w'k}} - \lambda_{j_k} + \mu'_{j_k} - 1} = \xi_{(i_{wk})'}^{(i_{k})',\mu_{i_k} - 1}.$$ It follows that $(w', \mu')$ is the required pair. Since $\sigma$ is non-degenerate on $\g_e$ and $\Theta_r(w',\mu')$ is a product of terms $\sigma(\xi_{i_k}^{i_{wk}, s_k})$, each with $\xi_{i_k}^{i_{wk}, s_k} \neq 0$, we conclude that $\Theta_r(w',\mu') \neq 0$.
\end{proof}

\p The following corollary is not strictly needed in what follows, however it gives a nice picture of how the elementary invariants behave when restricted. It also provides a generalisation of well known behaviour in the case $e = 0$.
\begin{cor}\label{vanishing1}
The following are true:
\begin{enumerate} 
\item{$x_{r}|_{\mathfrak{k}_e^\ast} = 0$  for $r$ odd;}
\item{$x_{r}|_{(\mathfrak{\pp}_e)^\ast} = 0$ for $r + d_r$ odd.}
\end{enumerate}
\end{cor}
\begin{proof}
Let $\BB_0$ be a basis for $\h_e$ and $\BB_1$ a basis for $\pp_e$, so that $\BB_0 \cup \BB_1$ is a basis for $\g_e$. Then the monomials in $\BB_0 \cup \BB_1$ form an eigenbasis for the action of $\sigma$ on $\K[\g_e^\ast]$. The eigenvalues are $\pm1$ depending on the parity of the number of factors coming from $\BB_1$ in a given monomial. 

Now fix $r=1,...,N$ and write $x_r$ in the above eigenbasis. If $r$ is odd then by Proposition~\ref{sigmaxr} we have $\sigma x_r = - x_{r}$. It follows that there are an odd number of factors from $\BB_1$ in each monomial summand of $x_r$. In particular that number of factors is non-zero. Each of these factors vanishes on $\h_e^\ast$; so must $x_r$, whence 1.

If $r + d_r$ is odd then there are two possibilities: either $r$ is odd and $d_r$ even or vice versa. Assume the former so that $x_r$ is a sum of monomials of even degree and $\sigma x_r = - x_r$. There must be an odd number of factors from $\BB_1$ in each monomial and since each monomial has even degree there is also an odd number of factors from $\BB_0$ in each monomial. In particular the number of factors coming from $\BB_0$ is non-zero. Each of these factors restrict to zero on $\pp_e^\ast$ and so must $x_r$. If, on the other hand, $r$ is even and $d_r$ is odd then $\sigma x_r = x_r$ so each monomial summand of $x_r$ contains an even number of factors from $\BB_1$. Since each such monomial has odd degree, there is a non-zero number of factors from $\BB_0$ in each monomial. Again each of these factors restrict to zero on $\pp_e^\ast$, and so does $x_r$. This completes the proof.
\end{proof}

\section{Jacobian loci of the invariants}\label{JacLoc}

\subsection{The general linear case}\label{JacLoc1}
\renewcommand{\theparna}{\thesubsection.\arabic{parna}}
\renewcommand{\p}{\refstepcounter{parna}\noindent\textbf{\theparna .} \space} 
\setcounter{parna}{0}

\p If $f_1,...,f_l \in \K[\g_e^\ast]$ and $U \subseteq \g_e^\ast$ is a subspace then we denote by $$J_U(f_i : i = 1,...,l)$$ the Jacobian locus of the $f_i$ in $U$, ie. those $\gamma \in \g_e^\ast$ for which the restricted differentials $d_\gamma f_1|_U,...,d_\gamma f_l|_U $ are linearly dependent. Make the notation $J = J_{\g_e^\ast}(x_r : r=1,...,N)$. We aim to show that the condition on the codimension of the Jacobian locus in Skryabin's theorem is satisfied for these invariants. We shall prove the following.
\begin{prop}\label{Codimension 2}
\emph{codim}$_{\g_e^\ast} J \geq 2$.
\end{prop}

\p In the style of \cite[\S~3]{PPY} we proceed by identifying a 2-dimensional plane in $\g_e^\ast$ intersecting $J$ only at 0. The proposition will then quickly follow. Calculating the differentials of the invariants polynomials explicitly is an unappealing task, and so we infer the necessary properties implicitly. Fix $\gamma \in \g_e^\ast$ and consider the polynomial
\begin{eqnarray*}
x_r^\gamma &: &\mathfrak{g}_e^\ast \mapsto \K;
\\ & & v \mapsto x_r(\gamma + v).
\end{eqnarray*}
By expanding the definition of the differential and the definition of $x_r^\gamma$ it is easy to see by direct comparison that
\begin{eqnarray}\label{xrdiff}
x_r^\gamma = \text{ constant } + d_{\gamma} x_r + \text{ higher degree polynomials in } \K[\g_e^\ast]
\end{eqnarray}

Hence in order to show that $d_{\gamma} x_1,..., d_{\gamma} x_N$ are linearly independent it will suffice to show that the linear components of $x_1^\gamma|_U,...,x_N^\gamma|_U$ are linearly independent for some appropriate choice of vector space $U \subseteq \g_e^\ast$.

\p Let us define
\begin{eqnarray}\label{alphabeta}
\alpha = \sum_{1 \leq i\leq n} a_i (\xi_i^{i,\lambda_i-1})^\ast \text{ and } \beta = \sum_{i=2}^{n} (\xi_{i-1}^{i,\lambda_{i}-1})^\ast
\end{eqnarray}
where the coefficients $a_i \in \K$ are subject to the constraint that $a_i = a_j$ implies $i=j$. These elements were first presented by Yakimova in \cite{Yak1} as examples of regular elements of $\g_e^\ast$. The following lemma is based upon \cite[Lemma~4.2]{BB}.
\begin{lem}\label{beta}
For all $y \in \K^\times$ we have $y \beta \in \g_e^\ast \backslash J$.
\end{lem}
\begin{proof}

Let $U = \spn\{(\xi_i^{1,s})^\ast: 1 \leq i \leq n, \lambda_1 - \lambda_i \leq s < \lambda_1\}\subseteq \g_e^\ast$ and let
\begin{eqnarray*}
v = \sum_{i=1}^n \sum_{s=\lambda_1 - \lambda_i}^{\lambda_1 - 1} c_{i,s}( \xi_i^{1,s} )^\ast
\end{eqnarray*}
be an arbitrary element of $U$, with $c_{i,s} \in \K$. We have
\begin{eqnarray}\label{xievaluate}
\xi_i^{j,s}(y\beta + v) = \left\{ \begin{array}{ll}
         y & \mbox{ if $i=j-1$ and $s = \lambda_j -1$}\\
        c_{i,s} & \mbox{ if $j = 1$}\\
        0 & \mbox{ otherwise} \end{array} \right.
\end{eqnarray}

Observe that if $1 \leq r \leq \lambda_1$ then $x_r = \sum_i \xi_i^{i, r-1}$ so that $x_r(y \beta + v) = c_{1,r-1}$ and $x_r^{y\beta}|_U= \xi_1^{1,r-1}$. If $n=1$ then $\lambda_1 = N$ and we have shown that the linear terms of the $x_1^{y\beta}|_U, ..., x_N^{y\beta}|_U$ are linearly independent and by formula (\ref{xrdiff}) we are done. Assume not, so that we may choose $\lambda_1 < r \leq N$ and by definition we shall have $d := d_r > 1$. With the definition of $x_r$ in mind let us consider the possible choice of composition $\mu$ of $\lambda$, and permutation $w \in \Sf_d$ such that the monomial $\xi_{i_1}^{i_{w1}, s_1} \cdots \xi_{i_d}^{i_{wd}, s_d}$ evaluates as nonzero at $v$. We require that $\xi_{i_k}^{i_{wk}, s_k}(v) \neq 0$ for $k=1,...,d$.

For any $w \in \Sf_d$ we can be sure that $wk \leq k$ for some $k \in \{1,...,d\}$. Fix such a $k$. Due to (\ref{xievaluate}) we have $i_{wk} = 1$. Clearly $wl \neq 1$ for $l \neq k$ so, again by (\ref{xievaluate}), $i_{wl} = i_l + 1$ for all $l \neq k$. This gives us $i_{wk} = 1$, $i_{w^2k} = 2$, $i_{w^3k} = 3,...$ and $i_{w^dk} = i_{k} = d$. Furthermore by (\ref{xievaluate}) we must have $\mu_1 = \lambda_1, \mu_2 = \lambda_2, ..., \mu_{d-1} = \lambda_{d-1}$. Since $|\mu | = r$ and $l(\mu) = d$ we have $\mu_d = r - \sum_{m=1}^{d-1} \lambda_m$. The upshot is that $w$ must be the $d$-cycle $(123\cdots d)$ and that the composition $\mu$ must be $(\lambda_1, \lambda_2,..., \lambda_{d-1}, r - \sum_{k=1}^{d-1} \lambda_k, 0, 0,..., 0)$.

We make the notation $t_r = r -\sum_{m=1}^{d-1} \lambda_m \in \{1,...,\lambda_d\}$. Since $\sgn(123\cdots d) = (-1)^{d-1}$ we have shown that $x_r(y\beta +v) = (-y)^{d-1}  c_{d, t_r-1}$. Thus $$x_r^{y\beta}|_U = (-y)^{d-1} \xi_d^{1,t_r-1}|_U$$ for all $r=1,...,N$. These linear functions are clearly linearly independent and the lemma is proven.
\end{proof}

\p For the next lemma we shall need to make use of a decomposition of $\g_e$ similar to the well known triangular decomposition of $\g$. We define 
\begin{eqnarray*}
\mathfrak{n}^- &:=& \spn\{\xi_i^{j,s} : i < j\} \\
\smallskip
\mathfrak{h} &:=& \spn\{\xi_i^{j,s} : i = j\} \\
\smallskip
\mathfrak{n}^+ &:=& \spn\{\xi_i^{j,s} : i > j\}
\end{eqnarray*}
where $\lambda_j - \min(\lambda_i, \lambda_j) \leq s < \lambda_j$ in all three of the above. If we order the basis $\{e^s w_i\}$ so that $e$ is in Jordan normal form then $\mathfrak{n}^-$ is strictly lower triangular and $\mathfrak{n}^+$ is strictly upper triangular, and $\g_e = \n^- \oplus \hh \oplus \n^+$. Of course $\hh$ is not actually a torus unless $e = 0$, however it is proven over $\C$ in \cite[$\S 5$]{Yak1} that $\hh$ is a generic stabiliser in $\g_e$, and that $(\g_e)_\alpha = \hh$. We shall see later that this generic stabiliser also exists when we work over $\K$ (Theorem~\ref{genstabal}).

\p There is a dual decomposition $$\g_e^\ast = (\mathfrak{n}^-)^\ast \oplus \mathfrak{h}^\ast \oplus (\mathfrak{n}^+)^\ast$$ where $\hh^\ast$ is defined to be the annihilator of $\n^-\oplus \n^+$ in $\g_e^\ast$, and similar for $(\n^-)^\ast$ and $(\n^+)^\ast$. We have $\alpha \in \mathfrak{h}^\ast$ and $\beta \in (\mathfrak{n}^-)^\ast$.
\begin{lem}\label{alpha}
There exists $g \in G_e$ such that \emph{Ad}$^\ast(g) \alpha = \alpha + \beta$.
\end{lem}

\begin{proof}
Since each $v \in \mathfrak{n}^+$ is nilpotent (as an endomorphism of $V$), the translation morphism $v \ra 1 + v$ takes $v \in \mathfrak{n}^+$ to a unipotent matrix in $G_e$. We denote the subgroup generated by all $1 + v$ with $v \in \mathfrak{n^+}$ by $N^+$. It is easily checked that $1 + \mathfrak{n}^+$ is closed under matrix multiplication. Due to formula~(4) in the proof of Corollary~\ref{invariants} the set $1 + \mathfrak{n}^+$ is also closed under the map $g \mapsto g^{-1}$. Hence $N^+ = 1 + \mathfrak{n}^+$. We aim to prove that $\Ad^\ast(N^+) \alpha = \alpha + (\mathfrak{n}^-)^\ast$, from which our proposition will quickly follow. The one dimensional subspaces $\{1 + t \xi_i^{j,s} : t \in \K\}$ with $i > j$ generate $N^+$. Again the proof of Corollary~\ref{invariants} informs us that $(1 + t\xi_i^{j,s})^{-1} = 1 - t \xi_i^{j,s}$. A quick calculation then shows that
\begin{eqnarray*}
\Ad^\ast(1 + t \xi_i^{j,s}) \alpha = \alpha + t(a_j(\xi_j^{i, \lambda_j - 1- s})^\ast - a_i(\xi_j^{i, \lambda_i - 1- s})^\ast) \in \alpha + (\mathfrak{n}^-)^\ast
\end{eqnarray*}
The conditions on the $a_i$ ensure that the linear forms $\{a_j(\xi_j^{i, \lambda_j - 1- s})^\ast - a_i(\xi_j^{i, \lambda_i - 1- s})^\ast : i > j \}$ are linearly independent, hence span $(\mathfrak{n}^-)^\ast$. We see that $\dim\, \Ad^\ast(N^+) \alpha = \dim\, \alpha + (\mathfrak{n}^-)^\ast$. Thanks to \cite[Theorem~2]{Ros} we know that $\Ad^\ast(N^+) \alpha$ is a closed subvariety of $\alpha + (\mathfrak{n}^-)^\ast$. The dimensions coincide and so we have equality $\Ad^\ast(N^+) \alpha = \alpha + (\mathfrak{n}^-)^\ast$. Now $\beta \in (\mathfrak{n}^-)^\ast$ so there exists some $g \in N^+$ such that $\Ad^\ast(g) \alpha = \alpha + \beta$ as required.
\end{proof}

\p Let $a : \K^\times \mapsto G_e$ be the cocharacter given by $a(t)w_i = t^i w_i$. Define a rational linear action $\rho : \K^\times \mapsto GL(\g_e^\ast)$ by
\begin{eqnarray*}
\rho(t) \gamma = t \Ad^\ast (a(t)) \gamma
\end{eqnarray*}
where $\gamma \in \g_e^\ast$ and $t \in \K^\times$. Clearly we have $\rho(t) (\xi_i^{j,s})^\ast = t^{i-j+ 1} (\xi_i^{j,s})^\ast$.

\begin{lem}\label{Jstable}
The Jacobian locus $J$ is
\begin{enumerate}
\item{$G_e$-stable;}
\item{$\rho(\K^\times)$-stable.}
\end{enumerate}
\end{lem}
\begin{proof}
Since $x_r$ is $G_e$-invariant
\begin{eqnarray*}
x_r(\Ad^\ast (g)(\gamma + \delta)) = x_r (\gamma + \delta)
\end{eqnarray*}
for all $g \in G_e$ and $\gamma,\delta \in \g_e^\ast$. This equates to $x_r^{\Ad^\ast(g)\gamma} \circ (\Ad^\ast g) = x_r^\gamma$. The linear part of the left hand side of this equation is $d_{\Ad^\ast(g)\gamma} x_r \circ (\Ad^*g)$ and the same of the right hand side is $d_\gamma x_r$. As $\Ad^\ast (g)$ is invertible, the dimension of the linear span of $d_{\Ad^\ast(g)\gamma} x_1,..., d_{\Ad^\ast(g)\gamma} x_N$ equals that of the $d_{\gamma} x_1,..., d_{\gamma} x_N$, whence 1.

Turning our attention to $\rho(\K^\times)$, fix $t \in \K^\times$, $r = 1,...,N$, $(w,\mu) \in \mathfrak{S}_d \times \mathcal{C}_\lambda$ and observe that
\begin{eqnarray*}
\Theta_r(w,\mu) \circ \rho(t) &=& (\sgn(w) \xi_{i_1}^{i_{w1}, s_1} \cdots \xi_{i_d}^{i_{wd}, s_d})\circ \rho(t) \\
&=& (\prod_{k=1}^d t^{i_k - i_{wk} +1})\sgn(w)\xi_{i_1}^{i_{w1}, s_1} \cdots \xi_{i_d}^{i_{wd}, s_d}\\
& =& t^{d}(\sgn(w) \xi_{i_1}^{i_{w1}, s_1} \cdots \xi_{i_d}^{i_{wd}, s_d}) = t^d \Theta_r(w, \mu).
\end{eqnarray*}
So that $x_r \circ \rho(t) = t^{d} x_r$. Next let $\gamma, v \in \g_e^\ast$ and observe that $$x_r(\rho(t)\gamma + v) = x_r\circ \rho(t) (\gamma + \rho(t)^{-1} v) = t^d x_r(\gamma + \rho(t)^{-1}v)$$ which is written as $x_r^{\rho(t)\gamma} = t^d x_r^\gamma\circ \rho(t)^{-1}$ in our notations. We conclude that the linear terms must coincide, so that $$d_{\rho(t)\gamma} x_r = t^d d_{\gamma} x_r \circ \rho(t)^{-1}.$$ However, $\rho(t)^{-1}$ is evidently invertible so 2 follows.
\end{proof}

\p We are now ready to assemble the above ingredients. Proposition \ref{Codimension 2} shall immediately follow from the lemma.
\begin{lem}\label{plane1}
$(\K \alpha \oplus \K\beta) \cap J = 0$
\end{lem}

\begin{proof}
Let $t_1, t_2 \in \K$ and $\gamma = t_1 \alpha + t_2 \beta \neq 0$. We shall show that $\gamma \in \g_e^\ast \backslash J$. If $t_2 = 0$ then the element $g \in G_e$ constructed in Lemma \ref{alpha} sends $\gamma$ to $t_1\alpha + t_1\beta$ so by part 1 of Lemma \ref{Jstable} it suffices to prove that $\gamma \in \g_e^\ast \backslash J$ whenever $t_2 \neq 0$. By Lemma \ref{beta} we may assume that $t_1 \neq 0$.

It is clear that $\rho(t)\alpha = t \alpha$ and $\rho(t) \beta = \beta$. Consider the variety $\K \alpha + t_2\beta$. Since $J$ is closed it follows that $(\g_e^\ast \backslash J) \cap (\K \alpha + t_2\beta)$ is a Zariski open subset of $\K \alpha + t_2\beta$. By Lemma \ref{beta} that intersection is non-empty. We deduce that $$\rho(\K^\times)\gamma \cap (\g_e^\ast \backslash J) = (\K^\times\alpha + t_2\beta) \cap (\g_e^\ast \backslash J)  \neq \emptyset.$$ By part 2 of Lemma \ref{Jstable}, $\g_e^\ast \backslash J$ is $\rho(\K^\times)$-stable implying $\gamma \in \g_e^\ast \backslash J$ as required.

\end{proof}

We are now ready to prove Proposition \ref{Codimension 2}.

\begin{proof}
The Jacobian locus is conical and Zariski closed. Apply the above lemma.
\end{proof}

\subsection{The symplectic and orthogonal cases}\label{JacLoc2}
\setcounter{parna}{0}

\p In this section we aim to prove an analogue of Proposition~\ref{Codimension 2} for centralisers in other classical types.
\begin{prop}\label{jaclocsymorth}The following are true:
\begin{enumerate}
\item{Let $\epsilon = -1$ so that $K$ is of type $\sf C$. Then \emph{codim}$_{\h_e^\ast} J_{\h_e^\ast}(x_{r}: r $ even$) \geq 2$;}
\smallskip
\item{Let $\epsilon = 1$ so that $K$ is of type $\sf B$ or $\sf D$. Then \emph{codim}$_{\pp_e^\ast} J_{\pp_e^\ast}(x_{r}: r + d_r $ even$) \geq 2$.}
\end{enumerate}
\end{prop}

\p As was discussed in the introduction, the proofs of parts 1 and 2 shall be identical. We shall supply all details of the proof of part 1, whilst our proof of part 2 shall simply consist of a description of necessary changes to that proof. As an immediate corollary to Proposition~\ref{jaclocsymorth} we obtain.
\begin{cor}\label{nonzero1} The following are true:
\begin{enumerate} 
\item{If $\epsilon = -1$ then the restrictions $x_r|_{\h_e^\ast}$ with $r$ even are non-zero and distinct;}
\smallskip
\item{If $\epsilon = 1$ then the restrictions $x_r|_{\pp_e^\ast}$ with $r+d_r$ even are non-zero and distinct.}
\end{enumerate}
\end{cor}
\begin{proof}
If $x_r|_{\h_e^\ast} = x_s|_{\h_e^\ast}$ for some $r \neq s$ then we would have $\h_e^\ast = J_{\h_e^\ast}(x_r : r \text{ even})$, contradicting part 1 of Proposition~\ref{jaclocsymorth}. Similarly the restrictions are non-zero. The same argument applies for part 2.
\end{proof}

\p \emph{Until otherwise stated we assume $\epsilon = -1$ so that $K$ is of type $\sf C$}. Since $\g_e = \h_e \oplus \pp_e$ there is a natural inclusion $\iota : S(\h_e) \ra S(\g_e)$ which is also a $\K$-algebra homomorphism.
\begin{lem}\label{dees}
$d_\gamma x_r = d_\gamma \iota(x_r|_{\h_e^\ast})$ for $r$ even and for all $\gamma \in \h_e^\ast$.
\end{lem}
\begin{proof}
Let $\gamma \in \h_e^\ast$. Let $\BB_0$ and $\BB_1$ be bases for $\h_e$ and $\pp_e$ respectively, so that $\BB := \BB_0 \cup \BB_1$ is a basis for $\g_e$. We aim to show that $d_\gamma (x_r - \iota(x_r|_{\h_e^\ast})) = 0$. Now $x_r - \iota(x_r|_{\h_e^\ast})$ may be written as a finite sum $\sum_{i=1}^k c_i m_i$ where $c_i \in \K^\times$ are constants and  the $m_i \in S(\g_e)$ are monomials in the basis $\BB$. Since $\iota(x_r|_{\h_e^\ast})$ is the sum of those monomial summands of $x_r$ which contain no factors from $\BB_1$ we conclude that each $m_i$ possesses a factor from $\BB_1$. We have $\sigma \eta = -\eta$ for each $\eta \in \BB_1$ and, since $r$ is even, $\sigma x_r = x_r$ by Proposition~\ref{sigmaxr} so there must be an even number of factors from $\BB_1$ in each monomial summand of $x_r$. This implies that each $m_i$ possesses at least two factors from $\BB_1$ and that for all $x \in \BB$ the partial derivative $\frac{\partial m_i}{\partial x}$ either is zero or possesses at least one nonzero 
factor from $\BB_1$. The functionals $\BB_1$ annihilate $\h_e^\ast$ so $$\frac{\partial m_i}{\partial x}(\gamma) = 0$$ for all $x \in \BB$. But now $$d_\gamma (x_r - \iota(x_r|_{\h_e^\ast})) = \sum_{x \in \BB} \frac{\partial (x_r - \iota(x_r|_{\h_e^\ast}))}{\partial x}(\gamma)x = \sum_{x \in \BB} \sum_{i=1}^k c_i \frac{\partial m_i}{\partial x}(\gamma)x = 0.$$ The lemma follows.
\end{proof}

\p The following shall allow us to infer our Jacobian locus 2 condition from the type $\sf A$ case.
\begin{lem}\label{jal}
$J_{\h_e^\ast}(x_r : r \text{ even}) = \h_e^\ast \cap J_{\g_e^\ast}(x_r : r \text{ even})$
\end{lem}
\begin{proof}
Fix $\gamma \in \h_e^\ast$. If $\sum_r c_{2r} d_\gamma x_{2r} = 0$ then $$\left(\sum_r c_{2r} d_\gamma x_{2r}\right)|_{\h_e^\ast} = \sum_r c_{2r} d_\gamma (x_{2r}|_{\h_e^\ast}) = 0$$ which gives one inclusion. Conversely suppose $\sum_r c_{2r} d_\gamma (x_{2r}|_{\h_e^\ast}) = 0$. Then $$\iota\left(\sum_r c_{2r} d_\gamma (x_{2r}|_{\h_e^\ast})\right) = \sum_r c_{2r} d_\gamma \iota(x_{2r}|_{\h_e^\ast}) =  \sum_r c_{2r} d_\gamma x_{2r} = 0,$$ which gives the other inclusion.
\end{proof}

\p Let $\alpha$ be as defined in the previous section, with the additional constraint that $a_i = - a_{i'}$ for all $i \neq i'$. We define $$\bar\beta = \beta + \beta' \text{ where } \beta' = \sum_{i+1 \neq i'} \varepsilon_{i,i+1,0} (\xi_{(i+1)'}^{i', \lambda_i-1})^\ast.$$
\begin{rem}\label{remremree}
\rm{These definitions for $\alpha$ and $\bar{\beta}$ are rather unclear at first glance. They first appeared in \cite{Yak1}, were used again in \cite{PPY}, and have a simple rationale behind them which we shall briefly discuss. The obvious guess of how to construct analogues of $\alpha$ and $\beta$, but lying in $\h_e^\ast$, is to define an automorphism $\sigma^\ast$ which acts by $+1$ on $\h_e^\ast$ and acts by $-1$ on $\pp_e^\ast$, and extends to all of $\g_e^\ast$ by linearity. Just as we obtained a spanning set for $\h_e$ by considering expressions $x + \sigma(x)$ with $x \in \g_e$ we may obtain analogues for $\alpha$ and $\beta$ by considering $\alpha + \sigma^\ast \alpha$ and $\beta + \sigma^\ast \beta$. This is roughly what we do here, although some of the summands $(\xi_i^{j,s})^\ast$ are rescaled in order to simplify notation.}
\end{rem}

\p In Section~\ref{elemInv} we introduced dual vectors $(\zeta_i^{j,s})^\ast := (\xi_i^{j,\lambda_j-1-s})^\ast + \varepsilon_{i,j,s} (\xi_{j'}^{i',\lambda_i - 1 - s})^\ast$. In order to carry out explicit calculations using $\alpha$ and $\bar\beta$ it will be necessary to express these two elements in terms of the $(\zeta_i^{j,s})^\ast$.
\begin{lem}\label{albeso} If $a_i = - a_{i'}$ for $i \neq i'$ then
\begin{eqnarray*}
\alpha = \frac{1}{2} \sum_{i = i'} a_i(\zeta_i^{i,0})^\ast +  \sum_{i < i'}a_i(\zeta_i^{i,0})^\ast; \\
\bar\beta = \frac{1}{2}\sum_{i+1 = i'}(\zeta_i^{i+1, 0})^\ast + \sum_{i+1 \neq i'}(\zeta_i^{i+1,0})^\ast;
\end{eqnarray*}
and in particular $\alpha, \bar\beta \in \h_e^\ast$.
\end{lem}
\begin{proof}
Since $\epsilon = -1$, Lemma~\ref{nilpotents} implies that $i = i'$ if and only if $\lambda_i$ is even. Using Lemma~\ref{spanningdetails} it follows that $(\zeta_i^{i,0})^\ast = 2(\xi_i^{i,\lambda_i-1})^\ast$ for all $i = i'$ and $(\zeta_i^{i,0})^\ast = (\xi_i^{i,\lambda_i-1})^\ast -(\xi_{i'}^{i',\lambda_i-1})^\ast$ for all $i < i'$. The formula for $\alpha$ follows. Similarly $\varepsilon_{i,i+1,0} = (-1)^{\lambda_{i+1}}\varpi_{i\leq i'} \varpi_{i+1\leq (i+1)'}$. If $i+1 = i'$ then Lemma~\ref{nilpotents} implies $\lambda_{i+1}$ is odd and that $\varepsilon_{i,i+1,0} = 1$. We conclude that $(\zeta_i^{i+1,0})^\ast = 2(\xi_i^{i+1, \lambda_{i+1} - 1})^\ast$ which completes the proof.
\end{proof}
\p Recall that $J := J_{\g_e^\ast}(x_1,...,x_N)$ and that there is a rational linear action $\rho : \K^\times \ra GL(\g_e^\ast)$ defined preceding Lemma~\ref{xrdiff}. If $i + 1 \neq i'$ then $(i+1)' > i'$ and so $$\rho(t) (\xi_{(i+1)'}^{i', \lambda_i - 1})^\ast = t^{k_i}  (\xi_{(i+1)'}^{i', \lambda_i - 1})^\ast$$ where $k_i = (i+1)' - i' + 1 \geq 2$.

\begin{lem}\label{plane2}
$(\K \alpha \oplus \K\bar\beta) \cap J = 0$
\end{lem}
\begin{proof}
Let $\gamma = t_1 \alpha + t_2\bar\beta \neq 0$ with $t_1,t_2 \in \K$. We shall show that $\gamma \in \g_e^\ast \backslash J$. Supposing $t_1 \neq 0$ and $t_2 = 0$ we may invoke Lemma~\ref{plane1} to conclude that $\gamma \in \g_e^\ast \backslash J$.

Suppose $t_2\neq 0$, then consider the set $E = \{t\alpha + (\beta + \sum \varepsilon_{i,i+1, 0} t^{k_i}\xi_{(i+1)'}^{i',\lambda_i-1}) : t \in \K\}$ where $k_i \in \N$ is defined preceding the statement of the lemma. It is a one dimenisonal variety containing $\beta$, hence by Lemma~\ref{beta} it must intersect the set $\g_e^\ast \backslash J$ in a non-empty open subset. Since $\rho(\K^\times)\gamma = \{t\alpha + (\beta + \sum \varepsilon_{i,i+1, 0} t^{k_i}\xi_{(i+1)'}^{i',\lambda_i-1}) : t \in \K^\times\} \subseteq E$ is also non-empty and open in $E$ the intersection $\rho(\K^\times)\gamma \cap (\g_e^\ast \backslash J)$ is non-empty. By part 2 of Lemma~\ref{Jstable}, $J$ is $\rho(\K^\times)$-stable and so $\gamma \in \g_e^\ast \backslash J$.
\end{proof}

\p We can now give a proof for part 1 of Proposition~\ref{jaclocsymorth}.
\begin{proof}
By Lemmas~\ref{albeso} and \ref{plane2} there is a 2 dimensional plane contained in $\h_e^\ast$ intersecting $J$ only at zero. By Lemma~\ref{jal} we have $J_{\h_e^\ast}(x_r : r \text{ even}) = \h_e^\ast \cap J_{\g_e^\ast}(x_r : r \text{ even}) \subseteq \h_e^\ast \cap J$ so that same plane intersects $J_{\h_e^\ast}(x_r : r \text{ even})$ only at zero. As $J_{\h_e^\ast}(x_r : r \text{ even})$ is conical and Zariski closed, the proposition follows.
\end{proof}

\p In order to prove part 2 of Proposition~\ref{jaclocsymorth} we follow exactly the same scheme of argument as above. Since we treated the symplectic case so carefully, our proof here will constantly refer back to previous arguments. \emph{For the remnant of this subsection take $\epsilon = 1$}. Again we have an inclusion $\iota : S(\pp_e) \ra S(\g_e)$. The next lemma is analogous to Lemma~\ref{jal}.
\begin{lem}\label{jal2}
$J_{\pp_e^\ast}(x_r : r + d_r \text{ even}) = \pp_e^\ast \cap J_{\g_e^\ast}(x_r : r + d_r \text{ even})$
\end{lem}
\begin{proof}
First we prove a version of Lemma~\ref{dees}: that $d_\gamma x_r = d_\gamma i(x_r|_{\pp_e^\ast})$ for all $\gamma \in \pp_e^\ast$. Resume notations $\BB_0$, $\BB_1$ and $\BB$ from Lemma~\ref{dees}. This time write $x_r - \iota(x_r|_{\pp_e^\ast}) = \sum_{i=1}^k c_i m_i$ where $c_i \in \K^\times$ are non-zero constants and $m_i$ are monomials in $\BB$. Since $\iota(x_r|_{\pp_e^\ast})$ is just the sum of those monomials summands of $x_r$ which contain no terms from $\BB_0$, so each $m_i$ possesses a factor from $\BB_0$. Using a reasoning identical to Lemma~\ref{vanishing1} we see that the number of such factors is even. The proof now concludes exactly as per Lemma~\ref{dees}. In order to finish the current proof we use identical calculations to Lemma~\ref{jal}, simply replacing the set $\{x_r : r \text{ even}\}$ with $\{x_r: r + d_r \text{ even}\}$, and restricting our functions to $\pp_e^\ast$ rather than $\h_e^\ast$.
\end{proof}

\p Next we identify a 2-dimensional plane contained in $\pp_e^\ast$ intersecting $J_{\pp_e^\ast} (x_r : r + d_r \text{ even})$ only at zero. As was noted in Remark~\ref{remremree} our construction in type $\sf C$ is essentially to take $\alpha + \sigma^\ast \alpha$ and $\beta + \sigma^\ast \beta$. The obvious choice when constructing elements in $\pp_e^\ast$ is to consider $\alpha - \sigma^\ast \alpha$ and $\beta - \sigma^\ast \beta$. This is essentially what we do here. 

Define $\alpha$ in the same way as in Section~\ref{JacLoc1}, with $a_i = -a_{i'}$ for all $i \neq i'$, and define $\bar\beta = \beta - \beta'$.
\begin{lem}\label{albeos} We have
\begin{eqnarray*}
\alpha = \frac{1}{2} \sum_{i = i'} a_i(\eta_i^{i,0})^\ast +  \sum_{i < i'}a_i(\eta_i^{i,0})^\ast;\\
\bar\beta = \frac{1}{2}\sum_{i+1 = i'}(\eta_i^{i+1, 0})^\ast + \sum_{i+1 \neq i'}(\eta_i^{i+1,0})^\ast;
\end{eqnarray*}
and in particular $\alpha, \bar\beta \in \pp_e^\ast$.
\end{lem}
\begin{proof}
Simply follow the proof of Lemma~\ref{albeso} verbatim, replacing each occurrence of $\zeta_i^{j,s}$ with $\eta_i^{j,s}$, and exchanging the words odd and even.
\end{proof}

\p We can now give a version of Lemma~\ref{plane2}. The statement of the lemma is precisely the same, but we remind the reader that now we have $\epsilon = 1$, and our definition of $\bar\beta$ is slightly different.
\begin{lem}
$\K \alpha \oplus \K \bar\beta \cap J = 0$
\end{lem}
\begin{proof}
The argument is identical to Lemma~\ref{plane2} except in this instance the correct definition of $E$ is  $\{t\alpha + (\beta - \sum \varepsilon_{i,i+1, 0} t^{k_i}\xi_{(i+1)'}^{i',\lambda_i-1}) : t \in \K\}$, which reflects the fact that $\bar\beta = \beta - \beta'$.
\end{proof}

\p We may now supply the proof of part 2 of Proposition~\ref{jaclocsymorth}. 
\begin{proof}
By Lemmas~\ref{plane2} and \ref{albeos} there is a two dimensional plane contained in $\pp_e^\ast$ intersecting $J$ only at zero. By lemma \ref{jal2} we have $J_{\pp_e^\ast}(x_r : r + d_r \text{ even}) = \pp_e^\ast \cap J_{\g_e^\ast}(x_r : r + d_r \text{ even}) \subseteq \pp_e^\ast \cap J$ so that same plane intersects $J_{\pp_e^\ast}(x_r : r + d_r \text{ even})$ only at zero. As $J_{\pp_e^\ast}(x_r : r + d_r \text{ even})$ is conical and Zariski closed, the proposition follows.
\end{proof}

\section{Generic stabilisers}\label{GenEl}
\renewcommand{\theparno}{\thesection.\arabic{parno}}
\renewcommand{\p}{\refstepcounter{parno}\noindent\textbf{\theparno .} \space} 
\setcounter{parno}{0}

\p In the current section we make a quick detour to discuss the existence of generic stabilisers in certain cases. We recorded the definition of a generic stabiliser in Section~\ref{Index} and proved an elementary but fundamental lemma, which we restate for the reader's convenience.
\begin{lem}\label{genrestate}
Suppose $G \subseteq GL(V)$. Fix $\alpha \in V$ and let $\mathcal{W}_\alpha = \{\gamma \in V : \g_\alpha \subseteq \g_\gamma\}$. If
\begin{eqnarray*}\varphi: G \times \mathcal{W}_\alpha \ra V;\\ \varphi(g,v) = g\cdot v\end{eqnarray*} is a dominant morphism then $\alpha$ is a regular, generic point in $V$.
\end{lem}
\noindent The remaining results in this section hold in any good characteristic. Some of the results were known previously although the proofs were different. We shall attribute known results where necessary.

\p Retain all notations and conventions of the previous sections. Recall that we have a linear form $\alpha \in \g_e^\ast$ such that $\alpha \in \h_e^\ast$ when $\epsilon = -1$ and $\alpha \in \pp_e^\ast$ for $\epsilon = 1$ (see Lemmas \ref{albeso} and \ref{albeos}). In \cite[$\S 5$]{Yak1} it is shown that $\alpha \in \g_e^\ast$ is a regular element with stabiliser $\mathfrak{h} = \spn\{\xi_i^{j,s} : i = j\}$, and that $\alpha$ is also regular for the coadjoint representation of $\h_e$ when $\epsilon = -1$. Yakimova goes on to prove that $\alpha$ is a generic point over $\C$ in these two cases. Her method makes use of a powerful criterion due to Elashvili, which is particular to characteristic $0$ \cite{Ela}. With a little extra work we can use the same technique to extend this result to the characteristic $p > 0$ case.
\begin{thm}\label{genstabal} The linear form $\alpha$ is
\begin{enumerate}
\item{a generic, regular point for the action of $G_e$ on $\g_e^\ast$;}
\item{a generic, regular point for the action of $K_e$ on $\h_e^\ast$ when $\epsilon = -1$;}
\item{a generic, regular point for the action of $K_e$ on $\pp_e^\ast$ when $\epsilon = 1$.}
\end{enumerate}
\end{thm}
\p The method is to satisfy the assumptions of Lemma~\ref{genrestate}. Much as was the case in Subsection~\ref{JacLoc2} the proof of part 2 and part 3 is almost identical. In order to prove 3 we give a detailed account of how to modify the argument in part 2. In \cite{Yak1} some of these statements from parts 1 and 2 were proven: the regularity was demonstrated over fields of good characteristic and the the generic stabilisers were exhibited over $\C$. 

\p Recall that $\g_e = \mathfrak{n}^- \oplus \mathfrak{h} \oplus \mathfrak{n}^+$ where $\mathfrak{n}^- = \spn\{\xi_i^{j,s} : i < j\}$, $\mathfrak{h} = \spn\{\xi_i^{j,s} : i = j\}$ and $\mathfrak{n}^+ = \spn\{\xi_i^{j,s} : i > j\}$. We have an induced decomposition $\g_e^\ast = (\mathfrak{n}^-)^\ast \oplus \mathfrak{h}^\ast \oplus (\mathfrak{n}^+)^\ast$ where $\alpha \in \mathfrak{h}^\ast$.
\begin{prop}\label{shost}
The following map is dominant
\begin{eqnarray*}
\vartheta : & G_e \times \mathfrak{h}^\ast \ra \g_e^\ast \\
\vartheta : & (g, \gamma) \mapsto \Ad^\ast(g) \gamma. 
\end{eqnarray*}
\end{prop}
\begin{proof}
By \cite[Theorem 3.2.20(i)]{Spr} it suffices to show that the differential $d_{(1, \alpha)} \vartheta : \g_e \oplus \hh^\ast \ra \g_e^\ast$ is surjective. In \cite[Lemma~1.6]{TY} the differential is calculated $$d_{(1,\alpha)} \vartheta (x, v) = \ad^\ast(x) \alpha + v.$$ Therefore $\hh^\ast \subseteq d_{(1,\alpha)} \vartheta (\g_e, \hh^\ast)$. To complete the proof we shall show that $(\mathfrak{n}^-)^\ast, (\mathfrak{n}^+)^\ast \subseteq \ad^\ast(\g_e)\alpha =d_{(1,\alpha)} \vartheta (\g_e, 0)$. A quick calculation shall confirm that 
\begin{eqnarray}\label{adform}
\ad^\ast (\xi_i^{j,s}) (\xi_k^{l,r})^\ast = \delta_{ik} (\xi_j^{l,r-s})^\ast - \delta_{jl} (\xi_k^{i,r-s})^\ast
\end{eqnarray}
and that
\begin{eqnarray}\label{eqnn}
\ad^\ast(\xi_i^{j,s}) \alpha = a_i (\xi_j^{i,\lambda_i-1-s})^\ast - a_j (\xi_j^{i,\lambda_j-1-s})^\ast.
\end{eqnarray} Fix $i > j$. If $\lambda_i = \lambda_j$ then the restriction $a_i \neq a_j$ implies that $(\xi_j^{i,s})^\ast \in \ad^\ast(\mathfrak{n}^+)\alpha$ for $s = 0,1,...,\lambda_i$. If $\lambda_i < \lambda_j$ then substituting $s=\lambda_j-1$ into equation (\ref{eqnn}) gives $(\xi_j^{i,0})^\ast \in \ad^\ast(\mathfrak{n}^+)\alpha$. Substituting successively smaller values of $s$ into equation (\ref{eqnn}) we obtain by induction $(\xi_j^{i,s})^\ast \in \ad^\ast(\mathfrak{n}^+)\alpha$ for all $s = 0,1,..., \lambda_i -1$. We conclude that $(\mathfrak{n}^-)^\ast \subseteq \ad^\ast(\mathfrak{n}^+) \alpha$. An identical argument shows that $(\mathfrak{n}^+)^\ast \subseteq \ad^\ast(\mathfrak{n}^-) \alpha$, completing the proof.
\end{proof}

\p In order to prove analogues of the above proposition for classical subalgebras of $\g$ we must first place further restriction on the coefficients $a_i$ which appear in the definition of $\alpha$. Define $$l_i = \left\{ \begin{array}{ll}
         1 & \mbox{if $i=i'$};\\
        2 & \mbox{if $i \neq i'$}.\end{array} \right.$$
and impose the restriction $l_i a_i = l_j a_j$ only if $i = j$.
\begin{prop}\label{alphaKegeneric}
The following map is dominant when $\epsilon = -1$
\begin{eqnarray*}
\vartheta : & K_e \times (\mathfrak{h}^\ast\cap \h_e^\ast) \ra \h_e^\ast \\
\vartheta : & (g, \gamma) \mapsto \Ad^\ast(g) \gamma. 
\end{eqnarray*}
\end{prop}
\begin{proof}
The proof is very similar to that of Proposition~\ref{shost}. Once again we may prove that the differential $d_{(1, \alpha)} \vartheta$ is surjective. The differential is $d_{(1,\alpha)} \vartheta (x,v) = \ad^\ast(x) \gamma + v$. Since $d_{(1,\alpha)} \vartheta (0,\hh^\ast\cap \h_e^\ast) = \mathfrak{h}^\ast \cap \h_e^\ast = \spn\{ (\zeta_i^{i,s})^\ast\}$ we may complete the proof by showing that $\spn\{(\zeta_i^{j,s})^\ast : i \neq j\}\subseteq \ad^\ast(\h_e) \alpha$. Once again we shall need explicit calculations. From formula (\ref{adform}) in the proof of the previous proposition it follows that
\begin{eqnarray*}
\ad^\ast (\zeta_i^{j,s}) (\zeta_k^{l,r})^\ast & = & \delta_{ik} (\zeta_j^{l, \lambda_j - 1 + r-s})^\ast - \delta_{jl} (\zeta_k^{i, \lambda_i - 1 +r-s})^\ast \\
& + & \delta_{il'} \varepsilon_{klr}(\zeta_j^{k', \lambda_j - 1+r-s})^\ast - \delta_{jk'} \varepsilon_{klr} (\zeta_{l'}^{i,\lambda_i - 1 +r-s})^\ast.
\end{eqnarray*}
Recall that there is an expression for $\alpha$ in terms of $(\zeta_i^{i,0})^\ast$ derived in Lemma~ \ref{albeso} $$\alpha = \frac{1}{2} \sum_{i = i'} a_i(\zeta_i^{i,0})^\ast +  \sum_{i < i'}a_i(\zeta_i^{i,0})^\ast = \sum_{i \leq i'}a_i (1 - \frac{1}{2} \delta_{i,i'}) (\zeta_i^{i,0})^\ast.$$
By Lemma~\ref{spanningdetails} we have $\varepsilon_{i,i,0} = (-1)^{\lambda_i}$ and as a consequence we have
\begin{eqnarray*}
\ad^\ast(\zeta_i^{j,s}) \alpha & = & (1 - \frac{1}{2}\delta_{i,i'}) a_i  (\zeta_j^{i,\lambda_j - 1 - s})^\ast - (1 - \frac{1}{2}\delta_{j,j'})a_j (\zeta_j^{i,\lambda_i - 1 - s})^\ast\\
& - & \varepsilon_{i',i',0} (1 - \frac{1}{2}\delta_{i,i'}) a_{i'}  (\zeta_j^{i,\lambda_j - 1 - s})^\ast - \varepsilon_{j',j',0}(1 - \frac{1}{2}\delta_{j,j'})a_j (\zeta_j^{i,\lambda_i - 1 - s})^\ast\\
& = & (1 - \frac{1}{2}\delta_{i,i'}) (a_i + (-1)^{\lambda_i}a_{i'})  (\zeta_j^{i,\lambda_j - 1 - s})^\ast \\& - & (1 - \frac{1}{2}\delta_{j,j'})(a_j + (-1)^{\lambda_j}a_{j'}) (\zeta_j^{i,\lambda_i - 1 - s})^\ast.
\end{eqnarray*}
By Lemma~\ref{nilpotents}, $\lambda_i$ is even if and only if $i = i'$. Furthermore $a_i = a_{i'}$ whenever $i=i'$, and $a_i = -a_{i'}$ whenever $i\neq i'$. In either case $a_i + (-1)^{\lambda_i} a_{i'} = 2a_i$. We deduce that $$\ad^\ast(\zeta_i^{j,s}) \alpha = l_i a_i (\zeta_j^{i,\lambda_j-1-s})^\ast - l_j a_j (\zeta_j^{i,\lambda_i-1-s})^\ast.$$ Fix $i\neq j$. Suppose $\lambda_i = \lambda_j$. Then the restrictions imposed on the coefficients $a_i$ and $l_i$ ensure that $(\zeta_j^{i,s})^\ast \in \ad^\ast(\h_e)\alpha$ for $s=0,1,...,\lambda_i-1$. Now suppose $\lambda_i < \lambda_j$. Taking $s=0$ we get $-l_ja_j(\zeta_j^{i,\lambda_i-1})^\ast \in \ad^\ast(\h_e)\alpha$. Taking successively larger values of $s$ we obtain $\zeta_j^{i,\lambda_i-1-s} \in \ad^\ast(\h_e)\alpha$ for $s=0,1,...,\lambda_i-1$ by induction. Now suppose $\lambda_i > \lambda_j$. It follows that  $j' \neq i < j$ so that $i' < j'$. By part 3 of Lemma~\ref{spanningdetails} we have $\zeta_i^{j,s} = \varepsilon_{i,j,s} \zeta_{j'}^{i',s}$ which 
lies in $\ad^\ast(\h_e) \alpha$ by our previous remarks. We conclude that $\spn\{(\zeta_i^{j,s})^\ast : i \neq j\}\subseteq \ad^\ast(\h_e) \alpha$ and $d_{(1, \alpha)} \varphi (\h_e, \mathfrak{h}^\ast \cap \h_e^\ast) = \h_e^\ast$, which completes the proof.
\end{proof}

\p We now state the analogous proposition for orthogonal algebras. The reader will notice that we denote the representation of $K_e$ in $\pp_e^\ast$ by $\Ad^\ast$. Technically this representation is the restriction of the coadjoint representation of $G_e$, although this will cause no confusion.
\begin{prop}\label{alphaKegeneric2}
The following map is dominant when $\epsilon = 1$
\begin{eqnarray*}
\vartheta : & K_e \times (\mathfrak{h}^\ast\cap \pp_e^\ast) \ra \pp_e^\ast \\
\vartheta : & (g, \gamma) \mapsto \Ad^\ast(g) \gamma. 
\end{eqnarray*}
\end{prop}
\begin{proof}
Following the previous line of reasoning, the proof hinges on showing that $\spn\{(\eta_i^{j,s})^\ast : i \neq j\}\subseteq \ad^\ast(\h_e) \alpha$. Once again we shall need explicit calculations. From formula (\ref{adform}) it follows that
\begin{eqnarray*}
\ad^\ast (\zeta_i^{j,s}) (\eta_k^{l,r})^\ast & = & \delta_{ik} (\eta_j^{l, \lambda_j - 1 + r-s})^\ast - \delta_{jl} (\eta_k^{i, \lambda_i - 1 +r-s})^\ast \\
& - & \delta_{il'} \varepsilon_{klr}(\eta_j^{k', \lambda_j - 1+r-s})^\ast - \delta_{jk'} \varepsilon_{klr} (\eta_{l'}^{i,\lambda_i - 1 +r-s})^\ast 
\end{eqnarray*}
Recall that there is an expression for $\alpha$ in terms of $(\eta_i^{i,0})^\ast$ derived in Lemma~\ref{albeso}. Using this with the above expression for $\ad^\ast (\zeta_i^{j,s}) (\eta_k^{l,r})^\ast$ and going through a series of calculations almost identical to those in Proposition~\ref{alphaKegeneric} we arrive at the assertion $$\ad^\ast(\zeta_i^{j,s}) \alpha = l_i a_i (\zeta_j^{i,\lambda_j-1-s})^\ast - l_j a_j (\zeta_j^{i,\lambda_i-1-s})^\ast.$$ The proof then concludes by making precisely the same observations as those concluding the previous proposition.
\end{proof}

\p We may now supply a proof of theorem \ref{genstabal}.
\begin{proof}
Thanks to Lemma~\ref{genrestate} we need to show that $\varphi : G_e \times \V_\alpha \ra \g_e^\ast$ is dominant where $\V_\alpha = \{\gamma \in \g_e^\ast : (\g_e)_\alpha \subseteq (\g_e)_\gamma\}$. By \cite[Theorem~1]{Yak1} we have $(\g_e)_\alpha = \mathfrak{h}$. Since $\mathfrak{h}$ is abelian and stabilises $\n^-$ and $\n^+$, any linear form $\gamma \in \mathfrak{h}^\ast$ is annihilated by $\mathfrak{h}$. As such $\mathfrak{h}^\ast \subseteq \V_\alpha$. By Proposition~\ref{shost} the map $\vartheta = \varphi|_{G_e\times \mathfrak{h}^\ast}$ is dominant, and so too is $\varphi$. By Lemma~\ref{genreg}, part 1 of the theorem follows. 

The proofs of part 2 and 3 are similar. We shall reset the definition of $\varphi$, $\vartheta$ and $\V_\alpha$. Let $\epsilon = -1$, let $\V_\alpha = \{ \gamma\in \h_e^\ast : (\h_e)_\alpha \subseteq (\h_e)_\gamma\}$ and let $\varphi : K_e \times \V_\alpha \ra \h_e^\ast$. Since $(\h_e)_\alpha = \hh \cap \h_e$ is abelian and preserves $\spn\{\zeta_i^{j,s} : i\neq j\}$ we have $\hh^\ast\cap \h_e^\ast \subseteq \V_\alpha$. Then by Proposition~\ref{alphaKegeneric} the map $\vartheta = \varphi|_{K_e \times (\hh^\ast\cap\h_e^\ast)}$ is dominant and so is $\varphi$. Then by Lemma~\ref{genreg}, part 2 of the current theorem follows.

For part 3 we set $\epsilon = 1$, $\V_\alpha = \{\gamma \in \pp_e^\ast : (\h_e)_\alpha\subseteq (\h_e)_\gamma\}$ and $\varphi : K_e \times \V_\alpha \ra \pp_e^\ast$. In this case $(\h_e)_\alpha = \hh\cap \h_e$ annihilates $\hh \cap \pp_e$ and stabilises $\spn \{ \eta_i^{j,s} : i\neq j\}$ so $\hh^\ast\cap \pp_e^\ast \subseteq \V_\alpha$ and so by Proposition~\ref{alphaKegeneric2} the map $\varphi$ is dominant. The theorem follows by Lemma~\ref{genreg}.
\end{proof}

\p Finally we calculate the index of $\h_e$ in $\pp_e^\ast$ when $\epsilon = 1$. This comprises a small but vital ingredient of the proof of Theorem \ref{main3}.
\begin{cor}\label{orthindex}
$\ind(\h_e, \pp_e) = \frac{1}{2}(N - |\{i : \lambda_i \text{ \emph{odd}}\}|)$ when $\epsilon = 1$.
\end{cor}
\begin{proof}
By part 3 of Theorem~\ref{genstabal} we know that $\ind(\h_e, \pp_e) = \dim\, (\h_e)_\alpha = \dim\, \mathfrak{h}\cap \h_e = \dim\,  \spn\{\zeta_i^{i,s}: 1\leq i \leq n, 0 \leq s < \lambda_i\}$. Now $\zeta_i^{i,s} = 0$ only if $\xi_i^{i,\lambda_i-1-s} = -\varepsilon_{i,i,s} \xi_{i'}^{i',\lambda_i-1-s}$, which is only if $i=i'$ and $\varepsilon_{i,i,s} = 1$. In this case $\lambda_i$ is odd, by Lemma~\ref{nilpotents}, and $\varepsilon_{i,i,s} = (-1)^{\lambda_i-s} = 1$ implies $s$ is odd. Hence for each $i=i'$ we obtain $\frac{\lambda_i-1}{2}$ nonzero maps $\zeta_i^{i,s}$. In case $i\neq i'$ we have the relations $\zeta_i^{i,s} = \varepsilon_{i,i,s} \zeta_{i'}^{i',s}$ by part 3 of Lemma~\ref{spanningdetails}. Since these are the only relations we have $\dim\,  \spn \{\zeta_i^{i,s}, \zeta_{i'}^{i',s} : 0 \leq s < \lambda_i\} = \lambda_i$ for each pair $(i,i')$ with $i\neq i'$. Since $\lambda_i$ is even in this case we conclude that
$$\ind(\h_e, \pp_e) = \sum_{i} \lfloor\frac{\lambda_i}{2}\rfloor = \frac{1}{2}(N - |\{i : \lambda_i \text{ odd}\}|).$$
\end{proof}

\section{The structure of the invariant algebras}
\setcounter{parno}{0}

\p The deductions of sections \ref{elemInv} through \ref{GenEl} aim to satisfy the assumptions of Theorem \ref{Skryabin}. Before we apply the theorem we must clarify one detail. We remind the reader that the first assumption of the theorem is that we have found precisely $$m = m(\tg, X) = \dim\, X - \dim\, \tg + \emph{\ind}(\tg, X)$$ invariants. We shall now confirm that we have the correct number in each case.
\begin{lem}\label{noinv} The following are true:
\begin{enumerate} 
\item{$m(\g_e, \g_e^\ast) = N$;}
\smallskip
\item{$m(\h_e, \h_e^\ast) = |\{x_r : r \text{ even}\}| = N/2$ when $\epsilon = -1$;}
\smallskip
\item{$m(\h_e, \pp_e^\ast) = |\{x_r : r + d_r \text{ even}\}| = \frac{1}{2}(N + |\{i : \lambda_i \text{ odd}\}|)$ when $\epsilon = 1$.}
\end{enumerate}
\end{lem}
\begin{proof}
We have $m(\g_e, \g_e^\ast) = \ind(\g_e, \g_e^\ast) = N$ by \cite{Yak1}. If we take $\epsilon = -1$ so that $N$ is even then $m(\h_e, \h_e^\ast) = \ind(\h_e, \h_e^\ast) = N/2 = |\{ 1\leq r \leq N : r \text{ even}\}|$ by the same reasoning.

Now let $\epsilon = 1$. By \cite[3.2(3)]{Jan} we have $$\dim\, \h_e = \frac{1}{2}( \dim\, \g_e - |\{ i : \lambda_i \text{ odd}\}|).$$ Since $\dim\, \pp_e^\ast = \dim\, \pp_e = \dim\, \g_e - \dim\, \h_e$ we have $\dim\,  \pp_e^\ast - \dim\,  \h_e = |\{ i : \lambda_i \text{ odd}\}|.$ By Corollary \ref{orthindex} we get $m(\h_e, \pp_e^\ast) = \frac{1}{2}(N + |\{i : \lambda_i \text{ odd}\}|)$. We must show that this number equals $|\{1\leq r \leq N: r + d_r \text{ even}\}|$. Partition the set $\{1,...,N\}$ into disjoint subsets $\II_i = \{1\leq r \leq N : d_r = i\}$ where $i=1,...,n$. Then $|\II_i| = \lambda_i$ and by the definition of the sequence $d_1,d_2,...,d_N$ we have $r \in \II_i$ if and only if
\begin{eqnarray}\label{IIprop}
0 < r - \sum_{k=1}^{i-1} \lambda_k \leq \lambda_i.
\end{eqnarray}
If $i \neq i'$ then $|\II_i| = \lambda_i$ is even by Lemma \ref{nilpotents}. In this case, regardless of the parity of $\lambda_i$ there are exactly $\lambda_i/2$ values $r$ fulfilling (\ref{IIprop}) with $r+d_r$ even. Now suppose $i = i'$ so that $|\II_i| = \lambda_i$ is odd. We consider two cases: $i$ odd or $i$ even. In the first case there must be an even number of indexes $j$ with $1 \leq j < i$ and $j = j'$, since $j' \in \{j-1, j, j+1\}$. Thus there an even number of indexes $1 \leq j < i$ with $\lambda_j$ odd by Lemma \ref{nilpotents}. We deduce that $\sum_{k=1}^{i-1} \lambda_k$ is even. If $r \in \II_i$ so that $d_r = i$, then $r + d_r$ even implies $r$ is odd. Thus there are exactly $(\lambda_i + 1)/2$ values of $r$ fulfilling (\ref{IIprop}) with $r+d_r$ even. Now suppose $i=i'$ and $i$ is even. Similar to the case $i$ odd, there must be an odd number of indexes $1\leq j < i$ with $\lambda_j$ odd. Thus
$\sum_{k=1}^{i-1}\lambda_k$ is odd. For $r \in\II_i$, If $r+d_r$ is even then $r$ is even. Thus there are exactly $(\lambda_i + 1)/2$ values of $r$ fulfilling (\ref{IIprop}) with $r+d_r$ even. With Lemma \ref{nilpotents} in mind we are able to conclude that $$|\{1\leq r \leq N : r + d_r \text{ even}\}| = \sum_{i\neq i'} \lambda_i/2 + \sum_{i=i'} (\lambda_i+1)/2 = \frac{1}{2}(N + |\{i : \lambda_i \text{ odd}\}|).$$
\end{proof}

\p The proofs are identical for both parts of Theorem~\ref{main1} and also for Theorem~\ref{main3}, and so we may unify notation. In the remnant of this subsection, $\tG$ shall be our underlying group, $e \in \tg = \Lie(\tG)$ shall be a choice of nilpotent element, $\epsilon = \pm 1$ will indicate whether $K$ is orthogonal or symplectic (this, of course, is redundant when $\tG$ is of type $\sf A$), $X$ shall be the $Q_e$-module of interest, $\mathcal{I}$ shall be a finite subset of $\K[X]^{\tG_e}$ and $m$ shall be the integer defined in the statement of Skryabin's theorem. The following table explains how our notations are unified:
\begin{center}
  \begin{tabular}{ | c | c | c | c | c | c|}
    \hline
    & $\tG$ & $\epsilon$ & $X$ & $\mathcal{I}$ & $m$\\ \hline\hline
    \text{Case 1} & $G$ & $\text{-}$ & $\g_e^\ast$ & $\{x_1,...,x_N\}$ & $m(\g_e, \g_e)$  \\ \hline
    \text{Case 2} & $K$ & $-1$ &$ \h_e^\ast$ & $\{x_r|_{\h_e^\ast} : r \text{ even}\}$& $m(\h_e, \h_e)$ \\ \hline
    \text{Case 3} & $K$ & $1$ & $\pp_e^\ast$ & $\{x_r|_{\pp_e^\ast} : r + d_r \text{ even}\}$ & $m(\h_e, \pp_e)$\\
    \hline
  \end{tabular}
\end{center}

\p Theorems~\ref{main1} and \ref{main3} are both convenient ways of stating our first results without discussing the explicit constructions of the invariant algebras. Our method tells us more than is expressed by those theorems, especially in the nilpotent case. Using the above notations we may precisely state our result as follows.
\begin{thm}\label{main13}
Suppose that we are in Case 1, 2 or 3. If $e\in \tg$ then $|\mathcal{I}| = m$. Furthermore $\K[X]^{Q_e} = \K[\mathcal{I}]$ is a polynomial algebra generated by $\mathcal{I}$, and $$\K[X]^{\tg_e} = \K[X]^p[\mathcal{I}]$$ is a free $\K[X]^p$-module of rank $p^m$.
\end{thm}
\begin{proof}
That $|\mathcal{I}| = m$ follows immediately from Lemma~\ref{noinv} and Corollary~\ref{nonzero1}.

We now prove $\K[X]^{\tg_e} = \K[X]^p[\mathcal{I}]$ is a free $\K[X]^p$-module of rank $p^m$ by applying Skryabin's theorem (Theorem~\ref{Skryabin}). Corollary~\ref{invariants} tells us that the polynomials $\mathcal{I}$ are invariant. Lemma~\ref{noinv} ensures that $\mathcal{I}$ is of the correct size. Propositions \ref{Codimension 2} and \ref{jaclocsymorth} confirm that the condition on codimension of the Jacobian locus is satisfied, and so the the assumptions of Skryabin's theorem are satisfied. The claim follows.

In order to prove that $\K[X]^{Q_e} = \K[\mathcal{I}]$ we use induction on degree. Let $f \in \K[X]^{\tG_e}$. If $\deg(f) < p$ then by the previous paragraph, $f \in \K[\mathcal{I}]$. Suppose $\deg(f) \geq p$ and that all $g \in \K[X]^{\tG_e}$ with $\deg(g) < \deg(f)$ actually lie in $\K[\mathcal{I}]$. We shall call a product of elements of $\mathcal{I}$ a monomial in $\mathcal{I}$, and this terminology shall cause no confusion. Since $f \in \K[X]^{\tG_e} \subseteq \K[X]^{\tg_e}$ we may write $f$ as a sum of monomials in $\mathcal{I}$ with coefficients in $\K[X]^p$. Since $f$ and all monomials in $\mathcal{I}$ are fixed by $\tG_e$ we deduce that each coefficient is also fixed by $\tG_e$. Since $(\K[X]^p)^{\tG_e} = (\K[X]^{\tG_e})^p$ we conclude by induction that each coefficient is a $p^\text{th}$ power of an element of $\K[\mathcal{I}]$. It follows that $f \in \K[\mathcal{I}]$.
\end{proof}
\begin{rem}
\rm{Case 3 of Theorem~\ref{main13} quickly implies Theorem~\ref{main3}. Cases 1 and 2 imply the nilpotent case of Theorem~\ref{main1}. As was noted in the discussion following the statement of that theorem, there is a simple reduction to the nilpotent case. Therefore Theorems~\ref{main1} and \ref{main3} follow immediately from Theorem~\ref{main13}.}
\end{rem}
 
\p Next we prove Theorem~\ref{main2}, using the theory introduced in Section~\ref{Symmetrisation} and a simple filtration argument. Fortunately the proof in type $\sf A$ and $\sf C$ is identical and may be dealt with at once. Unfortunately we will have to reset our notation again. We let $\epsilon = -1$, let $\tG \in \{G, K\}$ and choose $x \in \tg$. Thanks to Theorem~\ref{milmap} there is a filtered $\tG_x$-module isomorphism $$\beta : U(\tg_x) \ra S(\tg_x).$$ According to Theorem~\ref{main1}, $S(\tg_x)^{Q_x}$ is polynomial on $\rank \,\tg$ generators. In keeping with the previous notations we denote a set of homogeneous generators by $\mathcal{I}$. Since $\gr (\beta)$ is the identity, the top graded component of $\beta^{-1} f$ is equal to $f$ for all $f \in \mathcal{I}$. We shall denote by $\widetilde{\mathcal{I}}$ the preimage under $\beta$ of $\mathcal{I}$ in $U(\tg_x)$.  Of course
$\widetilde{\mathcal{I}} \subseteq U(\tg_x)^{\tG_x} \subseteq Z(\tg_x)$. We are now ready to present a proof of Theorem~\ref{main2}
\begin{proof}
The graded algebra of the centre $\gr Z(\tg_x)$ is contained in $S(\tg_x)^{\tg_x}$, which is equal to $S(\tg_x)^p[\mathcal{I}]$ by the previous theorem. Since $\mathcal{I}$ consists of homogeneous polynomials we have $S(\tg_x)^p[\mathcal{I}] = \gr Z_p(\tg_x)[\widetilde{\mathcal{I}}]$ and furthermore $Z_p(\tg_x)[\widetilde{\mathcal{I}}]$ is central in $U(\tg_x)$. Placing all of these inclusions together we get
$$\gr Z(\tg_x) \subseteq S(\tg_x)^{\tg_x} = S(\tg_x)^p[\mathcal{I}] \subseteq \gr Z_p(\tg_x)_[\widetilde{\mathcal{I}}] \subseteq \gr Z(\tg_x).$$
Thence $\gr Z_p(\tg_x)[\widetilde{\mathcal{I}}] = \gr Z(\tg_x)$ and the dimensions of the graded components of $\gr Z_p(\tg_x)[\widetilde{\mathcal{I}}]$ and $\gr Z(\tg_x)$ coincide. It follows that the inclusion $Z_p(\tg_x)[\widetilde{\mathcal{I}}] \subseteq Z(\tg_x)$ is actually an equality. Part 1 of the theorem follows.

Our proof of part 2 is very similar. We denote by $\K[\widetilde{\mathcal{I}}]$ the subalgebra of $Z(\tg_x)$ generated by $\widetilde{\mathcal{I}}$, and similar for $\K[\mathcal{I}]$. We have $$\K[\mathcal{I}] = \gr \K[\widetilde{\mathcal{I}}] \subseteq \gr U(\tg_x)^{\tG_x} \subseteq S(\tg_x)^{\tG_x} = \K[\mathcal{I}]$$ and so we have equality throughout. In particular the dimensions of the graded components of $\gr U(\tg_x)^{\tG_x}$ and $\gr \K[\widetilde{\mathcal{I}}]$ coincide. Since $ \K[\widetilde{\mathcal{I}}]\subseteq U(\tg_x)^{\tG_x}$ we must actually have an equality and part 2 follows.

Part 3 is an immediate consequence of parts 1 and 2, and the fact that the $Z_p(\tg_x)$-basis of $Z(\tg_x)$ is also a $Z_p(\tg_x)^{\tG_x}$-basis of $Z(\tg_x)$.
\end{proof}

\p We remind the reader that the dimensions of simple modules for modular Lie algebras are bounded. This fact essentially follows from the fact that the enveloping algebra is a finite module over its centre, and sits in stark contrast to the characteristic zero case, where each simple algebra has simple modules of arbitrarily high dimension. For each Lie algebra $\tg$ over $\K$ we denote by $M(\tg)$ the maximal dimension of simple modules. The first Kac-Weisfeiler conjecture asserts that $$M(\tg) = p^{\frac{1}{2}(\dim\,\tg - \ind\,\tg)}.$$ This conjecture is extremely general and should be approached with some trepidation.

\p Using Theorem~\ref{kacweisover}, and our results on invariant theory, we now prove Theorem~\ref{KW1} which states that the first Kac-Weisfeiler conjecture holds for centralisers in types $\sf A$ and $\sf C$. We actually give two proofs, very different in nature. The first of these is more general and applies to all centralisers in types $\sf A$ and $\sf C$ whilst the second proof only applies only in the case $x$ nilpotent with $\tg$ of type $\sf A$. We now supply the first proof, which is purely algebraic and makes use of Theorem~\ref{kacweisover}.
\begin{proof}
Let $\tG \in \{G, K\}$ and $x \in \tg$. Define $F := F(\tg_x)$ and $F_p := F_p(\tg_x)$. By Zassenhaus' theorem we may prove that $[F : F_p] = p^{\ind\, \tg}$. In the notation of the previous theorem we let $F' = F_p [\widetilde{\mathcal{I}}] = Z(\tg_x)\otimes_{Z_p(\tg_x)} F_p$. Clearly $F' \subseteq F$. We claim that this inclusion is actually an equality. It will suffice to prove that every element of $F$ can be written in the form $f/g$ where $f \in Z(\tg_x)$ and $g \in Z_p(\tg_x)$. This is actually a very general result which holds whenever we have a module-finite inclusion of $\K$-algebras. We shall recall the proof for the reader's convenience.

Suppose $f/g \in F$. Since $Z(\tg_x)$ is a finite $Z_p(\tg_x)$-module we have $a_n g^n  + a_{n-1} g^{n-1} + \cdots + a_0 = 0$ for some $n \in \N$ and certain coefficients $a_i \in Z_p(\tg_x)$ with $a_0 \neq 0$. We conclude that $$\frac{f}{g} = \frac{f(a_ng^{n-1} + a_{n-1}g^{n-2} + \cdots + a_1)}{g(a_ng^{n-1} + a_{n-1}g^{n-2} + \cdots + a_1)} =  \frac{f(a_ng^{n-1} + a_{n-1}g^{n-2} + \cdots + a_1)}{-a_0} \in F'.$$
Therefore $F = F'$ as claimed. Since $Z(\tg_x)$ is free of rank $p^{\ind\, \tg}$ over $Z_p(\tg_x)$ (Theorem~\ref{main2}) we conclude that $F$ is of dimension $p^{\ind\, \tg}$ over $F_p$. Thanks to \cite{Yak1} this is equal to $p^{\ind \, \tg_x}$ and the theorem follows.
\end{proof}

\p We now supply our second, more geometric proof of Theorem~\ref{KW1}. We remind the reader that the following proof applies only to the case $x$ nilpotent and $\tG = G = GL(V)$. As such we shall retain the notations of \ref{genlinnots}. 
\begin{proof}
Define
\begin{eqnarray*}
\Omega = \{ \chi \in \g_e^\ast : \dim\, W = M(\g_e) \text { for all simple } U_\chi(\g_e) \text{-modules } W \}
\end{eqnarray*}
Let $\chi \in \g_e^\ast$. Every $g \in G_e$ acts on $\g_e^\ast$ and induces an isomorphism of algebras
\begin{eqnarray*}
U_\chi(\g_e) \xrightarrow{\sim} U_{\Ad(g)\chi}(\g_e)
\end{eqnarray*}
Hence $\Omega$ is $G_e$-stable. In \cite[Proposition~4.2(1)]{PS} it was shown that $\Omega$ is a non-empty Zariski open subset of $\mathfrak{g}_e^\ast$. By Theorem~\ref{genstabal} there exists a nonempty open subset $\mathcal{O} \subseteq\g_e^\ast$ such that the $\g_e$-stabilisers of points in $\mathcal{O}$ are $G_e$-conjugate to $\mathfrak{h} = \{\xi_i^{i,\lambda_i - 1-s} :1\leq i\leq n, 0 \leq s < \lambda_i\}$. Since $\mathfrak{h}$ is abelian and stabilises $\mathfrak{n}^-\oplus \mathfrak{n}^+$ the stabilisers of points in $\mathfrak{h}^\ast$ include $\mathfrak{h}$.

Hence we can find $\chi \in \Omega \cap \mathfrak{h}^\ast$, ie. a linear function vanishing on $(\mathfrak{n}^-)^\ast$ and $(\mathfrak{n}^+)^\ast$ with the additional property that every irreducible $U_\chi(\mathfrak{g}_e)$-module has dimension $M(\mathfrak{g}_e)$. Fix such a $\chi$ and identify it with its restriction to $\mathfrak{h}$. The decomposition of $\g_e$ induces a decomposition of enveloping algebras
\begin{eqnarray*}
U_\chi(\mathfrak{g}_e) = U_0(\mathfrak{n}^-) \otimes_\K U_\chi(\mathfrak{h}) \otimes_\K U_0(\mathfrak{n}^+)
\end{eqnarray*}
Let $W$ be a simple $U_\chi(\mathfrak{g}_e)$-module so that $\dim\, W = M(\g_e)$. By Engel's theorem $\mathfrak{n}^+$ has a common eigenvector in $W$ of eigenvalue 0 and $\mathfrak{h}$ acts diagonally on $W^{\mathfrak{n}^+} \neq 0$. We can find a 1-dimensional $U_\chi(\mathfrak{h}) \otimes_\K U_0(\mathfrak{n}^+)$-module $\K w_\mu \subseteq W^{\mathfrak{n}^+}$ where $h w_\mu = \mu(h) w_\mu$ for all $h \in \mathfrak{h}$. 

The induced module $\Ind_{\mathfrak{h} + \mathfrak{n}^+}^{\mathfrak{g}_e} (\K w_\mu) = U_\chi(\mathfrak{g}_e) \otimes_{U_\chi(\mathfrak{h}) \otimes U_0(\mathfrak{n}^+)} \K w_\mu$ has dimension $\dim\, U_0(\mathfrak{n}^-) = p^{\dim\, \mathfrak{n}^-}$. We have $\dim\, \mathfrak{n}^- = \dim\, \mathfrak{n}^+$ and $\dim\, \mathfrak{h} = N = \ind\, \g_e$ by \cite[Theorem~1]{Yak1}. We obtain $\dim\, \g_e = 2\dim\, \mathfrak{n}^- + \ind\, \g_e$. Therefore
$$\dim\, \Ind_{\mathfrak{h} + \mathfrak{n}^+}^{\mathfrak{g}_e} (\K w_\mu) = p^{\frac{1}{2}(\dim\, \mathfrak{g}_e - \ind\, \mathfrak{g}_e)}.$$
By standard theory of induced modules, $W$ is a homomorphic image of $\Ind_{\mathfrak{h} + \mathfrak{n}^+}^{\mathfrak{g}_e} (\K w_\mu)$ so the dimension of the former is bounded above by that of the latter. According to \cite[Theorem~5.4(2)]{PS} it is possible to choose $\chi$ so that all finite dimensional $U_\chi(\g_e)$-modules have dimension divisible by $p^{\frac{1}{2}(\dim\, \mathfrak{g}_e - \ind\, \mathfrak{g}_e)}$. This ensures that $W = \Ind_{\mathfrak{h} + \mathfrak{n}^+}^{\mathfrak{g}_e} (\K w_\mu)$, and that $M(\g_e) = p^{\frac{1}{2}(\dim\, \mathfrak{g}_e - \ind\, \mathfrak{g}_e)}$. 
\end{proof}
\begin{rem}
\rm{Although our second proof is much less general it does have the virtue of implying that the Verma modules for $\g_e$ are generically simple. Part 3 of Lemma~\ref{spanningdetails} ensures that there is, in general, no `triangular' decomposition for classical types other than type $\sf A$, and it seems this proof will extend no further.}
\end{rem}

\p The proof of the KW1 conjecture allows us to make some powerful deductions about the Zassenhaus variety
$\mathcal{Z}(\tg_x) := \Specm\,  Z(\tg_x)$ when $\tg$ is of type $\sf A$ or $\sf C$ and $x \in \tg$. There is a natural map
from the isomorphism classes of simple $\tg_x$-modules to $\Zas(\tg_x)$
defined by taking the kernel of the central character of a given simple module. This map has finite fibres thanks to \cite{Zas}.

\begin{thm}\label{zassen}
If $\tg$ is a simple Lie algebra of type $\sf A$ or $\sf C$ over $\K$ and $x \in \tg$ then $\m$ is a smooth point of $\ZZ(\tg_x)$ if and only if $U(\tg_x)/ \m U(\tg_x) \cong \text{\emph{Mat}}_{M(\tg_x)}(\K)$.
\end{thm}
\begin{proof}
Recall that the singular locus $(\tg_x^\ast)_\text{sing}$ is defined to be the set of all $\chi \in \tg_x^\ast$ such that $\dim \, (\tg_x)_{\chi} > \ind\, \tg_x$. By \cite[Theorems~3.4 \& 3.11]{PPY}, $\codim_{\tg_x^\ast} \, (\tg_x^\ast)_{\text{sing}} \geq 2$ provided $x$ is nilpotent. Strictly speaking they worked over $\C$ there but the assumption was not necessary; indeed \cite{Yak1} works in good characteristic and that article is the basis for the aforementioned two theorems. Now the fact that $\codim_{\tg_x^\ast}\,  (\tg_x^\ast)_{\text{sing}} \geq 2$ when $x$ is not nilpotent may be proven using a reduction to the nilpotent case very similar to the one used for Theorem~\ref{main1}. For a non-negative integer $m$ we let $\mathscr{X}_m$ denote the set of all $\chi \in \tg_x^\ast$ such that $U_\chi(\tg_x)$ has a module of finite dimension not divisible by $p^m$. According to 
\cite[Proposition~5.2]{PS}  $\dim \, \mathscr{X}_{M(\tg_x)} \leq \dim \, \tg^\ast_\text{sing}$, from whence we deduce that $\codim_{\tg_x^\ast} \, \mathscr{X}_{M(\tg_x)} \geq 2$. 

By \cite[Theorems~5 \& 6]{Zas}, we may infer that there is a closed subset $\CC \subseteq \ZZ$ such that $$\ZZ \backslash \CC = \{ \m \in \ZZ : U(\tg_x)/ \m U(\tg_x) \cong \text{Mat}_{M(\tg_x)}(\K)\}.$$ We claim that $\codim_{\ZZ}\,  \CC \geq 2$. The inclusion $Z_p(\tg_x) \ra Z(\tg_x)$ induces as finite morphism of maximal spectra $\ZZ \ra \Specm \, Z_p(\tg_x)$ and by identifying $\Specm \, Z_p(\tg_x)$ and $\tg_x^\ast$ we obtain a map $\tau : \Specm\,  Z(\tg_x ) \ra \tg_x^\ast$. Explicitly, $\m \mapsto \chi$ where $\chi$ is the linear functional such that $\m \cap Z_p(\tg_x) = \m_\chi$ and $U_\chi(\tg_x) = U(\tg_x) / \m_\chi U(\tg_x)$. This map is closed and has finite fibres so $\codim_{\ZZ} \, \CC = \codim_{\tg_x^\ast}\, \tau(\CC)$. If $\chi \in \tau(\CC)$ then $U_\chi(\tg_x)$ has a module of dimension less than $M(\tg_x)$ so that $\chi \in \mathscr{X}_{M(\tg_x)}$. We conclude that $$\codim_{\ZZ} \, \CC = \codim_{\tg_x^\ast}\, \tau(\CC) \geq \codim_{\tg_x^\ast} \, \X_{M(\tg_x)} \geq 2.$$

Pick $\m \in \ZZ \setminus \CC$. Denote by $Z_\mm$ the localisation of $Z(\tg_x)$ at $\m \in \ZZ$ and set $U_\mm = U(\tg_x) \otimes_{Z(\tg_x)} Z_\mm$. Certainly $U_\m$ is a finite module over $Z_\m$. Furthermore the unique maximal ideal of $Z_\m$ is $\m Z_\m$ and $U_\m/\m U_\m \cong U(\tg_x) /\m U(\tg_x)$ is a matrix algebra over its centre $Z / \m Z \cong \K$. Now we may apply \cite[Theorem~7.1]{DI} to see that $U_\m$ is an Azumaya algebra over $Z_\m$. Since $\m \in \ZZ\setminus \CC$ was arbitrary we conclude that the non-Azumaya locus of $\ZZ$ lies inside $\CC$ and therefore has codimension $\geq 2$. Combining these deductions with \cite[2.2, 2.3]{BG}, we have satisfied the assumptions of \cite[Theorem~3.8]{BG} and the result follows.
\end{proof}

\chapter{The Derived Subalgebra and Sheets}\label{derivedchapter}
\setcounter{parno}{0}

\p A full overview of the current chapter is given in Section~\ref{shcnoverview}. Let's recap the important points. Our discussion shall be restricted to centralisers in simple algebras of type $\sf B$, $\sf C$ or $\sf D$. Fix $\epsilon = \pm 1$ and choose $e \in \Ni(\h)$. Our first task is to decompose $[\h_e \h_e]$ into a finite direct sum of subspaces. As a corollary we obtain a formula for the dimension of the maximal abelian quotient $\dim\, \h_e^\text{ab}$ where $\h_e^\ab := \h_e / [\h_e \h_e]$. We go on to define an algorithm which is used for parameterising the sheets of $\h$ which contain a given nilpotent orbit $\Oo_e$, and calculate their ranks. We provide a combinatorial formula for the maximal rank of sheets containing $\Oo_e$ and classify the nilpotent elements lying in a unique sheet in terms of partitions. These elements shall be called non-singular. Our efforts culminate in a proof of a conjecture of Izosimov \cite[Conjecture~4.1]{Izos} which states that for $e$ non-singular the derived 
subalgebra $[\h_e \h_e]$ is the orthogonal complement to the tangent space of the sheet containing $e$ (we actually prove this without the assumption that $e$ is nilpotent).

\section{Decomposing $\h_e$}\label{decomposingh_e}
\setcounter{parno}{0}

\p Our notations in this chapter shall be those laid out in Section \ref{basisforthecent}. In particular $\K$ is an algebraically closed field, $V$ is an $N$-dimensional vector space over $\K$, $G = GL(V)$, $\sigma : \g \ra \g$ is the involution associated to our choice of non-degenerate bilinear form $(\cdot, \cdot) : V \times V \ra \K$ and $\g = \h \oplus \pp$ is the associated $\Z_2$-grading. Our group of study is the orthogonal or symplectic group $K$ with $\Lie(K) = \h$, and so we request that $\chr(\K)\neq 2$. We fix a nilpotent element $e \in \h$ with ordered partition $\lambda = (\lambda_1,...,\lambda_n)$. Thanks to Lemma~\ref{subbasis} we have a very explicit basis for $\h_e$ which decomposes as $H \sqcup N_0 \sqcup N_1$. We shall principally make use of the spans of these sets: $\HH = \spn(H), \NN_0 = \spn(N_0), \NN_1 = \spn(N_1)$. We suggest that the reader keep the definitions of $H$, $N_0$ and $N_1$ close at hand in what follows.

\p\label{H_0inst} It is our intention to decompose $[\h_e \h_e]$ into subspaces. In order to do so we must first decompose $\HH$ and $\NN_1$. Let
\begin{eqnarray*}
&\HH_0 := \spn\{\zeta_i^{i,s}\in \HH : \lambda_i - s \text{ even} \}\\
&\HH_1 := \spn\{\zeta_i^{i,s} \in\HH : \lambda_i - s \text{ odd} \}
\end{eqnarray*}
so that $\HH = \HH_0 \bigoplus \HH_1$. The space $\HH_0$ can be further decomposed as $\bigoplus_{m=1}^{\lfloor\lambda_1/2\rfloor} \HH_0^{m}$ where $$\HH_0^{m} := \spn\{\zeta_i^{i,\lambda_i - 2m} \in \HH: 1 \leq i \leq n\}.$$ 

\p Next we must decompose each $\HH_0^{m}$ into subspaces $\HH_{0,j}^m$ for $j\ge 1$. Fix $0 < m \leq \lfloor \lambda_1/2 \rfloor$, put $a_{1,m}:=1$ and let $1=a_{1,m}< a_{2,m}<\cdots< a_{t(m), m}\le n+1$ be the set of all integers such that $$\lambda_{a_{j,m}-1} - \lambda_{a_{j,m}} \geq 2m, \qquad 2\le j\le t(m).$$ For $1 \leq j < t(m)$ we define $$\HH_{0,j}^m := \spn \{\zeta_i^{i,\lambda_i - 2m} \in \HH: a_{j,m} \leq i < a_{j+1, m}\}$$ and set $$\HH_{0,t(m)}^m := \spn \{\zeta_i^{i,\lambda_i - 2m} \in \HH: a_{t(m),m} \leq i < n+1\}.$$
\begin{lem}
The following are true:
\begin{enumerate}\label{bits}
\item{If $\lambda_{a_{t(m),m}} < 2m$ then $\HH_{0,t(m)}^m = \{0\}$;}

\smallskip

\item{$\HH_0^m = \bigoplus_{j=1}^{t(m)} \HH_{0,j}^m$.}
\end{enumerate}
\end{lem}
\begin{proof}
If $a_{t(m),m} = n+1$ then certainly $\HH_{0,t(m)}^m = 0$, so assume not. If $\lambda_{a_{t(m),m}} < 2m$ then the ordering $\lambda_1 \geq \cdots \geq \lambda_n$  implies that $\lambda_i - 2m < 0$ for all $i \geq a_{t(m),m}$. Then $\zeta_i^{i,\lambda_i-2m} = 0$ for all $\zeta_i^{i,\lambda_i - 2m} \in \HH_{0,t(m)}^m$ proving (1).
The choice of $m$ (and the fact that $a_{1,m} = 1$) ensures that $\bigoplus_{l=1}^{t(m)} \HH_{0,j}^m = \spn\{\zeta_i^{i,\lambda_i- 2m}\in \HH : 1\leq i \leq n\} = \HH_0^m$. Hence (2).
\end{proof}

\p It should be noted that if $i \neq i'$ then $\varepsilon_{i,i,\lambda_i - 2m} = 1$ by 
Lemma~\ref{spanningdetails}. In this case $\zeta_i^{i,\lambda_i - 2m} = \zeta_{i'}^{i',\lambda_{i'} - 2m}$ by that same lemma. In order to overcome this notational problem and concisely refer to a basis for $\HH_{0,j}^m$ it shall be convenient to use an indexing set slightly different from $\{1,...,n\}$. Extend the involution $i \mapsto i'$ to all of $\Z$ by the rule $i = i'$ for $i> n$ or $i < 1$. We adopt the convention $\lambda_i = 0$ for all $i > n$ or $i < 1$ which immediately implies $\zeta_i^{i,s} = 0$ for any such $i$. We shall index our maps and partitions by the set $\Z/\sim$ where $i\sim j$ if $i = j'$. We denote by $[i]$ the class of $i$ in $\Z /\sim$. We have $\lambda_i = \lambda_{i'}$ for all $i$ so we may introduce the notation $\lambda_{[i]}$. As was observed a moment ago, $\zeta_i^{i,\lambda_i - 2m} = \zeta_{i'}^{i',\lambda_{i'} - 2m}$. Hence we may also use the notation $\zeta_{[i]}^{[i], \lambda_{[i]} - 2m}$. Furthermore, since $i' \in \{i-1, i, i+1\}$ we have a well defined
order on $\Z /\sim$ inherited from $\Z$: let $[i] \leq [j]$ if $i \leq j$. As a result there exists a unique isomorphism of totally ordered sets $\psi : (\Z / \sim) \rightarrow \Z$ with $\psi([1]) = 1$. Using this isomorphism we define analogues of addition and subtraction $+ , - \colon\, (\Z / \sim) \times \Z \mapsto (\Z / \sim)$ by the rules
\begin{eqnarray*}
&[i] + j := \psi^{-1}(\psi(i) + j)\\
&[i] - j := \psi^{-1}(\psi(i) - j)
\end{eqnarray*}
To clarify, $[i] + 1$ is the class in $(\Z/\sim)$ succeeding $[i]$ and $[i]-1$ is that class preceding $[i]$ in the ordering.

\p For $1\leq j < t(m)$, Lemma~\ref{spanningdetails}(3) yields that the set $$\big\{ \zeta_{[i]}^{[i],\lambda_{[i]} - 2m}\in \HH:\, [a_{j,m}] \leq [i] < [a_{j+1, m}]\big\}$$ is a basis for $\HH_{0,j}^m$. Using this basis we may describe an important hyperplane $\HH_{0,j}^{m,+}$ of $\HH_{0,j}^m$. First we define the augmentation map
$\HH_{0,j}^m\twoheadrightarrow \K$ by sending
$\zeta_{[i]}^{[i],\lambda_{[i]} - 2m}$ to $1$ for all  $[a_{j,m}] \leq [i] < [a_{j+1, m}]$ and extending to $\HH_{0,j}^m$ by $\K$-linearity.
Let $\HH_{0,j}^{m,+}$ denote the kernel of this map. It was noted in lemma \ref{bits} that $\HH_{0,t(m)}^m$ might be zero. If this is not the case then a basis for $\HH_{0,t(m)}^m$ is the span of those $\zeta_{[i]}^{[i], \lambda_{[i]} - 2m}$ which are non-zero with $[a_{t(m),m}] \leq [i] \leq [n]$. Using this basis we can define the augmentation map $\HH_{0,t(m)}^m \twoheadrightarrow \K$ and hyperplane $\HH_{0,t(m)}^{m,+}$ of $\HH_{0,t(m)}^m$ in a similar fashion. Make the notation $$\HH_0^+\,:=\, \sum_{m=1}^{\lfloor \lambda_1/2 \rfloor} \Big(\textstyle{\bigoplus}_{j=1}^{t(m)-1}  \HH_{0,j}^{m,+}+ \HH_{0,t(m)}^m\Big)\subseteq \HH_0.$$

\p Before we continue we must decompose $\NN_1$ into a direct sum of two subspaces as follows. For $1 \leq k \leq n$ we define lower and upper column bounds by
\begin{eqnarray*}
L(k) = \min\{ 1 \leq j \leq n : \lambda_j = \lambda_k\} \\ U(k) = \max\{ 1 \leq j \leq n : \lambda_j = \lambda_k\}.
\end{eqnarray*}
Clearly $\lambda_j = \lambda_k$ for $L(k) \leq j \leq U(k)$. We say that a triple $(i,j,s)$ with $1\leq i,j\leq n$ and $0\leq s < \min(\lambda_i, \lambda_j)$ is \emph{tightly nested} when the following conditions hold:
\begin{itemize}
\item{$j = i+1$}

\smallskip

\item{$s = \lambda_{i+1} -1$}

\smallskip

\item{$i = i'$}

\smallskip

\item{$i+1 = (i+1)'$}

\smallskip

\item{$L(i) = i$}

\smallskip

\item{$U(i+1) = i+1$}
\end{itemize}
Finally we make the notation
\begin{eqnarray*}
\NN_1^- = \spn \{\zeta_{i}^{j,s} \in \NN_1: (i,j,s) \text{ is a tightly nested triple} \};\\
\NN_1^+ = \spn \{ \zeta_{i}^{j,s} \in \NN_1 : (i,j,s) \text{ is not a tightly nested triple} \}.
\end{eqnarray*}
so that $$\NN_1 = \NN_1^- \oplus \NN_1^+.$$

\section{Decomposing $[\h_e \h_e]$}\label{decomposingderived}
\setcounter{parno}{0}

\p It is the intention of this section to decompose $[\h_e \h_e]$ into a finite collection of those subspaces of $\h_e$ defined in the previous section. Our calculations shall be quite explicit and depend principally upon the following.
\begin{lem}\label{product} For all indexes
$i,j,s$ and $k,l,r$,
$$[\zeta_i^{j,s}, \zeta_k^{l,r}] = \delta_{il} \zeta_k^{j,r+s - (\lambda_i - 1)} - \delta_{jk} \zeta_i^{l,r+s - (\lambda_j - 1)}+ \varepsilon_{klr}\big(\delta_{k,i'} \zeta_{l'}^{j,r+s - (\lambda_i - 1)} - \delta_{j,l'} \zeta_i^{k',r+s - (\lambda_j - 1)}\big).$$
\end{lem}
\noindent The proof is a short calculation which we leave to the reader. 

\p The following proposition shall be central in the process of decomposing $[\h_e \h_e]$.
\begin{prop}\label{subs}
The following inclusions hold
\begin{eqnarray*}
[\HH,\HH] = \{0\} ,& [\HH, \NN_0]\subseteq \NN_0, & [\HH, \NN_1] \subseteq \NN_1, \\
& [ \NN_0, \NN_0 ] \subseteq \HH, & [ \NN_0, \NN_1 ] \subseteq \NN_1
\end{eqnarray*}
Furthermore, for any two elements $\zeta_i^{j,s}, \zeta_k^{l,r} \in N_1$ the commutator $[\zeta_i^{j,s}, \zeta_k^{l,r}]$ lies in either $\HH$, $\NN_0$ or $\NN^+_1$. More precisely $$[\zeta_i^{j,s}, \zeta_k^{l,r}] \in\, \left\{ \begin{array}{llll}
         \NN^+_1 & \ \mbox{ if $i=l$ or $j=k$;}\\
         \NN_0 \text{ or } \NN_1^+& \ \mbox{ if $k = i'$ or $j = l'$ but not both};\\
         \HH & \ \mbox{ if $k = i'$ and $j = l'$;}\\
         0 & \ \mbox{ otherwise.}\end{array} \right.$$
\end{prop}
\begin{proof}
We shall call on the characterisations of $\HH, \NN_0$ and $\NN_1$ given in lemma \ref{subbasis}. Thanks to \cite[Theorem~1]{Yak1} we have $\HH = \h \cap (\g_e)_\alpha$ where $\alpha$ is a regular element of $\g_e^\ast$. By \cite[1.11.7]{Dix} the stabiliser $(\g_e)_\alpha$ is abelian, hence $[\HH \HH] = 0$. The elements of $\NN_0$ are characterised by the fact that they exchange the spaces $V[i]$ and $V[i']$ with $i \neq i'$. Therefore the elements of $[\HH \NN_0]$ must exchange them also, implying $[\HH \NN_0] \subseteq \NN_0$. Each $\zeta_i^{j,s} \in \NN_1$ transports $V[i]$ to $V[j]$ and $V[j']$ to $V[i']$. Therefore $[\HH \zeta_i^{j,s}]$ does likewise and $[\HH \NN_1] \subseteq \NN_1$. Since each element of $\NN_0$ exchanges the spaces $V[i]$ and $V[i']$ with $i\neq i'$ and anihilates all $V[i]$ with $i=i'$ the poduct $[\NN_0 \NN_0]$ must stabilise all $V[i]$, hence be contained in $\HH$. The inclusion $[\NN_0 \NN_1] \subseteq \NN_1$ is best checked using lemma \ref{product}.
Let $i \neq i'$ and $l > k \neq l'$. Then $[\zeta_i^{i',s} \zeta_k^{l,r}]$ is nonzero only if $i = l$ or $i' = k$. Our restrictions on $i,l$ and $k$ ensure that these two possibilities are mutually exclusive. In the first case $[\zeta_i^{i',s} \zeta_k^{l,r}] = \zeta_k^{l', r + s - (\lambda_i-1)} - \varepsilon_{klr} \zeta_l^{k', r+s- (\lambda_i-1)}$ which lies in $\NN_1$. The second case is very similar.

We consider the final claim. Suppose $j > i\neq j'$ and $l > k \neq l'$. By lemma \ref{product} the bracket $[\zeta_i^{j,s} \zeta_k^{l,r}]$ is only nonzero when one or more of the following equalities hold: $i=l, j=k, i'=k, j'=l$. We shall consider these four possibilities one by one. Since the bracket is anticommutative the reasoning in the case $i=l$ is identical to the case $j=k$ and so we need consider only the first of these two possibilities. If $i=l$ then the relations $i' \neq j > i$ and $l > k \neq l'$ ensure that $j\neq k, i' \neq k$ and $j' \neq l$. Therefore $[\zeta_i^{j,s} \zeta_k^{l,r}] = \zeta_k^{j, r+s - (\lambda_i-1)} \in \NN_1$. In order for this map to lie in $\NN_1^-$ we would require $j = k+1$, however we have $j > i = l > k$ which makes this impossible. Thus $[\zeta_i^{j,s} \zeta_k^{l,r}] \in \NN_1^+$.

By lemma \ref{spanningdetails} we have $\zeta_i^{j,s} = \pm \zeta_{j'}^{i',s}$ and $\zeta_k^{l,r} = \pm\zeta_{l'}^{k',r}$ so the reasoning in case $i=k'$ is identical to the case $j'=l$. Therefore we need only to consider the first of these two possibilities. Suppose $i=k'$. Then certainly $i \neq l$ and $j \neq k$. If $j' = l$ then $[\zeta_i^{j,s} \zeta_k^{l,r}] = \varepsilon_{klr} (\zeta_j^{j,r+s-(\lambda_i - 1)} - \zeta_i^{i,r+s - (\lambda_j-1)}) \in \HH$ so assume from henceforth that $j' \neq l$. Then $$[\zeta_i^{j,s} \zeta_k^{l,r}] = \varepsilon_{klr} \zeta_{l'}^{j,r+s - (\lambda_i-1)}.$$ If $j = l$ then the product lies in $\NN_0$. Assume $j\neq l$. Thanks to the relation $\zeta_{l'}^{j,r+s - (\lambda_i-1)} = \pm \zeta_{j'}^{l,r+s - (\lambda_i-1)}$ from lemma \ref{spanningdetails} we may assume that $j>l'$, and from here it is easily seen that the product lies in $\NN_1$. In order for the product to lie in $\NN_1^-$ we require $L(l') = l'$ which implies $\lambda_l < \lambda_i$ since $i = k < l$. From 
the bounds $0 \leq r < \lambda_l$ and $0 \leq s < \lambda_j$ we deduce that $r + s - (\lambda_i - 1) < \lambda_j - 1$ which confirms that the term $\zeta_{l'}^{j,r+s - (\lambda_i-1)}$ does not lie in $\NN_1^-$.
\end{proof}

\p The following proposition tells us how $\NN$ intersects with $[\h_e \h_e]$.

\begin{prop}\label{Nprop} The following are true:
\begin{enumerate}
\item{$\NN_0 \subset [\h_e \h_e]$;}
\item{$\NN_1 \cap [\h_e \h_e] = \NN_1^+$.}
\end{enumerate}
\end{prop}
\begin{proof}
Throughout the proof we shall need to evaluate expressions $\varepsilon_{i,j,s}$. When doing so we shall call on part 1 of lemma \ref{spanningdetails}. Likewise, whenever we evaluate expressions $[\zeta_i^{i,s} \zeta_k^{l,r}]$ we shall call on lemma \ref{product}. Assume $i \neq i'$ and $\lambda_i - s$ is odd. We have $\varepsilon_{i,i,s} = (-1)^{\lambda_i-s}$ so
$$[\zeta_i^{i',s}\zeta_i^{i,\lambda_{i}-1}] = \zeta_i^{i',s} - \varepsilon_{i,i,\lambda_i-1}\zeta_i^{i',s} = 2\zeta_i^{i',s} \in [\h_e \h_e].$$ Hence $\NN_0 = [\HH \NN_0] \subseteq [\h_e \h_e]$. Here we use the fact that char$(\K) \neq 2$. This concludes 1. For the sake of clarity we shall divide the proof of part 2 of the current proposition into subsections \emph{(i), (ii),..., (ix).} In parts \emph{(i),...,(v)} we demonstrate that $\NN_1^+ \subseteq [\h_e \h_e]$ by showing that if one of the six criteria for $(i,j,s)$ to be a tightly nested triple fails to hold then some multiple of $\zeta_i^{j,s}$ may be found as a product of two basis elements in $\h_e$. In parts \emph{(v),...,(viii)} we show that the reverse inclusion holds by noting that $\NN_1\cap [\h_e\h_e]$ is the sum of those products $[\zeta_i^{j,s} \zeta_k^{l,r}]$ which lie in $\NN_1$, and showing that all such products actually lie in $\NN_1^+$. For the remainer of the proof we shall fix $l > k \neq l'$ so that
$\min(\lambda_k , \lambda_l) = \lambda_l$.

\emph{(i) If $l \neq l'$ or $k \neq k'$, then $\zeta_k^{l,r} \in [\h_e \h_e]$ for $r = 0,1,..., \lambda_l-1$:} Suppose first that $l \neq l'$. We have
$$[\zeta_l^{l,\lambda_l-1} \zeta_k^{l,r}] = \zeta_k^{l, r} \in [\h_e \h_e]$$
whence we obtain $\zeta_k^{l,r} \in [\h_e \h_e]$ for $r = 0,1,..., \lambda_l-1$. Now suppose $k \neq k'$. Then
$$[\zeta_k^{k,\lambda_k-1} \zeta_k^{l,r}] = - \zeta_k^{l, r} \in [\h_e \h_e]$$
so that $\zeta_k^{l,r} \in [\h_e \h_e]$ for all $r = 0,1,..., \lambda_l-1$.

\emph{(ii) If $l'= l$ and $k = k'$ then $\zeta_k^{l,r} \in [\h_e \h_e]$ for $r=0,1,...,\lambda_l -2$:} With $l$ and $k$ as above $$[\zeta_l^{l,\lambda_{l}-2} \zeta_k^{l,r}] = \zeta_k^{l,r -1} - \varepsilon_{k,l,r} \zeta_{l'}^{k', r-1} \in [\h_e \h_e].$$ By part 3 of lemma \ref{spanningdetails} this final expression is $(1 - \varepsilon_{k,l,r} \varepsilon_{k,l,r-1}) \zeta_k^{l,r-1}$. Since $\varepsilon_{k,l,r} = (-1)^{\lambda_l-r}$ this expression actually equals $2\zeta_k^{l,r-1}$. Allowing $r$ to run from $0$ to $\lambda_l-1$ we obtain the desired result.

\emph{(iii) If $l' = l$, $k = k'$ and $k \neq l-1$ then $\zeta_k^{l,r} \in [\h_e \h_e]$ for $r=0,1,...,\lambda_l-1$:} We may assume there exists $j$ fulfilling $l > j > k$. Then $k \neq l$ and $k' \neq j \neq l'$ so that $$[\zeta_j^{l,r} \zeta_k^{j,\lambda_j-1}] =  \zeta_k^{l,r}\in [\h_e \h_e].$$

\emph{(iv) If $l = l'$, $k = k'$ and either $L(k) \neq k$ or $U(l) \neq l$, then $\zeta_k^{l,r} \in [\h_e \h_e]$ for $r = 0, 1,..., \lambda_l - 1$:} First suppose that $L(k) \neq k$. Since $k = k'$ we have $L(k) = L(k)'$ and $\lambda_{L(k)} = \lambda_k$ so that for $r = 0, 1,..., \lambda_l - 1$ $$[\zeta_{L(k)}^{l,r} \zeta_{L(k)}^{k,\lambda_k-1}] = \varepsilon_{L(k), k, \lambda_k -1} \zeta_{k}^{l,r}\in [\h_e \h_e].$$ Next suppose that $U(l) \neq l$. Since $l=l'$ we have $U(l) = U(l)'$ and $\lambda_{U(l)} = \lambda_l$ so that for $r = 0, 1,..., \lambda_l - 1$ $$[\zeta_{k}^{U(l),\lambda_{U(l)}-1} \zeta_{l}^{U(l),r}] = -\varepsilon_{l, U(l),r} \zeta_k^{l, r}\in [\h_e \h_e].$$

\emph{(v) $\NN_1^+ \subseteq [\h_e \h_e]$: } Combine the deductions of parts \emph{(i) - (iv)}.

\emph{(vi) $[\HH \NN_1] \subseteq \NN_1^+$:} We continue to fix $l > k \neq l'$. The product $[\zeta_i^{i,s} \zeta_k^{l,r}]$ is nonzero only if $i = k$ or $i = l$. Assume $i = l$ (the case $i = k$ is similar). Then $[\zeta_i^{i,s} \zeta_k^{l,r}] = \zeta_k^{l,r+s-(\lambda_i - 1)}$ which lies either in $\NN_1^-$ or $\NN_1^+$. In order for $\zeta_k^{l,r+s-(\lambda_i - 1)} \in \NN^-$ we require that $l=l'$. But in that case $i = i'$ so by definition of $\HH$ we have $\lambda_i - s$ even. In particular, $s \leq \lambda_i - 2$ and $r + s - (\lambda_i-1) \leq r - 1 < \lambda_l - 1$ so $[\zeta_i^{i,s} \zeta_k^{l,r}] \in \NN_1^+$.

\emph{(vii) $[\NN_0 \NN_1] \subseteq \NN_1^+$:} The product $[\zeta_i^{i',s} \zeta_k^{l,r}]$ with $i\neq i'$ is nonzero only if $i = l$ or $i' = k$. Our restrictions on $i,l$ and $k$ ensure that these two possibilities are mutually exclusive. In the first case $$[\zeta_i^{i',s} \zeta_k^{l,r}] = \zeta_k^{l', r + s - (\lambda_i-1)} - \varepsilon_{klr} \zeta_l^{k', r+s- (\lambda_i-1)} = (1 - \varepsilon_{klr} \varepsilon_{k,l',r+s-(\lambda_i-1)})\zeta_k^{l',r+s-(\lambda_i-1)}$$ however $\zeta_k^{l',r+s-(\lambda_i-1)}$ lies in $\NN_1^-$ only if $l = l'$. The assumption $i = l$ implies $i = i'$ contrary to our assumptions. 

Now consider the case $i' = k$. A calculation similar to the above gives $$[\zeta_i^{i',s} \zeta_k^{l,r}] = (\varepsilon_{klr}  \varepsilon_{k,l,r+s - (\lambda_i - 1)} - 1)\zeta_i^{l,r+s-(\lambda_i-1)}.$$ Note that this term lies in $\NN_1^-$ only if $i = i'$ contrary to our assumptions, so \emph{(vii)} follows.

\emph{(viii) $\NN_1 \cap [\NN_1 \NN_1] \subseteq \NN_1^+$:} This follows immediately from the final statement of proposition \ref{subs}.

\emph{(ix) $\NN_1\cap [\h_e \h_e] = \NN_1^+$:} By \emph{(v)} we know that $\NN_1^+ \subseteq \NN_1\cap [\h_e \h_e] $. By proposition \ref{subs}, $\NN_1 \cap [\h_e \h_e]$ is equal to the span of those products $[\zeta_i^{j,s} \zeta_k^{l,r}]$ which lie in $\NN_1$. By that same proposition and parts \emph{(vi) - (viii)} we see that every product $[\zeta_i^{j,s} \zeta_k^{l,r}]$ which lies in $\NN_1$ actually lies in $\NN_1^+$. The claim follows.

\end{proof}

\p This is an analogue of the previous proposition for $\HH$. In particular it allows us to describe $\HH\cap [\h_e \h_e]$ entirely.

\begin{prop}\label{Hprop} The following are true:
\begin{enumerate}
\item{$\HH_1 \subset [\h_e \h_e]$;}
\item{$\HH_0 \cap [\h_e \h_e] = \HH_0^+$.}
\end{enumerate}
\end{prop}
\begin{proof}
$\HH_1$ has a basis consisting of vectors $\zeta_i^{i,s}$ with $i < i'$ and $\lambda_i - s$ odd. Fix such a choice of $i$ and $s$, and choose $r$ such that $\lambda_i - r$ is odd. Following lemma \ref{product}, we have $$[\zeta_{i'}^{i, s} \zeta_{i}^{i',r}] = (1 + \varepsilon_{i,i',r}) (\zeta_i^{i,s+r-(\lambda_i-1)} - \zeta_{i'}^{i',s+r-(\lambda_i - 1)}).$$ Since $\varepsilon_{i,i,r+s-(\lambda_i-1)} = (-1)^{\lambda_i-(s+r-(\lambda_i-1))} = (-1)^{(\lambda_i - s) + (\lambda_i-r) + 1} = -1$ it follows that $\zeta_{i'}^{i',s+r-(\lambda_i-1)} = - \zeta_i^{i,s+r-(\lambda_i-1)}$ by part 3 of lemma \ref{spanningdetails}. Furthermore $\varepsilon_{i, i', r} = (-1)^{\lambda_i - r + 1} = 1$. Therefore 
\begin{eqnarray*}
[\zeta_{i'}^{i, s} \zeta_{i}^{i',r}] = 4 \zeta_i^{i,s+r-(\lambda_i-1)}
\end{eqnarray*}
which is nonzero since char$(\K)\neq 2$. We make the observation that the above expression lies in $\HH_1$ for any choice of $r$ and $s$ with $\lambda_i -r$ and $\lambda_i - s$ both odd ($\ast$). Taking $r = \lambda_i -1$ we obtain $\zeta_i^{i,s} \in \HH \cap [\h_e \h_e]$. Since $\HH_1$ is spanned by those $\zeta_i^{i,s}$ such that $i < i'$ and $\lambda_i - s$ is odd we have $\HH_1 \subseteq [\h_e \h_e]$. This completes (1). For the sake of clarity we shall divide the proof of part 2 of the current proposition into subsection \emph{(i), (ii), ..., (vii)}. The approach is much the same as part 2 of proposition \ref{Nprop}. In parts \emph{(i),...,(iv)} we show that a spanning set for $\HH_0^+$ may be found in $[\h_e \h_e]$ and in the subsequent parts \emph{(v), (vi), (vii)} we demonstrate that any product $[\zeta_i^{j,s} \zeta_k^{l,r}]$ which lies in $\HH_0$ actually lies in $\HH_0^+$.

\emph{(i) $\HH_0 \cap [\h_e \h_e] = \HH_0 \cap [\NN_1 \NN_1]$:} By proposition \ref{subs} we see that $\HH \cap [\h_e \h_e] = [\NN_0 \NN_0] + (\HH \cap [\NN_1 \NN_1])$. The argument for part 1 of the current proposition shows that $[\NN_0 \NN_0] \subseteq \HH_1$. Since $\HH = \HH_0 \oplus \HH_1$ we get $\HH_0 \cap [\h_e \h_e] = \HH_0 \cap [\NN_1 \NN_1]$.

\emph{(ii) $\HH_0 \cap [\NN_1\NN_1] = $ \emph{span}$\{ \zeta_{[j]}^{[j], \lambda_{[j]} - 2m} - \zeta_{[i]}^{[i], \lambda_{[i]} - 2m} : [1] \leq [i] < [j] \leq [n], \lambda_i - \lambda_j < 2m < \lambda_j + \lambda_i\}$:} Again by proposition \ref{subs} see that $\HH \cap [\NN_1 \NN_1]$ is spanned by products $[\zeta_i^{j,s} \zeta_{i'}^{j',r}]$ with $[j] > [i]$. In turn $[\zeta_i^{j,s} \zeta_{i'}^{j',r}] = \varepsilon_{i', j', r} [\zeta_i^{j,s} \zeta_{j}^{i,r}] = \varepsilon_{i', j', r}(\zeta_j^{j,r+s-(\lambda_i-1)} - \zeta_i^{i,r+s-(\lambda_j-1)})$. The reader will notice that $$[\zeta_i^{j,s} \zeta_{j}^{i,r}] \in \left\{ \begin{array}{ll}
        \HH_1 & \mbox{$\lambda_i + \lambda_j - (r+s) -1$ odd};\\
        \HH_0 & \mbox{$\lambda_i + \lambda_j - (r+s)-1$ even}.\end{array} \right.$$ so that $\HH_0 \cap [\NN_1 \NN_1] = \spn \{ [\zeta_i^{j,s} \zeta_{j}^{i,r}] :[1] \leq [i] <  [j] \leq [n], 0 \leq s,r < \lambda_i, \lambda_i + \lambda_j - (r+s) -1\text{ even}\}.$ If we pick $[1] \leq [i] <  [j] \leq [n]$ and $0 \leq s,r < \lambda_i$ such that $\lambda_i + \lambda_j - (r+s) -1 = 2m$ then we have $[\zeta_i^{j,s} \zeta_{j}^{i,r}] = \varepsilon_{i',j',r} (\zeta_{[j]}^{[j], \lambda_{[j]} - 2m} - \zeta_{[i]}^{[i], \lambda_{[i]} - 2m})$. The constraints placed on $s$ and $r$ are equivalent to $\lambda_i - \lambda_j < 2m < \lambda_i + \lambda_j$, and \emph{(ii)} follows.

\emph{(iii) Each spanning vector from (ii) lies in a unique $\HH_{0,l}^m$, in particular $\HH_0 \cap [\NN_1 \NN_1] = \oplus_{l,m} (\HH_{0,l}^m \cap [\NN_1 \NN_1])$:} Fix $m$ in the appropriate range and suppose $1\leq i < a_{t(m),m}$ We claim that if $[j] > [i]$ then each $\zeta_{[j]}^{[j], \lambda_{[j]} - 2m} - \zeta_{[i]}^{[i], \lambda_{[i]} - 2m} \in \HH_0 \cap [\NN_1 \NN_1]$ lies in $\HH_{0,l}^m$ where $l$ is the unique integer fulfilling $[a_{l, m}] \leq [i] < [a_{l+1,m}]$. It will suffice to show that given $i,j,l$ and $m$ as above we have $[j] < [a_{l+1,m}]$. To see this, suppose that $[j] \geq [a_{l+1,m}]$. Then by our choice of $a_{l+1,m}$ we have $\lambda_{a_{l+1,m}-1} - \lambda_{a_{l+1,m}} \geq 2m$ which implies $\lambda_i - \lambda_j \geq 2m$ contrary to the restriction $\lambda_i - \lambda_j < 2m$ noted in the statement of \emph{(ii)}. We conculde that $[a_{l,m}] \leq [i] < [j] < [a_{l+1,m}]$ and that the corresponding spanning vector lies in $\HH_{0,l}^m$. In case $a_{t(m),m} \leq i$
we have $\zeta_{[j]}^{[j], \lambda_{[j]} - 2m} - \zeta_{[i]}^{[i], \lambda_{[i]} - 2m} \in \HH_{0,t(m)}^m$ by definition. Thus we have shown that the spanning vectors  of $\HH_0 \cap  [\NN_1 \NN_1]$ each lie in some $\HH_{0,l}^m$, as claimed.

\emph{(iv) $\HH_{0,l}^{m,+} \subseteq \HH_{0,l}^m \cap [\NN_1 \NN_1] $ for all $l,m$:} Suppose $1 \leq i < a_{t(m),m}$. Since $\lambda_{a_{t(m),m} -1} - \lambda_{a_{t(m),m}} \geq 2m$ we know that $\lambda_{a_{t(m),m}-1} \geq 2m$ and so $\lambda_i \geq 2m$. It follows that $\zeta_{[i]}^{[i], \lambda_{[i]} - 2m} \neq 0$ for all such $i$. Fix $[a_{l,m}] < [i] < [a_{l+1,m}]$. By our choice of integers $\{a_{1,m},...,a_{t(m),m}\}$ we know that $\lambda_{[i]-1} - \lambda_{[i]} < 2m$ and since $\lambda_{[i]-1}, \lambda_{[i]} \geq \lambda_{a_{t(m),m}} \geq 2m$ we have $\lambda_{[i]-1} + \lambda_{[i]} > 2m$. By these remarks, using \emph{(ii)}, it follows that $\zeta_{[i]-1}^{[i]-1, \lambda_{[i]-1} - 2m} - \zeta_{[i]}^{[i], \lambda_{[i]} - 2m}$ is a nonzero element of $\HH_{0,l}^m \cap [\NN_1 \NN_1]$. These vectors span all of $\HH_{0,l}^{m,+}$ so \emph{(iv)} follows for $l < t(m)$. The argument for $l = t(m)$ is similar. Let $k = \max\{ i: \lambda_i \geq 2m\}$. Then $\zeta_i^{i,\lambda_i - 2m} \neq 0 $ if and only 
if $i \leq k$ so $\HH_{0,t(m)}^m = \spn\{\zeta_{[i]}^{[i],\lambda_{[i]}-2m} : [a_{t(m),m}] \leq [i] \leq [k]\}$. Fix $[a_{t(m),m}] < [i] \leq [k]$. By our choice of integers $\{a_{1,m},...,a_{t(m),m}\}$ we know that $\lambda_{[i]-1} - \lambda_{[i]} < 2m$ and by our choice of $k$ we have $\lambda_{[i]-1} + \lambda_{[i]} > 2m$. The argument now concludes exactly as above.

\emph{(v) $\HH_{0,l}^m \cap [\NN_1 \NN_1] = \HH_{0,l}^{m,+}$ for all $1\leq l < t(m)$:} The discussion in \emph{(iii)} confirms that $\HH_{0,l}^m \cap [\NN_1 \NN_1] = \spn\{ \zeta_{[j]}^{[j], \lambda_{[j]} - 2m} - \zeta_{[i]}^{[i], \lambda_{[i]} - 2m} : [a_{l,m}] \leq [i] < [j] < [a_{l+1,m}], \lambda_i - \lambda_j < 2m < \lambda_j + \lambda_i\}$. This space is clearly contained in $\HH_{0,l}^{m,+}$. By \emph{(iv)} we have equality.

\emph{(vi) $\HH_{0,t(m)}^m \cap [\NN_1 \NN_1] = \HH_{0,t(m)}^m$:}  Again by \emph{(iv)} we have $\HH_{0,t(m)}^{m,+} \subseteq \HH_{0,t(m)}^m \cap [\NN_1 \NN_1]$. If $\lambda_{a_{t(m),m}} < 2m$ then $\HH_{0,t(m)}^m = 0$ by part 1 of lemma \ref{bits} and the statement holds trivially. Assume $\lambda_{a_{t(m),m}} \geq 2m$. Let $k = \max\{ i : \lambda_i \geq 2m\}$ so that $\HH_{0,t(m)}^m = \spn\{\zeta_{[i]}^{[i],\lambda_{[i]}-2m} : [a_{t(m),m}] \leq [i] \leq [k]\}$. We claim that $[k]+1\leq [n]$. If not then $[k] = [n]$ which implies that $ \lambda_{k} - \lambda_{k+1} = \lambda_k \geq 2m$ implying that $k+1 \in \{a_{1,m},...,a_{t(m),m}\}$, however $k+1 > a_{t(m),m}$ and $a_{1,m} \leq \cdots \leq a_{t(m),m}$, and this contradiction confirms the claim. By the very same reasoning we know that $\lambda_{[k]} - \lambda_{[k]+1} = \lambda_k - \lambda_{k+1} < 2m$ and the inequality $[k]+1\leq n$ gives us $\lambda_{[k]+1} > 0$ which in turn implies $\lambda_{[k]} + \lambda_{[k]+1} > 2m$. By \emph{(ii)} and \emph{(iii)} 
we have $\zeta_{[k]+1}^{[k]+1, \lambda_{[k]+1}-2m} - \zeta_{[k]}^{[k],\lambda_{[k]} - 2m} \in \HH_{0,t(m)}^m$. Since $\lambda_{k+1} < 2m$ we know that $\zeta_{[k]+1}^{[k]+1, \lambda_{[k]+1}-2m} = 0$. Since $\zeta_{[k]}^{[k],\lambda_{[k]} - 2m} \notin \HH_{0,t(m)}^{m,+}$ and $\HH_{0,t(m)}^{m,+}$ has codimension 1 in $\HH_{0,t(m)}^m$ part \emph{(vi)} follows.

\emph{(vii) $\HH_0 \cap [\h_e \h_e] = \HH_0^+$:} By \emph{(i)} and \emph{(iii)} we have $$\HH_0 \cap [\h_e \h_e] = \oplus_{l,m} (\HH_{0,l}^m \cap [\NN_1 \NN_1]).$$ The proposition now follows from \emph{(v)} and \emph{(vi)}.
\end{proof}

\p We may finally state and prove the decomposition theorem for $[\h_e \h_e]$. 
\begin{thm}\label{derived}
The derived subalgebra $[\h_e \h_e]$ is equal to the direct sum $$\NN_0 \oplus \NN_1^+ \oplus \HH_0^+\oplus \HH_1.$$
\end{thm}
\begin{proof}
The sum of the above spaces is direct by construction. By Proposition~\ref{subs} we know that $[\h_e \h_e]$ is the sum of the three spaces $$[\h_e \h_e] = \NN_0 \cap [\h_e \h_e] + \NN_1 \cap [\h_e \h_e] + \HH \cap [\h_e \h_e].$$ By proposition \ref{Nprop} we have $\NN_0 \cap [\h_e \h_e] + \NN_1 \cap [\h_e \h_e] = \NN_0 + \NN_1^+$. By proposition \ref{Hprop} using the fact that $\HH = \HH_0 \oplus \HH_1$ we have $\HH \cap [\h_e \h_e] = \HH_1 + \HH_0^+$. The theorem follows.
\end{proof}

\section{An expression for $\dim\,  \h_e^{\text{ab}}$}
\setcounter{parno}{0}

\p As a corollary to the previous theorem we obtain an expression for the dimension of the maximal abelian quotient $\h_e^\text{ab} = \h_e/[\h_e \h_e]$. Suppose $\lambda = (\lambda_1,...,\lambda_n)$ and recall from the introduction that $\Delta(\lambda)$ is the set of indexes between $1$ and $n-1$ such that $i = i', i+1 = (i+1)'$ and $\lambda_{i-1} \neq \lambda_i \geq \lambda_{i+1} \neq \lambda_{i+2}.$ We then set $$s(\lambda) := \sum_{i=1}^{n} \lfloor \frac{\lambda_{i} - \lambda_{i+1}}{2} \rfloor.$$ Here we have adopted the convention $\lambda_0 = \lambda_{n+1} = 0$. The pairs $(i, i+1)$ with $i \in\Delta(\lambda)$ are the 2-steps for $\lambda$. These notations shall be used repeatedly in all that follows.
 \begin{cor}\label{expression}
 If $\Oo_e$ has partition $\lambda$ then $c(\lambda) := \dim\,  \h_e^{\emph{\text{ab}}} = s(\lambda) + |\Delta(\lambda)|$.
 \end{cor}
 \begin{proof}
 Recall that $\h_e = \HH \oplus \NN_0 \oplus \NN_1$, that $\NN_1 = \NN_1^- \oplus \NN_1^+$, and that $\HH = \HH_0 \oplus \HH_1$ with $\HH_0^+ \subseteq \HH_0$. By theorem \ref{derived} we have $\h_e^\text{ab} \cong \NN_1/\NN_1^+ \oplus \HH_0/\HH_0^+$ as vector spaces. We claim that $\dim\, \NN_1/\NN_1^+ = |\Delta(\lambda)|$ and that $\dim\, \HH_0/\HH_0^+ = s(\lambda)$, from whence the theorem shall follow. First of all observe that $\dim\, \NN_1/\NN_1^+ = \dim\, \NN_1^-$. By part 3 of lemma \ref{spanningdetails} the maps $\zeta_{i-1}^{i,\lambda_i-1}$ spanning $\NN_1^-$ are all linearly independent. Define an injection from the set of $\NN_1^-$ spanning vectors $\{\zeta_{i}^{i+1,\lambda_{i+1} - 1} : i = i', i+1 = (i+1)', L(i) = i, U(i+1) = i+1\}$ to $\{1,...,n-1\}$ via $$\zeta_{i}^{i+1,\lambda_{i+1}-1} \mapsto i.$$ We claim that this map is actually a bijection onto $\Delta(\lambda)$. The condition $i = i'$ and $i+1 = (i+1)'$ hold by definition. The condition $L(i) = i$ is equivalent to
 $\lambda_{i-1} \neq \lambda_i$ and $U(i+1) = i+1$ is equivalent to $\lambda_{i+1} \neq \lambda_{i+2}$. The inequality $\lambda_i \geq \lambda_{i+1}$ holds by definition, so the map defined above is indeed a bijection. We deduce that $\dim\, \NN_1/\NN_1^+ = \dim\, \NN_1^- = |\Delta(\lambda)|$.

We must now show that $\dim\, \HH_0/ \HH_0^+ = s(\lambda)$. Observe that $\HH_0 = \oplus_{l,m} \HH_{0,l}^m$ (part 2 of lemma \ref{bits}) and that each $\HH_{0,l}^{m,+}$ has codimension 1 in $\HH_{0,l}^m$. Furthermore if $l < t(m)$ then $\HH_{0,l}^m \neq 0$. We conclude that $\dim\, \HH_0/ \HH_0^+ = |\mathcal{H}|$ where $$\mathcal{H} := \{(l,m) : 1\leq l \leq t(m)-1, 1\leq m\leq \lfloor\frac{\lambda_1}{2} \rfloor\}\} .$$ Now rewrite $s(\lambda) = |\Sigma|$ where $$\Sigma := \{(i,m) \in \{2,...,n+1\}\times \{1,...,\lfloor \frac{\lambda_1}{2} \rfloor\}: \lambda_{i-1} - \lambda_{i} \geq 2m \}.$$ If we define a bijection from $\mathcal{H}$ to $\Sigma$, then the result follows. Define a map from $\mathcal{H}$ to $\{2,...,n+1\}\times \{1,...,\lfloor \frac{\lambda_1}{2} \rfloor\}$ by the rule $$(l,m) \mapsto (a_{l+1,m}, m).$$ By the definition of the integers $\{a_{1,m}, a_{2,m}, ..., a_{t(m),m}\}$ it is a well defined injection into $\Sigma$. Fix $1\leq m \leq \lfloor \frac{\lambda_1}{2} \rfloor$. Since
$\lambda_0 = 0$ and $\lambda_1 \geq \cdots \geq \lambda_n$, we have $a_{1,m} = 1$ and $\{a_{2,m},...,a_{t(m),m}\}$ is the set of all integers $i$ such that $\lambda_{i-1} - \lambda_{i} \geq 2m$. Thus the map is surjective and $\dim\, \HH_0/ \HH_0^+ = s(\lambda)$.
\end{proof}

\section{The Kempken-Spaltenstein algorithm} \label{KSalg}
\setcounter{parno}{0}

\p As we explained in the introduction, the Kempken-Spaltenstein algorithm is a tool for combinatorially enumerating the sheets of $\h$ containing a given nilpotent orbit. By studying the classification of sheets and the partitions associated to induced orbits we are able to construct a bijection between the sheets containing $\Oo_e$ and the so called admissible sequences upto reordering (Corollary~\ref{bijection}). The latter data is sufficiently explicit to classify the nilpotent orbits lying in a unique sheet (Corollary~\ref{izosim}). This section shall introduce the algorithm in its abstract form.

\p First we shall need a notion of rigid partition. The rigid nilpotent orbits in $\h$ are classified by Kempken and Spaltenstein in \cite{K}, \cite{Sp} as follows:
\begin{thm}\label{rigids}
Let $e \in \Ni(\h)$ have partition $\lambda = (\lambda_1, ..., \lambda_n)\in \mathcal{P}_\epsilon(N)$. Then $e$ is rigid if and only if
\begin{itemize}
\item{$\lambda_i - \lambda_{i+1} \in \{0,1\}$ for all $1\leq i\leq n$;}
\item{the set $\{ i \in \Delta(\lambda) : \lambda_i = \lambda_{i+1}\}$ is empty.}
\end{itemize}
\end{thm}

\p We recover the following, which was first proven in \cite[Theorem~12]{Ya10}
\begin{cor}\label{rigidcod}
$[\h_{e} \h_{e}] = \h_{e}$ if and only if $e$ is rigid.
\end{cor}
\begin{proof}
Note that $[\h_{e} \h_{e}] = \h_{e}$ if and only if $\dim\, \h_{e}^\text{ab} = 0$. Apply Corollary~\ref{expression} and Theorem~\ref{rigids}.
\end{proof}

\p Given the above classification of rigid orbits we have a well defined notion of a rigid partition. Following \cite{CM} we denote by $\mathcal{P}_\epsilon(N)^\ast$ the subset of rigid partitions. Throughout this section $\lambda = (\lambda_1, \lambda_2, ..., \lambda_n)$ shall remain to denote an element of $\mathcal{P}_\epsilon(N)$ ordered in the usual manner $\lambda_1 \geq \cdots \geq \lambda_n$. Moreau describes an algorithm \cite{Mor} used in the computation of dimensions of sheets for classical Lie algebras. Her algorithm takes $\lambda \in \mathcal{P}_\epsilon(N)$ and returns an element of $\mathcal{P}_\epsilon(M)$ for some $M \leq N$, then iterates. It finally terminates when the output is a rigid partition. Our algorithm extends her method.

\p We say that Case 1 or 2 occur for $\lambda$ at index $i$ as follows:\\
\textbf{Case 1} $$\lambda^{\ii}_{i} \geq \lambda^{\ii}_{i}+2$$
\textbf{Case 2} $$i \in \Delta(\lambda^{\ii}) \text{ and } \lambda_{i}^\ii = \lambda^\ii_{i+1}.$$
In the hope of avoiding any confusion we shall use `Case' when referring to Case 1 or Case 2, and we shall use `case' to refer to a particular situation. We say that an index $i \in \{1,...,n\}$ is \emph{admissible for $\lambda$} if Case 1 or 2 occurs for $\lambda$ at $i$. If Case 1 occurs then we define  
$$\lambda^{(i)} := (\lambda_1 - 2, \lambda_2 - 2, ..., \lambda_{i}-2, \lambda_{i+1},..., \lambda_n)$$
whilst if Case 2 occurs at $i$ then define
$$\lambda^{(i)} := (\lambda_1-2, \lambda_2-2,...,\lambda_{i-1} -2, \lambda_{i} - 1, \lambda_{i+1} - 1, \lambda_{i+2},...,\lambda_n).$$
In each case the reader may check that we have $\lambda^i \in \mathcal{P}_\epsilon(N - 2i)$.

\p Let $\ii=(i_1,...,i_l)$ be a sequence of integers and $\ii'=(i_1,...,i_{l-1})$. We shall now explain what it means for $\ii$ to be an \emph{admissible sequence for $\lambda$}, and construct a partition $\lambda^\ii\in \mathcal{P}_\epsilon(N - 2\sum_{j=1}^l i_j)$ for all admissible sequences $\ii$. We may state both definitions using an inductive process. By convention the empty sequence $\emptyset$ is admissible for $\lambda$ and $\lambda^\emptyset = \lambda$. If $l = 1$ the we say that $\ii$ is an admissible sequence if $i_1$ is an admissible index. In this situation $\lambda^\ii$ is defined in the previous paragraph. For $l>1$ we say that $\ii$ is as admissible sequence for $\lambda$ provided $\ii'$ is admissible and $i_l$ is an admissible index for $\lambda^{\ii'}$. If $\ii$ is admissible then we define $$\lambda^\ii := (\lambda^{\ii'})^{(i)}.$$ 

\p A \emph{maximal admissible sequence for $\lambda$} is one which is not a proper subsequence of an admissible sequence for $\lambda$. The following shows that the KS algorithm always terminates with a rigid partition as the output.
\begin{lem}\label{maxadmiss}
Let \emph{$\ii$} be an admissible sequence for $\lambda$. Then $\bf i$ is maximal admissible if, and only if, $\lambda^\ii$ is a rigid partition.
\end{lem}
\begin{proof}
This follows from the definition of maximal admissible sequences, given Kempken's classification of rigid nilpotent orbits in terms of partitions.
\end{proof}
\begin{rem}\label{remark1}
The algorithm is transitive in the following sense: if $\emph{\ii}$ is an admissible sequence for $\lambda$ and $\emph{\jj}$ is an admissible sequence for $\lambda^{\emph{\ii}}$ then $(\emph{\ii, \jj})$ is an admissible sequence for $\lambda$, where $(\emph{\ii, \jj})$ denotes the concatenation of the two sequences $\emph{\ii}$ and $\emph{\jj}$. Furthermore $\lambda^{(\emph{\ii, \jj})} = (\lambda^\emph{\ii})^\emph{\jj}$.
\end{rem}

\section{Non-singular partitions and preliminaries of the algorithm}\label{KSdetails}
\setcounter{parno}{0}

\p Before placing the algorithm into the geometric context for which it was intended we shall discuss it purely combinatorially. This section will contain  one important definition and several useful lemmas.

\p Let $\lambda = (\lambda_1,...,\lambda_n) \in \PP(N)$. In Section~\ref{definesingular} of the introduction we introduced the 2-steps for $\lambda$, the bad 2-steps and explained what it means for $\lambda$ to be singular. Let us recap these definitions. The 2-steps for $\lambda$ are simply the pairs $(i, i+1)$ with $i \in \Delta(\lambda)$. A 2-step $(i, i+1)$ is said to have \emph{a bad boundary} if either of the following hold:
\begin{itemize}
\item{$\lambda_{i-1} - \lambda_i \in 2\N$;}
\item{$\lambda_{i+1} - \lambda_{i+2} \in 2\N$.}
\end{itemize}
A 2-step is then called \emph{bad} if it has a bad boundary. If $\lambda$ has a bad 2-step it is called \emph{singular} and otherwise it is called \emph{non-singular}. In the next section we shall interpret these singular and non-singular partitions in geometric terms. In particular we shall show that singular partitions correspond precisely to the nilpotent singular points on the varieties $\h^{(m)}$, hence their name.
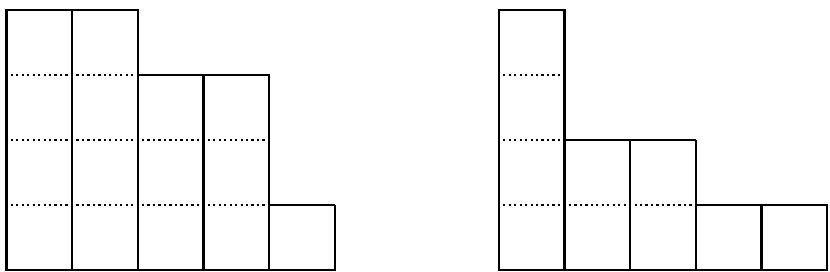
\begin{figure}[htb]
\setlength{\unitlength}{0.017in}
\begin{center}
\begin{picture}(250,80)(0,0)

\put(0,0){\line(0,1){80}}
\put(0,80){\line(1,0){40}}
\put(40,80){\line(0,-1){20}}
\put(40,60){\line(1,0){40}}
\put(80,60){\line(0,-1){40}}
\put(80,20){\line(1,0){20}}
\put(100,20){\line(0,-1){20}}
\put(100,0){\line(-1,0){100}}

\qbezier[24](0,60),(20,60),(40,60)
\qbezier[48](0,40),(40,40),(80,40)
\qbezier[48](0,20),(40,20),(80,20)

\put(20,0){\line(0,1){80}}
\put(40,0){\line(0,1){60}}
\put(60,0){\line(0,1){60}}
\put(80,0){\line(0,1){20}}

\put(150,0){\line(0,1){80}}
\put(150,80){\line(1,0){20}}
\put(170,80){\line(0,-1){40}}
\put(170,40){\line(1,0){40}}
\put(210,40){\line(0,-1){20}}
\put(210,20){\line(1,0){40}}
\put(250,20){\line(0,-1){20}}
\put(250,0){\line(-1,0){100}}

\qbezier[12](150,60),(160,60),(170,60)
\qbezier[12](150,40),(160,40),(170,40)
\qbezier[36](150,20),(180,20),(210,20)

\put(170,0){\line(0,1){40}}
\put(190,0){\line(0,1){40}}
\put(210,0){\line(0,1){20}}
\put(230,0){\line(0,1){20}}
\end{picture}
\end{center}
\caption{The Young diagrams of two singular partitions in $\mathcal{P}_1(15)$ and $\mathcal{P}_{-1}(10)$. The bad 2-steps are  $(3,4)$ and $(2,3)$, respectively.}\label{pikcha_A}
\end{figure}

\p We now collect some elementary lemmas about the behaviour of the algorithm. For the remnant of the subsection we assume that $\lambda \in \mathcal{P}_\epsilon(N)$ has the standard ordering $\lambda_1\geq \cdots \geq \lambda_n$. These first two lemmas tell us how $\Delta(\lambda)$ changes as we iterate the algorithm.
\begin{lemA}\label{deltain2}
Suppose $\emph{\ii} = (i)$ is a sequence of length $1$. If Case 2 occurs for $\lambda$ at index $i$ then $\Delta(\lambda^{\emph{\ii}}) = \Delta(\lambda) \setminus \{i\}$. Furthermore, if $(i,i+1)$ is a good $2$-step of $\lambda$ then $s(\lambda^\emph{\ii})=s(\lambda).$
\end{lemA}
\begin{proof}
We shall suppose that there is an index $$j \in \Delta(\lambda)\setminus \big( \Delta(\lambda^{\ii})\cup\{i\}\big)$$ and derive a contradiction. Observe that if $j < i-2$ (resp. $j> i+2$) then for $k \in\{ j-1, j, j+1, j+2\}$ we have that $\lambda_{k}^{\ii}=\lambda_{k}-2$ (resp.
$\lambda_{k}^{\ii}=\lambda_{k}$). So $j \in \Delta(\lambda)$ if and only if $j \in \Delta(\lambda^{\ii})$. It remains to show that if $j = i\pm1$ or $j = i \pm 2$ and $j \in \Delta(\lambda)$ then $j \in \Delta(\lambda^{\ii})$. If $j = i\pm1$ and $j \in \Delta(\lambda)$ then $\lambda_i \neq \lambda_{i+1}$ contradicting the fact that Case 2 occurs for $\lambda$ at index $i$.

Suppose $j = i-2$. Then $j, j+2 \in \Delta(\lambda)$ and hence $\lambda_{j+1} \neq \lambda_{j+2}$ and $(j+1)'=j+1$, $(j+2)'=j+2$.  As a consequence $\lambda_{j+1} - \lambda_{j+2}$ is even implying that $\lambda_{j+1} - \lambda_{j+2}\geq 2$ and  $\lambda_{j+1}^{\ii} \neq \lambda_{j+2}^{\ii}$. Since for $k\in\{j-1, j , j+1\}$ the equality $\lambda_{k}^{\ii}=\lambda_k-2$ holds,
we conclude that $j \in \Delta(\lambda^{\ii})$. A similar argument shows that if $j = i+2$ then $j \in \Delta(\lambda)$ implies $j \in \Delta(\lambda^{\ii})$. We conclude that $\Delta(\lambda^{\rm{\ii}}) = \Delta(\lambda) \setminus \{i\}$.

Now suppose $(i,i+1)$ is a good 2-step of $\lambda$.  Since $\lambda_{i+1}-\lambda_{i+2}$ and $\lambda_{i-1}-\lambda_i$ if $i>1$ are odd we have that
$$\big\lfloor(\lambda_{i+1}^\ii-\lambda_{i+2}^\ii)/2\big\rfloor
=\lfloor((\lambda_{i+1}-1)-\lambda_{i+2})/2\big\rfloor=
\lfloor(\lambda_{i+1}-\lambda_{i+2})/2\big\rfloor$$
and
$$\big\lfloor(\lambda_{i-1}^\ii-
\lambda_{i}^\ii)/2\big\rfloor=
\big\lfloor((\lambda_{i-1}-2)-(\lambda_i-1))/2\big
\rfloor=
\big\lfloor(\lambda_{i-1}-\lambda_{i})/2\big\rfloor$$
if $i>1$.
As $\lambda_j^\ii=\lambda_j$ for $j\not\in\{i,i+1\}$ it
follows that  $s(\lambda^\ii)=s(\lambda)$ as claimed.
\end{proof}

\begin{lemB}\label{deltain}
If $\emph{\ii}$ is an admissible sequence for $\lambda$ then $\Delta(\lambda^{\emph{\ii}}) \subseteq \Delta(\lambda)$.
\end{lemB}
\begin{proof}
In view of Lemma \ref{deltain2}(A) and Remark \ref{remark1} it will suffice to prove the current lemma when $\ii= (i)$ and $i$ is an index at which Case 1 occurs for $\lambda$. Suppose $j \in \Delta(\lambda^\ii)$. Then since Case 1 preserves the parity of the entries of $\lambda$ (that is to say $\lambda_k^\ii \equiv \lambda_k \mod 2$ for $1\leq k \leq n$), we deduce that $j'=j$ and $(j+1)'=j+1$.
If $j<i$ or $j>i+1$ then $\lambda_{j-1} - \lambda_j = \lambda_{j-1}^\ii - \lambda_j^\ii$ and  $\lambda_{j+1} - \lambda_{j+2}=\lambda_{j+1}^\ii - \lambda_{j+2}^\ii$ showing that $j \in \Delta(\lambda)$ in these cases. If $j=i+1$ then  $\lambda_{j-1} - \lambda_j = \lambda_{j-1}^\ii - \lambda_j^\ii+2$ and $\lambda_{j+1} - \lambda_{j+2} = \lambda_{j+1}^\ii - \lambda_{j+2}^\ii$. Hence $j\in\Delta(\lambda)$.
Finally, if $j=i$ then  $\lambda_{j-1} - \lambda_j = \lambda_{j-1}^\ii - \lambda_j^\ii+2$ and $\lambda_{j+1} - \lambda_{j+2} = \lambda_{j+1}^\ii - \lambda_{j+2}^\ii$. Thus $j\in\Delta(\lambda)$ in all cases and our proof is complete.
\end{proof}

\p Now we show that good 2-steps are preserved by the algorithm.
\begin{lem}\label{goodinherit}
If $(i, i+1)$ is a good 2-step for $\lambda$, $\emph{\ii}$ is an admissible sequence and $i \in \Delta(\lambda^\emph{\ii})$ then $(i,i+1)$ is a good 2-step for $\lambda^\emph{\ii}$.
\end{lem}
\begin{proof}
It suffices to prove the lemma when $\ii = (i_1)$ is an admissible of length 1. Suppose first that Case 1 occurs at index $i_1$. Then $\lambda_j^\ii - \lambda_{j+1}^\ii \equiv \lambda_j - \lambda_{j+1} \mod 2$ for all $j$. Since $(i,i+1)$ is good for $\lambda$ it follows that $\lambda_{i-1}^\ii - \lambda_{i}^\ii$ is odd (or $i=1$) and $\lambda_{i+1}^\ii - \lambda_{i+2}^\ii$ is odd, so that $(i,i+1)$ is a good 2-step for $\lambda^\ii$. Now suppose Case 2 occurs for $\lambda$ at index $i_1$. We may assume that $i_1 \neq i$. If $i_1 = i-1$ or $i_1 = i-2$ then $i_1 \in \Delta(\lambda)$ implies $\epsilon(-1)^{\lambda_{i-1}} = -1$ and $\lambda_{i-1} - \lambda_i$ is even, contrary to the assumption that the 2-step $(i,i+1)$ is good for $\lambda$. Similarly, if $i_1 = i+1$ or $i_1 = i+2$ then $\lambda_{i+1} - \lambda_{i+2}$ is even, contradicting the assumption that $(i,i+1)$ is good. It follows that $i_1<i-2$ or $i_1 > i+2$, from whence it immediately follows that $(i,i+1)$ is a good 2-step for $\lambda^\ii$.
\end{proof}

\begin{cor}\label{singinherit}
If $\lambda$ is non-singular then $\lambda^{\emph{\ii}}$ is non-singular for any admissible sequence $\emph{\ii}$.
\end{cor}
\begin{proof}
If $i \in \Delta(\lambda^\ii)$ then $i \in \Delta(\lambda)$ by Lemma~\ref{deltain}(B). Since $\lambda$ is non-singular, $(i,i+1)$ is a good 2-step for $\lambda$. By Lemma~\ref{goodinherit}, $(i,i+1)$ is good for $\lambda^\ii$.
\end{proof}

\section{The maximal length of admissible sequences} \label{KSmaxlength}
\setcounter{parno}{0}

\p In this section we shall give a combinatorial formula for the maximal length of admissible sequences for $\lambda$. The formula shall be of central importance to our results on sheets. First we shall need  some further terminology related to partitions $\lambda=(\lambda_1,\ldots,\lambda_n)\in \mathcal{P}_\epsilon(N)$.

\p \label{2clusters}A sequence $1\leq i_1 < i_2 < \cdots < i_k < n$ with $k\ge 2$ is called a \emph{2-cluster} of $\lambda$ whenever $i_j \in \Delta(\lambda)$ and $i_{j+1} = i_j + 2$ for all $j$.
Analogous to the terminology for 2-steps we say that a 2-cluster $i_1,...,i_k$ \emph{has a bad boundary} if either of the following conditions holds:
\begin{itemize}
\item{$\lambda_{i_1-1} - \lambda_{i_1} \in 2\N$;}
\smallskip
\item{$\lambda_{i_k+1} - \lambda_{i_k+2} \in 2\N$}
\end{itemize}
If $i_1=1$ the the first condition may be omitted, since $\lambda_{i_1-1} - \lambda_{i_1}$ is negative in this case.
A \emph{bad 2-cluster} is one which has a bad boundary, whilst a \emph{good 2-cluster} is one without a bad boundary.
\begin{lem}
A good 2-cluster is maximal in the sense that it is not a proper subsequence of any 2-cluster.
\end{lem}
\begin{proof}
If $i_1,...,i_k$ is a good 2-cluster then $\lambda_{i_1-1} - \lambda_{i_1}, \lambda_{i_k+1} - \lambda_{i_k+2} \notin 2\N$. The fact that $i_1, i_k \in \Delta(\lambda)$ means that $\epsilon(-1)^{\lambda_{i_1}} = \epsilon(-1)^{\lambda_{i_k+1}} = -1$. Combining these few observations we get $\epsilon(-1)^{\lambda_{i_1-1}} = \epsilon(-1)^{\lambda_{i_k+2}} = 1$ and so $i_1-2 \notin \Delta(\lambda)$ and $i_k + 2 \notin \Delta(\lambda)$.
\end{proof}

\p We introduce the notations:
\begin{eqnarray*}
\Delta_{\rm bad}(\lambda) &:=& \{\text{the bad 2-steps of $\lambda$}\};\\
\Sigma(\lambda) &:=& \{\text{the good 2-clusters of $\lambda$}\};
\end{eqnarray*}
and write
$$z(\lambda) = s(\lambda) + |\Delta(\lambda)| - \big(|\Delta_{\rm bad}(\lambda)| - |\Sigma(\lambda)|\big).$$
It is immediate from the definitions that $|\Delta_{\rm bad}(\lambda)| \ge |\Sigma(\lambda)|$ and $|\Delta_{\rm bad}(\lambda)| = |\Sigma(\lambda)|$ if and only if $\Delta_{\rm bad}(\lambda)=\emptyset$.
\begin{lem}\label{stairslemma}
$|\Sigma(\lambda)| \geq |\Sigma(\lambda^{\emph\ii})|$ for length 1 admissible sequences $\emph\ii = (i)$, unless Case 2 occurs at $i$ and $$i-4, i-2, i, i+2, i+4$$ is a subsequence of a good 2-cluster, in which case $|\Sigma(\lambda)| = |\Sigma(\lambda^\ii)| - 1$.
\end{lem}
\begin{proof}
We make the notation $\ii =(i)$. In this first paragraph we deal with the possibility that Case 1 occurs for $\lambda$ at index $i$. Let us consider some necessary conditions for $\Sigma(\lambda) \neq \Sigma(\lambda^{\ii})$. We require that $i-1$ or $i+1$ lie in $\Delta(\lambda)$, that the 2-steps $(i-1,i)$ or $(i+1,i+2)$ (or both) constitute a 2-step in a good 2-cluster, and that $\lambda_i - \lambda_{i+1} = 2$. Let us assume these conditions. If precisely one of the two $i-1, i+1$ lies in $\Delta(\lambda)$ (we may assume $i-1 \in \Delta(\lambda)$) then it follows that the good 2-cluster in question has the form $i_1\leq \cdots \leq i_k = i-1$. But $\lambda_{i_k + 1} - \lambda_{i_k + 2} = 2$ then implies that the 2-cluster has a bad boundary; a contradiction. It follows that both $i-1$ and $i+1$ lie in $\Delta(\lambda)$. Then we have a good 2-cluster $i_1 \leq \cdots \leq i-1 = i_l \leq i_{l+1} = i+1 \leq \cdots \leq i_k$. However the sequences $i_1,i_2,...,i_{l-1}$ and $i_{l+2},...,i_{k-1},i_k$ are either 
of length $\leq 1$, or are bad 2-clusters for $\lambda^\ii$, so $|\Sigma(\lambda)| = |\Sigma(\lambda^\ii)| + 1$.

Now suppose Case 2 occurs at index $i$. Similar to the previous case $\Sigma(\lambda)$ is only affected if $(i,i+1)$ is a bad 2-step in a good 2-cluster. If precisely one of $i-2$ and $i+2$ lie in $\Delta(\lambda)$ (we may assume $i-2\in \Delta(\lambda))$ then such a 2-cluster will take the form $i_1,...,i_k = i$. If $k > 2$ then $i_1,...,i_{k-1}$ is a good 2-cluster for $\lambda^\ii$ so that $|\Sigma(\lambda^\ii)| = |\Sigma(\lambda)|$. If $k = 2$ (we know $k\geq 2$) then the 2-cluster is eradicated by the iteration of the algorithm and $|\Sigma(\lambda^\ii)| = |\Sigma(\lambda)| - 1$.

Suppose that both $i-2$ and $i+2$ lie in $\Delta(\lambda)$. Then $\Sigma(\lambda)$ is unaffected unless $i_1,...,i_j = i,...,i_k$ is a good 2-cluster, which we shall assume from henceforth. Note that $j \geq 2$ and $k - j \geq 1$ by assumption. If $j=2$ and $k-j=1$ then the good 2-cluster is no longer present for $\lambda^\ii$ and $|\Sigma(\lambda)| = |\Sigma(\lambda^\ii)| - 1$. If $j> 2$ and $k-j=1$ then $i_1,...,i_{j-1}$ is a good 2-cluster for $\lambda^\ii$ and $|\Sigma(\lambda)| = |\Sigma(\lambda^\ii)|$. The situation when $j=2$ and $k-j>1$ is very similar. In the final case $j > 2$, $k-j>1$ and $i-4, i-2, i, i+2, i+4$ is a subsequence of a good 2-cluster, as in the statement of the lemma. Here both $i-2j, ..., i-2$ and $i,i+2,..., i+2k$ are good 2-clusters for $\lambda^\ii$ so that $|\Sigma(\lambda)| = |\Sigma(\lambda^\ii)| - 1$ as required.
\end{proof}

\p Before continuing we shall need some notation. We define a construction which takes $\lambda\in \PP(N)$ to $\lambda^S\in \mathcal{P}_\epsilon(N-2k)$ for some $k\geq 0$. It is based entirely on application of the algorithm. The partition $\lambda^S$ is call \emph{the shell of $\lambda$} and is constructed as follows: for all $1\leq i \leq n$ we apply Case 1 repeatedly; if $\lambda_{i} - \lambda_{i+1} \in 2\N$ and if $i-1$ or $i+1$ lie in $\Delta(\lambda)$ then apply Case 1 until $\lambda^\ii_{i} - \lambda^\ii_{i+1} = 2$; if we are not in the previous situation then apply Case 1 until $\lambda^\ii_{i} - \lambda^\ii_{i+1} \in \{0,1\}$; finally apply Case 2 at every index $i$ such that $(i,i+1)$ is a good 2-step. In order to keep the notation consistent we may regard $S$ as the admissible sequence of indices (chosen in ascending order) used to construct $\lambda^S$.
\begin{figure}[htb]
\setlength{\unitlength}{0.025in}
\begin{center}
\begin{picture}(80,70)(0,0)

\qbezier[70](0,0),(0,35),(0,70)
\qbezier[20](0,70),(10,70),(20,70)
\qbezier[10](20,70),(20,65),(20,60)
\qbezier[20](60,10),(70,10),(80,10)
\qbezier[10](80,10),(80,5),(80,0)
\qbezier[80](80,0),(40,0),(0,0)

\put(20,10){\line(0,1){50}}
\put(20,60){\line(1,0){10}}
\put(30,60){\line(0,-1){20}}
\put(30,40){\line(1,0){20}}
\put(50,40){\line(0,-1){20}}
\put(50,20){\line(1,0){10}}
\put(60,20){\line(0,-1){10}}
\put(60,10){\line(-1,0){40}}

\qbezier[42](10,0),(10,35),(10,70)
\qbezier[6](20,0),(20,5),(20,10)
\qbezier[24](30,0),(30,20),(30,40)
\qbezier[24](40,0),(40,20),(40,40)
\qbezier[12](50,0),(50,10),(50,20)
\qbezier[6](60,0),(60,5),(60,10)
\qbezier[6](70,0),(70,5),(70,10)

\qbezier[12](0,10),(10,10),(20,10)
\qbezier[30](0,20),(25,20),(50,20)
\qbezier[30](0,30),(25,30),(50,30)
\qbezier[18](0,40),(15,40),(30,40)
\qbezier[18](0,50),(15,50),(30,50)
\qbezier[12](0,60),(10,60),(20,60)

\end{picture}
\end{center}
\caption{The dotted perimeter represents the Young diagram of the partition $\lambda=(7,7,6,4,4,2,1,1) \in \mathcal{P}_{-1}(32)$. The solid perimeter represents the profile of $\lambda$ of type $(3,7)$.}\label{pikcha_B}
\end{figure}
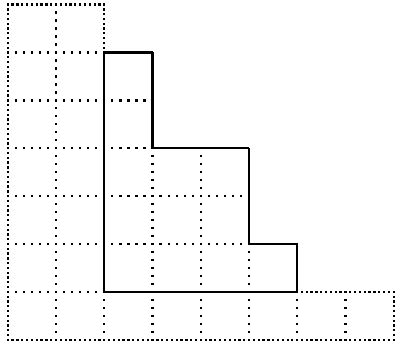

\p Retain the convention $\lambda = (\lambda_1,...,\lambda_n)$ with $\sum\lambda_i = N$. In order to make use of the shell $\lambda^S$ we shall interest ourselves firstly in the set of partitions which equal their own shell $\lambda =   \lambda^S$, and secondly in the relationship between a partition and its shell. It turns out that certain properties of a partition $\lambda = \lambda^S$ are controlled by the properties of certain special partitions constructed from $\lambda$. A \emph{profile} $\mu$ of $\lambda$ is a partition constructed in the following manner: choose indices $(j,k)$ with $0 < j \leq k \leq n+1$ such that $i = i'$ for all $j\leq i < k$, and such that $j-1\neq (j-1)'$ (or $j-1 = 0$) and $k\neq k'$ (or $k = n+1)$. Define $\mu = (\mu_1,...,\mu_{k-j})$ by the rule $$\mu_i = \lambda_{i + (j-1)} - \lambda_k.$$ If $k < n+1$ then in order to preserve the condition $i=i'$ we regard $\mu$ as an element of $\mathcal{P}_{1}(\sum_{i=j}^{k-1} \lambda_i - (k-j)\lambda_k)$. If $k = n+1$ then $\lambda_{k} 
= 0$ and we may regard $\mu$ is an element of $\mathcal{P}_{\epsilon}(\sum_{i=j}^{n} \lambda_i)$. We say that the  profile $\mu$ constructed in this manner is \emph{of type $(j,k)$}, and we include Figure~\ref{pikcha_B} to show what is intended by the definition.

\p Suppose $\mu$ is a profile of $\lambda$ of type $(j,k)$ and $\ii = (i_1,...,i_l)$ is an admissible sequence for $\mu$. Then the \emph{$j$-adjust} of $\ii$ is the sequence
$$\jmath(\ii) = (i_1 + (j-1), i_2 + (j-1), ..., i_l + (j-1)).$$
It is clear that $\jmath(\ii)$ is an admissible sequence for $\lambda$.
\begin{prop}\label{reductionprop}
Suppose $\lambda$ is equal to its shell and let $\mu(1), \mu(2),..., \mu(l)$ be a complete set of distinct profiles for $\lambda$, with $\mu(m)$ of type $(j_m, k_m)$. Then the following hold:
\begin{enumerate}
\item{$z(\lambda) = \sum_{i=1}^l z(\mu(i))$.}

\item{If $\ii(m)$ is admissible sequence for $\mu(m)$ then $$(\jmath_1(\ii(1)), \jmath_2(\ii(2)),...,\jmath_l(\ii(l)))$$ is an admissible sequence for $\lambda$, where this last sequence is obtained by concatenating the sequences $\jmath_m(\ii(m))$.
}
\end{enumerate}
\end{prop}
\begin{proof}
Since $\lambda = \lambda^S$ all differences $\lambda_i-\lambda_{i+1}$ are equal to $0,1,$ or $2$. If $\lambda_i - \lambda_{i+1} = 2$ then necessarily $i-1\in \Delta(\lambda)$ or $i+1 \in \Delta(\lambda)$. In either case $i=i'$, $i+1 = (i+1)'$ (or $i = n$) and it follows that there exists a profile of type $(j,k)$ with $j \leq i$ and $i+1 < k$ (or $i < k$ when $i = n$). Then each index $i$ for which $\lambda_i - \lambda_{i+1} = 2$ contributes 1 to $s(\lambda)$ and 1 to $\sum_{j=1}^l s(\mu(j))$ so that $s(\lambda) = \sum_{j=1}^l s(\mu(j))$. The condition $\lambda = \lambda^S$ also implies that all 2-steps are bad 2-steps so that $|\Delta(\lambda)| = |\Delta_{\rm bad}(\lambda)|$. Similarly $\mu(m) = \mu(m)^S$ so $|\Delta(\mu(m))| = |\Delta_{\rm bad}(\mu(m))|$ for all $m$, and it remains to prove that $|\Sigma(\lambda)| = \sum_{i=1}^l |\Sigma(\mu(i))|$. This follows from the fact that all good 2-clusters $i_1\leq \cdots \leq i_l$ fulfil $i = i'$ for all $i_1\leq i \leq i_l+1$ so for each such 2-cluster there 
exists profile of type $(j,k)$ with $j\leq i_1$ and $i_l +1< k$. Part (1) follows.

The second actually holds even when $\lambda \neq \lambda^S$. For obvious reasons the indices of the distinct profiles do not overlap, and we may assume that $k_m < j_{m+1}$ for $m=1,...,l-1$. Then for $1 \leq i < l$ we set $\jj(i) = (\jmath_1(\ii(1)),...,\jmath_i(\ii(i)))$ and note that $\lambda_r^{\jj(i)} = \lambda_r$ for all $r \geq j_{i+1}$. Using that $\jmath_{i+1}(\ii(i+1))$ is admissible for $\lambda$ we obtain by induction that $\jmath_{i+1}(\ii(i+1))$ is an admissible sequence for $\lambda^{\jj(i)}$. By the transitivity of the algorithm we deduce then that $(\jmath_1(\ii(1)), \jmath_2(\ii(2)),...,\jmath_l(\ii(l)))$ is admissible for $\lambda$ as required.
\end{proof}

\p The previous lemma allows us to obtain admissible sequences of length $z(\lambda)$ by piecing together sequences for the profiles. Now we demonstrate the existence of a sequence of length $z(\lambda)$ for each profile.
\begin{prop}\label{profileprop}
Let $\lambda = (\lambda_1,..,\lambda_n)$ be a partition and suppose that $i=i'$ for all $1\leq i\leq n$. Then there exists an admissible sequence for $\lambda$ of length $z(\lambda)$.
\end{prop}
\begin{proof}
A partition $\lambda$ fulfilling $i=i'$ for all $1\leq i\leq n$ contains a good 2-cluster if and only if $1, 3, 5, ..., n-1$ is good 2-cluster. In this case it is the only good 2-cluster. Suppose that this is the case. Of course this implies that $n$ is even and $\epsilon = 1$, so $\lambda_n$ is odd. Construct a sequence $\ii$ by repeatedly applying Case 1 at indices $2i-1$ for $1\leq i \leq \frac{n}{2}$ so that $\lambda^\ii_{2i-1} - \lambda^\ii_{2i} = 0$ for all such $i$. Then $$|\ii| = \sum_{i=1}^{\frac{n}{2}} \lfloor \frac{\lambda_{2i-1} - \lambda_{2i}}{2} \rfloor.$$ We construct an admissible $\ii'$ by subsequently applying Case 1 at indices $2i$ for $1\leq i\leq n$ so that $\lambda^{\ii'}_{2i} - \lambda^{\ii'}_{2i+1} = 2$ for all such $i$. Our sequence $\ii'$ has length $$|\ii'| = s(\lambda) - (\frac{n}{2}-1).$$ At this point we are able to say precisely what $\lambda^{\ii'}$ looks like. We have $\lambda^{\ii'} = \lambda^S = (n-1,n-1,n-3,n-3,...,3,3,1,1)$. Finally we obtain $\ii''$ by applying Case 2 
precisely once at each index $2i-1$ for $1\leq i\leq \frac{n}{2}$. The partition $\lambda^{\ii''}$ is rigid, so $\ii''$ is maximal (Lemma~\ref{maxadmiss}) and $$|\ii''| = s(\lambda) + 1.$$ In order to complete this part of the proof we must show that $z(\lambda) = s(\lambda) + 1$. Notice that our assumptions on $\lambda$ imply that every 2-step is bad. Therefore $|\Delta(\lambda)| = |\Delta_{\rm bad}(\lambda)|$ and by our original remarks $z(\lambda) = s(\lambda) + 1$ as required.

Now assume that $\lambda$ has no good 2-clusters. Since $i=i'$ for all $i$ we may apply Case 1 repeatedly at all indices to obtain a maximal admissible partition. Clearly $|\ii| = s(\lambda)$. Once again all 2-steps are bad so that $|\Delta(\lambda)| = |\Delta_{\rm bad}(\lambda)|$, and by assumption $|\Sigma(\lambda)| = 0$. Hence $z(\lambda) =s(\lambda) =  |\ii|$ as promised.
\end{proof}

\p Finally we may state and prove the main theorem of this section.
\begin{thm}\label{zismax} We have that
$$z(\lambda) = \max |\emph\ii|$$ where the maximum is taken over all admissible sequences $\emph\ii$ for $\lambda$.
\end{thm}
\begin{proof}
We begin by showing that $z(\lambda) \geq z(\lambda^{\ii}) + 1$ where $\ii = (i)$ is an admissible sequence of length 1 for $\lambda$. First assume Case 1 occurs for $\lambda$ at $i$. Then $s(\lambda^\ii) = s(\lambda) - 1$. Furthermore, if the iteration at $i$ removes a 2-step (ie. if $\lambda_{i} - \lambda_{i+1} = 2$ and either $i-1 \in \Delta(\lambda)$ or $i+1 \in \Delta(\lambda)$ or both) then that 2-step is bad. Therefore $|\Delta(\lambda)| - |\Delta(\lambda^\ii)| = |\Delta_{\rm bad}(\lambda)| - |\Delta_{\rm bad}(\lambda^\ii)|$. It remains to be seen that the number of good 2-clusters does not increase as we pass from $\lambda$ to $\lambda^{\ii}$. This follows from Lemma~\ref{stairslemma}.

Now suppose that Case 2 occurs for $\lambda$ at index $i$. Certainly if $(i,i+1)$ is a good 2-step then $z(\lambda^\ii) = z(\lambda) - 1$, so we may assume that $(i, i+1)$ is a bad 2-step. Suppose first that this 2-step has precisely one bad boundary. We may assume that $\lambda_{i-1} - \lambda_i $ is even and $\lambda_{i+1} - \lambda_{i+2}$ is odd. We can deduce at this point that $s(\lambda^\ii) = s(\lambda) - 1$ and $|\Delta(\lambda^\ii)| = |\Delta(\lambda)| - 1$. If $i-2 \notin \Delta(\lambda)$ then $|\Delta_{\rm bad}(\lambda^\ii)| = |\Delta_{\rm bad}(\lambda)| - 1$. Similarly, if $i-2 \in \Delta(\lambda)$ and $\lambda_{i-3} - \lambda_{i-2}$ is even then $|\Delta_{\rm bad}(\lambda^\ii)| = |\Delta_{\rm bad}(\lambda)| - 1$. In either of these two situations the number of good 2-clusters decreases, thanks to Lemma~\ref{stairslemma}. Hence $z(\lambda) \geq z(\lambda^{\ii}) + 1$ once again. We must now consider the possibility that $i-2 \in \Delta(\lambda)$ and $\lambda_{i-3} - \lambda_{i-2}$ is odd. In this 
situation $s(\lambda^\ii) = s(\lambda) - 1$, $|\Delta(\lambda^\ii)| = |\Delta(\lambda)| - 1$ and $|\Delta_{\rm bad}(\lambda^\ii)| = |\Delta_{\rm bad}(\lambda)| - 2$. Notice that $i-2, i$ is a good 2-cluster for $\lambda$ but not for $\lambda^\ii$, so that $|\Sigma(\lambda^\ii)| = |\Sigma(\lambda)| - 1$ and $z(\lambda^\ii) = z(\lambda) - 1$. A similar argument works when $\lambda_{i-1} - \lambda_i$ is odd but $\lambda_{i+1} - \lambda_{i+2} = 2$.

Now we assume that $(i, i+1)$ is a bad 2-step and that both boundaries are bad. If neither $i-2$ nor $i+2$ lie in $\Delta(\lambda)$ then $s(-)$ decreases by 2, $|\Delta(-)|$ decreases by 1, and $|\Delta_{\rm bad}(-)|$ decreases by 1 upon passing from $\lambda$ to $\lambda^\ii$. Certainly $|\Sigma(-)|$ may only decrease, by lemma \ref{stairslemma}, and so $z(\lambda) \geq z(\lambda^\ii) + 1$ in this situation. Now move on and suppose that precisely one of $i-2$ and $i+2$ lie in $\Delta(\lambda)$. We shall examine the case $i-2\in \Delta(\lambda)$, the other being very similar.

When $\lambda_{i-3} - \lambda_{i-2}$ is odd $s(\lambda^\ii) = s(\lambda) - 2$, $|\Delta(\lambda^\ii)| = |\Delta(\lambda)| - 1$ and $|\Delta_{\rm bad}(\lambda^\ii)| = |\Delta_{\rm bad}(\lambda)| - 2$ (since $(i-2,i-1)$ is no longer a bad 2-step after this iteration). Furthermore $(i,i+1)$ cannot make up a 2-step in a good 2-cluster since $i+2 \notin \Delta(\lambda)$ and $\lambda_{i+1} - \lambda_{i+2}$ is even, therefore $|\Sigma(\lambda)|$ remains unchanged. So consider the possibility that $(i-2,i-1)$ has two bad boundaries: that $\lambda_{i-3} - \lambda_{i-2}$ is even. Then our conclusions are exactly the same as before, except that $|\Delta_{\rm bad}(\lambda^\ii)| = |\Delta_{\rm bad}(\lambda)| - 1$. In either situation $z(\lambda^\ii) \geq z(\lambda) - 1$.

Finally we have the situation $i-2,i+2 \in \Delta(\lambda)$. Once again we must distinguish between the number of bad boundaries attached to the 2-steps $(i-2,i-1)$ and $(i+2,i+3)$. Suppose that both of these 2-steps have a single bad boundary (they have at least 1). Then $i - 2, i, i+2$ is a good 2-cluster. It is immediately clear upon passing from $\lambda$ to $\lambda^\ii$ that $s(\lambda^\ii) = s(\lambda) - 2$, $|\Delta(\lambda^\ii)| = |\Delta(\lambda)| - 1$, $|\Delta_{\rm bad}(\lambda^\ii)| = |\Delta_{\rm bad}(\lambda)| - 3$, and $|\Sigma(\lambda^\ii)| = |\Sigma(\lambda)| - 1$. Once again $z(\lambda^\ii) \geq z(\lambda) - 1$ follows. The last two situations to consider are when precisely one of the two 2-steps $(i-2,i-1)$ and $(i+2, i+3)$ has two bad boundaries, and when both of them have two bad boundaries.

Take the former situation. We may assume that $(i-2, i-1)$ has two bad boundaries, and $(i+2,i+3)$ has one (the opposite configuration is similar). Upon iterating the algorithm, $s(\lambda^\ii) = s(\lambda) - 2$, $|\Delta(\lambda^\ii)| = |\Delta(\lambda)| -1$ and $|\Delta_{\rm bad}(\lambda^\ii)| = |\Delta_{\rm bad}(\lambda)| - 2$. By lemma \ref{stairslemma}, $|\Sigma(\lambda^\ii)| \leq |\Sigma(\lambda)|$. In the final case $(i-2,i-1)$ and $(i+2,i+3)$ both have two bad boundaries. The outcome is that $s(\lambda^\ii) = s(\lambda) - 2$, $|\Delta(\lambda^\ii)| = |\Delta(\lambda)| - 1$ and $|\Delta_{\rm bad}(\lambda^\ii)| = |\Delta_{\rm bad}(\lambda)| - 1$ both decrease by 1 and by Lemma~\ref{stairslemma} either $|\Sigma(\lambda^\ii)| = |\Sigma(\lambda)|$ or $|\Sigma(\lambda^\ii)| = |\Sigma(\lambda)| + 1$.

We have eventually shown that $z(\lambda) \geq z(\lambda^\ii) + 1$. Recall that for any maximal admissible sequence $\ii$ the partition $\lambda^\ii$ is rigid. Also notice that $z(\lambda) = 0$ for any rigid partition $\lambda$. We deduce for any maximal admissible sequence $\ii$ of length $l$, that
$$z(\lambda) \geq z(\lambda^{\ii_2}) + 1 \geq z(\lambda^{\ii_3}) + 2 \geq \cdots z(\lambda^{\ii_{l+1}}) + l = l.$$
Here $\ii_k$ denotes $(i_1,...,i_{k-1})$. In order to complete the proof we shall exhibit a maximal admissible sequence of length $z(\lambda)$. This shall require some reductions.

Notice first that $z(\lambda)$ decreases by 1 at each iteration when we apply Case 1 in constructing the shell $\lambda^S$. Therefore we may assume that $\lambda = \lambda^S$. Let $\mu(1), \mu(2),..., \mu(l)$ be a complete set of distinct profiles for $\lambda$, as in the statement of Proposition~\ref{reductionprop}. By Proposition~\ref{profileprop} we know that for each $1\leq m\leq l$ there is an admissible sequence of length $z(\mu(m))$ for $\mu(m)$. Using Part (2) of Proposition~\ref{reductionprop} we obtain an admissible sequence for $\lambda$ of length $\sum_{i=1}^l z(\mu(i))$, and by Part (1) of the same proposition that length is equal to $z(\lambda)$. Hence a sequence of the correct length exists, and the theorem follows.
\end{proof}

\p The following corollary shall be of some importance to our later work.
\begin{cor}\label{maxseq}
For all $\lambda \in \mathcal{P}_\epsilon(N)$ the following hold:
\begin{enumerate}
\item{$c(\lambda) \geq z(\lambda)$;}

\smallskip

\item{$c(\lambda) = z(\lambda)$ if and only if $\lambda$ is non-singular.}
\end{enumerate}
\end{cor}
\begin{proof}
Part (1) follows from the fact that $|\Delta_{\rm bad}(\lambda)| \geq |\Sigma(\lambda)|$ for all partitions $\lambda$. For Part (2) we observe that $|\Delta_{\rm bad}(\lambda)| - |\Sigma(\lambda)| = 0$ if and only if $\lambda$ is non-singular.
\end{proof}

\section{Conjugacy classes of Levi subalgebras}\label{conjlev}
\setcounter{parno}{0}

\p Before we describe the relationship between admissible sequences and sheets we shall need to classify the Levi subalgebras of the classical Lie algebra $\h$ of type $\sf B$, $\sf C$ or $\sf D$, and give a representative for each class. It is well known that every Levi subalgebra is isomorphic to a standard Levi subalgebra, and that these in turn are constructed from the subsets of the simple roots. It follows that the conjugacy classes of such algebras are 1-1 with subsets of simple roots modulo the action of the Weyl group. For our purposes we shall need a classification in terms of the standard form which we gave in Lemma~\ref{levistruct}.

\p Assume the notations $\hh$, $\Phi$ and $\Pi$ of Section~\ref{rootnotation} and write $\ell = \rank\, \h$. Give the simple roots $\Pi = \{\alpha_1,...,\alpha_\ell\}$ the standard ordering. From every sequence $\ii = (i_1,...,i_l)$ (possibly empty) of natural numbers with $\sum i_j \leq \ell$ we construct a subset $\Pi_\ii \subseteq \Pi$ by excluding the roots labelled by the integers $\sum_{j=1}^k i_j$ with $k = 1,...,l$. In this way we obtain a bijection between sequences with sum bounded by $\ell$, and subsets of $\Pi$. From each subset $\Pi_\ii$ we can define the Levi subalgebra $\li^\ii$ to be the algebra generated by $\hh$ and all roots spaces $\h^{\pm \alpha}$ with $\alpha \in \Pi_\ii$. In \ref{leviint} we denoted this algebra by $\li({\Pi_\ii})$.

\p\label{leviform1} Set $R_\ii = N - 2\sum_j i_j$ and say that $\ii$ is a \emph{restricted sequence} if $\sum_{j} i_j \leq \ell$ and if $R_\ii \neq 2$ when the type of $\h$ is $\sf D$. The advantage of this definition is that for $\ii$ restricted we have $\li^\ii \cong \gl_\ii \times \mm$ where $\gl_\ii := \gl_{i_1} \times \cdots \times \gl_{i_l}$ and $\mm$ is either zero or a classical simple algebra of the same type as $\h$ and natural representation of dimension $R_\ii$. In exchange for this nice description of the isomorphism types we have lost some descriptive power: the sequences in type $\sf D$ with $R_\ii = 2$ also correspond to subsets of $\Pi$ and so to standard Levi subalgebras. We would like to describe these remaining Levi subalgebras in slightly different terms.

\p \label{levparam} Let $\h$ be of type $\sf D$ and let $D$ denote the outer automorphism of $\h$ coming from the diagram automorphism which exchanges the last two nodes of the Dynkin diagram. Let $\ii = (i_1,...,i_l)$ with $R_\ii = 0$ and suppose the set $\Pi_\ii$ excludes one but not both of the last two simple roots of $\Pi$ (this supposition is equivalent to $i_l \neq 1$). Then we obtain a new Levi algebra $D(\li^\ii)$, which corresponds to the sequence $(i_1,...,i_{l-1}, i_l - 1)$. Now all standard Levi subalgebras associated to our choice $\Pi$ of simple roots may be described uniquely as one of the following:
\begin{itemize}
\item{$\li^\ii$ with $\ii$ a restricted sequence;}
\item{$D(\li^\ii)$ with $\ii$ restricted, $R_\ii = 0$ and $i_l \neq 1$.}
\end{itemize}

\noindent This parameterisation has the useful property that every algebra constructed from a sequence $\ii$ is isomorphic to $\gl_\ii \times \mm$.

\p We may classify the conjugacy classes of Levi subalgebras by deciding when two of these standard Levis are conjugate. This classification will follow immediately from the proposition:
\begin{prop}\label{leviconjprop}
Let $\ii$ and $\jj$ be restricted sequences. Then the following are equivalent:
\begin{enumerate}
\item{$\li^\ii$ and $\li^\jj$ are $K$-conjugate;}
\item{$\ii$ and $\jj$ have equal length $l$ and are $\mathfrak{S}_l$-conjugate.}
\end{enumerate}
Suppose $\h$ has type $\sf D$, that $R_\ii = 0$ and $i_l \neq 1$. Then $\li^\ii$ is $K$-conjugate to $D(\li^\ii)$ if and only if $\rank\, \h$ is odd or some term of $\ii$ is odd.
\end{prop}
\begin{proof}
Suppose that $\li^\ii$ and $\li^\jj$ are $K$-conjugate, and suppose $\Ad(g)$ sends $\li^\ii$ to $\li^\jj$. Set 
$\li^\ii = \gl_\ii \times \mm_\ii$ and $\li^\jj = \gl_\jj \times \mm_\jj$ where $\mm_\ii$ and $\mm_\jj$
are classical algebras described in \ref{leviform1}. By \ref{levicent} there is a semisimple element $h \in \hh$ such that
$\li^\ii = \h_h$ and so $\li^\jj = \Ad(g) \h_h = \h_{\Ad(g) h}$. Then $\mm_\ii$ is the annihilator in $\h$ of the
non-zero eigenspaces of $h$ in $V$ and $\mm_\jj$ is the annihilator in $\h$ of the non-zero eigenspaces for $\Ad(g)h$.
It follows that $\mm_\ii$ and $\mm_\jj$ are isomorphic, and that $\gl_\ii \cong \gl_\jj$. We conclude that
$\ii$ and $\jj$ are $\mathfrak{S}_l$-conjugate, where $l = |\ii|$

We shall show that $(2) \Rightarrow (1)$. The cases $|\ii| = 0$ or $|\ii| = 1$ are trivial. Suppose
that $|\ii| = 2$. If $\ii = \jj$ then there is nothing to prove, so suppose that $\ii = (i_1, i_2)$ and 
that $\jj = (i_2, i_1)$. Set $k = i_1 + i_2 - 1$ and note that the roots $\alpha_1,...,\alpha_k$ span
a root subsystem $\Phi_k \subseteq \Phi$ of type $\sf{A}_k$. Observe furthermore that $\alpha_k$ is
orthogonal with respect to the Killing form to all $\alpha_i$ with $i > k+1$   (if $\h$ has type $\sf D$ then this observation 
relies upon the assumption that $R_\ii \neq 2$). The aforementioned roots form a base
for $\Phi_k$ which we will denote by $\Pi_k$. The Weyl group $W_k$ of $\Phi_k$ includes
canonically into the Weyl group $W$ of $\Phi$. The longest element $w_0$ of $W_k$ acts on $\Pi_k$ by sending
$\alpha_i \mapsto -\alpha_{k+1-i}$ for $i=1,...,k$ and fixes all roots $\alpha_i$ with $i > k+1$.
Let $w_0 = gH \in N_K(H)/H = W$ where $\hh = \Lie(H)$ and $H$ is a closed subgroup of $K$.
According to the above, $\Ad(g)$ sends the simple factors of $[\li^\ii \li^\ii]$ to the simple factors of
$[\li^\jj \li^\jj]$, and preserves the torus $\hh$. It follows that $\Ad(g)$ sends $\li^\ii$ to $\li^\jj$.

The general case is very similar. Let $\ii$ have length $l$. Using the same observations
as above it is possible to construct for every transposition $\sigma \in \mathfrak{S}_l$ an element
$g \in K$ such that $\Ad(g)$ maps $\li^\ii$ to $\li^{\sigma \ii}$. Since $\mathfrak{S}_l$ is generated
by transpositions we get $(2) \Rightarrow (1)$. 

For the remnant of the proof we assume that $\h$ has type $\sf D$. Let $D$ be the automorphism of $\h$ coming from
the non-trivial automorphism $d$ of the Dynkin diagram which exchanges the last two roots. We must show
that $\li^\ii$ is conjugate to $D(\li^\ii)$ if and only if $\ii$ has an odd term or $\rank \, \h$ is odd. When the rank
is odd the longest element $w_0$ of the Weyl group of $(\h, \hh)$ acts by $-d$. If $w_0 = gH \in N_K(H)/H = W$ then
$\Ad(g)$ sends $\li^\ii$ to $D(\li^\ii)$. Now suppose that some term of $\ii$ is odd and that $\rank \, \h$ is even.
Since we have ascertained that $\li^\ii$ is conjugate to all $\li^{\sigma \ii}$ with $\sigma \in \mathfrak{S}_l$
we might assume that $i_1$ is odd. The assumption that $R_\ii \neq 2$ also ensures that $i_1 \neq \rank\, \h - 1$.
Since $\rank \, \h$ is even, $\rank\, \h - i_1$ is odd. It follows that the roots $\alpha_{i_1+1}, \alpha_{i_1+2},...,\alpha_\ell$
span a root system of type ${\sf D}_{\rank\, \h - i_1}$ and the longest element of the Weyl group of that subsystem is
represented by an element $g \in K$ such that $\Ad(g) \li^\ii = D(\li^\ii)$ (arguing as above).

Finally, we must show that when $\rank\, \h$ is even and all parts of $\ii$ are even the algebras $\li^\ii$ and
$D(\li^\ii)$ are not conjugate. Consider the orbit $\Ind^\h_{\li^\ii}(\Oo_0)$ induced from the zero orbit in $\li^\ii$. We shall see later in the text
that the partition of the induced orbit is very even and that the label of the orbit is $I$. This shall follow from
Proposition~\ref{indpart} and is completely independent of our deductions here. It follows from the same proposition
that $\Ind_{D(\li^\ii)}^\h(\Oo_0)$ has the same partition but the label is $I\!I$. Since the orbits labelled $I$ and $I\!I$ are distinct we deduce that the pair $(\li^\ii, \Oo_0)$ is
not $K$-conjugate to $(D(\li^\ii), \Oo_0)$ and so $\li^\ii$ is not conjugate to $D(\li^\ii)$.
\end{proof}

\p In type $\sf{D}_{\ell}$ with $\ell$ even we shall call a restricted sequence $\ii$ \emph{very even} if $R_\ii = 0$ and all parts of $\ii$ are even. 
We shall say that two restricted sequences are equivalent if they are conjugate under the action of a symmetric group.
Let $\mathcal{R}(\ell)$ denote the equivalence classes of restricted sequences and let $\LL(\h)$ denote the set of Levi subalgebras
of $\h$. Finally we can state the classification
theorem.
\begin{thm}\label{leviclassification}
Let $\h$ have rank $\ell$. There is a surjection $$\pi: \LL(\h)/K \xtwoheadrightarrow{} \mathcal{R}(\ell).$$ The fibre above $(\ii/\!\sim) \, \in \mathcal{R}(\ell)$ is a singleton
containing the conjugacy class of $\li^\ii$ unless $\ell$ is even and $\ii$ is very even, in which case the fibre contains two elements:
the conjugacy classes of $\li^\ii$ and $D(\li^\ii)$, which are permuted by the outer automorphism of $\h$.
\end{thm}
\begin{proof}
The theorem follows immediately from the previous proposition. 
\end{proof}
\begin{rem}
This classification bears a satisfying resemblance to the classification of nilpotent orbits given in \ref{classsurj}.
\end{rem}

\p\label{labelsonlevis} Extending the conventions of \cite[Lemma~7.3.2(ii)]{CM} we attach labels to the conjugacy classes of
Levi subalgebras which lie in the fibre above a very even restricted sequence. Let $\rank\, \h$ be even and let $\ii$ be very
even. The conjugacy class of $\li^\ii$ is labelled $I$, whilst the conjugacy class of $D(\li^\ii)$ is labelled $I\!I$.
The similarity with the labels for very even nilpotent orbits is no coincidence, as we shall see when we study induced orbits in
Proposition~\ref{indpart} of the next section. 

\section{A geometric interpretation of the algorithm}\label{KScontext}
\setcounter{parno}{0}

\p We would like to characterise the non-singular partitions in geometric terms. This characterisation  shall be given in the corollary to the next theorem. The remainder of this section shall be spent preparing to prove that theorem. The symmetric group $\mathfrak{S}_l$ acts on the set of sequences in $\{1,...,n\}$ of length $l$  by the rule $\sigma (i_1,...,i_l) = (i_{\sigma(1)},...,i_{\sigma(l)})$. Let $$\Phi_\lambda := \{ \text{the maximal admissible sequences for }\lambda \} /\sim$$ where $\ii \sim \mathbf{j}$ if $\ii$ and $\jj$ have equal length and are conjugate under an element of a symmetric group. What follows is the main theorem of this section.
\begin{thm}\label{nosheets}
The following are true for any $\lambda\in\mathcal{P}_\epsilon(N)$ and any nilpotent element $e(\lambda)$ with partition $\lambda$:
\begin{enumerate}
\item{$e(\lambda)$ lies in $|\Phi_\lambda|$ distinct sheets;}

\smallskip

\item{$|\Phi_\lambda| = 1$ if and only if $\lambda$ is non-singular.}
\end{enumerate}
\end{thm}

\noindent The next corollary explains our choice of terminology.
\begin{cor}\label{izosim} Suppose $\lambda\in\mathcal{P}_\epsilon(N)$.
Then the following are equivalent:
\begin{enumerate}
\item{the partition $\lambda$ is non-singular;}

\smallskip

\item{$c(\lambda) = z(\lambda)$;}

\smallskip

\item{$e(\lambda)$ lies in a unique sheet;}
\end{enumerate}
If $\K$ has characteristic zero then $1$, $2$ and $3$ hold if and only if
$e(\lambda)$ is a non-singular point on the quasi-affine variety $\h^{(\dim\, \h_e)}$.
\end{cor}
\begin{proof}
The statements 1, 2, and 3 are equivalent by
Corollary~\ref{maxseq} and Theorem~\ref{nosheets}. Now suppose that
the characteristic of $\K$ is 0.  Then
it is proven in \cite[Chapter 6]{IH} that the sheets of $\h$
are smooth. In this
situation it follows from \cite[Chapter II, \S2, Theorem 6]{Sh} that
$e$ is a non-singular point of the algebraic variety
$\h^{(m)}$ if and only if $e$ belongs to a unique irreducible
component of $\h^{(m)}$. This completes the proof.
\end{proof}

\p We shall now assemble all of the necessary information required to prove Theorem \ref{nosheets}. Recall that the rigid nilpotent orbits are those which are not induced. We stated the classification of rigid orbits in $\h$ in terms of partitions in Theorem~\ref{rigids}. It is well known that every orbit in the special linear algebra is Richardson (induced from the zero orbit in some Levi subalgebra) so that the zero orbit is the only rigid orbit (see \cite[7.2.3]{CM} for example). The first part of the next lemma follows quickly from these observations. The remaining parts are contained in  \cite{LS}, \cite{BKr}, \cite{Bor}.
\begin{prop}\label{inducedprops}
The following are true:
\begin{enumerate}
\item{If $\li$ is a Levi subalgebra isomorphic to $\gl_\ii \times \mm$ as per $\S\ref{conjlev}$ then the rigid nilpotent orbits in $\li$ take the form $$\Oo = \Oo_0 \times \Oo_{\mu}$$ with $\mu \in \mathcal{P}_\epsilon(N - 2\sum_j i_j)^\ast$ a rigid partition, $\Oo_\mu$ a nilpotent orbit in $\mm$ with partition $\mu$ and $\Oo_0$ the zero orbit in $\gl_\ii$.}

\item{If $\mathcal{S}$ is a sheet with data $(\li, \Oo_\li)$ then \emph{Ind}$_\li^\h(\Oo_\li)$ is the unique nilpotent orbit contained in $\mathcal{S}$;}

\item{If $\li_1$ and $\li_2$ are Levi subalgebras of $\h$, $\Oo$ is a nilpotent orbit in $\li_1$ and $\li_1 \subseteq \li_2$, then $$\text{\emph{Ind}}_{\li_2}^\h(\text{\emph{Ind}}_{\li_1}^{\li_2}(\Oo) ) = \text{\emph{Ind}}_{\li_1}^\h(\Oo).$$}
\end{enumerate}
\end{prop}

\p Fix an orbit $\Oo_\lambda$ with partition $\lambda \in \mathcal{P}_\epsilon(N)$. Recall the map $\pi$ of Theorem~\ref{leviclassification}. Let $\Psi_\lambda$ denote the set of all $K$-conjugacy classes of pairs $(\li, \Oo)$ where the class of $\li$ lies in the fibre of $\pi$ above some restricted sequence $\ii$, so that $\li \cong \gl_\ii\times \mm$ is a Levi subalgebra of $\h$, and where $\Oo = \Oo_0 \times \Oo_{\mu}$ a nilpotent orbit in $\li$, such that $\mu$ is a rigid partition and $\Oo_{\lambda} = \Ind_{\li}^\h(\Oo).$ The significance of this set is given by the following.
\begin{lem}\label{contsheets}
$\Oo_\lambda$ lies in $|\Psi_\lambda|$ distinct sheets.
\end{lem}
\begin{proof}
Let $\mathcal{S}$ be a sheet of $\h$ with data $(\li, \Oo)/K$. By Proposition~\ref{inducedprops} we see that $\Oo_\lambda \subseteq \mathcal{S}\cap \mathcal{N}(\h)$ if and only if $\Oo_\lambda = \Ind_\li^\h(\Oo)$. By Theorem~\ref{classifsheets} the orbit $\Oo$ is rigid and by part 1 of Lemma~\ref{inducedprops} it takes the form prescribed in the definition of $\Psi_\lambda$.
\end{proof}

\p In light of \ref{contsheets} our method for proving part 1 of Theorem~\ref{nosheets} is now easy to explain: we construct a bijection from $\Phi_\lambda$ to $\Psi_\lambda$. This construction is no miracle as the algorithm was defined with this bijection in mind. This shall be done in Corollary~\ref{bijection}. Given the definition of $\Psi_\lambda$ it is clear that we shall need to develop a precise understanding of the partitions associated to induced orbits. The result stated below may be deduced from \cite[Corollary~7.3.3]{CM}. We warn the reader that when interpreting the proposition for algebras of type $\sf B$ the unique nilpotent orbit in the trivial algebra $\mathfrak{so}_1$ is labelled by the partition $\lambda = (1)$ contrary to the common convention. Furthermore, our description of labels attached to induced orbits does not quite agree with the description in \cite{CM}. This is due to a small misprint in that book which we will explain in the remark below. 
\begin{prop}\label{indpart}
Recall that the natural representation of $\h$ is of dimension $N$. Choose $0 < 2i \leq N$ and set $R = N - 2i$. Let $\li \cong \gl_{i} \times \mathfrak{m}$ be a maximal Levi subalgebra of $\h$ in a conjugacy class constructed from $(i)$ and let $\Oo = \Oo_0\times \Oo_{\mu}$ be a nilpotent orbit in $\li$ where $\Oo_\mu$ has partition $\mu \in \mathcal{P}_\epsilon(R)$. Then $\Oo_{\lambda} = \emph{\Ind}_\li^\h(\Oo)$ has associated partition $\lambda$ where $\lambda$ is obtained from $\mu$ by the following procedure: add 2 to the first $i$ columns of $\mu$ (extending by zero if necessary); if the resulting partition lies in $\mathcal{P}_\epsilon(N)$ then we have found $\lambda$, otherwise we obtain $\lambda$ by subtracting 1 from the $i^\text{th}$ column and adding 1 to the $(i+1)^\text{th}$.

Suppose we are in type $\sf D$. If $\lambda$ is very even then either $\mu$ is very even or $R = 0$ and $\rank\, \h$ is even. If $R \neq 0$ then $\Oo_{\lambda}$ inherits its label from $\mu$, whilst if $R = 0$ then the induced orbit inherits its label from $\li$.
\end{prop}
\begin{rem}\label{amendment}
\rm{The above proposition is based on \cite[Corollary~7.3.3]{CM} however the reader will notice that the way in which the labels are chosen does not coincide with that lemma. The reason for this is that the book contains two small misprints which we must now amend.

The first problem stems from a comparison Lemmas~5.3.5 and 7.3.3(ii). We see, given the conventions of 5.3.5, that 7.3.3(ii) should actually state that the label of $\Ind^\h_{\gl_i\oplus\mm}(\Oo)$ is different to the label of $\Oo$ when $(\rank\, \h + \rank\, \mm)/2$ is odd. We could, of course, change 7.3.3(ii) but a better amendment is to change 5.3.5 so that the labelling convention for very even orbits is independent of $n$: in their notation we take $a = 2$ and $b = 0$ regardless of $n$. With this convention the statement of 7.3.3(ii) is correct however Lemma~7.3.3(iii) should now state that the label of the induced orbit coincides with the label of Levi from which it is induced. This is the convention we have taken in the above proposition.

The second misprint regards the number of conjugacy classes of maximal Levis in 7.3.2(ii). That lemma states that in type $\sf{D}_\ell$ there are two conjugacy classes of Levi subalgebras of the form $\gl_{\ell}$. Comparing with our Theorem~\ref{leviclassification} we see that their claim holds provided $\ell$ is even, but not when $\ell$ is odd, where there is a single class isomorphic to $\gl_\ell$.}
\end{rem}

\p In light of the above proposition we may explain the definition of the KS algorithm. We fix an orbit $\Oo_\lambda$ with partition $\lambda$ and want to decide when is it possible to find a pair consisting of a maximal Levi $\li = \gl_{i_1}\oplus \mathfrak{m}\cong\gl_{i_1}\times \mathfrak{m}$ and a nilpotent orbit $\Oo = \Oo_0 \times \Oo_{\mu}$ (with partition $\mu$) such that $\Ind_\li^\h(\Oo) = \Oo_{\lambda}$. It is now clear that this occurs precisely when we have an admissible index $i$ and a Levi subalgebra isomorphic to $\gl_i \times \mm$. In this case $\mu = \lambda^{(i)}$ and if $\Oo_{\mu}$ has a label then it is completely determined by that of $\Oo_{\lambda}$. The precise statement is contained in the following corollary.

Recall that in Theorem~\ref{leviclassification} we constructed a map $\pi$ from restricted sequences upto rearrangement to conjugacy classes of Levi subalgebras. 
\begin{cor}\label{case1or2}
Choose an orbit $\Oo_\lambda$ with partition $\lambda\in \mathcal{P}_\epsilon(N)$. Suppose that $(i)$ is a restricted sequence and that $\li$ is a maximal Levi subalgebra whose conjugacy class lies in the fibre of $\pi$ above $(i)$, so that $\li \cong \gl_i \times \mm$. Then the following are equivalent:
\begin{enumerate}
\item{$i$ is an admissible index for $\lambda$. If $(i)$ is a very even sequence then $\li$ belongs to the conjugacy class with the same label as $\Oo_{\lambda}$ (Cf. \ref{labelsonlevis});}
\item{There exists an orbit $\Oo = \Oo_0 \times \Oo_{\mu}$ with $\Oo_{\lambda} = \Ind_\li^\h(\Oo)$.}
\end{enumerate}
If these two equivalent conditions hold then $\Oo_{\mu}$ has partition $\mu = \lambda^{(i)}$. Furthermore, for every other orbit $\widetilde{\Oo} \subseteq \mathcal{N}(\li)$ with $\widetilde{\Oo} = \Oo_0\times \Oo_{\tilde{\mu}}$ such that $\Oo_{\lambda} = \Ind_\li^\h(\widetilde{\Oo})$, we have $(\li, \Oo)/K = (\li, \widetilde{\Oo})/K$.
\end{cor}
\begin{proof}
Fix $\li$ as in the statement of the lemma. Let $\Oo$ be an abritrary nilpotent orbit of the form $\Oo_0 \times \Oo_{\mu} \subseteq \li$. The previous proposition implies that if $\lambda$ is the partition of $\Ind_\li^\h(\Oo_0 \times \Oo_{\mu})$ then $i$ is admissible for $\lambda$ and $\lambda^{(i)} = \mu$. Suppose (1) holds and let $\Oo_{\lambda^{(i)}}$ be an orbit in $\mm$ with partition $\lambda^{(i)}$. Then the partition of $\Ind_\li^\h(\Oo_0 \times \Oo_{\lambda^{(i)}})$ is $\lambda$. If $\lambda$ is not very even then (2) follows. If we are in type $\sf D$ and $\lambda$ is very even then according to the previous proposition either $\lambda^{(i)}$ is very even or $\li \cong \gl_i$ where $i = N/2 = \rank\, \h$ is even. In the first case the orbit $\Oo_0 \times \Oo_{\lambda^{(i)}}$ with the same label as $\Oo_{\lambda}$ induces to $\Oo_{\lambda}$ whilst in the second case there is a unique orbit of the correct form (the zero orbit) and since the labels of $\li$ and $\Oo_{\lambda}$ coincide ex hypothesis,
 this orbit induces to $\Oo_{\lambda}$.

Now suppose that (2) holds. By the remarks at the beginning of the previous paragraph, $\mu = \lambda^{(i)}$ and the index $i$ is admissible for $\lambda$. If there are two conjugacy classes of Levis then again $\li \cong \gl_i$ and so $\Oo_{\lambda} = \Ind_\li^\h(\Oo)$ implies that the labels of $\li$ and $\Oo_{\lambda}$ coincide by the last part of the previous proposition.

The statement that $\mu = \lambda^{(i)}$ is immediate from the above discussion. Fix $\Oo = \Oo_0\times\Oo_{\mu}$ fulfilling $\Oo_{\lambda} = \Ind_\li^\h(\Oo)$. We must show that for every other orbit of the form $\tilde{\Oo} = \Oo_0 \times\Oo_{\tilde{\mu}}$ fulfilling $\Oo_{\lambda} = \Ind_\li^\h(\tilde{\Oo})$ that the pair $(\li, \tilde{\Oo})$ is $K$-conjugate to $(\li, \Oo)$. Since we know that $\mu = \lambda^{(i)} = \tilde{\mu}$ this is now obvious unless $\mu$ is very even, $\lambda$ is not very even, and the orbits $\Oo_\mu$ and $\Oo_{\tilde{\mu}}$ have opposite labels. So suppose that this is the case. It is clear that in this situation $(i,i+1)$ is the only 2-step for $\lambda$, it is a good 2-step, and all other parts are even. It follows that $\rank\, \h = N/2$ is odd. Now from the Bourbaki tables \cite{Bour} we see that the longest element $w_0$ of the Weyl group for $\h$ is the minus the outer diagram automorphism of the root system of $\h$. Therefore if $gT = w_0 \in W = N_K(T)/T$ then $\Ad(g)$ 
will preserve $\li$ and exchange the orbits with partition $\lambda^{(i)}$ labelled $I$ and $I\!I$. This complete the proof.
\end{proof}

\p The following proposition uses a similar kind of induction as \cite[Proposition~24]{Mor} and is central to our proof of Theorem \ref{nosheets}.
\begin{prop}\label{mainpropref}
Let $\ii = (i_1,...,i_l)$ be a restricted sequence and suppose that the class of $\li$ lies in the fibre of $\pi$ above $\ii$, so that $\li \cong \gl_\ii \times \mm$. Then following are equivalent:
\begin{enumerate}

\item{$\ii$ is an admissible sequence for $\lambda$. If $\ii$ is very even then $\li$ belongs to the conjugacy class with the same label as $\Oo_{\lambda}$ (Cf. \ref{labelsonlevis});}

\item{There exists an orbit $\Oo = \Oo_0 \times \Oo_{\mu}$ with $\Oo_{\lambda} = \Ind_\li^\h(\Oo)$.}
\end{enumerate}
If these two equivalent conditions hold then $\Oo_{\mu}$ has partition $\mu = \lambda^{\ii}$. Furthermore, for every other orbit $\widetilde{\Oo} \subseteq \mathcal{N}(\li)$ with $\widetilde{\Oo} = \Oo_0\times \Oo_{\tilde{\mu}}$ such that $\Oo_{\lambda} = \Ind_\li^\h(\widetilde{\Oo})$, we have $(\li, \Oo)/K = (\li, \widetilde{\Oo})/K$.
\end{prop}

\begin{proof}
The proof proceeds by induction on $l$. When $l = 0$ we have $\li = \h$ and the proposition holds by Part 2 of Proposition \ref{inducedprops} (note that $\lambda^\emptyset = \lambda$). If $\li$ is a proper Levi subalgebra then $l > 0$. The case $l = 1$ is simply the previous corollary. The inductive step is quite similar although to begin with we must exclude the possibility that $R_\ii = 0$ and $i_l = 1$ in type $\sf D$. We will treat this possibility at the end.

Suppose that the proposition has been proven for all $l' < l$. Since we have excluded this anomalous case in type $\sf D$ we may set $\ii' = (i_1,...,i_{l-1})$ and and obtain another restricted sequence. Since $R_{\ii'} > 0$ there is a unique conjugacy class of Levi subalgebras above $\ii'$. Choose and element $\li'$ of this class so that $\li' \cong \gl_{\ii'} \times \mm'$ where $\mm'$ has a natural representation of dimension $R_{\ii'}$ and the same type as $\h$. Let $M'$ be the closed subgroup of $K$ with $\mm' = \Lie(M')$. We may ensure that $\li \subseteq \li'$ by embedding $\gl_{i_l} \times \mm$ in $\mm'$.

Suppose that $\ii$ is admissible and, if possible, that the label of $\li$ coincides with that of $\Oo_{\lambda}$. We deduce that $\ii'$ is also admissible. By the inductive hypothesis we deduce that there exists an orbit $\Oo' = \Oo_0 \times \Oo_{\tau} \subseteq \li'$ with $\Oo_{\lambda} = \Ind_{\li'}^\h(\Oo')$. We also see that $\Oo_\tau$ has partition $\tau = \lambda^{\ii'}\in \mathcal{P}_\epsilon(N- 2\sum_{j=1}^{l-1}i_j)$ and that $(\li', \Oo')/K$ is the unique conjugacy class of pairs containing $\li'$ inducing to $\Oo_{\lambda}$. Clearly $i_l$ is admissible for $\lambda^{\ii'}$ and $(i_l)$ is a restricted sequence for $\mm'$. Studying the labelling conventions for Levi subalgebras in \ref{labelsonlevis} we see that the label of the $K$-conjugacy class of $\li$ equals the label of the $M'$-conjugacy class of $\gl_{i_l} \times \mm \subseteq \mm'$. Therefore we can apply Corollary~\ref{case1or2} to conclude that there exists an orbit
$\Oo = \Oo_0 \times \Oo_{\mu} \subseteq \mathcal{N}(\gl_{i_l} \times \mm)$ with $\Oo_{\tau} = \Ind^{\mm'}_{\gl_{i_l} \times \mm}(\Oo)$. We make use of Proposition \ref{inducedprops} in the following calculation:
\begin{eqnarray*}
\Oo_{\lambda} = \Ind_{\li'}^\h(\Oo') &= &\Ind^\h_{\li'}(\Oo_0 \times\cdots\times\Oo_0\times \Oo_\tau)\\
 &=& \Ind^\h_{\li'}(\Oo_0 \times\cdots\times\Oo_0\times \Ind^{\mm'}_{\gl_{i_l} \times \mm}(\Oo_0\times \Oo_{\mu}))\\& = & \Ind^\h_{\li}(\Oo_0\times\cdots \times \Oo_0 \times \Oo_{\mu})
\end{eqnarray*}
We have shown that $(1) \Rightarrow (2)$. Before proving $(2) \Rightarrow (1)$ we shall take a quick detour to show that the final remarks in the statement of the proposition follow from (1). We certainly have $\mu = \tau^{(i_l)} = (\lambda^{\ii'})^{(i_l)} = \lambda^\ii$ by the transitivity of the algorithm. Suppose $\widetilde{\Oo} = \Oo_0\times\cdots \times\Oo_0\times\Oo_{\tilde{\mu}}$ is another orbit in $\li$ inducing to $\Oo_{\lambda}$. By the inductive hypothesis the partition of $\Ind^{\mm'}_{\gl_{i_l}\times \mm}(\Oo_0 \times \Oo_{\tilde{\mu}})$ is $\lambda^{\ii'}$ and so we get $\tilde{\mu} = \lambda^\ii = \mu$. The uniqueness assertion is therefore obvious unless $\mu$ is very even and $\lambda$ is not. In this case, reasoning in a similar manner to the third paragraph of the proof of Corollary~\ref{case1or2}, we can be sure that some term of the sequence $\ii$ is odd. After conjugating by some element of $K$ we can assume that $i_l$ is odd. The proof of uniqueness then concludes just as with the 
previous corollary, with $\gl_{i_l} \times \mm$ playing the role of our Levi and $\mm'$ playing the role of $\h$.

Now we must go the other way. Keep $\li$, $\li'$, $\mm'$, etc as above. Suppose that there exists an orbit $\Oo = \Oo_0\times\cdots\times\Oo_0 \times \Oo_{\mu}\subseteq \li$ with $\Oo_{\lambda} = \Ind_\li^\h(\Oo)$. Then we set $\Oo_{\tau} := \Ind_{\gl_{i_l} \times \mm}^{\mm'}(\Oo_0 \times \Oo_{\mu})$, $\Oo' := \Oo_0 \times\cdots \times\Oo_0\times \Oo_{\tau} \subseteq \li'$ and conclude that $\Oo_{\lambda} = \Ind^\h_{\li'}(\Oo')$ using a calculation very similar to the above one. Applying the inductive hypotheses we get that $\ii'$ is admissible for $\lambda$. There is no label associated to the conjugacy class of $\li'$ since $R_{\ii'} > 0$. Now Corollary~\ref{case1or2} tells us that $i_l$ is an admissible index for $\tau = \lambda^{\ii'}$ and so $\ii$ is admissible for $\lambda$. The same corollary tells us that if the $M'$-conjugacy class of the Levi $\gl_{i_l}\times \mm\subseteq \mm'$ has a label then it coincides with that of $\Oo_{\tau}$. The inductive hypothesis tells us that this label coincides with 
that of $\Oo_{\lambda}$.

Finally we must turn our attention to those sequences $\ii$ in type $\sf D$ for which $R_\ii = 0$ and $R_{\ii'} = 2$ ($\ii'$ remains to denote $\ii$ with the last term removed). In this case $\ii'$ is not restricted and so there does not exist a Levi subalgebra of the form $\gl_{\ii'}\times \mm$ and the induction falls down. In order to resolve this we define $\ii'' = (i_1,...,i_{l-2})$ and let $\li'' = \gl_{\ii''} \oplus \mm''$. Since $\li$ has the form $\gl_\ii$ we may embed $\gl_{i_{l-1}}\times \gl_{i_l}\subseteq \mm''$ to get $\li \subseteq \li''$. Since $i_l = 1$ there is a unique class of Levi subalgebras conjugate to $\li^\ii \cong \gl_\ii$. Furthermore, since the $\mm$ part is zero, there is only one orbit of the prescribed form in $\li$. We let $\Oo$ equal the zero orbit in $\li$. The proposition in this case is therefore reduced to the statement that $\ii$ is admissible if and only if $\Oo_\lambda = \Ind_\li^\h(\Oo_0)$.

Suppose $\ii$ is admissible for $\lambda$. Then so is $\ii''$ and by the inductive hypothesis there exists and orbit $\Oo'' = \Oo_0\times\cdots\times\Oo_0\times \Oo_{\tau}$ in $\li''$ with $\Oo_{\lambda} = \Ind^\h_{\li''}(\Oo'')$. Since $i_{l-1}$ is an admissible index for $\tau$ and $\tau^{(i_{l-1})}$ is $(1,1)$ we conclude that $\tau = (3,1)$ if $i_l = 1$, that $\tau = (3,3)$ if $i_l = 2$, that $\tau = (3,3,2^{i_{l-1}})$ if $i_{l-1}> 2$ is even or finally that $\tau = (3,3,2^{i_{l-1}-1},1,1)$ if $i_{l-1}>2$ is odd. None of these partitions are very even and so there is a unique orbit with partition $\tau$. According to \cite[Theorem~7.2.3]{CM} the induced orbit $\Ind_{\gl_{i_{l-1}}\times \gl_{i_l}}^{\gl_{i_{l-1}+1}}(\Oo_0\times\Oo_0)$ is the minimal nilpotent orbit in $\gl_{i_{l-1}+1}$, with partition $(2,1,...,1)$. Furthermore, if we induce this orbit into $\mm''$ then \cite[Lemma~7.3.3(i)]{CM} tells us that $\Ind_{\gl_{i_{l-1}+1}}^{\mm''}(\Oo_\text{min}) = \Oo_{\tau}$. Placing these ingredients together 
we get
\begin{eqnarray*}
\Ind^\h_{\li}(\Oo_0) = \Ind^\h_{\li''}(\Ind^{\li''}_{\li}(\Oo_0)) &=& \Ind^\h_{\li''}(\Oo_0\times \cdots \times\Oo_0 \times \Ind^{\mm''}_{\gl_{i_{l-1}}\times\gl_{i_l}}(\Oo_0 \times \Oo_0) \\
&=& \Ind^\h_{\li''}(\Oo'') = \Oo_{\lambda}
\end{eqnarray*}
as required. To go the other way we assume that such an orbit $\Oo$ exists and go backwards through the above deductions. We will conclude that $\tau$ has one of the prescribed forms so that $(i_{l-1}, 1)$ is an admissible sequence for $\lambda^{\ii''}$ and conclude that $\ii$ is admissible for $\lambda$.
\end{proof}

\p The next corollary will be useful in the proof of part 2 of Theorem~\ref{contsheets}
\begin{cor}\label{tau}
Let $\lambda\in\mathcal{P}_\epsilon(N)$. If $\emph\ii=(i_1,\ldots,i_l)$ is a restricted admissible sequence for $\lambda$ then so is $\sigma(\emph\ii)$ for every $\sigma \in \mathfrak{S}_l$. Furthermore,
$$\lambda^\ii = \lambda^{\sigma(\ii)}$$ for all such $\ii$ and $\sigma$.
\end{cor}
\begin{proof}
Since $\li^\ii$ is $K$-conjugate to $\li^{\sigma\ii}$ (Proposition~\ref{leviconjprop}), this follows from Proposition \ref{mainpropref}.
\end{proof}

\p We now define a map $\varphi$ from the set of all restricted admissible sequences for $\lambda$ to the set of all $K$-orbits of pairs $(\li, \Oo)$ where $\li \subseteq \h$ is a Levi subalgebra of $\h$ and $\Oo \subset \li$ is a nilpotent orbit. Let $\ii = (i_1,...,i_l)$ be a restricted admissible sequence for $\lambda$. Let $\li$ be a Levi subalgebra in a conjugacy class lying in $\pi^{-1}(\ii)$. If $|\pi^{-1}(\ii)| = 2$ then it is not hard to see that $\lambda$ is very even, and we request that the class of $\li$ has the same label as $\Oo_{\lambda}$. Let $\varphi(\ii) = (\li, \Oo)/K$ be the unique $K$-pair described in Proposition~\ref{mainpropref}. The next lemma is vital to our later work.
\begin{lem}\label{lenn}
If $\SS$ is a sheet with data $\varphi(\ii)$ then $\rank\, \SS = |\ii|$.
\end{lem}
\begin{proof}
According to \ref{levparam} every $\li \in \mathcal{C} \in \pi^{-1}(\ii)$ is isomorphic to $\gl_\ii \times \mm$ and so $\rank \, \SS = \dim\, \z(\li) = |\ii|$.
\end{proof}

\p The following corollary to \ref{mainpropref} will complete the proof of part 1 of Theorem~\ref{nosheets}. 
\begin{cor}\label{bijection}
 The restriction of $\varphi$ to the set of maximal admissible sequences for $\lambda$ descends to a well defined bijection from $\Phi_\lambda$ onto $\Psi_\lambda$. In particular $|\Phi_\lambda| = |\Psi_\lambda|$.
\end{cor}
\begin{proof}
First of all, note that every maximal admissible sequence is restricted. We shall show that $\varphi$ maps the set of maximal admissible sequences for $\lambda$ to $\Psi_\lambda$. Take $\ii$ maximal admissible and $\varphi(\ii) = (\li, \Oo)/K$ with $\Oo = \Oo_0\times\cdots\times\Oo_0\times \Oo_{\mu}$. By Proposition~\ref{mainpropref} we have $\mu= \lambda^\ii$ and so by Lemma \ref{maxadmiss} $\Oo_{\mu}$ is a rigid orbit. By part 1 of Proposition~\ref{inducedprops} the orbit $\Oo$ is also rigid. Furthermore we have that $\Oo_{\lambda} = \Ind_\li^\h(\Oo)$. Hence $\varphi(\ii) \in \Psi_\lambda$.

We now claim that the map is well defined on $\Phi_\lambda$, that is to say that $\varphi(\ii) = \varphi(\jj)$ whenever $\ii \sim \jj$. Let $\varphi(\ii) = (\li_1, \Oo_1)/K$ and $\varphi(\jj) = (\li_2, \Oo_2)/K$ where $\li_1 \cong \gl_\ii \times \mm_1$ and $\li_2 \cong \gl_\jj \times \mm_2$. Since $\ii = \sigma(\jj)$ for some $\sigma \in \mathfrak{S}_{|\ii|}$ and the labels of $\li_1$ and $\li_2$ are the same (if they exist), we conclude that they are $K$-conjugate by Proposition~\ref{leviconjprop}. Thus we may assume that $\li_1 = \li_2$. Now the uniqueness statement at the end of Proposition~\ref{mainpropref} asserts that $(\li_1, \Oo_1)/K = (\li_2, \Oo_2)/K$, so $\varphi$ is well-defined. For the rest of the proof $\varphi$ shall denote the induced map $\Phi_\lambda \ra \Psi_\lambda$.

Let us prove that $\varphi$ is surjective. Suppose $(\li, \Oo)/K \in \Psi_\lambda$ with $\li$ and $\Oo$ as in the definition of $\Psi_\lambda$. Then by Proposition~\ref{mainpropref} the sequence $\ii = (i_1,...,i_l)$ is admissible for $\lambda$ and by Lemma~\ref{maxadmiss} it is a maximal admissible. Therefore $\varphi(\ii) = (\li, \tilde{\Oo})/K$ for some orbit $\tilde{\Oo} = \Oo_0 \times \Oo_{\lambda^\ii}$. Since $\Oo = \Oo_0 \times \Oo_{\mu}$ by construction, the uniqueness statement in Proposition~\ref{mainpropref} tells us that $(\li, \Oo)/K = (\li, \tilde{\Oo})/K$. Hence $\varphi$ sends the equivalence class of $\ii$ in $\Phi_\lambda$ to $(\li, \Oo)/K$.

In order to prove the corollary we must show that $\varphi$ is injective. Suppose that $\ii$ and $\jj$ are maximal admissible for $\lambda$ and $\varphi(\ii) = \varphi(\jj)$. Again we make the notation $\varphi(\ii) = (\li_1, \Oo_{1})/K$ and $\varphi(\jj) = (\li_2, \Oo_{2})/K$. Since $\li_1$ and $\li_2$ are $K$-conjugate the sequences $\ii$ and $\jj$ are $\mathfrak{S}_l$-conjugate by Proposition~\ref{leviconjprop}. This completes the proof.
\end{proof}

\p Part 1 of Theorem \ref{nosheets} follows quickly from the above and Lemma~\ref{contsheets}.

\section{Counting admissible sequences}
\setcounter{parno}{0}

\p We now prepare to prove Part 2 of Theorem~\ref{nosheets}. Before we proceed we shall need two lemmas.  Define a function $\kappa : \PP(N) \longrightarrow (\Z_2)^\N$ by setting
$$\kappa(\lambda)_i := \lambda_i - \lambda_{i+1} \mod 2 \quad\text{ for all }\, i>0.$$
The reader should keep in mind here that $\lambda_i = 0$ for all $i > n$ by convention.
\begin{lem}\label{kappa}
Let $M,N \in \N$. If $\mu \in \mathcal{P}_\epsilon(M)^\ast$ and $\lambda \in \mathcal{P}_\epsilon(N)^\ast$ then $\mu = \lambda$ if and only if $\kappa(\mu) = \kappa(\lambda)$.
\end{lem}
\begin{proof}
Evidently $\mu = \lambda$ if and only if $\mu_i - \mu_{i+1} = \lambda_i - \lambda_{i+1}$ for all $i>0$. Since $\lambda$ is rigid
$\lambda_i - \lambda_{i+1} \in \{0,1\}$ by Theorem~\ref{rigids}. The lemma follows.
\end{proof}

\p The following proposition will render the proof of Theorem~\ref{contsheets} complete. We shall use the notation $\ii_l = (i_1,..., i_{l-1})$ for $1\leq l \leq |\ii| + 1$.
\begin{prop}\label{phising}
$|\Phi_\lambda| = 1$ if and only if $\lambda$ is non-singular.
\end{prop}
\begin{proof}
Suppose $\lambda$ is non-singular. We shall show that all maximal admissible sequences for $\lambda$ are conjugate under the action of the symmetric group. Let $\ii$ and $\jj$ be two such sequences. By Corollary~\ref{tau} we may put them both in ascending order, and still retain the fact that they are admissible sequences. It is not hard to see that they are still maximal after reordering. We shall show that they are now equal. Suppose not. Then either there exists an index $k$ such that $i_k \neq j_k$, or one sequence is shorter than the other, say $|\ii| < |\jj|$, and $i_k = j_k$ for all $k=1,...,|\ii|$. In this latter situation it is clear that $\ii$ is not maximal, so assume the former situation. We may assume without loss of generality that $i_k < j_k$. We shall prove that $\jj$ is not maximal and derive a contradiction.

If Case 1 occurs for $\lambda^{\ii_{k}}$ at index $i_k$ then it follows that $\lambda^{\ii_k}_{i_k} - \lambda^{\ii_k}_{i_{k}+1} > 1$. Since $\lambda$ is non-singular, so too is $\lambda^{\ii_k}$ by Corollary~\ref{singinherit}. This implies that $i_k + 1 \notin \Delta(\lambda^{\ii_k})$ and so $\lambda^{\jj_l}_{i_k} - \lambda^{\jj_l}_{i_k + 1} > 1$ for all $l \geq k$. In particular, $\lambda^{\jj}_{i_k} - \lambda^{\jj}_{i_k + 1} > 1$ which contradicts the maximality of $\jj$. Now we suppose Case 2 occurs for $\lambda^{\ii_{k}}$ at index $i_k$. This implies that $i_k \in \Delta(\lambda^{\jj})$, once again contradicting the maximality of $\jj$. We have now proven that if $\lambda$ is non-singular then all maximal admissible sequences are conjugate, and that $|\Phi_\lambda| = 1$.

In order to prove the converse we assume that $\lambda$ is singular. Let $(i, i+1)$ be a bad 2-step with $i$ maximal. We shall exhibit two maximal admissible sequences, $\ii$ and $\jj$, for $\lambda$ such that $\kappa(\lambda^{\ii}) \neq \kappa(\lambda^{\jj})$. In view of Lemma \ref{kappa} and Corollary~\ref{tau} the proposition shall follow. There are two possibilities: either $\lambda_{i+1} - \lambda_{i+2}$ is even, or $i > 1$ and $\lambda_{i-1} - \lambda_i$ is even. Assume the first of these possibilities, so that $\lambda_{i+1} - \lambda_{i+2}$ is even. Let $$\ii' = (\underbrace{i+1,i+1,...,i+1}_{(\lambda_{i+1} - \lambda_{i+2})/2 \text{ times}}).$$ We have $\lambda_{i+1}^{\ii'} = \lambda_{i+2}^{\ii'}$. Let $\ii$ be any maximal admissible sequence for $\lambda$ extending $\ii'$. Then $\kappa(\lambda^{\ii})_{i+1} = 0$. Now let $\jj' = (i)$ so that $\kappa(\lambda^{\jj'})_{i+1} = 1$. Let $\jj$ be any maximal admissible sequence extending $\jj'$. By Lemmas \ref{deltain2}(A) and \ref{deltain}(B), Case 2 does 
not occur for $\lambda^{\jj_k}$ at any index $j_k = i$ with $k > 1$ and since $(i, i+1)$ is a maximal bad 2-step Case 2 cannot occur at index $j_k = i+2$, so $\kappa(\lambda^{\jj})_{i+1} = \kappa(\lambda^{\jj'})_{i+1} = 1$. We conclude that $\kappa(\lambda^{\ii}) \neq \kappa(\lambda^{\jj})$, $\lambda^{\ii} \neq \lambda^{\jj}$ and hence
$|\Phi_\lambda| > 1$.

The other case is quite similar. This time we assume that $i > 1$, that $\lambda_{i-1} - \lambda_i$ is even and $\lambda_{i+1} - \lambda_{i+2}$ is odd. Our deductions will depend upon whether or not $i-2 \in \Delta(\lambda)$. Let us first assume that $i-2 \notin \Delta(\lambda)$. We take $$\ii' = (\underbrace{i-1,i-1,...,i-1}_{(\lambda_{i-1} - \lambda_{i})/2 \text{ times}}).$$ Let $\ii$ be any maximal admissible sequence extending $\ii'$. Much like before $\kappa(\lambda^\ii)_{i-1} = 0$. Now let $\jj' = (i)$ and let $\jj$ be a maximal admissible extension of $\jj'$. Since $i-2 \notin \Delta(\lambda)$ Lemma \ref{deltain}(B) shows that Case 2 does not occur for $\lambda^{\jj_k}$ at index $j_k = i-2$ for any $k$. The same can be said for $j_k = i$ at any index $k > 1$, so $\kappa(\lambda^\jj)_{i-1} = \kappa(\lambda)_{i-1} - 1 = 1$. It follows that $\lambda^{\ii} \neq \lambda^{\jj}$ and so $|\Phi_\lambda| > 1$ as desired.

To conclude the proof we must consider the final possibility: $i > 1$, $\lambda_{i-1} - \lambda_i$ even, $\lambda_{i+1} - \lambda_{i+2}$ odd and $i-2 \in \Delta(\lambda)$. We let $\ii'$ and $\ii$ be defined exactly as it was in the previous paragraph. We have $\lambda^{\ii'}_{i-1} = \lambda^{\ii'}_{i}$, so that Case 2 cannot occur at index $i_k = i$ for any $k$. Since $(i, i+1)$ is a maximal bad 2-step for $\lambda$ we know that $i+2 \notin \Delta(\lambda)$. Then Lemma \ref{deltain}(B) implies that Case 2 cannot occur at index $i_k = i+2$ for any $k$ yielding $\kappa(\lambda^{\ii})_{i+1} = \kappa(\lambda)_{i+1} = 1$. Let $$\jj' = (i, \underbrace{i+1,i+1,...,i+1}_{(\lambda_{i+1} - \lambda_{i+2})/2 \text{ times}})$$ and $\jj$ be any maximal admissible sequence extending $\jj'$. Since $\lambda_{i+1} - \lambda_{i+2}$ is odd, $\lambda^{\jj_2}_{i+1} - \lambda^{\jj_2}_{i+2}$ is even, and $\lambda^\jj_{i+1} = \lambda^\jj_{i+2}$. Hence $\kappa(\lambda^\jj)_{i+1} = 0$ and $|\Phi_\lambda| > 1$ as before.
\end{proof}
\p We can finally complete the proof of Theorem \ref{nosheets}.
\begin{proof}
Part 1 follows directly from Corollary \ref{bijection} and Lemma \ref{contsheets}. For Part 2 use Part 1 along with Proposition \ref{phising}.
\end{proof}

\section{The orthogonal complement to $[\h_x \h_x]$}
\setcounter{parno}{0}

\p \label{ffr}We shall finish the chapter by furnishing a proof of Izosimov's conjecture. Before we do so, however, we would like to
introduce some notation which will be used in the next chapter. Let $\mathcal{S}$ be a sheet with data $(\li, \Oo)/K$. Recall
that the rank of $\mathcal{S} \subseteq \h^{(k)}$ is defined to be $\dim \, \z(\li)$. Alternatively, by Corollary~\ref{dimformula}, we have
$$\rank \, \SS = \dim \, \SS - \dim\, \h + k.$$
\begin{cor}\label{newz}
If $\lambda \in \mathcal{P}_\epsilon(N)$ and $e \in \Oo$ which has partition $\lambda$, then
$$r(e):=\max_{\ e \in \mathcal{S}}\, \rank \, \mathcal{S} = z(\lambda)$$ where the maximum is taken over all sheets
of $\h$ containing $e$.
\end{cor}
\begin{proof}
Use Lemma~\ref{lenn}, Corollary~\ref{bijection} and and Theorem~\ref{zismax}.
\end{proof}

\p Suppose from henceforth that $\chr(\K)=0$. We shall set out to prove Izosimov's conjecture.
The following holds true for an arbitrary semisimple Lie algebra.
\begin{lem}
Suppose $x \in \h$ has Jordan decomposition $x = s + n$ and set $\li = \h_s$. There is a rank preserving bijection between the sheets
of $\h$ containing $x$ and the sheets of $\li$ containing $n$.
\end{lem}
\begin{proof}
This follows from the description of sheets and Jordan classes given in \ref{jordanclassif} and \ref{classifsheets}.
\end{proof}

\p We now relieve Corollary~\ref{izosim} of the assumption of nilpotency.
\begin{prop}
Let $x \in \h$ be arbitrary. The following are equivalent:
\begin{enumerate}
\item{
$x$ belongs to a unique sheet of $\h$;
}

\item{
the maximal rank of the sheets of $\h$ containing $x$ equals $\dim \, \h_x/[\h_x,\h_x]$;}

\item{
$x$ is a smooth point of the quasi-affine variety $\h^{(\dim\, \h_x)}$.
}

\end{enumerate}
\end{prop}
\begin{proof}
If $x=s+n$ is the Jordan decomposition of $x$ then $\li := \h_{s}$ is a Levi subalgebra of $\h$ (\ref{levistruct}) and
$\dim \, \h_x/[\h_x,\h_x] = \dim \, \li_{n}/[\li_{n},\li_{n}]$ (\ref{jordanreduction}). Now we may appeal to the previous lemma
to see that (1) is equivalent to (2). The equivalence to (3) follows by the same argument given in Corollary~\ref{izosim}.
\end{proof}

\p We may now give a proof of Izosimov's conjecture.
\begin{thm}
Let the characteristic of $\K$ be zero. Suppose $x \in \h$ lies in a unique sheet $\SS$. Then $$[\h_x \h_x]^\perp = T_x \SS$$ where orthogonality is taken with respect to the Killing form.
\end{thm}
\begin{proof}
In \cite{Izos} Izosimov shows that the conclusion of the theorem is equivalent to the condition that $\dim \,[\h_x \h_x] + \dim \, T_x \SS = \dim \, \h$.
Now suppose that $x$ lies in a unique sheet $\SS \subseteq \h^{(k)}$. Then by the previous proposition $\rank \, \SS = \dim \, \h_x - \dim \, [\h_x \h_x]$.
Combining with \ref{ffr} and using the smoothness of the sheets of $\h$, proven in \cite[Chapter~6]{IH}, we deduce the required equality.
\end{proof}

\chapter{Abelian Quotients of Finite $W$-algebras}\label{abelianquot}
\setcounter{parno}{0}

For the rest of this thesis we shall assume that $\K$ is an algebraically closed field of characteristic $0$. The results of this chapter
are summarised in the introduction but we shall recap the important details. A finite $W$-algebra is an associative algebra $U(\tg,e)$
constructed from the Lie algebra $\tg$ of a reductive group $\tG$, and a nilpotent element $e \in \mathcal{N}(\tg)$. The one dimensional representations
of $U(\tg, e)$ are parameterised by the maximal spectrum of the maximal abelian quotient $\EE(\tg,e) = \Specm \, U(\tg,e)^\ab$. We begin by
giving a criterion for $\EE(\tg,e)$ to be isomorphic to affine space $\mathbb{A}^d_\K$. We then apply this to classify the nilpotent elements $e$
in a classical algebra $\h$ for which $\EE(\h,e)$ is an affine space: they are precisely the non-singular nilpotent elements of $\h$, and $d = r(e)$ in this case. The
component group acts upon $\EE(\h,e)$ and we show that $\EE(\h,e)^\Gamma$ is always isomorphic to affine space $\mathbb{A}^d_\K$, and give its dimension.
Finally we apply Skryabin's equivalence and Losev's embedding to discuss the representation theory of $U(\h)$. We show that every multiplicity free primitive ideal whose associated variety
is the closure of an induced nilpotent orbit is induced from an appropriate completely prime primitive ideal with nice properties, generalising a theorem of M{\oe}glin and contributing towards
the long-standing problem of classifying completely prime primitive ideals of the enveloping algebra.

\section{Preliminary theory of $W$-algebras}\label{3.1}
\setcounter{parno}{0}

\p In this section we shall give a more detailed description of the maximal abelian quotients of finite $W$-algebras, and introduce the tools needed to
discuss them. The characteristic of $\K$ shall be zero from henceforth.

\p \label{firstt} Our first criterion for polynomiality shall not require our group to be classical and so we begin by refreshing the notation. Let $\tG$ be a connected 
reductive algebraic group over $\K$ and pick a non-zero element $e \in \Ni(\tg)$. The finite $W$-algebra $U(\tg ,e)$ is a non-commutative filtered algebra with associated graded algebra $S(\tg_e)$ (endowed
with a non-standard grading). Fix an $\sl_2$-triple $\phi = (e, h, f)$ containing $e$. By $\sl_2$ theory we have $\tg = [\tg, e] \oplus \tg_f$ and since 
$T_e \Oo_e = [\tg, e]$ we see that the \emph{Slodowy slice} $e + \tg_f$ is a transverse slice to $\Oo_e$ the $\tG$-orbit of $e$. Using the Killing form of $\tg$
we may identify $\K[e + \tg_f]$ with the symmetric algebra $S(\tg_e)$, and we view $U(\tg, e)$ as a deformation of the coordinate ring $\K[e + \tg_f]$.

\p \label{1dreps} Denote by $U(\tg,e)^\ab$ the largest commutative quotient of $U(\tg,e)$. We construct it explicitly by letting $I_c$ be the derived ideal of $U(\tg, e)$
generated by all commutators $x \cdot y - y \cdot x$ with $x, y \in U(\tg, e)$ and setting $U(\tg,e)^\ab = U(\tg,e)/I_c$. The one dimensional representations of $U(\tg, e)$ are in bijective correspondence
with those of $U(\tg, e)^\ab$, and these are parameterised by $\EE = \EE(\tg, e) = \Specm\, U(\tg, e)^\ab$. In \cite{Pre6} Premet showed that 
finite dimensional representations exist and in \cite[Conjecture~3.1]{Pre5} he went on to conjecture that every finite $W$-algebra
should possess a one dimensional representation. Since then they have been studied extensively, and considerable effort has been needed to prove their existence. In \cite[Theorem~1.1]{Pre4} a reduction to the case of rigid nilpotent elements was given and in \cite[Theorem~1.2.3]{Lo1} the conjecture was settled in classical cases. The work of \cite{Ub} was extended in \cite[Theorem~1.1]{GRU} to prove the existence of one dimensional representations for $W$-algebras associated to all but 3 rigid orbits in exceptional Lie algebras (all in type ${\sf E}_8$). In a forthcoming paper \cite{Pr} the remaining cases will be dealt with.
Let us record these results for later use.
\begin{thm}\label{EEnonempty}
If $\tG$ is reductive and $e \in \tg$ then the set $\EE(\tg, e)$ is non-empty.
\end{thm}

\p \label{quotientrecap} Let $\SS_1,...,\SS_l$ be the pairwise distinct sheets of $\tg$ containing $e$. In \cite{Kat}, Katsylo constructed
a geometric quotient for each $\SS_i$. The procedure is as follows. First of all set $X_i = \SS_i \cap (e + \tg_f)$ and  observe that $\tG_\phi$ acts naturally on $X_i$.
It can be shown that $\tG_\phi^\circ$ fixes $X_i$, and this induces an action of $\tG_\phi/\tG_\phi^\circ$. Recall that the component group $\Gamma = \Gamma(e) := \tG_e/\tG_e^\circ$
is isomorphic to $\tG_\phi / \tG_\phi^\circ$ (\ref{componentgroupintro}), hence $\Gamma$ acts on each $X_i$. Katsylo proved that each $\tG$-orbit in $\SS_i$ intersects $X_i$ in a finite set, and that this set is permuted transitively by $\Gamma$. This gives a map $\SS_i / G \rightarrow X_i / \Gamma$ which turns out to be a geometric quotient for $\SS_i$ (\cite[Lemma~6.3]{Kat}) and so $$r_i := \rank \, \SS_i = \dim\, \SS_i - \dim\, \Oo_e = \dim\, X_i/\Gamma = \dim\, X_i.$$

\p \label{sheetrecap} One of or main tools shall be Theorem~\ref{premettheorem}. Combining with the previous remarks it is immediate that for $e$ induced, the dimension of 
$U(\tg,e)^\ab$ equals $$r(e) := \max\{r_1,\ldots, r_t\}$$ and the number of irreducible components of $\EE$ is greater than or equal to the total number of all
irreducible components of the $X_i$'s. If $\EE$ is isomorphic to an affine space $\mathbb{A}^d_\K$ for some $d$ then this theorem ensures that $e$ lies in a unique sheet. We shall prove the
remarkable fact that when $\tg$ is classical, this condition is also sufficient to ensure that $\EE$ is an affine space. In order to do so we shall identify
a general criterion which is sufficient to imply the polynomiality of $U(\tg,e)^\ab$ when $e$ is induced.

\p \label{Thetamap} The discussion of this criterion shall rely heavily on the filtration of $U(\tg, e)$ mentioned in \ref{firstt}. If we take $U(\tg, e)$ to
be defined as per the construction of Gan and Ginzburg (\ref{ganginz}) then it is clear that the action of $\tG_\phi$ on $U(\tg)$ descends to an action on $U(\tg, e)$.
\begin{lem}\rm{\cite[Remark~2.1]{Pre5}}
There exists an injective $\tG_\phi$-module homomorphism $\Theta : \,\tg_e\ra U(\tg,e)$ with the property that $\Theta(\tg_e)$ generates $U(\tg,e)$ as an algebra.
\end{lem}

\p \label{slodowygrad}The action of $\ad(h)$ gives $\tg$ a $\Z$-graded Lie algebra structure $\tg=\bigoplus_{i\in\Z}\,\tg(i)$. We have that $e\in\tg(2)$ and $\tg_e=\bigoplus_{i \geq 0} \, \tg_e(i)$ 
where $\tg_e(i):=\tg_e\cap \tg(i)$. Let $x_1,\ldots, x_r$ be a basis for $\tg_e$ such that $x_i\in \tg_i(n_i)$ for some $n_i\ge 0$. The \emph{Slodowy grading} on $S(\tg_e)$ is defined
by letting each $x_i$ have degree $n_i + 2$. This grading arises naturally when studying the special transverse slice $e + \tg_f$ (see \cite[\S 7.4]{Slo}).

\p\label{kazhdanfilt} The current paragraph is contained in \cite[Theorem~4.6]{Pre2}. For $i = 1,...,r$ make the notation $\Theta_i := \Th(x_i) \in U(\tg, e)$.
The finite $W$-algebra $U(\tg,e)$ has a  Poincar\'{e}--Birkhoff--Witt basis consisting of monomials
$\Th^{\ii} := \Th_1^{i_1}\cdots \Th_r^{i_r}$ with $i_j \in \N_0$. We assign to $\Th^{\ii}$ filtration degree $$|\ii|_e := \sum_{j=1}^r i_j (n_j + 2)$$
The resulting filtration on $U(\tg, e)$ is called the {\it Kazhdan filtration}. The following relations hold in $U(\tg,e)$ for $1\leq i\leq j\leq r$:
\begin{eqnarray}\label{eqn1}
[\Th_i, \Th_j ] = \Th [x_i, x_j] + q_{ij}(\Th_1,...,\Th_r)+ \mbox{terms of lower Kazhdan deree}
\end{eqnarray}
where $q_{ij}$ is a polynomial of Kazhdan degree $n_i + n_j + 2$ whose constant and linear parts are both zero. We write ${\sf K}_l\,U(\tg,e)$ for of the
$l^{\rm th}$ component of the Kazhdan filtration of $U(\tg,e)$. In general, the Kazhdan degree $\deg_{\sf K}(x)$ of $x \in U(\tg,e)$ is the smallest $l \in \N_0$
with $x \in {\sf K}_l \, U(\tg,e)$ (we take ${\sf K}_{-1} U(\tg,e) = 0$ so that $\deg_{\sf K}(\K) = 0$).

\p Thanks to \cite[Proposition~6.3]{Pre2} there exists an isomorphism of graded algebras $\grk U(\tg, e) \overset{\sim}{\ra} S(\tg_e)$ where the latter algebra is
endowed with the Slodowy grading. Identifying $S(\tg_e)$ with $\K[e + \tg_f]$ we view $U(\tg, e)$ as a filtered deformation of the coordinate ring. This justifies the alternative
nomenclature for $U(\tg,e)$ as the enveloping algebra of the Slodowy slice, although this perspective will not feature in the current work.

\p The action of $\tG_\phi$ in \ref{Thetamap} preserves every ${\sf K}_l\,U(\tg,e)$. It also preserves the Slodowy grading and the usual grading on $S(\tg_e)$. We should not expect $\grk \Theta(\tg_e)$ to coincide with $\tg_e \subseteq S(\tg_e)$ as $\grk \Theta(x_i)$ will usually involve non-linear terms in $\Th_1,...,\Th_r$ of Kazhdan degree $n_i + 2$, however,
\cite[Lemma~4.5]{Pre2} tells us that the projection $S(\tg_e) \rightarrow \tg_e$ restricts to an isomorphism of graded $\tG_\phi$-modules $\grk\Theta(\tg_e) \rightarrow \tg_e \subseteq S(\tg_e)$. In what follows, we shall denote by $\grko$ the composition of $\grk$ and the projection onto linear terms in $S(\tg_e)$. For any filtered vector subspace $F \subseteq U(\tg, e)$ it is very common to abuse notation and write $\grk(F)$ for the associated graded subspace of $S(\tg_e)$ and to also write $\grk(v)$ for the top graded component of $v \in U(\tg,e)$. We shall proliferate this abuse by using $\grko$ in the same way. The reader should note that $\grko : U(\tg, e) \rightarrow \tg_e \subseteq S(\tg_e)$ is \emph{not} a linear map. Our earlier remarks may now be written as $\grko\Theta(\tg_e) = \tg_e$.

\p Since we are no longer in the classical case we shall need notation for $\codim_{\tg_e} \, [\tg_e \tg_e]$ which does not rely on the partition associated to $\Oo_e$.
Write $\mathfrak{c}_e = \tg_e^{\rm ab}=\tg_e/[\tg_e,\tg_e]$ and $c(e) := \dim\,  \mathfrak{c}_e$. If $\tg$ is classical and $\Oo_e$ has partition $\lambda$ then $c(e)$ coincides with $c(\lambda)$
defined in Corollary~\ref{expression}. Since $[\tg_e(0),\tg_e]\subset [\tg_e,\tg_e]$ and $\tg_e(0)=\Lie(\tG_\phi)$, it follows from Weyl's theorem that $\tG_\phi^\circ$ acts trivially on $\mathfrak{c}_e$. This
gives rise to a natural linear action of the component group $\Gamma$ on the vector space $\mathfrak{c}_e$. We denote by $\mathfrak{c}_e^\Gamma$ the fixed
point space of this action and set $c_\Gamma(e) :=\dim\,  \mathfrak{c}_e^\Gamma.$

\p Since the group $\tG_\phi$ operates on $U(\tg,e)$ by algebra automorphisms, it acts on the variety
$\mathcal E$ which identifies naturally with the set of all ideals of codimension $1$ in $U(\tg,e)$. By \cite[Lemma~2.5]{Pre5}
the differential of the action of $\tG_\phi$ on $U(\tg,e)$ is just $\ad \circ \Th$ and it follows that  $\tG_\phi^\circ$ preserves any
two-sided ideal of $U(\tg,e)$, and acts trivially on $\mathcal E$.
We thus obtain a natural action of $\Gamma=\tG_\phi/\tG_\phi^\circ$ on the affine variety
$\mathcal E$. We denote by $\mathcal{E}^\Gamma$  the corresponding fixed point set and let $I_\Gamma$ be the ideal of $U(\tg,e)^{\rm ab}$ generated by all $x-x^\gamma$ with $x\in U(\tg,e)^{\rm ab}$ and $\gamma\in\Gamma$. It is clear that $\mathcal{E}^\Gamma$
is contained in the zero locus of $I_\Gamma$. Conversely, if $\mathfrak{m} \in\mathcal{E}$ is such that $x(\mathfrak{m})=0$ for all $x\in I_\Gamma$, then
$\gamma(\mathfrak{m})=\mathfrak{m}$  for all $\gamma\in\Gamma$. Indeed, otherwise $\mm$ and $\gamma_0^{-1}(\mm)$ would be distinct maximal ideals of $U(\tg,e)^\ab$ for
some $\gamma_0\in \Gamma$ and we would be able to find an element
$x\in U(\tg,e)^{\rm ab}$ with $x(\mm)=0$ and $x(\gamma_0^{-1}(\mm))\ne 0$. But this would imply that $(x-x^{\gamma_0})(\mm)\ne 0$, a contradiction.  As a result, $\mathcal{E}^\Gamma$
coincides with the zero locus of $I_\Gamma$ in $\mathcal E$. We denote by $U(\tg,e)_\Gamma^{\rm ab}$ the commutative
$\K$-algebra $U(\tg,e)^{\rm ab}/I_\Gamma$. The above discussion shows that
$$\mathcal{E}^\Gamma\,=\,{\rm Specm}\,U(\tg,e)_\Gamma^{\rm ab}.$$

\section{The polynomiality of quotients}\label{polynomialityofquot}
\setcounter{parno}{0}

\p Continue to assume $\chr(\K) = 0$, and now take $\tG$ to be reductive and connected. The goal of this subsection is to give a sufficient condition for the polynomiality
of $U(\tg,e)^\ab$ and $U(\tg,e)_\Gamma^{\rm ab}$ and use it to classify those nilpotent elements in the Lie algebras
of classical groups for which $U(\tg,e)^\ab$ is a polynomial algebra. For $k\in\N_0$, we continue to use $S^k(\tg_e)$ to denote the $k^\th$ part of the usual grading
where $x^\ii = x_1^{i_1}\cdots x_r^{i_r}$ has degree $|\ii| := \sum_{j = 1}^r i_j$. 
\begin{lem}\label{generate}
Let $I$ be a proper two-sided ideal of $U(\tg,e)$ and let $V_I$ and $V_I'$ be two $\tG_\phi$-submodules
of $\Theta(\tg_e)$, such that $\tg_e = S^1(\tg_e) =\grko(V_I)\oplus\grko(V_I')$ as graded $\Ad(\tG_\phi)$-modules. Suppose further that
$$\grko(V_I)\subseteq \grko(I)$$
Then the unital algebra $U(\tg,e)/I$ is generated by the subspace $V'_I$.
\end{lem}

\begin{proof} Throughout the proof we identify $\tg_e$ with $S^1(\tg_e)$. Since $\tg_e =\grko(V_I)\oplus\grko(V_I')$, it follows that $\Theta(\tg_e) = V_I\oplus V_I'$.
Since $V_I$ and $V_I'$ are graded submodules of $\tg_e$ we can
choose a basis $x_1,...,x_r$ of $\tg_e$ such that each $x_i$ has Slodowy degree $n_i + 2$, where $x_1,...,x_{q-1}$ spans $V_I$ and $x_{q},...,x_r$ spans $V_I'$.
Let $\Th_i  = \Th(x_i)$ and denote by $\mathcal{A}$ the $\K$-span in $U(\tg,e)$ of all $\Th_{q}^{i_{q}}\cdots\Th_r^{i_r}$ with $i_j \in\N_0$. We claim that every
monomial $\Theta^\ii \in U(\tg,e)$ with $|\ii|_e=k$ lies in $\mathcal{A}+I$. The proof is by induction on $k$.

The statement is obviously true for $k = 0$, since ${\sf K}_0 U(\tg,e) = \K$. Suppose that it holds for all $\Th^{\ii'}$ with $|\ii'| < k$.
Notice also that if the claim holds for $|\ii| = 1$ then it holds for all $|\ii| > 1$ by a simple inductive argument. Hence we may assume that $|\ii|=1$, i.e.
$\Theta^{\ii}=\Th_s$ for some $s\in\{1,..., r\}$ and $k = n_s + 2$. The claim is trivially true when $q \leq s \leq r$ so assume $1\leq s < q$. Lemma 4.5 of \cite{Pre2} tells us that $x_s = \grko \Th_s \in \grko(I)$. Let $u \in I$ be such that $\grko u = x_s$ and $\deg_{\sf K}(I) = n_s +2 $. It follows that
\begin{eqnarray*}\label{eqn2}
\Th_s - u =  \sum_{|\ii|_e = n_s + 2,\ |\ii|\ge 2} \lambda_{s, \ii} \Th^\ii + \mbox{ terms of lower Kazhdan degree}
\end{eqnarray*}
for some constants $\lambda_{s,\ii}\in\K$. Therefore,
$$\Th_s\,\equiv \sum_{|\ii|_e = n_s + 2,\ |\ii|\geq 2} \lambda_{s,\,{\bf i}}\,\Th^{\bf i}\mod (\mathcal{A}+I)$$
by the inductive hypothesis. Take a term $\Theta^\ii$ in the above sum. Since $|\ii| > 1$ we can define $\ii' \in \N_0^r$ to have a 1 in the $j^\th$ position and 0 elsewhere, where $j$ is the minimal index such that $i_j \neq 0$, and define $\ii'' = \ii - \ii'$. We have $|\ii'|, |\ii''| >0$ and $\Theta^\ii$ can be written as $\Th^{\ii'} \Th^{\ii''}$. But $|\ii'|_e$ and $|\ii''|_e$ are
strictly less than $|\ii|_e = k$. By the inductive hypothesis $\Th^{\ii'}$ and $\Th^{\ii''}$ both lie in $\mathcal{A} + I$, and so too does $\Th^\ii$.
Therefore $\Th_s \in  \mathcal{A} + I$ as well. This completes the inductive step and the lemma follows.
\end{proof}

\p It should be stressed at this point that in Lemma \ref{generate} we do not require $V_I$  to be contained in $I$.
\begin{prop}\label{abgen}
Let $e$ be any nilpotent element of $\tg$. Then the following are true:
\begin{enumerate}
\item{
If $\EE\ne \emptyset$, then the unital algebra $U(\tg,e)^\ab$ is generated by
a $\tG_\phi$-submodule of $\Th(\tg_e)$ isomorphic to $\mathfrak{c}_e$. In particular, it is generated by $c(e)$ elements.}

\item{
If $\mathcal{E}^\Gamma\ne \emptyset$, then the unital algebra $U(\tg,e)^\ab_\Gamma$ is generated by a $\tG_\phi$-submodule
of $\Th(\tg_e)$ isomorphic to $\mathfrak{c}_e^\Gamma$. In particular it is generated by $c_\Gamma(e)$ elements.}

\end{enumerate}
\end{prop}
\begin{proof}

(i) Retain the notations of \ref{kazhdanfilt}. The defining ideal $I_c$ of $U(\tg,e)^{\rm ab}$
contains all commutators $[\Theta_i,\Theta_j]$ with $1\leq i, j\leq r$. According to (\ref{eqn1}) the linear part of $\grk [\Theta_i, \Theta_j]$ is
$\grko [\Theta_i, \Theta_j] = [x_i, x_j]$. Setting $V_{I_c} = \Th([\tg_e, \tg_e])$ we have $\grko V_{I_c} = [\tg_e \tg_e]$. It follows
that $\grko(V_{I_c}) \subseteq \grko(I)$. Since $\tG_\phi$ is a reductive group, $\tg_e$ contains
a graded $\Ad(\tG_\phi)$-submodule $M$ of dimension $c(e)$ complementary to the derived subalgebra $[\tg_e,\tg_e]$. If we set
$V_{I_c}' = \Th(M)$ then the pair $(V_{I_c}, V_{I_c}')$ fulfil the assumptions of the previous lemma and we may conclude that $V_{I_c}'$ generates $U(\tg,e)^\ab$.

\smallskip

\noindent
(ii) Let $\widetilde{I}_c$ be the preimage of the ideal $I_\Gamma$ of $U(\tg,e)^{\rm ab}$ under the canonical homomorphism $U(\tg,e)\twoheadrightarrow
U(\tg,e)^{\rm ab}$. Then $\widetilde{I}_c$ is a
two-sided ideal of $U(\tg,e)$ and $U(\tg,e)/\widetilde{I}_c\cong U(\tg,e)^{\rm ab}_\Gamma$ as algebras. Since $[\tg_e(0),M]\subseteq [\tg_e,\tg_e]$ and
$M\cap [\tg_e,\tg_e]=0$, it follows from Weyl's theorem that the connected reductive group $\tG_\phi^\circ$ acts trivially on $M$. Therefore, $M$ has a
natural structure of a $\Gamma$-module. There exists a $\Gamma$-submodule $M'$ of $M$ complementary to $M^\Gamma:=\{x\in M\colon\, \gamma(x)=x\}$.
We choose $V_{\widetilde{I}_c} = \Th(M' \oplus [\tg_e \tg_e])$ and $V'_{\widetilde{I}_c} = \Th(M^\Gamma)$. In order to apply Lemma~\ref{generate} we
need to show that $\grko(V_{\widetilde{I}_c}) \subseteq \grko(\widetilde{I}_c)$. Since $$\grko[\tg_e \tg_e] \subseteq \grko(I_c) \subseteq \grko(\widetilde{I}_c)$$
we must show that $M' \subseteq \grko(\widetilde{I}_c)$. But $M'$ is spanned by elements of the form $x-\gamma(x)$ with $x \in M$. Now for every $x \in \tg_e = S^1(\tg_e)$
the image of $$\Th(x) - \Th(\gamma(x)) = \Th(x) - \gamma(\Th(x))$$ in $U(\tg,e)^\ab$ lies in $I_\Gamma$. Therefore $\Th(x) - \Th(\gamma(x)) \in \widetilde{I}_c$ and $x - \gamma(x) \in \grko(\widetilde{I}_c)$.
Finally we have obtained $\grko(M') \subseteq \grko(\widetilde{I}_c)$. We apply Lemma~\ref{generate} to complete the proof.
\end{proof}
 
\p We now record our criterion for polynomiality. Recall that $r(e)$ denote the maximal rank of the sheets of $\tg$ containing $e$.

\begin{cor}\label{poly} Let $e$ be an induced nilpotent element of $\tg$. Then the following hold:
\begin{enumerate}
\item{If $c(e) = r(e)$, then $U(\tg,e)^\ab \cong S(\mathfrak{c}_e) \text{ and } U(\tg,e)^\ab_\Gamma\cong S(\mathfrak{c}_e^\Gamma)$
as $\K$-algebras and as $\Gamma$-modules.}

\item{ If $\mathcal{E}^\Gamma\ne\emptyset$ and $\dim\,  \mathcal{E}^\Gamma\geq c_\Gamma(e)$, then
$U(\tg,e)^{\rm ab}_\Gamma\,\cong \, S(\mathfrak{c}_e^\Gamma)$ as $\K$-algebras and $\Gamma$-modules.}
\end{enumerate}
\end{cor}
\begin{proof}
(i) By \ref{sheetrecap} we have $\dim\,  U(\tg,e)^{\rm ab} = r(e)$. On the other hand part 1 of Propostion \ref{abgen} implies that there exists a surjective $\tG_\phi$-equivariant algebra homomorphism
$S(\mathfrak{c}_e)\twoheadrightarrow U(\tg,e)^\ab$. Since $c(e) =\dim\,  S(\mathfrak{c}_e)$ equals $r(e) = U(\tg,e)^\ab$, the map $\psi$ must be injective. Since 
$\tG_\phi^\circ$ acts trivially on $\mathfrak{c}_e$ we deduce that $U(\tg,e)^\ab\cong S(\mathfrak{c}_e)$ as $\K$-algebras and $\Gamma$-modules. But then
$\mathcal{E}\cong \mathfrak{c}_e^*$ as $\Gamma$-varieties implying that $\mathcal{E}^\Gamma\cong(\mathfrak{c}_e^*)^\Gamma$. Since the defining ideal in $S(\mathfrak{c}_e)\cong\K[\mathfrak{c}_e^*]$ of the linear subspace $(\mathfrak{c}_e^*)^\Gamma$ is generated by all $f-f^\gamma$ with $f\in S(\mathfrak{c}_e)$ and $\gamma\in\Gamma$, its image under the surjection onto $U(\tg,e)^\ab$ coincides with $I_\Gamma$.
This implies that $S(\mathfrak{c}_e^\Gamma)\cong U(\tg,e)^\ab_\Gamma$ as $\K$-algebras. They are both trivial as $\Gamma$-modules.

\noindent
(ii) As $\mathcal{E}^\Gamma\ne\emptyset$, it follows from Proposition~\ref{abgen}(ii) that there is a surjective algebra homomorphism
$S(\mathfrak{c}_e^\Gamma)\twoheadrightarrow U(\tg,e)^{\rm ab}_\Gamma$.
As a consequence, $c_\Gamma(e)\geq \dim\,  U(\tg,e)^\ab_\Gamma.$ If
 $\dim\,  U(\tg,e)_\Gamma^\ab = \dim\, \mathcal{E}^\Gamma\geq c_\Gamma(e)$, then it must be that $U(\tg,e)^\ab_\Gamma \cong S(\mathfrak{c}_e^\Gamma)$ as $\K$-algebras.
 Once again they are trivially isomorphic as $\Gamma$-modules.
\end{proof}

\section{Computing the varieties $\EE$ and $\EE^\Gamma$}\label{computingvarieties}
\setcounter{parno}{0}

\p In this subsection we are going to apply our results on non-singular nilpotent elements from Chapter~\ref{derivedchapter}
to give a complete description of those nilpotent elements $e$ in classical Lie algebras for which $\EE$ is an affine space. 
Recall that the situation in type $\sf A$ was dealt with by Premet in \cite{Pre4}. He showed that the space $\EE(\sl_n, e)$ is
always isomorphic to affine space of dimension $c(e)$.

\p We shall revert here to the notations designated in \ref{classnots} which we have used to describe classical algebras throughout.
In particular, $K$ shall be the connected simple algebraic group
of type $\sf B$, $\sf C$ or $\sf D$ preserving a form $(\cdot, \cdot) : V \times V \ra \K$ where $V = \K^N$ with $(u, v) = \epsilon(v, u)$. The set of partitions associated to
the nilpotent orbits $\Ni(\h)/K$ is denoted $\mathcal{P}_\epsilon(N)$ and once $e$ is chosen, $\phi = (e,h,f)$ is an $\sl_2$-triple of $\h$ containing $e$.
 
\begin{thm}\label{class}
Let $e \in \Ni(\h)$. Then the following are equivalent:
\begin{enumerate}
\item{
$e$ belongs to a unique sheet of $\h$;}
\item{$U(\h,e)^\ab$ is isomorphic to a polynomial algebra.
}
\end{enumerate}
If these equivalent criteria hold then $U(\h,e)^\ab$ is generated by $c(e)$ variables.
\end{thm}
\begin{proof}
If $e$ belongs to a unique sheet then then $c(e) = r(e)$ by Corollary \ref{izosim} and Theorem~\ref{zismax}. Thanks to \ref{sheetrecap} we know that $\dim(U(\h,e)^\ab) = c(e)$. Comparing with Proposition~\ref{abgen} we see that $U(\h,e)^\ab$ is generated by $c(e)$ elements. It follows that $U(\h,e)^\ab$ is a polynomial algebra on $c(e)$ generators.

If $U(\h,e)^\ab$ is isomorphic to a polynomial algebra, then the variety $\mathcal{E}$ is irreducible. If $e$ is induced, then after the discussion in \ref{sheetrecap}
we see that $e$ belongs to a unique sheet. If $e$ is rigid, this holds automatically as $\Ad(K)e$ is a sheet of $\h$. This completes the proof.
\end{proof}
\p As an immediate consequence we obtain:
\begin{cor}\label{uniquerigids}
$U(\h,e)$ has a unique one dimensional representation if and only if $e$ is rigid.
\end{cor}
\begin{proof}
$U(\h,e)$ has a unique one dimensional representation if and only if $\EE(\h, e)$ is a point, which is if and only if $U(\h,e)^\ab = \K$. By the previous theorem this is equivalent to
$e$ being non-singular and $c(e) = 0$. Given that every rigid partition is non-singular, this is equivalent to $e$ being rigid by Corollary~\ref{rigidcod}.
\end{proof}

\p In the midst of our discussion of finite $W$-algebras we are able to retrieve some information of a geometric nature on the Katsylo sections of $e$.
\begin{cor}
If $e$ lies in a unique sheet $\SS$ of $\h$ then $X:= (e + \h_f) \cap \SS$ is a smooth, irreducible variety.
\end{cor}
\begin{proof}
By Theorem~\ref{class} the algebra $U(\h,e)^\ab$ is polynomial and so $\EE$ is irreducible. In view of Threorem~\ref{premettheorem} the variety $X$ is
irreducible. Since $\SS$ contains both $X$ and $\Oo_e = \Ad(K)e$ the tangent space $T_e\SS$ contains $T_eX + T_e\Oo_e$, however
$T_e X\subset T_e(e+\h_f)=\h_f$, whilst $T_e\Oo_e = [e,\h]$, so that that $T_e(\mathcal{S})$ contains $T_e X\oplus [e,\h]$ by $\sl_2$-theory. Thanks to \cite{IH}
the variety $\SS$ is smooth and by \cite{Kat} we have $\dim\,  \SS = \dim\,  X + \dim\,  \Oo_e$. Therefore $T_e \SS = T_e X \oplus [e,\h]$ and $\dim \, T_e X=\dim\, X$. As a consequence,
$e$ is a smooth point of $X$. But there exists a $\K^\times$-action on $X$, contracting to $e$, induced by letting $t \in \K^\times$ act on $\h(i)$ by $t^{i-2}$ where
$\h = \oplus_{i\in \Z}\h(i)$ is the grading induced by $\ad(h)$ (see \cite{Slo}). It follows that $X$ is smooth, as required.
\end{proof}

\p Our next result relies heavily on Losev's Theorem which we outlined in \ref{losevtheorem}. Together with Corollary~\ref{poly}(ii) this will enable us to describe the variety $\mathcal{E}(\h,e)^\Gamma$
for all nilpotent elements $e\in \h$, even the singular ones which cannot be treated by Theorem~\ref{poly}(i). Strictly speaking the following result works in the full generality of reductive
groups, however we shall not interest ourselves in such possibilities anymore.
\begin{prop}\label{mult1}
Let $P$ be a proper parabolic subgroup of $K$ with Levi decomposition $L\ltimes U$ and corresponding decoposition of Lie algebras $\pp = \li \ltimes \mathfrak{u}$. Suppose that a nilpotent element $e = e_0+e_1\in\pp$ with $e_0\in \li$ and  $e_1\in\mathfrak{u}$ is induced from $e_0$ and that $$K_e\subset P.$$ Let $\mathcal{E}_0=\EE([\li \li], e_0)$ and suppose further that $\EE_0^{\Gamma_0}\neq \emptyset$ where $\Gamma_0=L_{e_0}/(L_{e_0})^\circ$. Then
$\EE^\Gamma\ne\emptyset$ and  $$\dim\, \EE^\Gamma\ge\dim\, \z(\li).$$
\end{prop}
\begin{proof}
(a) Let $\phi_0 = (e_0,h_0,f_0)$ be an $\mathfrak{sl}_2$-triple of $\li$ containing $e_0$ (if $e_0 = 0$ then $(e_0,h_0,f_0)$ is the zero triple).
By $\mathfrak{sl}_2$-theory, the reductive group $L_{\phi_0}$ is a Levi subgroup of the centraliser $L_{e_0}$. The $\mathfrak{sl}_2$-triple for $e$
is still denoted $\phi = (e,h,f)$ and let $\lambda_e$ be the cocharacter in $X_\ast(K)$ with $h\in \Lie(\lambda_e(\K^\times))$. Note that $K_\phi = K_e\cap K_h = K_e\cap Z_K(\lambda_e(\K^\times))$.
Since $K_\phi\subset P$ by our assumption on $e$, it follows from \cite[Proposition~6.1.2(4)]{Lo3} that the reductive group $\lambda_e(\K^\times)K_\phi$ is contained in $P$. Since any 
reductive subgroup of $P$ is conjugate under $P$ to a subgroup of $L$ by Mostow's theorem \cite{Mos},
we may assume without loss of generality that $ \lambda_e(\K^\times)K_\phi \subseteq L$.
Since $K_\phi \subseteq L$ fixes $e$ and preserves both $\li$ and $\mathfrak{u}$,
it must be that $K_\phi$ fixes $e_0$ too, so $K_\phi\subseteq L_{e_0}$. Since the group $K_\phi$ is reductive, it follows again from \cite{Mos}
that it is conjugate under $L$ to a subgroup of $L_{\phi_0}$. We shall assume for the rest of the proof that, in fact, $$K_\phi \subseteq L_{\phi_0}.$$

\medskip

\noindent (b) Recall from \ref{losevtheorem} that Losev defines a completion $U(\li,e_0)'$ of the finite
$W$-algebra $U(\li,e_0)$ and an injective algebra homomorphism $\Xi\colon\,U(\h,e)\rightarrow U(\li,e_0)'$.
By construction, the reductive group $L_{\phi_0}$ acts on $U(\li,e_0)'$ by algebra automorphisms. One can see by
inspection that all maps involved in Losev's construction are $L_{\phi_0}$-equivariant (a related discussion
can also be found in \cite[2.5]{Lo5}). This implies, in particular, that in our situation Losev's embedding
is $K_\phi$-equivariant. Furthermore, it follows from \cite[Proposition~6.5.1]{Lo3} that the isomorphism
$U(\li, e_0)^\ab \cong U(\li, e_0)'^\ab$ is $L_{\phi_0}$-equivariant. We thus obtain a $K_\phi$-equivariant
algebra homomorphism $$\xi_\li\colon\,U(\tg,e)^\ab\rightarrow U(\li,e_0)^\ab.$$

\noindent (c) Let $\widetilde{\EE}_0=\Specm \, U(\li,e_0)^\ab$. According to Theorem~\ref{losevtheorem} the morphism of affine varieties $\xi^\ast_\li :  \,\widetilde{\EE}_0 \rightarrow\EE$
associated with $\xi_\li$ is finite. In particular, it has finite fibres. Since $\xi_\li$ is
$K_\phi$-equivariant and $K_\phi^\circ$ acts trivially on both $U(\h,e)^\ab$ and $U(\li,e)^\ab$, the morphism $\xi^\ast_\li$
maps $\widetilde{\EE}_0^\Gamma$ into $\mathcal{E}^\Gamma$. It follows that
$$\dim\,  \widetilde{\mathcal{E}}_0^\Gamma = \dim\, \xi^*(\widetilde{\mathcal{E}}_0^\Gamma)\leq \dim\,  \mathcal{E}^\Gamma.$$

\smallskip

\noindent (d) Write $\z$ for the centre $\z(\li)$ of $\li$. Clearly, $\z$ is a toral subalgebra of $\h$ and $\li=\z\oplus [\li,\li]$. It follows that $U(\li,e_0)\cong S(\z)\otimes U([\li,\li],e_0)$. This, implies that $U(\li,e)^\ab\cong S(\z)\otimes U([\li,\li],e_0)^\ab$ as algebras.  Since the subalgebra $U([\li,\li],e_0)$ of $U(\li, e_0)$
is stable under the action of $K_\phi$ on $U(\li,e_0)$, we have a natural action of
$\Gamma$ of the affine variety $\mathcal{E}_0:=\Specm\, U([\li,\li],e_0)^\ab$. Since $K_\phi\subseteq L_{\phi_0}$, and $L_{\phi_0}^\circ$ acts trivially on $\mathcal{E}_0$, the variety  $\mathcal{E}_0^\Gamma$ contains $\EE_0^{\Gamma_0}$ and hence is non-empty by
our assumption on $\EE_0^{\Gamma_0}$. 

\noindent (e) Note that $L$ acts trivially on $\z$ and hence so does $K_\phi\subset L$. It follows that $\widetilde{\EE}_0^\Gamma \cong \z^*\times
\mathcal{E}_0^\Gamma$ as affine varieties. In particular, $$\dim\,  \widetilde{\mathcal{E}}_0^\Gamma
=\dim\,\mathcal{E}_0^\Gamma+\dim\,\z\ge \dim\,\z.$$
But then $\mathcal{E}^\Gamma \supseteq \xi^*(\widetilde{\mathcal{E}}^\Gamma)\neq \emptyset$ and $\dim\,\mathcal{E}^\Gamma\ge\dim\,\widetilde{\mathcal{E}}_0^\Gamma\ge \dim\,\z(\li)$ as claimed.
\end{proof}

\p\label{gammaexplicit} Pick a nilpotent element $e \in \h$ with partition $\lambda$. We shall now describe the component
group $\Gamma$ by giving a canonical set of representatives in $K_e$. The group $K$ is
the connected component of the stabiliser of a non-degenerate form $(\cdot,\cdot)$. We shall denote the full stabiliser
by $\tK$. This group coincides with $K$ if $\epsilon = -1$ but if $\epsilon = 1$ then $\tK = O(V)$ is the full orthogonal group.
Observe that $K = SL(V) \cap \tK$ in either case. Let $\overline{I} =\{1\le i\le n\colon i = i', \lambda_i> \lambda_{i+1}\}$ and set $\nu(\lambda):=|\overline{I}|$. Note that $\nu(\lambda)$ is the
number of distinct values $\lambda_i$ for which $i = i'$. For $i\in \overline{I}$ we let $g_i$ denote the involution of $V$ lying in $\tK_e$ such
that $$g_i(e^sw_j)=(-1)^{\delta_{i,j}}(e^sw_j)$$ for all $1\le j\le n$ and $0\leq s < \lambda_j$. Here
we use Kronecker delta notation. Then $g_i$ acts as $-1$ on $V[i]$ and by $1$ on each $V[j]$ with $j \neq i$.
Define $\widetilde{\Gamma}:= \langle g_i\colon\,i\in I\rangle$, a subgroup of $\tK$.
\begin{lem}
The natural map $\tK_e \twoheadrightarrow \tK_e/\tK_e^\circ$ restricts to an isomorphism from $\widetilde{\Gamma}$ onto $\tK_e / \tK_e^\circ$.
If $K$ has type $\sf C$ then $\tK_e/\tK_e^\circ = \Gamma$ whilst if $K$ has type $\sf B$ or $\sf D$ then $\Gamma$ is generated by the images of
products $\{g_i g_j : i, j \in \overline{I}\}$ in $\tK_e/\tK_e^\circ$.
\end{lem}
\begin{proof}
As the involutions $g_i$ pairwise commute, $\widetilde{\Gamma}$ is an elementary abelian $2$-group of order $2^{\nu(\lambda)}$.
We claim that the restriction of $\tK_e \twoheadrightarrow \tK_e/\tK_e^\circ$ to $\widetilde{\Gamma}$ is injective. We must show that
the elements $\{g_i : i \in \overline{I}\}$ lie in distinct connected components of $K_e$. This follows from the description of $\tK_\phi$
given in \cite[3.8]{Jan}. Using \cite[3.13]{Jan} it is now evident that the projection $\tK_e \twoheadrightarrow \tK_e/\tK_e^\circ$ restricts to an isomorphism
$\widetilde{\Gamma} \overset{\sim}{\longrightarrow} \tK_e/\tK_e^\circ$ as stated. Since $K = \tK \cap SL(V)$ we see that the restriction of
the map $K_e \twoheadrightarrow \Gamma$ to $\widetilde{\Gamma} \cap SL(V)$ is an isomorphism. For the rest of the proof identify
$\tK_e/\tK_e^\circ$ with $\widetilde{\Gamma}$ .

Suppose $K$ is of type $\sf C$. Then $K = \tK$ and so $\Gamma = \widetilde{\Gamma}$. In case the type is $\sf B$ or $\sf D$ the dimension $\dim\,  V[i]$ is odd
for all $i = i'$ so $\det(g_i) = -1$ for all $i \in \overline{I}$. Therefore the $g_i$ do not lie in $K_e$, however the products $g_i g_j$ with $i, j \in \overline{I}$ do, and it is clear
that they generate all of $\Gamma$.
\end{proof}

\p We call a partition $\lambda \in \mathcal{P}_1(N)$ {\it exceptional} if there exists a unique $i \in \Delta(\lambda)$ such that all parts $\lambda_j$ with $j\not\in \{i, i+1\}$ are even. It is clear that if $\lambda$ is exceptional then the unique 2-step is good. This shows that exceptional partitions are non-singular. Using the KS algorithm it is straightforward to see that every orbit with an exceptional partition is actually Richardson (i.e. is induced from the zero orbit). Define
$$\overline{s}(\lambda) = \left\{ \begin{array}{cc}
s(\lambda) & \lambda \text{ not exceptional}\\
s(\lambda)  +1 & \lambda \text{ exceptional}
\end{array} \right.$$
\smallskip
\p Before we describe the varieties $\EE(\h, e)^\Gamma$ we shall prove two more lemmas.
\begin{lem}\label{c_egamma}
$\dim\, \mathfrak{c}_e^\Gamma = \overline{s}(\lambda)$.
\end{lem}
\begin{proof}
Retain the notation of \ref{gammaexplicit}, identify $\tK_e/\tK_e^\circ$ with $\widetilde{\Gamma}$ and identify $\Gamma$ with the subgroup
of $\widetilde{\Gamma}$ described in Lemma~\ref{gammaexplicit}.

The space $\HH_0$ was defined in \ref{H_0inst} to be the span of all $\zeta_i^{i,\lambda_i-2m}$ with $1\leq i \leq n$ and $0 \leq m <\lfloor\frac{\lambda_i}{2} \rfloor$.
By Lemma~\ref{subbasis} these elements preserve every $V[i]$ and so are fixed by the $g_i$. It follows that $\Gamma$ fixes
$\HH_0$. Let $\overline{\HH}_0$ denote the image of $\HH_0$ in $\mathfrak{c}_e=\tg_e/[\tg_e \tg_e]$. By the inclusion $\Gamma \subseteq
\widetilde{\Gamma}$ we deduce that $\overline{\HH}_0\subseteq \mathfrak{c}_e^\Gamma$. In view of Corollary~\ref{expression} this yields
$$\dim\, \mathfrak{c}_e^\Gamma \geq \dim\, \HH_0/\HH_0^+=s(\lambda).$$

The proof of Corollary~\ref{expression} also shows that the images of $\zeta_i^{i+1,\lambda_{i+1}-1}$ with $i\in\Delta(\lambda)$ in the
quotient space $\overline{\mathfrak{c}}_e:= \mathfrak{c}_e/\overline{\HH}_0$ form a $\K$-basis of $\overline{\mathfrak{c}}_e$.
We recall for the reader's convenience that $\Delta(\lambda)$ is the set of 2-steps for $\lambda$: the indexes $1\leq i < n$ such that
$\lambda_{i-1} \neq \lambda_i \geq \lambda_{i+1} \neq \lambda_{i+2}$ and $i, i+1$ are both fixed by $j \mapsto j'$. Note that $g_{i+1} \in \widetilde{\Gamma}$ for every 2-step
$(i,i+1)$ of $\lambda$ and, moreover, $g_i, g_{i+1}$ both lie in $\widetilde{\Gamma}$ if $(i,i+1)$ is a 2-step of $\lambda$ such that $\lambda_i \neq \lambda_{i+1}$.
If we choose $j\in \Delta(\lambda)$ then direct computations shows us that

\begin{equation}\label{zeta}
g_i\big(\zeta_j^{j+1,\lambda_{j+1}-1}\big)=\begin{cases}\ \, \zeta_j^{j+1,\lambda_{j+1}-1}\ \quad \mbox{ if } j\notin\{i-1,i\},
\\-\zeta_j^{j+1,\lambda_{j+1}-1} \ \,\ \,\mbox{ if } j\in\{i-1,i\}.\end{cases}
\end{equation}

Now we must consider the type of $K$. Suppose $K$ is of type $\sf C$. By Lemma~\ref{gammaexplicit}, $\widetilde{\Gamma} = \Gamma$. By the above calculation
we see that if $0 \neq \alpha \in \overline{\mathfrak{c}}_e$ then $\alpha$ cannot be fixed by all $g_i$, therefore $\overline{\mathfrak{c}}_e^\Gamma = 0$ and
$\dim\, \mathfrak{c}_e^\Gamma = s(\lambda)$ as promised.

To complete the proof we suppose that $K$ is of orthogonal type and show that if $\Gamma$ fixes a non-zero vector in $\overline{\mathfrak{c}}_e$
then $\lambda$ is exceptional, and the fixed point space is one dimensional. If the type of $\h$ is  $\sf B$ or $\sf C$ then $\Gamma = \langle g_ig_j : i,j
\in \overline{I}\rangle \subseteq \widetilde{\Gamma}$ (Lemma~\ref{gammaexplicit}). Suppose that $\Gamma$ has a non-trivial fixed point in $\overline{\mathfrak{c}}_e$. Let $\bar{\zeta}_i^{i+1,\lambda_{i+1}-1}$
denote the image of $\zeta_i^{i+1,\lambda_{i+1}-1}$ in $\overline{\mathfrak c}_e$ and set $$\alpha:=\sum_{k\in\Delta(\lambda)}\,
a_k\bar{\zeta}_k^{k+1,\lambda_{k+1}-1}$$ where $a_k\in\K$ are constants. Fix $k\in \Delta(\lambda)$ such that $a_k \neq 0$ and assume $\Gamma$ fixes $\alpha$.
It follows that $\Gamma$ also fixes $\zeta_k^{k,\lambda_k-1}$. First of all we suppose that $k' \in \Delta(\lambda)$ where $k' +1 < k$. Then
 it follows from (\ref{zeta}) that $g_{k+1}g_{k'+1}(\zeta_k^{k,\lambda_k-1}) \ne\zeta_k^{k,\lambda_k-1}$, a contradiction. Similar
 reasoning shows that all $k' \in \Delta(\lambda)$ fulfil $k' \leq k$. This tells us that $\Delta(\lambda) \subseteq \{k-1, k\}$.
Suppose $k-1 \in \Delta(\lambda)$. Then it follows that $\lambda_{k-1} > \lambda_{k}$, and that $k-1 = (k-1)'$.
But now $g_{k-1} g_{k+1}(\zeta_k^{k,\lambda_k-1}) \neq\zeta_k^{k,\lambda_k-1}$ by (\ref{zeta}).
We conclude that $\Delta(\lambda) = \{k\}$. It follows that $\alpha = a_k \zeta_k^{k+1,\lambda_{k+1}-1}$. It is clear that if any index $i \notin \{k, k+1\}$ fulfils $i = i'$ then (choosing the greatest such index)
we have $g_k g_i \zeta_k^{k,\lambda_k-1} \neq \zeta_k^{k,\lambda_k-1}$. Therefore $\lambda$ is an exceptional partition. In this case we see that
$\zeta_k^{k,\lambda_k-1}$ is fixed by $\Gamma$ and since $\overline{\mathfrak{c}}_e = \K \zeta_k^{k,\lambda_k-1}$ we have the desired result.
\end{proof}

\p We temporarily say that an admissible sequence $\ii = (i_1,...,i_l)$ for $\lambda$ is \emph{of type 1} if Case 1 occurs at index $i_k$ for $\lambda^{\ii_k}$ for every $k=1,...,l$. 
\begin{lem}\label{specialparabolic}
Suppose that $\Oo_e$ has partition $\lambda$ and that $\ii$ is a restricted admissible sequence for $\lambda$ of type 1. Then there exists a parabolic subgroup $P$
such that $K_e \subseteq P$ with the following properties. There is a Levi decomposition $P = U \rtimes L$ with $\widetilde{\Gamma} \cap SL(V) \subseteq L$.
The Lie algebra $\li = \Lie(L)$ lies in the fibre of $\pi$ above $\ii$ so that $\li^\ii \cong \gl_\ii \times \mm$ (see \ref{leviclassification}). Finally there exists an orbit $\Oo \subseteq \li$
such that $\Oo_e = \Ind_\li^\h(\Oo)$ where $\Oo = \Oo_0 \times \Oo_{e_0}$ and $e_0 \in \li$ has partition $\lambda^\ii$.
\end{lem}
\begin{proof}
In this proof we shall assume that $|\ii| = 1$ and, by inspection, the parabolic defined upon iterating this construction shall contain $K_e$, by the transitivity of
induction the orbit will induce to $\Oo_e$, and by the transitivity of the algorithm the partition shall be $\lambda^\ii$. Since Case 1 occurs for $\lambda$ at $i$
we get $\lambda_i-\lambda_{i+1}\geq 2$.

For  each $j\in\{1,\ldots, i\}$ we denote by $V'[j]$ the linear span of all $e^s w_j$ with $0< s<\lambda_j-1$
and set $$V':=\Big(\textstyle{\bigoplus}_{j=1}^k\,V'[j]\Big) \bigoplus\Big(\textstyle{\bigoplus}_{j>k}\,V[j]\Big),$$ a non-degenerate subspace of $V$ with respect
to the form $(\cdot, \cdot)$. Let $W^+$ be the span of $\{e^{\lambda_j-1} w_j : 1\leq j \leq i\}$ and $W^-$ be the span of $\{w_j : 1\leq j \leq i\}$ so that
$W = W^+ \oplus W^-$ is a complement to $V'$ in $V$. Notice that $\dim \, V' = N - 2i$. Since $\ii$ is restricted, this dimension is not equal to 2 in type $\sf D$
and the annihilator of $W$ is a simple Lie algebra $\mm \subseteq \h$ of the same type as $\h$. Meanwhile, the annihilator of $V'$ is isomorphic to $\gl_i$ and
the internal direct sum $\li = \gl_i \oplus \mm$ is a Levi subalgebra lying in the fibre of $\pi$ above $\ii$. Here we use that $\ii$ is restricted.
Set $L$ equal to the Levi subgroup of $K$ with $\li = \Lie(L)$.

The inclusions $$W^+ \subseteq V' \oplus W^+ \subseteq V$$ form a partial flag and there is a parabolic group $P \subseteq K$ stabilising it.
Since $L$ is a maximal Levi subgroup contained in $P$ it is a Levi factor. Let $\mathfrak{u}$ denote the Lie algebra of the unipotent radical of $P$. Then $\mathfrak{u}$ is precisely the
subspace of $\h$ which maps $V$ into $V' \oplus W^+$ and maps $V'\oplus W^+$ into $W^+$.

Let $e_0$ be the nilpotent element of $\mm$ which sends $e^s w_j$ to $e^{s+1} w_j$ for $1\leq j \leq i$ and $0 < s < \lambda_j - 2$, sends $e^{\lambda_j-2}w_j$ to
zero for $1\leq j \leq i$, and acts as $e$ on all $V[j]$ with $j > i$. Set $\Oo = \Ad(L) e_0$. It is not hard to see that the partition of $e$ is $\lambda^{(i)}$.
Since $e - e_0$ lies in $\mathfrak{u}$ we see that $e \in \Oo + \mathfrak{u}$ so that $\Oo_e$ lies in the closure of $\Ind_\li^\h(\Oo)$. However, since the partition of $e_0$ is obtained from $\lambda$
by an iteration of the algorithm at index $i$, the partitions of $\Oo_e$ and $\Ind_\li^\h(\Oo)$ coincide (Corollary~\ref{case1or2}) and it must be that these two orbits are equal
by dimension considerations.

Using the description of $\h_e$ in \ref{basisforthecent} it is easy to see that $\h_e\subset\Lie(P)$ which, in turn, implies that $K_e^\circ\subset P$.
The set $SL(V)\cap \widetilde{\Gamma}$ described in \ref{gammaexplicit} is contained in $L \subseteq P$. Since $K_e = (SL(V)\cap\widetilde{\Gamma}) \cdot K_e^\circ$
we get $K_e \subseteq P$ as required. It is clear that we may repeat this construction any number of times and still conclude that
$K_e \subseteq P$ and $\Oo_e = \Ind_\li^\h(\Oo)$ where $\Oo = \Oo_0 \times \Oo_{\lambda^\ii}$.
\end{proof}

\p We call a nilpotent element $e\in \h$ {\it almost rigid} if the partition $\lambda \in\mathcal{P}_\epsilon(N)$ of $\Oo_e$
fulfils $\lambda_i-\lambda_{i+1}\in\{0,1\}$ for all $i>0$. Recall that this is the first of the two criteria given in Theorem~\ref{rigids} for a partition to be rigid.
Since any such partition has no bad 2-steps every almost rigid nilpotent elements of $\h$ is non-singular.
 \begin{thm}\label{class1}
 If $e$ has partition $\lambda$ then $U(\h,e)^\ab_\Gamma$ is a polynomial algebra in $\overline{s}(\lambda)$ variables.
\end{thm}
\begin{proof}
We shall use notation $\Gamma := \Gamma(e)$ and $\EE := \EE(\h,e)$. Let $\ii = (i_1,...,i_l)$ be the admissible sequence for $\lambda$ of type 1 which is
formed by applying Case 1 repeatedly until $s(\lambda^\ii) = 0$, so that $\lambda^\ii$ is almost rigid. We must distinguish the cases where $\lambda^\ii$ is
exceptional or not exceptional.

If $\lambda^\ii$ is exceptional then, since we only applied Case 1 of the KS algorithm to $\lambda$ we see that $\lambda$ must also be exceptional.
But then $e$ is non-singular and so by Corollary~\ref{izosim} and Corollary~\ref{newz} we see that $c(e) = r(e)$. Now we make invoke part 1 of
Theorem~\ref{poly} to see that $U(\h,e)^\ab_\Gamma$ is polynomial in $\dim\, \mathfrak{c}_e^\Gamma$ variables. By Lemma~\ref{c_egamma} that is $\overline{s}(\lambda)$ variables.

Now suppose that $\lambda^\ii$ is not exceptional. We claim that $\ii$ is restricted. Indeed, if it is not restricted then the type of $\h$ is $\sf D$ and $\lambda^\ii = (1,1)$. But
now it is clear that $\lambda$ is exceptional, contrary to our assumptions. Hence we may apply the previous lemma to deduce that there is a parabolic subgroup $P$ containing $K_e$
such that $P = U \rtimes L$ and $\Lie(L) = \li \cong \gl_\ii\times \mm$ with a nilpotent element $e_0 \in \mm$ such that $\Oo = \Ad(L)e_0$ induces to
$\Oo_e$, and $e_0$ has partition $\lambda^\ii$. It is immediate that $\dim\, \z(\li) = |\ii| = s(\lambda)$. Since Case 1 does not occur for $\lambda^\ii$ at any
index, $e_0$ is almost rigid.

Let $M$ be simple subgroup of $K$ with $\Lie(M)=\mm$ and denote by $\Gamma_0$ the component group $M_{e_0}/M_{e_0}^\circ$, and make the
notation $\mathfrak{c}_{e_0} = \mm_{e_0}/[\mm_{e_0} \mm_{e_0}]$. Then combining Corollary~\ref{izosim} with
Corollary~\ref{newz} and Corollary~\ref{poly}(i) we  deduce that $U(\mm, e_0)^\ab_{\Gamma_0}\cong S\big(\mathfrak{c}_{e_0}^{\Gamma_0}\big)$.
Since in the present case $\dim \, \mathfrak{c}_{e_0}^{\Gamma_0} = 0$ by Lemma~\ref{c_egamma} we conclude that $\EE(\mm, e_0)^{\Gamma_0}$
is a single point. In particular it is non-empty. Note that $U([\li \li], e_0) \cong U(\sl_\ii)\otimes U(\mm, e_0)$ as $\K$-algebras and both tensor factors
are stable under the natural action of the reductive part of $L_{e_0}$ on $U([\li \li], e_0)$. Since $U([\li \li], e_0)^\ab \cong U(\mm, e_0)^\ab$
we may deduce that $\EE([\li \li], e_0)^{L_{e_0}/L_{e_0}^\circ} \neq \emptyset$ and so Proposition~\ref{mult1} yields that
$\mathcal{E}^\Gamma\neq\emptyset$ and $\dim\,\mathcal{E}^\Gamma\geq \dim\,\z(\li)$. On the other hand, $\dim\,\z(\li) = s(\lambda) = \dim\,\mathfrak{c}_e^\Gamma$
by Lemma~\ref{c_egamma}. It follows by part 2 of Corollary~\ref{poly} that $U(\h, e)^\ab_\Gamma\cong S(\mathfrak{c}_e^\Gamma)$ is polynomial in $s(\lambda)$ variables.
\end{proof}

\section{Completely prime primitive ideals of $U(\h)$}\label{CPPideals}
\setcounter{parno}{0}

\p Let $e$ be an induced nilpotent element of $\h$. Let $\Prim U(\h)$ be the set of all primitive ideals
of the universal enveloping algebra $U(\h)$ and let $\Primo_\Oo$ be the set of those $I \in \Prim U(\h)$ with $\VA(I)=\overline{\Oo}$.

\p The traditional way to classify the completely prime primitive ideals $I\in\Primo_\Oo$ parallels Borho's classification of the sheets. One aims to show that if the orbit $\Oo$ is induced
from a rigid orbit $\Oo_0$ in a Levi subalgebra $\li$ of $\h$, then
the majority of $I$ as above can be obtained as the annihilators in $U(\h)$ of certain induced $\h$-modules
$$\Ind_\pp^{\h}(E):=U(\h)\otimes_{U(\pp)}E,$$ where $\pp = \li\oplus\n$ is a parabolic subalgebra of $\h$ with nilradical $\n$
and $E$ is an irreducible $\pp$-module acted on trivially by $\n$ and such that the annihilator $I_0:=\Ann_{U(\li)}\, E$ is a completely prime primitive ideal of $U(\li)$
with $\VA(I_0)= \overline{\Oo}_0$. It should be noted that the induced module does not necessarily have to be irreducible and $\Ind_\li^\h(I_0)$ does not have to
be primitive and completely prime, in general, but this holds under the additional assumption that $I_0$ is completely prime thanks to Theorem~\ref{inducedideals}.

\p The multiplicity $\mult_\Oo U(\h)/I$ of a primitive quotient with $I \in \Primo_\Oo$ is defined in \ref{multiplicity}. The results of the previous section
can be applied to describe how primitive ideals $I\in\Primo_{\Oo}$ for which ${\rm mult}_{\Oo}(U(\tg)/I)=1$ may be induced.
The characterisation we obtain can be regarded as a generalisation of M{\oe}glin's theorem \cite{Moe} on completely prime primitive ideals of $U(\mathfrak{sl}_n)$
to simple Lie algebras of other types (that theorem was recently reproved by Brundan \cite{Br} by using the theory of finite $W$-algebras).

\p The rest of this section is devoted to proving the following:
\begin{thm}\label{E} Let $I\in\MF_{\Oo}$ be a multiplicity-free primitive ideal where $\Oo$ is an induced orbit. Then there exists a proper parabolic subalgebra $\mathfrak p$ of $\g$ with a Levi decomposition $\pp = \li \ltimes \n$, a rigid nilpotent orbit $\Oo_{e_0}$ in $\li$ and a completely prime primitive ideal $I_0 \in \Primo_{\Oo_{e_0}}$ such that the following hold:
\begin{enumerate}
\item{$\VA(I_0) = \overline{\Oo}_{e_0}$;}
\item{$\Oo_e = \Ind^\h_\li(\Oo_{e_0})$;}
\item{$I = \Ind^\h_\li(I_0)$.}
\end{enumerate}
\end{thm}

\p Recall from \ref{losevtheorem} that whenever $\Oo_e = \Ind_\li^\h(\Oo_{e_0})$, Losev's embedding $\Xi : U(\h,e) \hookrightarrow U(\li, e_0)'$ actually induces an embedding
$\xi_\li : U(\h,e)^\ab \hookrightarrow U(\li,e)^\ab$. We study the comorphism $$\xi^\ast_\li : \EE(\li, e_0) \ra \EE(\h, e).$$
Our proof of the above theorem shall hinge upon the following proposition.
\begin{prop}
For every induced nilpotent element $\Oo_e \in \Ni(\h)/K$ there exists a Levi subalgebra $\li$ and a rigid nilpotent element $e_0 \in \Ni([\li \li])$ such that:
\begin{enumerate}
\item{$\Oo_e = \Ind_\li^\h(\Oo_{e_0})$;}
\item{$\EE(\h,e)^\Gamma \subseteq \xi^\ast_\li \EE(\li, e_0)$.}
\end{enumerate}
\end{prop}
\begin{proof}
Suppose that $\Oo_e$ has partition $\lambda$. First of all we consider the case where $\lambda$ is non-singular. By Theorem~\ref{nosheets} our element $e$ is contained
in a unique sheet and by Proposition~\ref{inducedprops} there is a unique $K$-pair $(\li, \Oo)/K$ with $\Oo \subseteq \li$ rigid and $\Oo_e = \Ind_\li^\h(\Oo)$. Fix a Levi $\li$ from this
$K$-pair, and choose any $e_0 \in \Oo$. By the theory developed in~\ref{conjlev} and \ref{KScontext} we know that the conjugacy class of $\li$ lies in the fibre of $\pi$ above some
admissible sequence $\ii$ for $\lambda$, and so $\li \cong \li^\ii \cong \gl_\ii \times \mm$, with $e_0 \in \mm$ having partition $\lambda^\ii$.
Now $U(\li, e_0)^\ab \cong U(\gl_\ii, 0)^\ab \otimes U(\mm, e_0)^\ab \cong S(\z(\li))\otimes U(\mm, e_0)^\ab$ as algebras.
Since $e_0$ is rigid we deduce that $U(\mm, e_0)^\ab = \K$ by Corollary~\ref{uniquerigids}.
It follows that $U(\li, e_0)^\ab$ is a polynomial algebra in $\dim \, \z(\li)$ variables. Since $\lambda$ is non-singular we may apply Corollary~\ref{izosim} and \ref{newz} to deduce that
$\dim \, \z(\li) = z(\lambda) = c(\lambda) = c(e)$. By Theorem~\ref{class} the algebra $U(\h,e)^\ab$ is also polynomial in $c(e)$ variables. Now the finite morphism
$\xi^\ast_\li : \EE(\li, e_0) \rightarrow \EE(\h, e)$ is forced to be surjective (bear in mind that $\EE(\h,e)$ is irreducible whilst $\xi^\ast_\li$ is closed and has finite fibres). Therefore
$$\EE(\h,e)^\Gamma \subseteq \EE(\h,e) = \xi^\ast_\li \EE(\li, e_0).$$

Now turn our attention to the case where $\lambda$ is singular. We shall reduce to the non-singular case as follows. Let $\ii$ be the admissible sequence of type 1 for $\lambda$
which is obtained by applying Case 1 of the KS algorithm to $\lambda$ as many times as possible, so that $s(\lambda^\ii) = 0$. By Lemma~\ref{specialparabolic} there is
a special parabolic $P$ containing $K_e$, where the conjugacy class of the Lie algebra of a Levi factor lies in the fibre above $\ii$. Then
$\li \cong \gl_\ii \times \mm$ and there exists a nilpotent element $e_0 \in \mm$ such that $\Oo_e = \Ind^\h_\li(\Oo_{e_0})$.
Set $\Gamma_0 = L_{e_0} / L_{e_0}^\circ$. After the discussion in parts (b) of the proof of Proposition~\ref{mult1} the maps which are used in Losev's construction of $\xi^\ast_\li$ are
$K_\phi$-equivariant and it follows that $\xi_\li^\ast \EE(\li, e_0)^{\Gamma_0} \subseteq \EE(\h,e)^\Gamma$. Since $\lambda$ is singular it is not exceptional and it follows that
$\lambda^\ii$ is not exceptional either. Now $U(\li, e_0)^\ab = S(\z(\li)) \otimes U(\mm, e_0)^\ab$ and so $\dim \, \EE(\li, e_0)^{\Gamma_0} = \dim\, \z(\li) + s(\lambda^\ii) = \dim \, \z(\li)$ by Theorem~\ref{class1}. 
But $\dim \, \z(\li) = s(\lambda)$ by construction, and this coincides with $\dim\, \EE(\h,e)^\Gamma$ by the same theorem. We deduce that $$\xi_\li^\ast : \EE(\li, e_0)^{\Gamma_0} \longrightarrow \EE(\h,e)^\Gamma$$
is a finite morphism between irreducible varieties of the same dimension. In particular, it is surjective.

Finally we may apply the construction from the first paragraph of the proof to $\mm$ and we will obtain a Levi subalgebra and nilpotent orbit with the correct properties.
\end{proof}

\p We can now supply a proof for the Theorem~\ref{E}, the last theorem of the thesis.
\begin{proof}
Let $(\li, e_0)$ be the pair constructed in the previous proposition and pick $I \in \MF_{\Oo_e}$. According to \ref{bigtheoremF} there is a unique $\eta \in \EE(\h,e)^\Gamma$ with
$I = I_{\FF(\eta)}$. Thanks to part 2 of the previous proposition we can find $\eta_0 \in \EE(\li, e_0)$ such that $\xi_\li^\ast \eta_0 = \eta$. Let $I_0 := I_{\FF_0(\eta_0)} \unlhd U(\li)$, where
$\FF_0 : U(\li,e_0)\text{-mod} \rightarrow U(\li)\text{-mod}\chi_0$ is the Skryabin functor for $U(\li,e_0)$. According to part 3 of Theorem~\ref{bigtheoremF} this ideal is completely prime,
whilst part 1 tells us that the associated variety is $\overline{\Oo}_{e_0}$. Thanks to \cite[Corollary~6.4.2]{Lo3}, we have the following commutative diagram:
\begin{center}
\begin{tikzpicture}[node distance=4cm, auto]
 \node (ug) {$\EE(\h,e)$};
 \node (prim1) [right of=ug] {$\Prim_{\Oo_e}$};
 \node (ul) [below of=ug] {$\EE(\li, e_0)$};
 \node (prim2) [below of=prim1] {$\Prim_{\Oo_{e_0}}$};
 \draw[->] (ug) to node {$I_{\FF(\text{-})}$} (prim1);
 \draw[->] (ul) to node {$\xi^\ast_\li(\text{-})$} (ug);
 \draw[->] (ul) to node [swap] {$I_{\FF_0(\text{-})}$} (prim2);
 \draw[->] (prim2) to node [swap] {$\Ind_\li^\h(\text{-})$} (prim1);
\end{tikzpicture}
\end{center}
It follows that $I$ is obtained from $I_0$ by parabolic induction.



\end{proof}


\end{document}